%% file: submission.tex
\documentclass{amsart}

\input{arxivheader}

\usepackage{pgfplots} 
\newtheorem{theorem}{Theorem}[section]
\newtheorem{lemma}[theorem]{Lemma}

\theoremstyle{definition}
\newtheorem{definition}[theorem]{Definition}
\newtheorem{example}[theorem]{Example}
\newtheorem{proposition}[theorem]{Proposition}
\newtheorem{corollary}[theorem]{Corollary}

\theoremstyle{remark}
\newtheorem{remark}[theorem]{Remark}
\numberwithin{equation}{section}

\lstdefinelanguage{Maple}%
{morekeywords={with, assume, Multiply, int, Matrix, log, arctan, sqrt, simplify}, sensitive=true,%
morecomment=[l]\#
}[keywords,comments]%
\usepackage{accsupp}
\newcommand{\noncopynumber}[1]{%
\BeginAccSupp{method=escape,ActualText={}}%
#1%
\EndAccSupp{}%
}
\begin{document}

\title{Discrete Geodesic Calculus in the Space of Sobolev Curves}

\author{Sascha Beutler}
\address{Sascha Beutler, Institute for Computational and Applied Mathematics, University of Münster, Einsteinstr. 62, 48149 Münster}
\curraddr{}
\email{}
\thanks{}

\author{Florine Hartwig}
\address{Florine Hartwig, Institute for Numerical Simulation, University of Bonn, Endenicher Allee 60, 53115 Bonn}
\curraddr{}
\email{}
\thanks{}

\author{Martin Rumpf}
\address{Martin Rumpf, Institute for Numerical Simulation, University of Bonn, Endenicher Allee 60, 53115 Bonn}
\curraddr{}
\email{}
\thanks{}
\author{Benedikt Wirth}
\address{Benedikt Wirth, Institute for Computational and Applied Mathematics, University of Münster, Einsteinstr. 62, 48149 Münster}
\curraddr{}
\email{}
\thanks{}

\date{}

\dedicatory{}

\begin{abstract}
The Riemannian manifold of curves with a Sobolev metric is an important and frequently studied model in the theory of shape spaces. 
	Various numerical approaches have been proposed to compute geodesics, but so far elude a rigorous convergence theory.
	By a slick modification of a temporal Galerkin discretization we manage to preserve coercivity and compactness properties of the continuous model
	and thereby are able to prove convergence for the geodesic boundary value problem.
	Likewise, for the numerical analysis of the geodesic initial value problem
	we are able to exploit the geodesic completeness of the underlying continuous model
	for the error control of a time-stepping approximation.
	In fact, we develop a convergent discretization of a comprehensive Riemannian calculus
	that in addition includes parallel transport, covariant differentiation, the Riemann curvature tensor, and sectional curvature,
	all important tools to explore the geometry of the space of curves.	
	Selected numerical examples confirm the theoretical findings and show the 
	qualitative behaviour. To this end,  a low-dimensional submanifold of
	Sobolev curves with explicit formulas for ground truth covariant derivatives and curvatures
	are considered.
\end{abstract}
\maketitle

\section{Introduction}
The manifold of closed, parametric immersed 
curves can be equipped with a Sobolev metric of order $m\in \N$. 
This Riemannian metric represents the Sobolev inner product on the
space of $2\pi$ periodic vector-valued functions of the real line,
but with derivatives computed for an arc-length parameterization. 
The second-order Sobolev metric is particularly interesting because this is the lowest order for which the corresponding
space of curves is known to be geodesically and metrically complete. 

These spaces of curves have been studied extensively in the literature.
For an earlier overview on spaces of curves using the Hamiltonian approach, see \cite{MiMu07}.
A more recent overview of the theory of the Riemannian manifold of parametrized and non-parametrized curves 
equipped with a Sobolev metric and its various extensions is given in \cite{bauer2021intrinsic}.

We derive a consistent discretization of the Riemannian calculus on this space of curves. The proposed discretization leads to a fully practical scheme allowing the computation of discrete geodesics and discrete geometric operators such as the exponential map, covariant derivatives, parallel transport, and the curvature tensor for Sobolev metrics of order $m \geq 2$.

Our results build on the general concept of the variational time discretization of Riemannian calculus from \cite{RuWi15} and the discrete covariant differentiation as well as the curvature tensor approximation from \cite{EfHeRu22}. This existing work applies to a general class of Riemannian manifolds modelled over a Hilbert space, but comes with the restriction of
the Riemannian metric being smooth in its footpoint \emph{with respect to a weaker normed space}, a Banach space into which the Hilbert space compactly embeds. In the context of Sobolev curves, this condition is not satisfied,
so one has to work with the same Sobolev Hilbert space throughout.
The associated difficulties are reflected by the fact that in the continuous setting as well as in the design of an appropriate discretization, one has to ensure certain coercivity properties of the Riemannian path energy,
which in essence translates as the avoidance of degenerate curve parameterizations (that are not immersions).

\subsection{Outline}
In \cref{sec:background} we introduce the space of Sobolev immersions of $\Sone$ and fix notation. \Cref{sec:review} contains a collection of existing results for the analytical treatment of this space of curves based on \cite{Br15}. \Cref{subsec:existenceGeodesics} recapitulates the existence of shortest geodesics, where compared to \cite{Br15} we avoid using fractional Sobolev spaces and treat the case of reparameterization invariance differently. The proof later serves as a foundation to establish the existence results in the time-discrete setting.
In \cref{sec:timeDiscrete} we introduce the variational time discretization and the space discretization and show the existence of minimizers and Mosco convergence of the discrete to the continuous path energy, which implies convergence of discrete geodesics to continuous
ones. In fact, we propose and analyse two different discrete approximations of the Riemannian metric (resulting in two different time discretizations), both with their distinct advantages.
The space discretization is based on a truncated Fourier series.
In \cref{sec:discreteGeodCalc}, finally, we introduce discrete approximations of geometric operators and derive their convergence rates,
in particular, a discrete exponential map in \cref{sec:expMap} with discrete shot geodesics coinciding with discrete variational geodesics,
a discrete covariant derivative computed via covariant difference quotients and discrete parallel transport in \cref{sec:PTCovDeriv},
and discrete approximations of the Riemann curvature tensor in \cref{sec:consistencyCurvature}.
The technically involved proofs of the well-posedness and convergence of the discrete exponential map, covariant derivative, parallel transport, and Riemann curvature approximation are collected in \cref{sec:appendix},
which also contains explicit calculations of covariant derivatives and curvature at a perfectly circular curve, to be used for numerical verification of the obtained convergence rates.

\subsection{Related Work} \label{sec:related}
The space of planar Sobolev immersions of order $2$ and higher is shown in \cite{BrMiMu14} 
to be geodesically complete, \ie the initial value problem of the geodesic flow is
well-posed and long-time existence is ensured.
Furthermore, lower bounds for the geodesic distance depending on the curvature and its derivatives are given.
Sobolev metrics with non-constant coefficients are considered in \cite{MeYeSu08}.
The existence of minimizing geodesics in higher space dimensions is proved in \cite{Br15} 
in the space of pa\-ra\-metrized curves with non-degenerate length element and
Sobolev metric of order at least $2$. Here, the verification
of the weak convergence of arclength derivatives requires a compensated compactness argument.
In addition, it is shown that the space of Sobolev immersions is metrically and geodesically complete. 
Furthermore,  the space of unparametrized curves, identified with the quotient space of Sobolev
immersions under diffeomorphic reparameterizations, is shown to be a geodesic space.
In \cite{bauer2018fractional} a 
fractional Sobolev metric of non-integer order $s>\tfrac32$ is discussed and the existence of geodesic paths is demonstrated.
Furthermore, the connection to right-invariant metrics on the associated diffeomorphism group is explored.
In \cite{bauer2023sobolev} 
completeness properties of the space of non-parametrized immersed curves in a complete, smooth, bounded manifold are discussed,
following the outline given in \cite{Br15} for flat image manifold
metric, and geodesic completeness is shown for a broad range of classes of Sobolev metrics with order at least $2$.
Incompleteness due to collapse is observed in the case of the space of open curves.
The space of planar curves equipped with a $BV^2$-Finsler metric is investigated in \cite{nardi2016geodesics}, where
the $BV^2$-norm is defined as the sum over the $W^{1,1}$-norm and the total variation of the second variation considered as a Radon measure
on $S^1$. A finite norm is compatible with curves that are piecewise smooth with jumping velocities.
The existence of shortest paths is shown, and a gradient descent method for a finite element discretization is used to
compute approximations of geodesic paths.
The difference to the Sobolev metric of order $2$ is shown based on selected numerical examples.

A Riemannian distance in the space of curves can also be based on the square root velocity approach.
To this end one defines for a curve $c$ the square root velocity function $q[c]=\tfrac{c'}{\sqrt{|c'|}}$; the distance between two curves
$b$ and $c$ is then defined as $\left\Vert\tfrac{b'}{\sqrt{|b'|}}-\tfrac{c'}{\sqrt{|c'|}}\right\Vert_{L^2}$.
In this context, the quotient space of unparametrized, absolutely continuous curves is studied analytically and topologically.
In \cite{bruveris2016optimal} it is shown that the square root velocity transform 
mapping $c$ to $q[c]$ is a homeomorphism and that the reparameterization 
group action is continuous. Furthermore, the existence of an optimal reparameterization realizing the minimal distance is 
demonstrated.
The square root velocity transform is generalized
in \cite{needham2020simplifying}, and the proposed modified transform allows for computing any first-order Sobolev metric controlled by two constant
parameters as the pullback of the $L^2$ metric under this transform.
A further extension allows to apply the transform in the case of piecewise affine curves.
A slightly more general class of transforms compared to the square root velocity transform
is discussed in \cite{sukurdeep2019inexact}.
In \cite{bauer2019inexact}, the square root normal fields for the efficient computation of elastic distances between parametrized curves
is combined with the concept of varifold distances to relax the terminal constraint in an elastic matching approach.

Reparameterization invariant Sobolev metrics on spaces of regular curves are introduced and motivated by applications in
shape analysis in \cite{bauer2016use}. On this background, the authors advocate the development of robust
numerical methods for higher-order Sobolev metrics.
A tensor product B-spline discretization in space and time with arc length pa\-ra\-me\-trization in space is used in \cite{bauer2015second} to compute solutions of the initial and boundary value problem of 
geodesics in the space of Sobolev immersions.
Furthermore, it is applied to compute the Fr\'echet mean and a principal component analysis on ensembles of curves.
In \cite{BaBrHa17}, a class of algorithms is proposed to numerically solve the
geodesic initial and boundary value problem on the space of Sobolev immersions using
tensor product B-splines. The corresponding approximate Galerkin discretization 
is explained in detail. For the initial value problem, 
the time-discretization approach from \cite{RuWi15} is used.
A particular focus in \cite{bauer2019metric} is on the computation of solutions of the geodesic boundary value problem. 
A Lagrangian formulation of this variational problem is taken into account with the 
end point boundary condition considered as a constraint.
This problem is then relaxed using so-called chordal distances, and the resulting approximation is solved using an optimal control approach.
In \cite{sukurdeep2022new}, an extension of the Sobolev metric to structures composed of curve segments 
and the matching of topologically different structures is discussed. 
The matching is defined variationally, and the existence of optimal matches as
well as their numerical approximation is investigated.
The finite-dimensional space of piecewise affine, reparameterization invariant curves equipped with a discrete Sobolev metric is analyzed
in \cite{cerqueira2024sobolev}.
In particular, metric and geodesic completeness are demonstrated. The computed Gaussian curvature on this space sheds light on the global geometry of the space.

Michor and Mumford \cite{MiMu04} studied the Riemann curvature tensor\index{Riemann curvature tensor}
on the space of smooth planar curves equipped with the metric
$
g_c(v,v) \coloneqq \int_{\Sone} (1+A \kappa_c^2) \vert v \vert^2 \,\d s
$
for $\kappa_c$, the curvature of the curve.
They gave an explicit formula for the sectional curvature and 
could show that for large, smooth curves, with a suitable notion of large and smooth, 
all sectional curvatures turn out to be nonnegative, 
whereas for curves with high frequency characteristics, sectional curvatures are nonpositive.
Younes \etal 
\cite{YoMiSh08} investigated a metric on the space of plane curves derived as the limit case of 
a scale invariant metric of Sobolev order $1$ from \cite{MiMu07} which allows the explicit computation of geodesics
and sectional curvature.
Micheli \etal 
\cite{MiMiMu12} showed how to compute sectional curvature on the space of 
landmarks with a metric induced by the flow of diffeomorphisms.
They evaluated the sectional curvature for pairs of tangent directions to special geodesics along which only two landmarks move.
A formula for the derivatives of the inverse of the metric is at the core of this approach. 
In \cite{MiMiMu13}, the formula for the sectional curvature on the landmark space was generalized to special infinite-dimensional weak Riemannian manifolds.
A brief overview of the geometry of shape spaces with a particular emphasis on sectional curvature in different shape spaces 
can be found in the contribution by Mumford \cite{Mu12}.

\section{Review of the space of Sobolev immersions} \label{sec:background}
We consider closed curves in $\R^d$, $d\geq2$, \ie continuous maps from the unit circle $\Sone\subset\R^2$ into $\R^d$,
$$c:\Sone \to \R^d,\qquad \theta \mapsto c(\theta).$$
Note that $\Sone$ can be identified with the real line modulo
a $2\pi$ periodicity, \ie we identify all points $\theta + 2\pi j$ for $j\in \Z$.
Thus, a curve $c$ defined on $\Sone$ can be considered as a 
map $c:\R \to \R^d$ with $c(\theta)=c(\theta+2\pi)$ for all $\theta\in \R$.

More specifically, we will work with \emph{immersions} of $\Sone$ into $\R^d$ of different regularity.
Let us first introduce the space of smooth immersions
\begin{equation*}
	\immersion\coloneqq\{c \in C^{\infty}(\Sone,\R^d)\,|\,c'(\theta)\neq 0\ \forall \theta \in \Sone\},
\end{equation*}
which is an open subset of the space of smooth functions $C^{\infty}(\Sone,\R^d)$.
Above, differentiation of the curve with respect to the parameter $\theta$ is denoted by $c'=\dtheta c$.
The space of \emph{Sobolev-immersions} is of lower regularity; for $m\geq2$ it is given by
\begin{equation}\label{eq:SobImm}
	\immersion^m \coloneqq \{c \in W^{m,2}(\Sone,\R^d)\,|\, c'(\theta)\neq 0\ \forall \theta \in \Sone\}
	\supset\immersion,
\end{equation}
where $m\geq2$ is required so that $c'$ is defined for all $\theta$ due to Sobolev embedding.
Here, $W^{m,2}(\Sone,\R^n)$ denotes the usual Sobolev space of functions ($n=1$)
or vector fields $u$ on $S^1$ with norm and inner product (we will use the notation $W^m_\theta$ for $W^{m,2}(\Sone,\R^d)$)
\begin{align}
	\Vert u \Vert^2_{W^m_\theta}\coloneqq& \int_{\Sone}|u|^2 + |\dtheta^m u|^2 \,\d \theta, \nonumber\\
	(u,w)_{W^m_\theta}\coloneqq&\int_{\Sone} u \cdot w + \partial^m_\theta u \cdot \partial^m_\theta w \,\d \theta,
\end{align}
where ``$\cdot$'' indicates the Euclidean inner product in $\R^d$.
Due to Poincar\'{e}'s inequality and the observation that $\int_{\Sone} \dtheta^k u \,\d \theta=0$ for $k\geq 1$,
the norm is equivalent to the norm $\sqrt{\int_{\Sone} \sum_{j=0}^m |\partial^j_\theta u|^2 \,\d \theta}$.

Geometrically, it is more natural to consider arclength differentiation and integration along curves $c \in \immersion$.
Defining the arclength function
$$s(\theta) = \int_0^{\theta} |c'(\vartheta)| \,\d \vartheta,$$
we obtain for the associated differentials $\d s = |c'(\theta)| \d \theta$ and correspondingly for differentiation with respect to the arclength parameter $\diffs=|c'(\theta)|^{-1}\dtheta$.
Arclength differentiation of the curve $c: \Sone \to \R^d$ itself leads to the unit speed velocity vector $v=\diffs c = \frac{c'}{|c'|}$ tangent to the curve.
Correspondingly, we  define an arclength Sobolev space $W^m_s$ with norm and inner product
\begin{align*}
	\Vert u \Vert^2_{W^m_s}\coloneqq& \int_{\Sone}|u|^2 + |\partial_s^m u|^2 \,\d s,\\
	(u,w)_{W^m_s}\coloneqq&\int_{\Sone} u \cdot w + \partial^m_s u \cdot \partial^m_s w \,\d s,
\end{align*}
which implicitly depend on the underlying curve $c$.

We will identify $L^2_\theta$, $L^2_s$ with $W^0_\theta$,  $W^0_s$, respectively.

The space of Sobolev immersions $\immersion^m$ is an open subset of $W^{m,2} (\Sone,\R^d)$.
By equipping it with a Riemannian metric, it can be turned into a Riemannian manifold.

\begin{definition}[Sobolev metric on $\immersion^m$]\label{def:SobolevMetricCurves}
	The \emph{Sobolev metric}\index{Sobolev metric} of order $m\geq2$ on $\immersion^m$ is defined
	for a curve $c\in \immersion$ as footpoint and two infinitesimal variations $\xi,\zeta\in W^{m,2}(\Sone,\R^d)$ of $c$ as
	\begin{align*}
		g_c(\xi,\zeta)\coloneqq \int_{\Sone}\sum_{j=0}^m a_j \diffs^j \xi \cdot \diffs ^j \zeta \,\d s
	\end{align*}
	with weights $a_0,a_m>0$, $a_1,\ldots,a_{m-1}\geq0$.
	For ease of presentation, we set $a_0=\ldots=a_m = 1$ unless explicitly stated otherwise.
\end{definition}

\begin{remark}[Sobolev metric is smooth and strong]\label{rem:smoothStrongMetric}
	This metric is clearly well-posed at any $c\in\immersion$,
	but it is even well-posed and smooth on $\immersion^m$
	due to the smoothness of the map $(c,\xi)\mapsto\diffs\xi$ from $\immersion^m\times W^{k,2}(\Sone,\R^d)$ into $W^{k-1,2}(\Sone,\R^d)$ for $1\leq k\leq m$ \cite[Lemma 3.3]{BrMiMu14}.
	Moreover, this map is linear in $\xi$, its kernel being the constant functions.
	As a direct consequence, $g_c$ is coercive (\ie $g_c(\xi,\xi)=\|\xi\|_{W^m_s}^2\geq C\|\xi\|_{W^m_\theta}^2$ for some $C>0$ depending on $c\in\immersion^m$)
	and thus a strong Riemannian metric.
\end{remark}

For $m=2$, we exemplarily expand the metric and obtain
\begin{align}
	g_c(\xi,\zeta)&=\int_{\Sone} \xi \cdot \zeta |c'|+\frac{1}{|c'|}\xi'\cdot \zeta' + \frac{1}{|c'|}\left(\dfrac{\xi'}{|c'|}\right)'\left(\dfrac{\zeta'}{|c'|}\right)'\,\d \theta \nonumber \\
	&=\int_{\Sone} \xi \cdot \zeta |c'|+\frac{1}{|c'|}\xi'\cdot \zeta' + \frac{1}{|c'|^3}\xi'' \cdot \zeta'' + \dfrac{(c'\cdot c'')^2}{|c'|^7}\xi' \cdot \zeta'  \label{eq:SobolevMetricTwo} \\
	& \qquad   - \dfrac{(c'\cdot c'')}{|c'|^5}(\xi'' \cdot \zeta' + \xi' \cdot \zeta'') \,\d \theta ,  \nonumber
\end{align}
using $\diffs^2 \xi=\frac{1}{|c'|}\left(\frac{\xi'}{|c'|}\right)'
=\frac{1}{|c'|^2} \xi'' - \frac{c' \cdot c''}{|c'|^4}\xi'$.
To study the Riemannian manifold $(\immersion^m,g_c)$ it is indispensable to control $|c'|$ pointwise from below and above.

We denote paths of curves $c \in \immersion^m$ by boldface $t \mapsto \cpath(t, \cdot)$. For a path $\cpath$ we define the associated \emph{path energy} \begin{align} \label{eq:pathenergySobolev}
	\pathenergy[\cpath]&=\int_0^1 g_{\cpath(t)} (\dot{\cpath}(t),\dot{\cpath}(t)) \,\d t 	\end{align}
with path velocities $\dot\cpath=\partial_t\cpath\in L^2((0,1),W^{m,2}(\Sone, \R^d))$.
We define geodesics as minimizers of the path energy for fixed end points.
The Riemannian distance is defined as the square root of the infimum of the path energy over 
regular paths connecting $c_1$ and $c_2$, \ie 
\begin{align}
	\dist (c_1,c_2) = \inf  \sqrt{\int_0^1 g_{\cpath(t)} (\dot{\cpath}(t),\dot{\cpath}(t)) \,\d t }
\end{align}
where the minimization is over all $\cpath \in W^{1,2}((0,1),\immersion^m)$ with $\cpath(0,\cdot) = c_1$, $\cpath(1,\cdot)=c_2$.
The Euler-Lagrange equation is given by
\begin{multline*}
	0 = \partial_\cpath \pathenergy[\cpath](\vartheta)
	=\left.\frac{\d}{\d \epsilon} \pathenergy[\cpath + \epsilon \vartheta]\right\vert_{\epsilon=0}
	= \int_0^1 (D_c g_{\cpath(t)})(\vartheta) (\dot \cpath(t),\dot \cpath(t)) + 2 g_{\cpath(t)}(\dot \cpath(t), \dot \vartheta) \,\d t\\
	= \int_0^1 (D_c g_{\cpath(t)})(\vartheta) (\dot \cpath(t),\dot \cpath(t)) - 2 (D_c g_{\cpath(t)})(\dot \cpath(t))(\dot \cpath(t), \vartheta) - 2 g_{\cpath(t)} (\ddot \cpath(t),  \vartheta)\,\d t
\end{multline*}
for all smooth vector fields $\vartheta$ with compact support.
Hence, geodesics also solve
\begin{equation} \label{eq:geodesicIVP}
	0 = -2 \tfrac{\mathrm{d}}{\mathrm{d}t} \left. \left(g_{\cpath(t)} (\dot \cpath(t), \vartheta)\right)\right\vert_{t=t_0} \!\!+ D_c\, g_{\cpath(t_0)} (\vartheta) (\dot \cpath(t_0),\dot \cpath(t_0))
	\qquad\forall\vartheta\in C^\infty(\Sone,\R^d)
\end{equation}
at every $t_0 \in (0,1)$. For details on the geodesic equation in this context, see, \eg, \cite{MiMu07}.
It turns out that $\immersion^m$ is the metric completion of $\immersion$ \cite[Thm.\,4.5]{Br15}.

\subsection{A priori estimates on metric balls} \label{sec:review}
Concerning the analytical treatment of this space of curves, we first collect some useful insights and technical tools mainly given in \cite{Br15,BrMiMu14,MiMu04} that we will exploit for our discrete theory:

We collect a priori estimates for curves and arclength derivatives in a metric ball $$B_r(c_0)=\{c\in\immersion^m\,|\, \dist (c,c_0)<r\}$$
in $\immersion^m$ around some immersion $c_0 \in \immersion^m$.

\begin{proposition}[{A priori estimates \citep[Prop.\,3.1 \& 3.4]{Br15}}] \label{prop:SobCurveUniform}
	Let $g$ be the Sobolev metric of order $m\geq 2$ and $\dist(\cdot,\cdot)$ the induced distance and let  $l_c \coloneqq \int_{\Sone} \,\d s$ denote the length of the curve $c$. Then the maps
	\begin{align*}
		\log |c'|&:(\immersion^m,\dist) \to L^{\infty}(\Sone,\R),\\
		l_c^{\frac12},l_c^{-\frac12} &: (\immersion^m,\dist) \to \R^{+},  \\
		\partial_s^k c&:(\immersion^m,\dist)\rightarrow L^p(\Sone,\R^d),\\
		\partial_s^{k-1}\vert c'\vert&:(\immersion^m,\dist)\rightarrow L^p(\Sone,\R)
	\end{align*}
	with $p=2$ for $k=m$ and $p=\infty$ for $1\leq k\leq m-1$
	are 	Lipschitz continuous on every metric ball $B_r(c_0)\subset\immersion^m$. In particular, $\Vert c'\Vert_{L^\infty}, \Vert |c'|^{-1} \Vert_{L^\infty}, l_c$ are bounded
	and the $L^2(\d \theta)$- and $L^2_s$-norms are uniformly equivalent on every metric ball, \ie $\exists C>0$ st. $\tfrac1C \|w\|_{L^2(\d \theta)} \leq \|w\|_{L^2_s} \leq C  \|w\|_{L^2(\d \theta)}$ for all $c \in B_r(c_0)$ and all $w \in L^2(\d \theta)$.
	Moreover, the following expressions are bounded on every metric ball:
	\begin{align*}
		&\Vert c\Vert_{L^\infty},\ldots,\;\Vert\diffs^{m-1}c\Vert_{L^\infty},\; \Vert\vert c'\vert\Vert_{L^\infty},\ldots,\;\Vert\diffs^{m-2}\vert c'\vert\Vert_{L^\infty}, \\&\Vert\diffs^mc\Vert_{L^2(\d \theta)},\;\Vert\diffs^mc\Vert_{L^2(\d  s)},\; \Vert c\Vert_{W^{m,2}(\d  s)},\; \Vert\diffs^{m-1}\vert c'\vert\Vert_{L^2(\d \theta)},\; \Vert\diffs^{m-1}\vert c'\vert\Vert_{L^2_s}.
	\end{align*}
\end{proposition}

\begin{remark}[Logarithm of length element]
	The boundedness of $\log |c'|$ in $L^\infty$ is the crucial observation. The uniform bounds for the
	length element $|c'|$  from above and below follow immediately from this, and these are the cornerstone of the subsequent analysis.
\end{remark}

\begin{remark}[Extension of estimates to $\immersion^m$]
	In \citep[Prop.\,3.1 \& 3.4]{Br15}, the proposition is actually only stated for metric balls in the space of smooth immersions
	(\ie every occurrence of $\immersion^m$ is replaced with $\immersion$),
	but our version above is an immediate consequence.
	Indeed, by the smoothness of $(c,\xi)\mapsto\diffs\xi$ stated in \cref{rem:smoothStrongMetric} and the smoothness of $\immersion^m\ni c\mapsto|c'|\in W^{m-1,2}(\Sone,\R)$,
	all considered maps are smooth on $\immersion^m$.
	Thus, for any $c_0\in\immersion^m$ and $c_1,c_2\in B_r(c_0)$
	we can find $\tilde c_0,\tilde c_1,\tilde c_2\in\immersion$ arbitrarily close to $c_0,c_1,c_2$,
	with respect to both the metric (since $\immersion^m$ is the metric completion of $\immersion$)
	and the $W^m_\theta$-topology (since the metric is smooth and strong, \cf \cite[VII Prop.\,6.1]{La99}),
	where the latter implies that the values of the considered (smooth) maps at $\tilde c_1,\tilde c_2$ are arbitrarily close to those of $c_1,c_2$.
	Hence, the Lipschitz constants on metric balls of smooth immersions remain valid for metric balls of Sobolev immersions.
	An analogous remark holds for the following result.
\end{remark}

\begin{lemma}[{Coercivity and boundedness of the metric \citep[Prop.\,3.5]{Br15}, \citep[Lemma 5.1]{BrMiMu14}}]\label{lem:coercivityBoundednessMetric}
	For any metric ball $B_r(c_0)\subset\immersion^m$ there exists $C>0$ such that
	for all $c\in B_r(c_0)$ and all $\xi\in W^m_\theta$ it holds
	$$\tfrac1C \Vert \xi\Vert^2_{W^m_\theta}\leq g_{c}(\xi,\xi)\leq C\Vert \xi\Vert^2_{W^m_\theta}.$$
\end{lemma}

\subsection{Existence of shortest geodesics} \label{subsec:existenceGeodesics}
The next lemma provides weak convergence results for arclength derivatives of vector fields.
This allows to show weak continuity of the arc\-length derivatives of the path velocity for paths of bounded path energy in the space of Sobolev immersions.
Below, we use the following shortcuts for Bochner spaces
\begin{gather*}
	W^n_t W^m_\theta = W^{n,2}((0,1),W^{m,2}(\Sone,\R^d)),~ W^1_t\immersion^m=W^{1,2}((0,1),\immersion^m),\\
	C^{n,\alpha}_t W^m_\theta=C^{n,\alpha}([0,1],W^{m,2}(\Sone,\R^d)),~C^{n,\alpha}_tC^k_\theta = C^{n,\alpha}([0,1],C^k(\Sone,\R^d)),
\end{gather*}
where we also write $W^0_t=L^2_t$ or $W^0_\theta=L^2_\theta$.
Further, we indicate by $\partial_ch$ the derivative $\partial_sh$ of some function $h:\Sone\to\R^n$ with respect to arclength of a curve $c\in\immersion^m$.

The following result is given in a slightly more general form in \cite[Lemma 5.9]{Br15} for $m>\frac{3}{2}$, however, we will only need integer orders $m\geq2$ to prove the existence result \cref{thm:SobolevExistence}. For its proof, one can follow \cite[Lemma 5.9]{Br15} using only integer orders.

\begin{lemma}[{Weak convergence of arclength derivatives\index{arclength derivative} \citep[Lemma 5.9]{Br15}}]\label{lem:weakcontSobolev}
	Let $m\geq 2$, $0\leq k \leq m$, and let $\cpath^j,\cpath \in W^1_t\immersion^m$ with $|(\cpath^j)'| \geq \delta >0$ on $\Sone$ and $\cpath^j \rightharpoonup \cpath$ in $W_t^1 W^m_\theta$. If $h^j \rightharpoonup h$ in $L^2_tW^k_\theta$,  then $\partial^k_{\cpath^j} h^j \rightharpoonup \partial^k_\cpath h$ in $L^2_tL^2_\theta$.
\end{lemma}

We next consider the existence of geodesic paths between an element $c_A\in\immersion^m$ and a set $B\subset\immersion^m$.
We denote the set of admissible paths by
\begin{equation}\label{eqn:admissiblePaths}
	\mathcal{P}_{c_A,B}\coloneqq\lbrace \cpath \in W^1_t\immersion^m\,\vert\, \cpath(0,\cdot)=c_A,\, \cpath(1,\cdot)\in B\rbrace,
\end{equation}
in which $\cpath(0,\cdot)$ and $\cpath(1,\cdot)$ represent the (temporal) trace of $\cpath\in W^1_t\immersion^m$.
In the proof, we avoid the use of fractional Sobolev spaces employed in \cite{Br15}.

\begin{theorem}[{Existence\index{existence} of geodesic paths\index{geodesic} \cite[Thm.\,5.2]{Br15}}]\label{thm:SobolevExistence}
	Let $g_c(\cdot,\cdot)$ be the Sobolev metric of order $m\geq 2$, $c_A\in\immersion^m$,  and $B\subset W^m_\theta$ a $W^m_\theta$-weakly sequentially closed set which intersects the connected component of $c_A$ in $\immersion^m$. 	Then there exists a minimizer $\cpath$ of the path energy $\pathenergy$ on $\mathcal{P}_{c_A,B}$.
\end{theorem}
\begin{proof}
	By \cref{rem:smoothStrongMetric} and \cite[VII Prop.\,6.1]{La99} $\dist$ induces the $W^m_\theta$-topology on $\immersion^m$
	so that the connected components of $\immersion^m$ with respect to $\dist$ and the $W^m_\theta$-norm coincide.
	Since $\immersion^m$ is open and locally connected in $W^m_\theta$, its connected components are open. 	Since connected open subsets of normed spaces are polygonally connected 
	and a polygon in $\immersion^m$ has finite path energy,
	any two points from the same connected component can be joined by a path of finite path energy,
	thus there exists a path $\cpath$ in $\mathcal{P}_{c_A,B}$ with $\pathenergy[\cpath]=\bar E<\infty$.
	
	Now consider a minimizing sequence $(\cpath^j)_{j=1,2,\ldots}$ of $\pathenergy[\cdot]$ satisfying $\cpath(0,\cdot)=c_A$ and $\cpath(1,\cdot)\in B$.
	Using \cref{lem:coercivityBoundednessMetric} and $\pathenergy[\cpath^j]\leq\bar E$
	we obtain that $\Vert\dot{\cpath}^j\Vert_{L^2_tW^m_\theta}\leq C$ for some $C>0$.
	This implies the existence of a subsequence, again denoted $(\cpath^j)_{j=1,\ldots}$, which converges weakly in $W_t^1W^{m}_\theta$ to a
	path $\cpath^*\in W_t^1W^{m}_\theta$. Due to the trace theorem in $W_t^{1}W_\theta^{m}$ with respect to the time variable $t$
	and the weak sequential closedness of $B$ the path $\cpath^*$ lies in $\mathcal{P}_{c_A,B}$.
	
	Next note that the embedding $W_t^1W^m_\theta\hookrightarrow C^{0}_tC^{m-1}_\theta$ is compact:
	Indeed, $W_t^1W^m_\theta\hookrightarrow C^{0,1/2}_tW^m_\theta$ is continuous by $\|\cpath(s)-\cpath(r)\|_{W^m_\theta}=\left\|\int_r^s\dot\cpath(t)\,\d t\right\|_{W^m_\theta}\leq\sqrt{s-r}\|\cpath\|_{W^1_tW^m_\theta}$,
	and $C^{0,1/2}_tW^m_\theta\hookrightarrow C^{0}_tC^{m-1}_\theta$ is compact by Arzel\`a--Ascoli \cite[rem.\ on p.\,382]{DuSc58}.
	
	Since $\cpath^j(t)\in B_{\sqrt{\bar E}}(c_A)$ for all $j$ and all $t\in[0,1]$,
	\cref{prop:SobCurveUniform} implies that $|\cpath^j(t)'|$ is bounded away from $0$ uniformly in $j$ and $t$.
	Thus, by the compact embedding $W_t^1W^m_\theta\hookrightarrow C^{0}_tC^1_\theta$, also $|\cpath^*(t)'|$ is bounded away from zero and thus $\cpath^*\in W_t^1\immersion^m$.
	Furthermore, we may apply \cref{lem:weakcontSobolev} for $h^j=\dot\cpath^j$ to obtain
	$$\partial_{\cpath^j}^k\dot{\cpath}^j \xrightharpoonup[]{L^2_tL_\theta^2}\partial_{\cpath^*}^k \dot \cpath^*
	\qquad\text{for }k=0,\ldots,m.$$
	Since additionally $\dtheta\cpath^j\to\dtheta\cpath^*$ in $C^{0}_tC^0_\theta$ we even have
	$$\partial_{\cpath^j}^k\dot{\cpath}^j\sqrt{\dtheta \cpath^j} \xrightharpoonup[]{L^2_tL_\theta^2}\partial_{\cpath^*}^k \dot \cpath^*\sqrt{\dtheta \cpath^*}
	\qquad\text{for }k=0,\ldots,m.$$
	
	The lower semicontinuity of the path energy is now a direct consequence of the weak lower semicontinuity of the $L^2$-norm, which implies
	\begin{multline*}
		\pathenergy[\cpath^*]
				=\sum_{k=0}^m\left\Vert\partial_{\cpath^*}^k\dot{\cpath}^*\sqrt{\vert\dtheta \cpath^*\vert}\right\Vert^2_{L^2_tL^2_\theta}
		\leq\liminf_{j\rightarrow\infty} \sum_{k=0}^m \left\Vert\partial_{\cpath^j}^k\dot{\cpath}^j\sqrt{\vert\dtheta \cpath^j\vert}\right\Vert_{L_t^2L_\theta^2}^2\\
		=\liminf_{j\rightarrow\infty} \pathenergy[\cpath^j]
		=\inf\pathenergy,
	\end{multline*}
	which proves the claim.
\end{proof}

\begin{remark}[Reachable curves]
	There are examples where $c_A$ and $B$ lie in different connected components of $\immersion^m$.
	For instance, in the space of planar curves ($d=2$), curves with different winding numbers lie in different connected components and thus cannot be connected by a continuous path.
\end{remark}

So far, we focused on paths $\cpath$ in the space of \emph{parametric} curves $\immersion^m$,
\ie points $\cpath(0,\theta)$ are mapped along the path to points $\cpath(1,\theta)$ for all $\theta\in\Sone$.
If we do not enforce this point-to-point correspondence,
but instead ask for a shortest connecting path between two curves $c_A,c_B\in\immersion^m$ modulo reparameterization,
then one needs to solve the minimization problem from \cref{thm:SobolevExistence} for $B=C_B$, the equivalence class
$$C_B \coloneqq \left\{c\!\in\!\immersion^m \,\middle| \, c\!=\!c_B\!\circ\!\phi \text{ with } \phi\!\in\!C^0(\Sone,\Sone)\text{ bijective \& orientation-preserving}\right\}.$$
Unfortunately, $C_B$ is not even strongly closed in $W^m_\theta$,
as can be seen from $c_n=c_B\circ\phi_n\to_{n\to\infty}c\circ\phi\notin B$ for diffeomorphisms $\phi_n\in C^m(\Sone,\Sone)$ that converge uniformly to $0$ on $[0,\frac12]$ (identifying $\Sone$ with $[0,2\pi)$).
Thus, \cref{thm:SobolevExistence} does not apply.
However, the following result states that one can extend $B$ to a weakly sequentially closed set without changing the minimization problem, thus shortest geodesics between $c_A$ and $B$ exist.

\begin{theorem}[Reparameterization orbits]\label{thm:reparamOrbit}
	The closure and the weak sequential closure of $C_B$ coincide and are given by
	\begin{multline*}
		\overline{C_B}=\{c:\Sone\to\R^d\,|\,c=c_B\circ\phi\text{ with }\\\phi\in W^{m,2}(\Sone,\Sone)\text{ orientation-preserving of winding number }1\}.
	\end{multline*}
	Moreover, $\overline{C_B}\cap\immersion^m=C_B$ so that \cref{thm:SobolevExistence} even holds for $B=C_B$.
\end{theorem}
\begin{proof}
	Let $c_n\to c$ strongly (or weakly) in $W^m_\theta$ as $n\to\infty$ for $c_n=c_B\circ\phi_n$ and orientation-preserving homeomorphisms $\phi_n$ of $\Sone$.
	From $c_n\to c$ in $C^0(\Sone,\R^d)$ and the injectivity of $c_B$ it follows $\phi_n\to\phi$ in $C^0(\Sone,\Sone)$, which implies that $\phi$ is orientation-preserving and of winding number $1$.
	
	We will next show $\phi_n\to\phi$ strongly (or weakly) in $W^{m,2}(\Sone,\Sone)$.
	By induction it first follows that $\phi_n\to\phi$ in $C^k(\Sone,\Sone)$ and $c_B'\circ\phi_n\to c_B'\circ\phi$ in $C^{k-1}(\Sone,\R^d)$ for $k=1,\ldots,m-1$.
	Indeed, for $k=1$ we have $c_B'\circ\phi_n\to c_B'\circ\phi$ in $C^0(\Sone,\R^d)$ due to $c_B\in C^{m-1}(\Sone,\R^d)$
	as well as $c_B'\circ\phi_n\,\phi_n'=c_n'\to c'=c_B'\circ\phi\,\phi'$ in $C^0(\Sone,\R^d)$, which together with the other limit and the boundedness of $c_B'$ away from zero yields the uniform convergence $\phi_n'\to\phi'$.
	For general $k\leq m-1$, $c_B'\circ\phi_n$ converges to $c_B'\circ\phi$ in $C^{k-1}(\Sone,\R^d)$ as the composition of a function $c_B'$ in $C^{k-1}(\Sone,\R^d)$
	with a sequence $\phi_n$ that converges to $\phi$ in $C^{k-1}(\Sone,\R^d)$ by the induction hypothesis.
	Thus, using again the induction hypothesis, also
	\begin{multline*}
		\textstyle
		c_B'\circ\phi_n\,\dtheta^{k}\phi_n
		=\dtheta^kc_n-\sum_{j=1}^{k-1}{k-1\choose j}\dtheta^{j}(c_B'\circ\phi_n)\dtheta^{k-j}\phi_n\\
		\textstyle
		\to\dtheta^kc-\sum_{j=1}^{k-1}{k-1\choose j}\dtheta^{j}(c_B'\circ\phi)\dtheta^{k-j}\phi
		=c_B'\circ\phi\,\dtheta^{k}\phi
	\end{multline*}
	uniformly,
	which in turn implies $\dtheta^k\phi_n\to\dtheta^k\phi$ uniformly.
	Finally,
	\begin{multline*}
		\textstyle
		c_B'\circ\phi_n\,\dtheta^{m}\phi_n
		=\dtheta^mc_n-\sum_{j=1}^{m-1}{m-1\choose j}\dtheta^{j}(c_B'\circ\phi_n)\dtheta^{m-j}\phi_n\\
		\textstyle
		\to\dtheta^mc-\sum_{j=1}^{m-1}{m-1\choose j}\dtheta^{j}(c_B'\circ\phi)\dtheta^{m-j}\phi
		=c_B'\circ\phi\,\dtheta^{m}\phi
	\end{multline*}
	strongly (or weakly) in $L^2_\theta$
	(where we used $\dtheta^{m-1}(c_B'\circ\phi_n)\to\dtheta^{m-1}(c_B'\circ\phi)$ strongly in $L^2_\theta$ since composition of functions is continuous on $W^{m-1,2}$).
	Consequently, $\dtheta^m\phi_n\to\dtheta^m\phi$ strongly (or weakly) in $L^2_\theta$.
	
	This shows that the closure is a subset of our expression.
	That it is also a superset follows by approximating any orientation-preserving $\phi\in W^{m,2}(\Sone,\Sone)$ with winding number $1$ by a smooth homeomorphism in $W^{m,2}(\Sone,\Sone)$.
	
	It holds $C_B\subset\overline{C_B}\cap\immersion^m$ by definition of the closure.
	For the reverse inclusion, let $c=c_B\circ\phi\in\overline{C_B}\cap\immersion^m$.
	From $0<|c'|=|c_B'\circ\phi||\phi'|$ it follows that $|\phi'|\neq 0$, which together with the continuity of $\phi$, its orientation preservation and its winding number $1$ implies that $\phi$ is a homeomorphism of $\Sone$,
	proving $c\in C_B$.
	
	Finally, for any path $\cpath$ of finite energy and with $\cpath(0,\cdot)=c_A$, the end point $\cpath(1,\cdot)$ lies in $\immersion^m$
	so that shortest geodesics between $c_A$ and $\overline C_B$ (which exist by \cref{thm:SobolevExistence}) actually end in $C_B$.
\end{proof}

\section{Variational time discretization and fully discrete geodesics} \label{sec:timeDiscrete}
We propose a variational time discretization\index{variational time discretization} based on a
conforming, piecewise affine Galerkin discretization\index{Galerkin discretization} in time 
for the path energy:
Following the paradigm of \cite{RuWi15}, the path energy $\pathenergy[\cpath]$ is approximated by the \emph{discrete path energy}
\begin{equation}\label{eqn:discretePathEnergy}
	\Pathenergy^K[(c_0,c_1,\ldots, c_K)] \coloneqq K \sum_{k=1}^K \W[c_{k-1},c_k]
\end{equation}
with $c_k=\cpath(k/K)$ and $\W$ a second order approximation (still to be chosen) of the squared Riemannian distance $\dist^2$, thus satisfying
\begin{gather}
	\label{eqn:symmetry}
	\W[c,c]\!=\!0,\quad
	\partial_1\W[c,c]\!=\!\partial_2\W[c,c]\!=\!0,\quad
	\partial_1^2\W[c,c]\!=\!\partial_2^2\W[c,c]\!=\!-\partial_1\partial_2\W[c,c],\\
	\label{eqn:consistency}
	\partial_1^2\W[c,c]=2g_c.
\end{gather}
Minimizers of $\Pathenergy^K$ for fixed $c_0,c_K$ are called \emph{discrete geodesics}.
\cite{RuWi15} provides sufficient conditions on the Riemannian manifold and on $\W$ under which $\Pathenergy^K$ indeed $\Gamma$-converges to $\pathenergy$ as $K\to\infty$,
however, the setting of $\immersion^m$ is not covered and thus developed afresh here.

For two curves $\hat c,\check c$ consider the linear interpolation $\cpath_{\hat c,\check c}(t, \theta)=(1-t)\hat c(\theta)+t\check c(\theta)$ as a connecting path. As a first attempt, we choose
\begin{equation} \label{eq:GalerkinSobCurve}
	\W[\hat c,\check c] = \WGalerkin[\hat c,\check c]\coloneqq \pathenergy[\cpath_{\hat c,\check c}]
\end{equation}
and straightforwardly verify that
\begin{equation}\label{eqn:GalerkinDiscretization}
	\Pathenergy^K[(c_0,c_1,\ldots, c_K)] = \EGalerkin{K}[(c_0,c_1,\ldots, c_K)] \coloneqq \pathenergy[\interpolation^K[\cpath]],
\end{equation}
where $\interpolation^K[\cpath]$ is the piecewise affine interpolant of the path $\cpath$ with respect to the (temporal) nodes $\tfrac{k}{K}$.
Hence, the minimization of $\Pathenergy^K[(c_0,c_1,\ldots, c_K)]$ subject to the boundary constraints $c_0=c_A$ and $c_K=c_B$
for two curves $c_A,c_B\in\immersion^m)$
is a classical Galerkin discretization (a linear finite element discretization) of minimizing the path energy $\pathenergy[\cpath]$ subject to $\cpath(0,\cdot)=c_A$
and $\cpath(1,\cdot)=c_B$.
We have
\begin{align*} 
	\dot \cpath_{\hat c,\check c}(t, \theta) &= \check c(\theta)-\hat c(\theta),\\
	\diffs \dot \cpath_{\hat c,\check c}(t, \theta)  &= \frac1{\vert \cpath_{\hat c,\check c}'(t, \theta)\vert}  (\check c'(\theta)-\hat c'(\theta)),\\
	\diffs^2 \dot \cpath_{\hat c,\check c}(t, \theta)  &= \frac1{\vert \cpath_{\hat c,\check c}'(t, \theta)\vert^2} \dot \cpath_{\hat c,\check c}''(t, \theta) - \frac1{\vert \cpath_{\hat c,\check c}'(t, \theta)\vert^4} (\cpath_{\hat c,\check c}'(t, \theta)\cdot \cpath_{\hat c,\check c}''(t, \theta))  \dot \cpath_{\hat c,\check c}'(t, \theta) \\
	&=  \frac{\vert \cpath_{\hat c,\check c}'(t, \theta)\vert^2 (\check c''(\theta) -\hat c''(\theta)) - (\cpath_{\hat c,\check c}'(t, \theta)\cdot \cpath_{\hat c,\check c}''(t, \theta))  (\check c'(\theta) -\hat c'(\theta))}{\vert \cpath_{\hat c,\check c}'(t, \theta)\vert^4}
\end{align*}
so that for $m=2$
\begin{align*}
	\W[\hat c,\check c]
	=&\int_0^1\int_{\Sone} 
	\vert \cpath_{\hat c,\check c}'(t, \theta) \vert \left( \vert\dot \cpath_{\hat c,\check c}(t, \theta)\vert^2  + \vert \diffs \dot \cpath_{\hat c,\check c}(t, \theta)\vert^2 + \vert \diffs^2 \dot \cpath_{\hat c,\check c}(t, \theta) \vert^2
	\right) \d \theta\,\d t \\
	= &\int_{\Sone} \int_0^1 \vert \cpath_{\hat c,\check c}'(t, \theta) \vert  \vert \check c(\theta)-\hat c(\theta)\vert^2
	+   \vert \cpath_{\hat c,\check c}'(t, \theta) \vert^{-1}\vert \check c'(\theta)-\hat c'(\theta) \vert^2  \\
	&   +  \frac{\left\vert \vert \cpath_{\hat c,\check c}'(t, \theta)\vert^2 (\check c''(\theta) -\hat c''(\theta)) - (\cpath_{\hat c,\check c}'(t, \theta)\cdot \cpath_{\hat c,\check c}''(t, \theta))  (\check c'(\theta) -\hat c'(\theta))  \right\vert ^2}{\vert \cpath_{\hat c,\check c}'(t, \theta)\vert^{7} }
	\,\d t \,\d \theta.
\end{align*}
Due to the occurrence of the length element $\vert \cpath'(t, \theta)\vert$ an explicit evaluation of the time integration appears impractical.
To turn this ansatz into a fully practical numerical approach, we proceed as follows:
\begin{itemize}
	\item 
	We approximate the length element $|\cpath_{\hat c,\check c}'|$ and its negative powers
	such that the time integration can be performed explicitly. 
	We design the approximation in such a way that the resulting sequence
	of approximating functionals $\Gamma$-converges in the suitable topology to the path energy $\pathenergy$.
	\item
	To finally implement the resulting approach,
	one has to choose a suitable discretization in the spatial variable $\theta$.
	We suggest a spectral approach and consider the Fourier coefficients of $c_1,\ldots,c_{K-1}$ as the degrees of freedom.
\end{itemize}
We first focus on the approximation of $|\cpath_{\hat c,\check c}'|$, resulting in two alternative variational time discretizations for which we prove $\Gamma$-convergence.
The $\Gamma$-convergence of the fully discretized energy, including the spatial discretization, will be discussed subsequently.

\subsection{Upper and lower bounds on length element and resulting $\epsilon$-regularized path energy}
\label{sec:SobolevMetricTimeDiscretization}
\begin{figure}
	\centering
	\begin{tikzpicture}[scale=1.5]
		\coordinate (p) at (sqrt{(3/4)},0.5);
		\coordinate (u) at ($3/2*(p)$); 
		\coordinate (t) at ($2.5*(p)$);
		
		\coordinate (q) at (-sqrt{(1/2)},sqrt{(1/2)});
		\coordinate (r) at ($3/2*(q)$);
		
		\coordinate (m) at ($1/2*(p)+1/2*(q)$);
		\coordinate (w) at ($3/4*(p)+3/4*(q)$);
		
		\coordinate (s) at ($(t)!(0,0)!(r)$);

				\coordinate (M) at ($(p)!1/2!(q)$);
	
		\draw let \p1 = (t), \p2 = ($(r)-(t)$),
		\n1 = {1/scalar(veclen(\p2))}, 		\p3 = ($\n1*(\p2)$), 		\n2 = {scalar(\x1*\x3+\y1*\y3)} 		in coordinate (S) at ($(t) - \n2*(\p3)$);       

		\draw (0,0) circle (1);
		\draw (0,0) -- (p);
		\draw (0,0) -- (q);
		\draw (0,0) -- (w);
		\draw[dotted] (p) -- (q);
		\draw[-latex] (0,0) -- (r);
		\draw[-latex] (0,0) -- (t);
		\draw[dotted] (r) -- (t);
		
		\fill (p) circle[radius=0.02cm];
		\fill (q) circle[radius=0.02cm];
		\fill (m) circle[radius=0.02cm] node[anchor=north west] {$p$};
		\fill (s) circle[radius=0.02cm] node[anchor=south] {$r$};
		\fill (w) circle[radius=0.02cm] node[anchor=north west] {$q$};
		\node[anchor=west] at (1.0,-0.1)
		{$p = \frac12 \left( \frac {\hat c'(\theta)}{\vert \hat c'(\theta)\vert} + \frac {\check c'(\theta)}{\vert \check c'(\theta)\vert}\right)$};
		\node[anchor=west] at (1.0,-0.5)
		{$q = \frac12 \left( \frac {\hat c'(\theta)}{\vert \hat c'(\theta)\vert} + \frac {\check c'(\theta)}{\vert \check c'(\theta)\vert}\right)  \min \left(\vert \hat c'(\theta)\vert, \vert \check c'(\theta)\vert \right) $};
		\node[anchor=west] at (1.0,-0.9)
		{$r = \cpath_{\hat c,\check c}'(t,\theta)$ with $t= \mathrm{argmin}_{t\in [0,1]} \vert \cpath_{\hat c,\check c}'(t,\theta)\vert$};
		\node[anchor=north] at (0,0) {$0$};
		\node[anchor=west] at (t) {$\hat c'$};
		\node[anchor=north] at (p) {$\frac{\hat c'}{\vert \hat c'\vert}\;$};
		\node[anchor=south east] at (r) {$\check c'$};
		\node[anchor=north] at (q) {$\frac{\check c'}{\vert \check c'\vert}$};
		\fill (t) circle[radius=0.02cm];
		\fill (r) circle[radius=0.02cm];
		\fill (0,0) circle[radius=0.02cm];
		\draw[dotted] (r) -- (u);
		\draw[-latex] (0,0) -- (t);
	\end{tikzpicture}
	\caption[A sketch of the selection of a lower bound for $\vert \cpath_{\hat c,\check c}'(t)\vert$]{Illustration of the lower bound $|q|$
		for $\vert \cpath_{\hat c,\check c}'(t)\vert$. }\label{fig:SobolevSketch}
\end{figure}
To turn the previous $\W$ into a more practical expression
we consider suitable upper and lower bounds for the length element $\vert \cpath_{\hat c,\check c}'(t, \theta)\vert$, \begin{align*}
	\max_{t\in [0,1]}  \vert \cpath_{\hat c,\check c}'(t, \theta)\vert &= \max \left(\vert \hat c'(\theta)\vert, \vert \check c'(\theta)\vert \right),\\
	\min_{t\in [0,1]}  \vert \cpath_{\hat c,\check c}'(t, \theta)\vert &= \min_{t\in [0,1]}  \vert (1-t) \hat c'(\theta)+ t \check c'(\theta)\vert \\
	&\geq \frac12 \left\vert \frac {\hat c'(\theta)}{\vert \hat c'(\theta)\vert} + \frac {\check c'(\theta)}{\vert \check c'(\theta)\vert}\right\vert \min \left(\vert \hat c'(\theta)\vert, \vert \check c'(\theta)\vert \right).
\end{align*}
The first estimate is straightforward; the geometric construction behind the second is depicted in \cref{fig:SobolevSketch}.

The above approximations are not differentiable in $\hat c$ or $\check c$, which will be inconvenient for minimization of $\Pathenergy^K$.
Therefore, abbreviating $\gamma^\pm= \tfrac{\gamma\pm\vert \gamma \vert}{2}$, we regularize
$	\min (\alpha, \beta) = \alpha + (\beta-\alpha)^-$ and $\max (\alpha, \beta) = \alpha + (\beta-\alpha)^+$
by replacing $\vert \gamma \vert$ with
$\vert \gamma \vert_\epsilon \coloneqq \sqrt{\vert \gamma \vert^2 +\epsilon^2}\geq \vert \gamma \vert$, resulting in
\begin{equation}
	{\max}^\epsilon (\alpha, \beta) \coloneqq \alpha +  \tfrac{(\beta-\alpha) + \sqrt{(\beta-\alpha)^2+\epsilon^2}}{2},\
	{\min}^\epsilon (\alpha, \beta) \coloneqq \alpha +  \tfrac{(\beta-\alpha) - \sqrt{(\beta-\alpha)^2+\epsilon^2}}{2}. \label{eq:maxmin}
\end{equation}
Setting $\underline{\min}^\epsilon(\alpha,\beta)=\max(0,\min^\epsilon(\alpha,\beta))$, we now define our upper and lower bounds on $\vert \cpath_{\hat c,\check c}'(t, \theta)\vert$ as
\begin{align*}
	\upperLength{\hat c'}{\check c'} (\theta) &\coloneqq
	{\max}^\epsilon\left(\vert \hat c'(\theta)\vert, \vert \check c'(\theta)\vert \right) &\geq \vert \cpath_{\hat c,\check c}'(t, \theta)\vert \enspace \forall t \in [0,1],\\
	\lowerLength{\hat c'}{\check c'} (\theta) &\coloneqq
	\frac12 \left\vert\frac {\hat c'(\theta)}{\vert \hat c'(\theta)\vert} + \frac {\check c'(\theta)}{\vert \check c'(\theta)\vert}\right\vert
	\underline{\min}^\epsilon \left(\vert \hat c'(\theta)\vert, \vert \check c'(\theta)\vert \right) &\leq \vert \cpath_{\hat c,\check c}'(t, \theta)\vert \enspace \forall t \in [0,1].
\end{align*}
Replacing $\vert \cpath_{\hat c,\check c}'(t, \theta)\vert$ in $\W[\hat c,\check c]$ by these bounds yields an upper approximation $\WEps{}[\hat c,\check c]$ of $\W[\hat c,\check c]$, which for $m=2$ reads
\begin{multline*}
	\WEps{}[\hat c,\check c]
	\coloneqq \int_{\Sone} \upperLength{\hat c'}{\check c'} (\theta)  \vert \check c(\theta)-\hat c(\theta)\vert^2
	+\frac{|\check c'(\theta)-\hat c'(\theta)|^2}{\lowerLength{\hat c'}{\check c'} (\theta)} \\
	+ \frac{\int_0^1 \left| \vert \cpath_{\hat c,\check c}'(t, \theta)\vert^2 (\check c''(\theta)\!-\!\hat c''(\theta)) - (\cpath_{\hat c,\check c}'(t, \theta)\!\cdot\! \cpath_{\hat c,\check c}''(t, \theta))  (\check c'(\theta) \!-\!\hat c'(\theta))\right|^2 \,\d t}{\lowerLength{\hat c'}{\check c'}^7 (\theta)}
	\,\d \theta.
\end{multline*}

\begin{remark}[Differentiability of $\WEps$]
	In contrast to $\min^\epsilon$, the function $\underline\min^\epsilon$ is not differentiable; however, for an appropriate choice of $\epsilon$ this will become unnoticeable:
	In the following analysis, as well as the numerics, we will see that we always (more or less by construction) stay within some metric ball $B_R$ of radius $R$.
	From \cref{prop:SobCurveUniform} 
	we know that on $B_R$ the length element $|c'(\theta)|$ is bounded from below by a constant $\kappa(R)>0$ for all curves $c$.
	If we now choose $\epsilon< 2\kappa(R)$, then for all $\alpha$, $\beta \geq \kappa(R)$ we have $4 \alpha \beta > \epsilon^2$ and thus
	${\min}^\epsilon (\alpha, \beta) >0$.
	Hence, as long as $\epsilon< 2 \kappa(R)$ we observe that
	$0<\lowerLength{\hat c'}{\check c'} (\theta)\leq \min (\vert \hat c'(\theta) \vert , \vert \check c'(\theta) \vert)$ and thus $\WEps{}[\hat c,\check c]$ is well-defined and differentiable.
	As an alternative to this argumentation one could also replace $\underline\min^\epsilon(\alpha,\beta)$ with a different nonnegative lower approximation of $\max(0,\min(\alpha,\beta))$,
	namely one that is smooth right from the start (such as a mollification of $(\alpha,\beta)\mapsto\max(0,\min(\alpha-\epsilon,\beta-\epsilon))$ with a smooth mollifier of support in the $\epsilon$-ball).
	However, this will be more complicated to evaluate numerically.
\end{remark}

The integrand in the definition of $\WEps{}[\hat c,\check c]$ is now polynomial in $t$ and thus can be explicitly integrated in $t$,
leaving $\WEps{}[\hat c,\check c]$ as an integral in $\theta$ only, whose integrand depends solely on $\hat c(\theta)$, $\check c(\theta)$ and their first $m$ derivatives.
For $m=2$ the polynomial in $t$ is quartic, and we obtain (dropping the argument $\theta$ for notational simplicity)
\begingroup
\allowdisplaybreaks
\begin{align*}
						&\int_0^1  \left|\vert \cpath_{\hat c,\check c}'(t)\vert^2 (\check c'' \!-\!\hat c'') \!-\! (\cpath_{\hat c,\check c}'(t)\!\cdot\! \cpath_{\hat c,\check c}''(t))  (\check c' \!-\!\hat c')\right|^2 \,\d t\\
	&= \int_0^1\vert \cpath_{\hat c,\check c}'(t)\vert^4 \,\d t \; \vert \check c'' \!-\!\hat c'' \vert^2
	\!-\! 2 \int_0^1 \vert \cpath_{\hat c,\check c}'(t)\vert^2  (\cpath_{\hat c,\check c}'(t)\!\cdot\! \cpath_{\hat c,\check c}''(t)) \,\d t \;  \left[(\check c'' \!-\!\hat c'')\!\cdot\! (\check c' \!-\!\hat c')\right] \\
	& \quad + \int_0^1  (\cpath_{\hat c,\check c}'(t)\!\cdot\! \cpath_{\hat c,\check c}''(t))^2 \,\d t\; \vert \check c' \!-\!\hat c'\vert^2 \\
		&= \vert \check c'' \!-\!\hat c''\vert^2 \! \left[\omega_0 \vert \hat c'\vert^4
	\!+\!4 \omega_1 \vert \hat c'\vert^2 (\hat c'\!\cdot\! \check c')
	\!+\! \omega_2\! \left(4(\hat c'\!\cdot\! \check c')^2 \!+\! 2 \vert \hat c'\vert^2\vert \check c'\vert^2 \right)\right. \\
	&\qquad \qquad \quad \;\,\left.
	\!+ 4 \omega_3 (\hat c'\!\cdot\! \check c') \vert \check c'\vert^2 \!+\! \omega_4 \vert \check c'\vert^4\right] \\
			& \quad \!-\! 2 (\check c'' \!-\!\hat c'')\!\cdot\! (\check c' \!-\!\hat c') \left[ \omega_0 \vert \hat c'\vert^2 (\hat c'\!\cdot\! \hat c'')  \right. \\
	& \qquad \qquad \qquad \qquad \quad \;\; \left. + \omega_1 \left(\vert \hat c'\vert^2 (\hat c'\!\cdot\! \check c'' + \hat c'' \!\cdot\! \check c')  + 2 (\hat c'\!\cdot\! \check c')(\hat c'\!\cdot\! \hat c'')\right)\right.\\
		&\qquad \qquad \qquad \qquad \quad \;\; \left.  +\omega_2 \left(  \vert \hat c'\vert^2 (\check c'\!\cdot\! \check c'')
	+2 (\hat c'\!\cdot\! \check c')(\hat c'\!\cdot\! \check c'' + \hat c''\!\cdot\! \check c') \right. +  \vert \check c'\vert^2 (\hat c'\!\cdot\! \hat c'') \right) \\
		&\qquad \qquad \qquad \qquad  \quad \;\;\left.
	+ \omega_3 \left( 2 (\hat c'\!\cdot\! \check c')(\check c'\!\cdot\! \check c'') + \vert \check c'\vert^2 (\hat c'\!\cdot\! \check c'' + \check c'\!\cdot\! \hat c'')\right) \right.  \\
	&\qquad \qquad \qquad \qquad  \quad \;\;\left. + \omega_4 \vert \check c'\vert^2 (\check c'\!\cdot\! \check c'')  \right]\\
			& \quad +\vert \check c' \!-\!\hat c'\vert^2 \left[ \omega_0  (\hat c'\!\cdot\! \hat c'')^2
	+2 \omega_1 (\hat c'\!\cdot\! \hat c'')(\hat c'\!\cdot\! \check c'' + \hat c''\!\cdot\! \check c') \right.\\
	&\qquad \qquad  \quad  \;\;\left.
	+\omega_2  \left(2(\hat c'\!\cdot\! \hat c'') (\check c'\!\cdot\! \check c'')+(\hat c'\!\cdot\! \check c''+\check c'\!\cdot\! \hat c'')^2 \right) \right. \\
	&\qquad \qquad  \quad  \;\;\left.
	+2 \omega_3 (\hat c'\!\cdot\! \check c'' + \hat c''\!\cdot\! \check c')(\check c'\!\cdot\! \check c'')
	+ \omega_4 (\check c'\!\cdot\! \check c'')^2 \right]  ,
\end{align*}
\endgroup
where we evaluated the integrals 
$\omega_j \coloneqq \int_0^1 (1-t)^{4-j} t^j \,\d t = \frac{j!(4-j)!}{5!}$ for $j=0,\ldots, 4$ resulting in
$\omega_0=\tfrac1{5}$, $\omega_1=\tfrac1{20}$, $\omega_2=\tfrac1{30}$, $\omega_3=\tfrac1{20}$, and $\omega_4=\tfrac1{5}$.
For $m>2$ we proceed analogously.
We observe that for $\cpath=\cpath_{\hat c,\check c}(t, \theta)$ and $j\geq 1$ with the notation $v^{(i)} =  \dtheta^i v$ for a vector field $v=v(t,\theta)$
\begin{equation}\label{eqn:arclengthDerivative}
	\diffs^{j} \dot \cpath = \frac1{\vert \cpath'\vert^{3j-2}} P_j(\cpath', \cpath'',\ldots,\cpath^{(j)};\dot \cpath',\dot \cpath'',\ldots, \dot \cpath^{(j)})  ,
\end{equation}
where $(X,Y) \mapsto P_j(X;Y)$ is a polynomial in $X$ and $Y$, which has degree $2j-2$ in $X$ for fixed $Y$,
is linear in $Y$ for fixed $X$ and linear in $\cpath^{(j)}$ for fixed $Y$.
Furthermore, there appears to be no product of the highest order terms 
$\cpath^{(j)}$ and $\dot\cpath^{(j)}$.
$P_j(X;Y)$ is recursively defined via
\begin{align}\label{eq:Polone}
	P_1(\cpath';\dot \cpath') &=\dot \cpath',\\
	P_{j+1}(\cpath', \ldots, \cpath^{(j+1)};\dot \cpath',\ldots, \dot \cpath^{(j+1)})
	&= \vert \cpath'\vert^2 \; \dtheta \left( P_j(\cpath',\ldots,\cpath^{(j)};\dot \cpath',\ldots, \dot \cpath^{(j)})\right)\label{eq:Poljpone} \\
	& -\!(3j\!-\!2) (\cpath'\!\!\cdot\! \cpath'') P_{\!j}(\cpath',\ldots,\cpath^{(j)};\dot \cpath',\ldots, \dot \cpath^{(j)}), \nonumber
\end{align}
since
\begin{align*}
	\diffs^{j+1} \dot \cpath =\frac1{\vert \cpath'\vert} \dtheta \left(   \diffs^j  \dot \cpath \right)
	&= \frac1{\vert \cpath'\vert} \dtheta \left( \frac1{\vert \cpath'\vert^{3j-2}} P_j(\cpath',\ldots,\cpath^{(j)};\dot \cpath',\ldots, \dot \cpath^{(j)})  \right)\\
	&= \frac1{\vert \cpath'\vert} \left( \frac1{\vert \cpath'\vert^{3j-2}} \dtheta \left(P_j(\cpath',\ldots,\cpath^{(j)};\dot \cpath',\ldots, \dot \cpath^{(j)})\right) \right.\\
	& \qquad \qquad   \left. - \frac{(3j\!-\!2) (\cpath'\!\cdot\! \cpath'')}{\vert \cpath'\vert^{3j}} P_j(\cpath',\ldots,\cpath^{(j)};\dot \cpath',\ldots, \dot \cpath^{(j)}) \right).
\end{align*}
The resulting upper approximating for $\W[\hat c,\check c]$ then reads
\begin{multline}\label{eqn:curveSpaceW}
	\WEps{}[\hat c,\check c]
	=\int_{\Sone}  \upperLength{\hat c'}{\check c'} (\theta)  \vert \check c(\theta)-\hat c(\theta)\vert^2\\
	+ \int_0^1 \sum_{j=1}^m 
	\frac{ \left\vert P_j\left((\cpath_{\hat c,\check c}',\ldots,\cpath_{\hat c,\check c}^{(j)})(t,\theta);(\check c'-\hat c',\ldots,\hat c^{(j)}-\check c^{(j)})(\theta)\right)\right\vert^2}
	{\lowerLength{\hat c'}{\check c'} (\theta)^{6j-5}}
	\,\d t\,  \,\d \theta.
\end{multline}
Again, the time integration can be performed explicitly, involving the integral evaluations
$$
\int_0^1 (1-t)^{4m-4-j} t^j \,\d t = \frac{j!(4m-4-j)!}{(4m-3)!}
$$ 
for $j=0,\ldots, 4m-4$. 
Note that, abbreviating the nonpolynomial coefficients $\upperLength{\hat c'}{\check c'}$ or $\lowerLength{\hat c'}{\check c'}^{5-6j}$ by $\ell^\epsilon[\hat c',\check c']$, these satisfy
\begin{equation}\label{eqn:derivativeBound}
	\|\partial_1^\alpha\partial_2^\beta\ell^\epsilon[\hat c,\check c]\|_{L^\infty}
	\leq C\epsilon^{1-\alpha-\beta}
\end{equation}
for a constant $C$ depending only on $\|\hat c'\|_{L^\infty},\|\check c'\|_{L^\infty},\|\hat c'-\check c'\|_{L^\infty}$ and $\alpha,\beta>0$.

Based on the definition of $\WEps$, we obtain the discrete path energy
\begin{equation}\label{eq:SobolevDiscretePathEnergy}
	\EEps{\epsilon}{K}[(c_0,c_1,\ldots, c_K)] \coloneqq K \sum_{k=1}^K \WEps{}[c_{k-1},c_k]
\end{equation}
for $c_j \in \immersion^m$.
By construction, we have
\begin{equation}\label{eqn:boundWEps}
	\WEps{}[\hat c,\check c]\geq\WGalerkin[\hat c,\check c]
	\qquad\text{and}\qquad
	\EEps{\epsilon}{K}[(\cpath(\tfrac0K),\cpath(\tfrac1K),\ldots,\cpath(\tfrac KK))]\geq\pathenergy[\interpolation^K[\cpath]].
\end{equation}
Here, $\WEps{}[\hat c,\check c]$ differs from $\WGalerkin[\hat c,\check c]$ only in 
the replacement of the length element by $\upperLength{\hat c'}{\check c'}$ and $\lowerLength{\hat c'}{\check c'}$, from which we get
\begin{equation}\label{eqn:coercivityWEps}
	\partial_i^2\WGalerkin[c,c]\leq\partial_i^2\WEps{}[c,c]\leq(1-\tfrac\epsilon{2\min_\theta|c'(\theta)|})^{5-6m}\partial_i^2\WGalerkin[c,c],\qquad i=1,2,
\end{equation}
in the sense of quadratic forms (note that all Hessian terms involving derivatives of the length element vanish at $\hat c=\check c$),
implying consistency \eqref{eqn:consistency} up to order $\epsilon$ (while \eqref{eqn:symmetry} holds by construction).

Note that the above notation also allows us to rewrite the continuous path energy as
\begin{align}
	\pathenergy[\cpath] 
	&=\int_0^1\int_{\Sone}  |\cpath'(t, \theta)|  \vert \dot \cpath(t, \theta)\vert^2 \label{eq:SobolevpathenergyRewritten} \\
	& \qquad \quad 
	+ \sum_{j=1}^m 
	\frac{\left\vert P_j\left((\cpath',\ldots,\cpath^{(j)})(t,\theta);(\dot \cpath',\ldots, \dot \cpath^{(j)})(t,\theta)\right)\right\vert^2}
	{|\cpath'(t, \theta)|^{6j-5}} \d \theta\, \d t. \nonumber
\end{align}

\subsection{Time discretization for $m=2$ based on $\epsilon$-free upper bounds of the length element}\label{sec:SobolevMetricTimeDiscretization2}
The approximation of the previous paragraph can readily be applied to Sobolev metrics of any order $m\geq2$.
For the special case $m=2$ we present here an alternative approximation
that can manage without an $\epsilon$-regularization.
The principle of this alternative approximation is independent of $m$.
However, the larger $m$, the more difficult it will be to implement a numerically stable version of it.

Recall the Galerkin ansatz for $m=2$
\begin{equation*}
	\WGalerkin[\hat c,\check c]
	=\int_0^1\int_{\Sone} 
	\vert \cpath_{\hat c,\check c}'(t, \theta) \vert \left( \vert\dot \cpath_{\hat c,\check c}(t, \theta)\vert^2  + \vert \diffs \dot \cpath_{\hat c,\check c}(t, \theta)\vert^2 + \vert \diffs^2 \dot \cpath_{\hat c,\check c}(t, \theta) \vert^2
	\right) \,\d \theta\,\,\d t
\end{equation*}
for $\cpath_{\hat c,\check c}(t, \theta)=(1-t)\hat c(\theta)+t\check c(\theta)$.
Exploiting the triangle inequality
\begin{equation}\label{eqn:L}
	|\cpath_{\hat c,\check c}'(t, \theta)|\leq L[\hat c,\check c](t,\theta)
	\qquad\text{with }L[\hat c,\check c](t,\theta)=(1-t)|\hat c'(\theta)|+t|\check c'(\theta)|  ,
\end{equation}
we can bound this energy from above by
\begingroup
\allowdisplaybreaks
\begin{flalign*}\label{eqn:GalerkinApproxSobolevMetric}
	\addtocounter{equation}{1}\tag{\theequation}
	&\bar\W[\hat c,\check c]\\
	&=\int_0^1\int_{\Sone} 
	L[\hat c,\check c](t,\theta) \left( \vert\dot \cpath_{\hat c,\check c}(t, \theta)\vert^2  + \vert \diffs \dot \cpath_{\hat c,\check c}(t, \theta)\vert^2 + \vert \diffs^2 \dot \cpath_{\hat c,\check c}(t, \theta) \vert^2
	\right) \,\d \theta\,\,\d t\\
		&= \int_{\Sone} \int_0^1 L[\hat c,\check c](t,\theta) \vert \check c(\theta)-\hat c(\theta)\vert^2
	+ \frac{L[\hat c,\check c](t,\theta)}{\vert \cpath_{\hat c,\check c}'(t, \theta) \vert^{2}}\vert \check c'(\theta)-\hat c'(\theta) \vert^2  \\
	&\qquad  \quad +  \frac{L[\hat c,\check c](t,\theta)}{\vert \cpath_{\hat c,\check c}'(t, \theta)\vert^{8}} \left\vert \vert \cpath_{\hat c,\check c}'(t, \theta)\vert^2 (\check c''(\theta) -\hat c''(\theta)) \right.\\
	&\qquad \qquad \qquad \qquad \quad \left.- (\cpath_{\hat c,\check c}'(t, \theta)\cdot \cpath_{\hat c,\check c}''(t, \theta))  (\check c'(\theta) -\hat c'(\theta))  \right\vert ^2
	\,\d t \,\d \theta\\
		&= \int_{\Sone} \int_0^1 L[\hat c,\check c](t,\theta) \vert \check c(\theta)-\hat c(\theta)\vert^2
	+   \frac{L[\hat c,\check c](t,\theta)}{\vert \cpath_{\hat c,\check c}'(t, \theta) \vert^{2}}\vert \check c'(\theta)-\hat c'(\theta) \vert^2 \\
	&\qquad \qquad +  \frac{L[\hat c,\check c](t,\theta)}{\vert \cpath_{\hat c,\check c}'(t, \theta)\vert^{4}} \left\vert \check c''(\theta) -\hat c''(\theta)\right\vert ^2  \\
	&\qquad  \qquad -2\frac{L[\hat c,\check c](t,\theta)}{\vert \cpath_{\hat c,\check c}'(t, \theta)\vert^{6}} (\cpath_{\hat c,\check c}'(t, \theta)\!\cdot\! \cpath_{\hat c,\check c}''(t, \theta))  ((\check c''(\theta) \!-\!\hat c''(\theta)) \!\cdot\! (\check c'(\theta) \!-\!\hat c'(\theta)))\\
	&\qquad  \qquad +\frac{L[\hat c,\check c](t,\theta)}{\vert \cpath_{\hat c,\check c}'(t, \theta)\vert^{8}} \left\vert (\cpath_{\hat c,\check c}'(t, \theta)\cdot \cpath_{\hat c,\check c}''(t, \theta))  (\check c'(\theta) -\hat c'(\theta))  \right\vert ^2
	\,\d t \,\d \theta\\
		&= \int_{\Sone}\!\!\int_0^1 \!\!w_0(t,\theta)\vert \check c(\theta)\!-\!\hat c(\theta)\vert^2  \!+\! w_1(t,\theta)\vert \check c'(\theta)\!-\!\hat c'(\theta) \vert^2 \!+\! w_{2\text a}(t,\theta) \left\vert \check c''(\theta) \!-\!\hat c''(\theta)\right\vert ^2  \\
	&\qquad  \quad -2w_{2\text b}(t,\!\theta) ((\check c''(\theta) \!-\!\hat c''(\theta)) \!\cdot\! (\check c'(\theta) \!-\!\hat c'(\theta))) \!+\!w_{2\text c}(t,\!\theta)  \left\vert \check c'(\theta) \!-\!\hat c'(\theta) \right\vert^2
	\d t \,\d \theta ,
\end{flalign*}

where we introduce the abbreviations
\begin{equation}\label{eqn:abbreviations}
	\begin{gathered}
		r=|\hat c'(\theta)|,\quad
		p=|\check c'(\theta)|,\quad
		q=\hat c'(\theta)\cdot\check c'(\theta),\\
		\rho(1-t)^2+\sigma t^2+2\tau t(1-t)=\cpath_{\hat c,\check c}'(t, \theta)\cdot \cpath_{\hat c,\check c}''(t,\theta),
	\end{gathered}
\end{equation}
\begin{equation}\label{eqn:timeIntegrals}
	\begin{aligned}
		&w_0(t,\theta) = ((1-t)r+tp), \\
		& w_1(t,\theta)  =\frac{(1-t)r+tp}{(1-t)^2r^2+2(1-t)tq+t^2p^2}, \\
		& w_{2\text a}(t,\theta) = \frac{(1-t)r+tp}{((1-t)^2r^2+2(1-t)tq+t^2p^2)^2}, \\
		&w_{2\text b}(t,\theta) = \frac{((1-t)r+tp)(\rho(1-t)^2+\sigma t^2+2\tau t(1-t))  }{((1-t)^2r^2+2(1-t)tq+t^2p^2)^3}, \\
		&w_{2\text c}(t,\theta) = \frac{((1-t)r+tp)(\rho(1-t)^2+\sigma t^2+2\tau t(1-t))^2}{((1-t)^2r^2+2(1-t)tq+t^2p^2)^4} \,.
	\end{aligned}
\end{equation}
Here and in the following, we drop the explicit dependence on $\theta$ for simplicity.
The advantage of our above approximation is that it turns the integrand into a rational function in $t$ that can explicitly be integrated.
Abbreviating further
\begin{equation}\label{eqn:abbreviationsII}
	v=\frac q{rp}
	\qquad\text{and}\qquad
	V=\frac{\arctan(\frac{r^2-q}u)+\arctan(\frac{p^2-q}u)}{u/q}
\end{equation}
for $u=\sqrt{r^2p^2-q^2}$
as well as
\begin{equation}\label{eqn:timeIntegralsII}
	\begin{aligned}
		\Phi_1&=\left(\begin{smallmatrix}1\\\frac{V-1}{1-v^2}\end{smallmatrix}\right),\;
		\Xi_1=\frac{\left(\begin{smallmatrix}3+2v&1\\3&1-2v\end{smallmatrix}\right)}{8v(1+v)},\;
		\Theta_1=\frac{\left(\begin{smallmatrix}\sigma r^3+\rho p^3\\rp((\sigma+2\tau)r+(\rho+2\tau)p)\end{smallmatrix}\right)}{(rp)^4}, \\
		\Phi_2&=\left(\begin{smallmatrix}1\\\frac{V-1+\frac{1-v^2}{3v^2}}{(1-v^2)^2}\end{smallmatrix}\right),\;
		\Xi_2=\frac{\left(\begin{smallmatrix}8v^3+10v^2-5&2v^2+4v-1&2v-1\\15v^2&-12v^3+3v^2&6v^4-6v^3+3v^2\end{smallmatrix}\right)}{48v^3(1+v)},\;\\
		\Theta_2&=\frac{\left(\begin{smallmatrix}\sigma^2 r^5+\rho^2 p^5\\r p (\sigma (\sigma+4 \tau) r^3+\rho (\rho+4 \tau) p^3)\\
				2 (r p)^2 ((\rho \sigma+2 \tau^2) (r+p)+2 \tau (\sigma r+\rho p))\end{smallmatrix}\right)}{(rp)^6},
	\end{aligned}
\end{equation}
it is tedious but straightforward to check (for instance, using the Maple computer algebra code given in \cref{sec:mapleCode}) that
\begin{align}\label{eqn:barW}
	\bar\W[\hat c,\check c]
	&=\int_{\Sone} \frac{r+p}2\vert \check c-\hat c\vert^2
	+\frac{(\frac1v-1)(r+p)V+(r-p)\log(r/p)}{r^2+p^2-2q}\vert \check c'-\hat c' \vert^2\\
	&\quad + \frac12\frac{\frac{r+p}qV+\frac1r+\frac1p}{rp+q}\left\vert \check c'' -\hat c''\right\vert ^2
	-2\Phi_1^T\Xi_1\Theta_1 ((\check c'' -\hat c'') \cdot  (\check c' -\hat c'))\nonumber\\
	&\quad +\Phi_2^T\Xi_2\Theta_2\left\vert \check c' -\hat c' \right\vert^2
	\,\d \theta\nonumber\\
	&=\int_{\Sone} \frac{r+p}2\vert \check c-\hat c\vert^2
	+(\tfrac1v-1)(r+p)V+(r-p)\log(r/p)\nonumber\\
	&\quad + \frac12\frac{\frac{r+p}qV+\frac1r+\frac1p}{rp+q}\left\vert \check c'' -\hat c''\right\vert ^2
	-2\Phi_1^T\Xi_1\Theta_1 (\rho+\sigma-2\tau) \nonumber\\
	&\quad +\Phi_2^T\Xi_2\Theta_2(r^2+p^2-2q)
	\,\d \theta.\nonumber
\end{align}
\endgroup
Next, note that $\arctan$ satisfies the following addition theorem,
\begin{equation*}
	\arctan a+\arctan b
	=\arctan\frac{a+b}{1-ab}+\begin{cases}
		\pi&\text{if }a,b,ab-1>0,\\
		0&\text{if }ab<1,\\
		-\pi&\text{if }a,b,1-ab<0,
	\end{cases}
\end{equation*}
which we would like to apply to $V$, thus $a=\frac{r^2-q}u$ and $b=\frac{p^2-q}u$.
Due to $|q|\leq rp$ at least one of $a$ or $b$ is nonnegative so we are either in the first or second case.
Now it is straightforward to check
\begin{equation*}
	ab<1
	\quad\Leftrightarrow\quad
	q((r-p)^2+2(rp-q))>0
\end{equation*}
so that for positive $q$ (which we will later enforce) we may set
\begin{equation*}
	V=\frac{\arctan(\frac{u}q)}{u/q}=\frac{\arctan\sqrt{v^{-2}-1}}{\sqrt{v^{-2}-1}}=\frac{\arccos v}{\sqrt{v^{-2}-1}}=\frac v{\mathop{\mathrm{sinc}}(\arccos v)}
\end{equation*}
with $\mathop{\mathrm{sinc}}(x)=\sin(x)/x$.
For positive $q$ we further have $v>0$, which together with $v\leq1$ implies $V\leq1$.
Note that for sufficiently many discrete steps along a discrete geodesic, we have $\hat c'=c_k'\approx c_{k+1}'=\check c'$ so that indeed $q>0$.
Moreover, in that case we even have $v\approx1$ so that replacing $V$ with its zeroth order Taylor expansion $1$ at $v=1$ is not too bad an approximation.
Doing so in $\bar\W$ will increase the energy even more to
\begin{equation}\label{eqn:curveSpaceW2}
	\WEpsFree[\hat c,\check c]
	=\begin{cases}\int_{\Sone} \frac{r+p}2\vert \check c-\hat c\vert^2
		+(\tfrac1v-1)(r+p)+(r-p)\log(r/p)\\
		\quad+ \frac12\left(\frac1{rq}+\frac1{pq}\right)\left\vert \check c'' -\hat c''\right\vert ^2\\
		\quad -2\Phi_1^T\Xi_1\Theta_1 (\rho\!+\!\sigma\!-\!2\tau)
		\!+\!\Phi_2^T\Xi_2\Theta_2(r^2\!+\!p^2\!-\!2q)
		\,\d \theta
		&\text{if }q>0,\\
		\infty&\text{else.}
	\end{cases}
\end{equation}
Note that the representation of the last two integrand terms with the particular choice of $\Phi_1$ and $\Phi_2$ turns out to be important for numerical stability:
$\Phi_1$ and $\Phi_2$ are chosen such that they stay bounded as $v$ approaches $1$
(which is the case if $\hat c$ approaches $\check c$ and thus if more and more points would be introduced along a discrete geodesic).
The corresponding discrete path energy reads
\begin{equation}\label{eqn:SobolevDiscretePathEnergy2}
	\EEpsFree{K}[(c_0,\ldots, c_K)] = K \sum_{k=1}^K \WEpsFree[c_{k-1},c_k].
\end{equation}
Again, by construction, we have
\begin{equation}\label{eqn:boundWEps2}
	\WEpsFree[\hat c,\check c]\geq\WGalerkin[\hat c,\check c]
	\qquad\text{and}\qquad
	\EEpsFree{K}[(\cpath(\tfrac0K),\cpath(\tfrac1K),\ldots,\cpath(\tfrac KK))]\geq\pathenergy[\interpolation^K[\cpath]].
\end{equation}

$\WEpsFree$ differs from $\WGalerkin$ by the approximations $|\cpath_{\hat c,\check c}'(t, \theta)|\approx L[\hat c,\check c](t,\theta)$ and $V\approx1$.
Since these approximations are exact for $\hat c=\check c$,
the first and second derivatives of $\WEpsFree[\hat c,\check c]$ at $\check c=\hat c$ are the same as of $\WGalerkin[\hat c,\check c]$
(note that at $\check c=\hat c$ the terms involving derivatives of $L$ or $V$ vanish),
\begin{equation}\label{eqn:coercivityWEpsFree}
	\partial_i\WGalerkin[\hat c,\hat c]\!=\!\partial_i\WEpsFree[\hat c,\hat c],\qquad
	\partial_i\partial_j\WGalerkin[\hat c,\hat c]\!=\!\partial_i\partial_j\WEpsFree[\hat c,\hat c]
	\qquad\text{for }i,j\!=\!1,2.
\end{equation}
Therefore $\WEpsFree$ is consistent in the sense of \eqref{eqn:symmetry}-\eqref{eqn:consistency}.

\subsection{Existence and Mosco convergence\index{Mosco convergence} of the time discretization} \label{sec:MoscoCurves}
We first prove the existence of discrete geodesics and subsequently analyse their convergence, treating all three variants $\EGalerkin{K},\EEps{\epsilon}{K},\EEpsFree{K}$ of the discrete path energy
(where the first, despite being impractical, serves as a preparation for the other two).

\begin{theorem}[Existence\index{existence} of discrete geodesics]\label{thm:ExistenceDiscreteMinimizersSobolev}
	Let $g_c(\cdot,\cdot)$ be the Sobolev metric of order $m\geq 2$, $c_A\in\immersion^m$,
	and $B\subset W^m_\theta$ either a $W^m_\theta$-weakly sequentially closed set or the equivalence class of some $c_B\in\immersion^m$ with respect to reparameterization.
	Let $\epsilon>0$, $K\in\N$ be fixed, and let $\Pathenergy^K\in\{\EGalerkin{K},\EEps{\epsilon}{K},\EEpsFree{K}\}$.
	If there exists a discrete path $(c_0,\ldots,c_K)$ with $c_0=c_A$ and $c_K\in B$ of finite discrete path energy $\Pathenergy^K$, then there also exists one of minimal discrete path energy.
\end{theorem}
\begin{proof}
	By \cref{thm:reparamOrbit}, it suffices to consider the case of sequentially weakly closed $B$.
	
	$\EGalerkin{K}$ coincides with the restriction of $\pathenergy$ to piecewise affine paths in time
	so that existence of minimizers for $\Pathenergy^K=\EGalerkin{K}$ is a straightforward adaption of the proof of \cref{thm:SobolevExistence}.
	
	Now let $\Pathenergy^K=\EEps{\epsilon}{K}$.
	Consider a minimizing sequence $(c^i_0,c^i_1,\ldots, c^i_K)$, $i=1,2,\ldots$, with $c^i_0=c_A$ and $c^i_K\in B$,
	where, without loss of generality, we may assume the path energy $\Pathenergy^K[(c^i_0,c^i_1,\ldots, c^i_K)]$ to be monotonically decreasing and bounded by some $\bar E>0$.
	Denote by $\cpath^{i,K}$ the piecewise affine interpolation of $(c^i_0,c^i_1,\ldots, c^i_K)$ with $c^i_k=\cpath^{i,K}(\tfrac{k}{K})$.
	By \eqref{eqn:boundWEps} we have
	$$\bar E\geq\EEps{\epsilon}{K}[(c^i_0,c^i_1,\ldots, c^i_K)] \geq \pathenergy[\cpath^{i,K}]$$
	so that $\pathenergy[\cpath^{i,K}]$ is bounded uniformly in $i$.
	In particular, the $\cpath^{i,K}$ lie within a metric ball around $c_A$ of finite radius.
	\Cref{lem:coercivityBoundednessMetric} thus implies $\|\cpath^{i,K}\|_{W_t^1W_\theta^m}\leq C\pathenergy[\cpath^{i,K}]\leq C\bar E$,
	from which we obtain, potentially after passing to a subsequence, the weak convergence $\cpath^{i,K}\rightharpoonup\cpath^K$ in $W_t^1W_\theta^m$ to some piecewise affine path $\cpath^K$.
	In particular, abbreviating $c_k\coloneqq \cpath^{K}(\tfrac{k}{K})$, we have $c_0=c_A$ and $c_K\in B$
	as well as $|(\cpath^{i,K})'|\to|(\cpath^{K})'|$ strongly in $C^{0}_tC^0_\theta$ (see proof of \cref{thm:SobolevExistence}).
	Thus also $\upperLength{(c^i_{k-1})'}{(c^i_{k})'}\to\upperLength{(c_{k-1})'}{(c_{k})'}$ and $\upperLength{(c^i_{k-1})'}{(c^i_{k})'}\to\upperLength{(c_{k-1})'}{(c_{k})'}$ strongly in $C^{0}(\Sone,\R)$.
	Using \cref{lem:weakcontSobolev} this implies that 
	\begin{align*}
		p_{j,k}^i&\coloneqq
		\frac{P_j\left((\cpath^{i,K})',\ldots,(\cpath^{i,K})^{(j)};(\dot\cpath^{i,K})',\ldots, (\dot\cpath^{i,K})^{(j)}\right)}{\lowerLength{(c^{i}_{k-1})'}{(c^{i}_{k})'}^{3j-5/2}}
						= \frac{{\vert (\cpath^{i,K})'\vert^{3j-5/2}}\diffs^{j} \dot\cpath^{i,K}}{\lowerLength{(c^{i}_{k-1})'}{(c^{i}_{k})'}^{3j-5/2}}\\
		&\rightharpoonup  \frac{{\vert (\cpath^{K})'\vert^{3j-5/2}}\diffs^{j} \dot c^{K}}{\lowerLength{(c_{k-1})'}{(c_{k})'}^{3j-5/2}}
		= \frac{P_j\left((\cpath^{K})',\ldots,(\cpath^{K})^{(j)};(\dot\cpath^{K})',\ldots, (\dot\cpath^{K})^{(j)}\right)}{\lowerLength{(c_{k-1})'}{(c_{k})'}^{3j-5/2}}
	\end{align*}
	weakly in $L^2((\frac{k-1}{K},\frac{k}{K}) \times \Sone,\R^d)$.
	The convergence of the term
	$$p_{0,k}^i\coloneq\sqrt{\upperLength{(c^{i}_{k-1})'}{(c^{i}_{k})'}}  (c^{i}_{k}-c^{i}_{k-1})$$
	in $L^2(\Sone,\R^d)$ follows in analogy. Since
	$$\textstyle\EEps{\epsilon}{K}[(c^i_0,c^i_1,\ldots, c^i_K)]=\sum_{k=1}^K\left[K\|p_{0,k}^i\|_{L^2(\Sone,\R^d)}^2+\sum_{j=1}^m\|p_{j,k}^i\|_{L^2((\frac{k-1}K,\frac kK)\times\Sone,\R^d)}^2\right]$$
	and the norms are weakly lower semicontinuous,
	\begin{equation*}
		\inf\EEps{\epsilon}{K}=\lim_{i\to\infty} \EEps{\epsilon}{K}[(c^i_0,c^i_1,\ldots, c^i_K)] \geq \EEps{\epsilon}{K}[(c_0,c_1,\ldots, c_K)].
	\end{equation*}
	
	The argument for $\Pathenergy^K=\EEpsFree{K}$ is essentially analogous;
	here we need to exploit the strong $L^\infty$-convergence of $L[c^i_{k-1},c^i_k](t,\theta)$ along the minimizing sequence (rather than of $\upperLength{c^i_{k-1}}{c^i_k}$ and $\lowerLength{c^i_{k-1}}{c^i_k}$).
\end{proof}
Next, we study the $\Gamma$-convergence of our discrete path energies to the continuous one; in fact, we will prove the slightly stronger Mosco convergence.
We will again restrict the energy to the set $\mathcal{P}_{c_A,B}$ from \eqref{eqn:admissiblePaths} such that minimizers are shortest geodesics between $c_A$ and $B$,
\ie\ we set
\begin{equation}\label{eq:refDefEnergySobCurv}
		\pathenergy_{c_A,B}[\cpath] \coloneqq \begin{cases}\pathenergy[\cpath] & \text{if } c \in  \mathcal{P}_{c_A,B}, \\
			\infty& \text{else}.\end{cases}
	\end{equation}
	Furthermore, we extend our discrete path energies to nondiscrete paths by infinity and set
	\begin{equation}\label{eqn:extPathEnergy}
		\pathenergy^K_{c_A,B}[\cpath] \coloneqq \begin{cases}\Pathenergy^K[(\cpath(\frac0K),\cpath(\frac1K),\ldots,\cpath(\frac kK))] & \text{if } \cpath=\interpolation^K[\cpath] \in  \mathcal{P}_{c_A,B}, \\
			\infty& \text{else}\end{cases}
	\end{equation}
	for $\Pathenergy^K\in\{\EGalerkin{K},\EEps{\epsilon}{K},\EEpsFree{K}\}$ and the Lagrangian interpolation operator $\interpolation^K$.
	
	\begin{theorem}[Mosco convergence of the discrete path energy]\label{thm:SobolevMoscoConv}
		Let us assume that $\Pathenergy^K\in\{\EGalerkin{K},\EEps{\epsilon_K}{K},\EEpsFree{K}\}$ in \eqref{eqn:extPathEnergy}, and let $m=2$ in the last case and $\epsilon_K\to0$ as $K\to\infty$ in the second.
		Under the assumptions of \cref{thm:SobolevExistence}, $\pathenergy_{c_A,B}^K$
		Mosco converges on $W_t^1W_\theta^m$ to $\pathenergy_{c_A,B}$ from \eqref{eq:refDefEnergySobCurv}.
	\end{theorem}
	\begin{proof}
		For the $\liminf$-inequality we consider a sequence $(\cpath^K)_{K=1,\ldots}$ which converges weakly to $\cpath^*$ in $W_t^1W_\theta^m$
		and assume without loss of generality that $\pathenergy^K_{c_A,B}[\cpath^K]$ is uniformly bounded, thus $\cpath^K=\interpolation^K[\cpath^K]\in\mathcal{P}_{c_A,B}$.
		As in the proof of \cref{thm:SobolevExistence} we thus observe that $\cpath^*\in\mathcal{P}_{c_A,B}$ and
		$$\liminf_{K\to \infty} \pathenergy^K_{c_A,B}[\cpath^K] \geq\liminf_{K\to \infty} \pathenergy[\cpath^K] \ge \liminf_{K\to \infty} \pathenergy[\cpath^K] \ge \pathenergy[\cpath^*]=\pathenergy_{c_A,B}[\cpath^*],$$
		where the first inequality is actually an equality in case of $\Pathenergy^K=\EGalerkin{K}$ and in the other cases holds due to \eqref{eqn:boundWEps} and \eqref{eqn:boundWEps2}.
		
		For the $\limsup$-inequality consider a path $\cpath\in W_t^1W_\theta^m$ with $\pathenergy_{c_A,B}[\cpath]<\infty$, thus $\cpath\in W_t^1\immersion^m$.
		As a recovery sequence $(\cpath^K)_{K=1,2,\ldots}$ we choose
		$\cpath^K\coloneqq \interpolation^K[\cpath]$.
		Let us first recall that $\cpath^K\to\cpath$ strongly in $W_t^1W_\theta^m$ as $K\to \infty$. Indeed,
		we may choose a family of approximations $\cpath^\delta\in C_t^2W_\theta^m$ with
		$\Vert \cpath - \cpath^\delta\Vert_{W^1_t W^m_\theta} \to 0 \text{ for } \delta \to 0$ and obtain
		\begin{align*}
			\Vert \cpath- \cpath^K\Vert_{W^1_t W^m_\theta} &
			\leq \Vert \cpath- \cpath^\delta \Vert_{W^1_t W^m_\theta}
			+ \Vert \cpath^\delta - \interpolation^K[\cpath^\delta]\Vert_{W^1_t W^m_\theta}
			+\Vert \interpolation^K[\cpath^\delta- \cpath]\Vert_{W^1_t W^m_\theta} .
		\end{align*}
		Using classical approximation theory, the second term can be estimated up to a constant factor by $K^{-1} \Vert \cpath^\delta \Vert_{C^2_t W^m_\theta}$,
		and the third term converges to $0$ for $\delta\to 0$ due to the boundedness of the Lagrangian interpolation operator $\interpolation^K$ uniformly in $K$.
		
		From the continuous embedding $W_t^1W_\theta^m\hookrightarrow C_t^0C_\theta^1$ (see proof of \cref{thm:SobolevExistence}) and
		$\tfrac1C \leq \Vert \vert \cpath '\vert \Vert_{L^\infty}, \Vert \vert \cpath'\vert^{-1} \Vert_{L^\infty} \leq C$ for some $C>0$ by \cref{prop:SobCurveUniform} we deduce
		$$
		\tfrac1{2C} \leq \Vert \vert (\cpath^K)'\vert \Vert_{L^\infty}, \Vert \vert (\cpath^K)'\vert^{-1} \Vert_{L^\infty} \leq 2 C
		$$
		for sufficiently large $K$. Furthermore, we have
		\begin{align*}
			(\cpath^K)^{(m)} &\rightarrow \cpath^{(m)} &&\text{ in } C_t^{0}L_\theta^2, \\
			(\dot \cpath^K)^{(m)} &\rightarrow \dot \cpath^{(m)} &&\text{ in } L_t^2L_\theta^2, \\
			(\cpath^K)^{(j)} &\rightarrow \cpath^{(j)} &&\text{ in } C_t^{0}C_\theta^{0} \text{ for } j < m, \\
			(\dot \cpath^K)^{(j)} &\rightarrow \dot \cpath^{(j)} &&\text{ in } L_t^2C_\theta^{0}  \text{ for } j < m .
		\end{align*}
		Thus, taking into account the properties of $P_j$ defined in \eqref{eq:Polone}, \eqref{eq:Poljpone}
		(in particular that it is affine in $\dot \cpath^{(j)}$  and in $\cpath^{(j)}$ and
		that there appears no product of $\cpath^{(j)}$ and $\dot \cpath^{(j)}$)
		we have $p_j^K\coloneq P_j((\cpath^K)', \ldots,(\cpath^K)^{(j)};(\dot\cpath^K)',\ldots, (\dot\cpath^K)^{(j)})
		\to P_j(\cpath', \ldots,\cpath^{(j)};\dot \cpath',\dot \ldots, \dot \cpath^{(j)})=:p_j$
		in $L^2((0,1),\R^d)$ as $K\to\infty$ for $j=1,\ldots,m$.
		Now we use
		\begin{equation*}
			\pathenergy_{c_A,B}^K[\cpath^K]
			=\int_0^1\int_{\Sone} \ell^+[(\cpath^K)'] |\dot\cpath^K|^2
			+\sum_{j=1}^m \frac{|p_j^K|^2}{\ell^-[(\cpath^K)']^{6j-5}} \,\d \theta\, \d t,
		\end{equation*}
		where for $\Pathenergy^K=\EGalerkin{K}$ we have $\ell^+[(\cpath^K)']=\ell^-[(\cpath^K)']=(\cpath^K)'$
		and for $\Pathenergy^K=\EEps{\epsilon_K}{K}$ we have $\ell^\pm[(\cpath^K)']=L^{\pm,\epsilon_K}[(\cpath^K)'(\frac{k-1}K),(\cpath^K)'(\frac{k}K)]$ on $(\frac{k-1}K,\frac kK)$.
		Due to the uniform convergence of $(\cpath^K)'$ we have $\ell^\pm[(\cpath^K)']\to\cpath'$ in $C^0([0,1])$ as $K\to\infty$
		so that altogether the integrand of $\pathenergy_{c_A,B}^K[\cpath^K]$ converges strongly in $L^1((0,1))$ to that of \eqref{eq:SobolevpathenergyRewritten}.
		Consequently, $\pathenergy_{c_A,B}^K[\cpath^K]\to\pathenergy[\cpath]=\pathenergy_{c_A,B}[\cpath]$ for $K\to\infty$, as desired.
		
		The argument for $\Pathenergy^K=\EEpsFree{K}$ is essentially analogous; here we need to exploit the strong uniform convergence $L[(\cpath^K)'(\frac{k-1}K),(\cpath^K)'(\frac{k}K)]\to\cpath'$ on $(\frac{k-1}K,\frac kK)$.
	\end{proof}
	
		\begin{corollary}[Convergence\index{convergence} of discrete geodesic\index{discrete geodesic} paths]
		Under the assumptions of \cref{thm:SobolevExistence,thm:SobolevMoscoConv},
		given a sequence of minimizers of $\pathenergy_{c_A,B}^K$ (\ie\ of discrete geodesics) indexed by $K$, any subsequence contains a converging subsequence, and its limit minimizes $\pathenergy_{c_A,B}$ (\ie\ is a shortest geodesic).
	\end{corollary}
	
	\begin{proof}
		This is a standard consequence of the Mosco convergence if one can show that the energies $\pathenergy_{c_A,B}^K$ are equi-coercive \cite[Thm.\,1.21]{Br02}.
		Now let $\cpath$ be a minimizer of $\pathenergy_{c_A,B}$ (which exists by \cref{thm:SobolevExistence}).
		By the Mosco convergence there exists a sequence $\cpath^K$ with $\pathenergy^K_{c_A,B}[\cpath^K]\to\pathenergy_{c_A,B}[\cpath]<\infty$ as $K\to\infty$,
		thus $\pathenergy^K_{c_A,B}[\cpath^K]$ is uniformly bounded for $K$ large enough. Consequently there exists $C>0$ with
		\begin{equation*}
			C\geq\pathenergy^K_{c_A,B}[\cpath^K]\geq\pathenergy[\cpath^K]
		\end{equation*}
		for $K$ sufficiently large.
		Using \cref{lem:coercivityBoundednessMetric} we obtain that $\|\dot\cpath^K\|_{L^2_tW^m_\theta}$ is uniformly bounded,
		which implies the existence of a subsequence weakly converging in $W_t^1W^{m}_\theta$, \ie\ the desired equicoercivity.
	\end{proof}
	
	\begin{remark}[Convergence of discrete geodesics in the space of nonparametric curves]
		By \cref{thm:reparamOrbit}, the previous results also hold if $B$ is replaced with the reparameterization orbit $C_B$ of a curve $c_B$.
	\end{remark}
	
	\subsection{Space discretization} \label{sec:SpaceDiscretization}
	For a fully practical scheme, we still require discretization in the $\theta$-variable.
	We represent closed curves by the Fourier coefficients of their coordinate functions up to a maximum absolute frequency
	(which corresponds to the reciprocal of the spatial discretization scale).
	In more detail, a spatially discretized curve $c:\Sone\to\R^d$ is parametrized by
	\begin{equation}\label{eqn:FourierRepresentationCurve}
		c(\theta)=\sum_{j=0}^Na_j\cos(j\theta)+b_j\sin(j\theta)
	\end{equation}
	for coefficients $a_j,b_j\in\R^d$, where as before $\theta\in\R$ parametrizes $\Sone$ by arc-length.
	The spatial integration to approximate $\WEps{}[\hat c,\check c]$ from \eqref{eqn:curveSpaceW} or $\WEpsFree[\hat c,\check c]$ from \eqref{eqn:curveSpaceW2} is then performed via trapezium rule quadrature with $M$ equispaced points on $\Sone$.
	By the Euler--Maclaurin summation formula \cite[Sec.\,7.1]{BrPe11}
	\begin{multline}\label{eqn:EulerMaclaurin}
		\int_0^{2\pi}f(x)\,\d x
		-\frac{2\pi}M\left[\frac{f(0)+f(2\pi)}2+\sum_{j=1}^{M-1}f\left(\frac{2\pi j}M\right)\right]\\
		=-\sum_{l=1}^p\left(\frac{2\pi}M\right)^{l}\frac{B_l}{l!}\left(f^{(l-1)}(2\pi)-f^{(l-1)}(0)\right)+R
	\end{multline}
	with $|R|\leq\frac C{M^p}\int_0^{2\pi}|f^{(p)}(x)|\,\d x$ for some constant $C>0$,
	this quadrature rule exhibits rapid convergence for periodic functions $f$.
	
	Concerning the implementation, note that the derivatives of $\WEps{}[\hat c,\check c]$ and $\WEpsFree[\hat c,\check c]$ with respect to $\hat c$ and $\check c$ are of a similar form as $\WEps$ and $\WEpsFree$ and thus are evaluated via the same quadrature.
	The spatially discretized energy $\WEps{}[\hat c,\check c]$ and $\WEpsFree[\hat c,\check c]$ is expressed as a function of the Fourier coefficients of the discretized curves $\hat c,\check c$,
	as are its derivatives with respect to these Fourier coefficients.
	An off-the-shelf optimization routine can then be used to minimize the geodesic energy.
	
	The fully discrete energy $\Gamma$-converges to the true path energy as the number of time steps, Fourier coefficients, and quadrature points tends to infinity.
	Also, its minimizers will converge to minimizers of the continuous path energy.
	To show this let
	\begin{equation*}\textstyle
		C^N=\left\{c:\Sone\to\R^d\,\middle|\,c(\theta)=\sum_{j=0}^Na_j\cos(j\theta)+b_j\sin(j\theta)\,a_j,b_j\in\R^d,\,j=0,\ldots,N\right\}
	\end{equation*}
	denote the (vector) space of curves of the form \eqref{eqn:FourierRepresentationCurve}, let
	\begin{equation}\label{eqn:quadrature}
		\W^M[\hat c,\check c]
		\text{ be the trapezium rule quadrature of }\W[\hat c,\check c]
		\text{ for }\W\in\{\WEps,\WEpsFree\}
	\end{equation}
	with $M$ equispaced quadrature points on $\Sone$,
	and set $$\Pathenergy^{K,M}[(c_0,\ldots,c_K)]=K\sum_{k=1}^K\W^M[c_{k-1},c_k].$$
	Also, denote the Fourier truncation operator by
	\begin{equation*}
		T_N:W^{m,2}(\Sone,\R^d)\to C^N,\quad
		\sum_{j=0}^\infty a_j\cos(j\theta)+b_j\sin(j\theta)\mapsto\sum_{j=0}^Na_j\cos(j\theta)+b_j\sin(j\theta),
	\end{equation*}
	and abbreviate $c^N=T_Nc$ for $c\in W^{m,2}(\Sone,\R^d)$ as well as $B^N=\{c^N\,|\,c\in B\}$ for $B\subset W^{m,2}(\Sone,\R^d)$.
	Further, for a path $\cpath\in W_t^1W_\theta^m$ we write $T_N\cpath$ for the Fourier truncation applied pointwise at each time.
	The fully discrete path functional can then be defined as
	\begin{equation*}
		\pathenergy_{c_A,B}^{K,N,M}[\cpath]
		\coloneqq\begin{cases}
			\Pathenergy^{K,M}[(\cpath(\frac0K),\ldots,\cpath(\frac KK))]
			&\text{if }\cpath=\interpolation^K[\cpath]=T_{N}\cpath\in\mathcal{P}_{c_A^N,B^N},\\
			\infty&\text{else.}
		\end{cases}
	\end{equation*}
	Before stating the convergence result, let us note that it poses no restriction to the previous theory and in particular to \crefrange{thm:SobolevExistence}{thm:SobolevMoscoConv}
	if one assumes the set $B$ to be bounded in $W^{m,2}(\Sone,\R^d)$.
	Indeed, paths $\cpath$ of finite energy, say $\pathenergy[\cpath]<C<\infty$, can connect $c_A$ only to a bounded subset of $B$ anyway by \cref{lem:coercivityBoundednessMetric}.
	Therefore, we may in the following additionally assume boundedness of $B$, which becomes important in the context of spatial discretization due to aliasing effects:
	For instance, if $c_A(\theta)=(\cos\theta,\sin\theta,0)$ and $B=\{c_A+c\,|\,c(\theta)=(0,0,\sin(n\theta)),\,n\in\N\}$, then we have $c_A^N=c_A$ and $c_A\in B^N$ for all $N$, but $c_A\notin B$.
	Thu,s the shortest geodesic between $c_A^N$ and $B^N$ is the constant one (with zero velocity) for any spatial discretization $N$, while the true continuous geodesic between $c_A$ and $B$ is completely different.
	
	\begin{theorem}[Mosco convergence\index{Mosco convergence} of fully discretized path energy\index{fully discretized path energy}]\label{thm:MoscoConvFullyDiscrete}
		Let the assumptions of \cref{thm:SobolevExistence} hold with $B$ bounded,
		let $\W\in\{\WEps[\epsilon_K],\WEpsFree\}$ in \eqref{eqn:quadrature} (in the former case $\epsilon_K\to0$ while in the latter we set $m=2$ and $\epsilon_K=1$),
		and let $K,N_K,M_K\to\infty$.
		If $\frac{K^{2/p}N_K}{\epsilon_KM_K}\to0$ for some $p>0$,
		then $\pathenergy_{c_A,B}^{K,N_K,M_K}$ Mosco converges on $W_t^1W_\theta^m$ to $\pathenergy_{c_A,B}$.
	\end{theorem}
	\begin{proof}
		We first argue that by slight modifications in the proof of \cref{thm:SobolevMoscoConv} we have Mosco convergence also of $\pathenergy_{c_A^{N_K},B^{N_K}}^K$ to $\pathenergy_{c_A,B}$.
		Indeed, in the $\liminf$-inequality one now needs to consider a weakly converging sequence $\cpath^K\rightharpoonup\cpath^*$ with $\cpath^K\in\mathcal{P}_{c_A^{N_K},B^{N_K}}$,
		and this again implies $\cpath^*\in\mathcal P_{c_A,B}$ (while the remainder of the argument stays unchanged):
		In detail, we have $\cpath^*(0)=c_A$ since both $\cpath^K(0)\rightharpoonup\cpath^*(0)$ and $\cpath^K(0)=c_A^{N_K}\to c_A$.
		Furthermore, $\cpath^K(1)=c_K^{N_K}$ for some sequence $c_K\in B$,
		which by the boundedness of $B$ converges weakly (up to a subsequence) to some $c\in B$. Consequently,
		\begin{multline*}
			(\cpath^K(1),\phi)_{W^m_\theta}
			=(c_K,\phi)_{W^m_\theta}-((\mathrm{id}-T_{N_K})c_K,\phi)_{W^m_\theta}\\
			=(c_K,\phi)_{W^m_\theta}-(c_K,(\mathrm{id}-T_{N_K})\phi)_{W^m_\theta}
			\to(c,\phi)
		\end{multline*}
		for any $\phi\in W^{m,2}(\Sone,\R^d)$ (due to $(\mathrm{id}-T_{N_K})\phi\to0$)
		so that $\cpath^*(1)=c\in B$ as desired.
		In the $\limsup$-inequality one just has to replace the employed recovery sequence $\cpath^K$ by $T_{N_K}\cpath^K$,
		which (by the same argument, only replacing $\interpolation^K$ with $T_{N_K}\circ\interpolation^K$) again strongly converges to $\cpath$ (while the remainder of the argument stays unchanged).
		
		After these preliminaries
		it is sufficient to show that for any sequence $(\cpath^K)_{K=1,2,\ldots}$ converging weakly in $W_t^1W_\theta^m$ with finite energy $\pathenergy_{c_A,B}^{K,N_K,M_K}[\cpath^K]$ to some $\cpath\in W_t^1\immersion^m$ we have
		\begin{equation*}
			\left|\pathenergy_{c_A,B}^{K,N_K,M_K}[\cpath^K]-\pathenergy_{c_A^{N_K},B^{N_K}}^K[\cpath^K]\right|\to0
			\qquad\text{as }K\to\infty.
		\end{equation*}
		Indeed, then the liminf-inequality for the Mosco convergence of $\pathenergy_{c_A^{N_K},B^{N_K}}^K$
		directly implies the one for $\pathenergy_{c_A,B}^{K,N_K,M_K}$ (except in case $\cpath\notin W_t^1\immersion^m$, for which we will provide an extra argument at the end),
		and the above-provided recovery sequence for the Mosco convergence of $\pathenergy_{c_A^{N_K},B^{N_K}}^K$
		then automatically also is a recovery sequence for $\pathenergy_{c_A,B}^{K,N_K,M_K}$.
		
		Hence, let $\cpath^K=\interpolation^K[\cpath^K]$ converge weakly in $W_t^1W_\theta^m$ to some $\cpath$ with $\pathenergy_{c_A,B}[\cpath]<\infty$,
		and assume $c_k^K\coloneqq \cpath^K(k/K)\in C^{N_K}$ for $k=0,\ldots,K$.
		By compact embedding we have $\cpath^K\to\cpath$ weakly in $C_t^0W_\theta^m$ and strongly in $C_t^0C_\theta^{m-1}$ (see proof of \cref{thm:SobolevExistence}).
		Due to $\cpath\in W_t^1\immersion^m$, $|\cpath'|$ and thus also $|(\cpath^K)'|$ are uniformly bounded away from zero.
		We will denote the integrand of $\W[\hat c,\check c]$ by $I_{\hat c,\check c}$.
		By the classical estimate \eqref{eqn:EulerMaclaurin} for the trapezium rule on periodic functions, we obtain
		\begin{equation*}
			\left|\W^{M_K}[c_{k-1}^K,c_k^K]-\W[c_{k-1}^K,c_k^K]\right|
			\leq\frac C{M_K^p}\int_{\Sone}|I_{c_{k-1}^K,c_k^K}^{(p)}(\theta)|\,\,\d\theta,
		\end{equation*}
		where $C>0$ is independent of $M_K$ or $p$.
		Now by \eqref{eqn:curveSpaceW} or \eqref{eqn:curveSpaceW2}, $I_{c_{k-1},c_k}$ is a polynomial in $c_i^{(j)}$, $i=k-1,k$ and $j=0,2,\ldots,m$,
		whose coefficients depend smoothly (recall that $|(\cpath^K)'|$ is bounded away from zero) but nonlinearly on $c_{k-1}',c_k'$.
		More explicitly, $I_{c_{k-1},c_k}$ has the form
		\begin{equation*}
			I_{c_{k-1},c_k}
			=\sum_{j=1}^Jf_j(c_{k-1}',c_k')\prod_{l=1}^{L_j}c_{i_{jl}}^{(n_{jl})}
		\end{equation*}
		for some $J\geq1$, smooth coefficient functions $f_j$, polynomial degrees $L_j\geq2$, indices $i_{jl}\in\{k-1,k\}$,
		and derivative orders $0\leq n_{j1}\leq n_{j2}\leq\ldots\leq n_{jL_j}\leq m$ (ordered by size) with $n_{j,L_j-2}\leq m-1$ (since the polynomial is at most quadratic in the $m$th derivatives).
		Thus, abbreviating $L$-dimensional multiindices of order $p$ by $M_{L,p}=\{\alpha\in\N_0^L\,|\,\alpha_0+\ldots+\alpha_L=p\}$,
		\begin{align*}
			I_{c_{k-1},c_k}^{(p)}
			&=\sum_{j=1}^J\sum_{\alpha\in M_{L_j,p}}f_j(c_{k-1}',c_k')^{(\alpha_0)}\prod_{l=1}^{L_j}c_{i_{jl}}^{(n_{jl}+\alpha_{l})}\qquad\text{ and }\\
			\|I_{c_{k-1},c_k}^{(p)}\|_{L^1}\!
			&\leq\sum_{j=1}^J\sum_{\alpha\in M_{L_j,p}}\!\!\|f_j(c_{k-1}',c_k')^{(\alpha_0)}\|_{L^\infty}\!\!\!
			\prod_{l=1}^{L_j-2}\!\!\|c_{i_{jl}}^{(n_{jl}+\alpha_{l})}\|_{L^\infty}\!\!\!\!
			\prod_{l=L_j\!-\!1}^{L_j}\!\!\!\!\|c_{i_{jl}}^{(n_{jl}+\alpha_{l})}\|_{L^2}.
		\end{align*}
		Now for trigonometric polynomials of order $N_K$ it holds
		\begin{equation*}
			\|c_{i_{jl}}^{(n_{jl}+\alpha_{l})}\|_{L^\infty}
			\leq N_K^{\alpha_{l}}\|c_{i_{jl}}^{(n_{jl})}\|_{L^\infty}
			\qquad\text{and}\qquad
			\|c_{i_{jl}}^{(n_{jl}+\alpha_{l})}\|_{L^2}
			\leq N_K^{\alpha_{l}}\|c_{i_{jl}}^{(n_{jl})}\|_{L^2}
		\end{equation*}
		by Bernstein's inequality for trigonometric polynomials and by Parseval's theorem.
		Thus we obtain
		\begin{align*}
			\|I_{c_{k-1},c_k}^{(p)}\|_{L^1}\!\!
			&\leq\sum_{j=1}^J\sum_{\alpha\in M_{L_j,p}}\!\!\!\!\!\!N_K^{\alpha_1+\ldots+\alpha_{L_j}}\|f_j(c_{k-1}',c_k')^{(\alpha_0)}\|_{L^{\!\infty}}\!\!\!\!
			\prod_{l=1}^{L_j-2}\!\!\|c_{i_{jl}}^{(n_{jl})}\|_{L^{\!\infty}}\!\!\!\!
			\prod_{l=L_j-1}^{L_j}\!\!\!\!\|c_{i_{jl}}^{(n_{jl})}\|_{L^2}\\
			&\lesssim\sum_{j=1}^J\sum_{\alpha_0=0}^pN_K^{p-\alpha_0}\|f_j(c_{k-1}',c_k')^{(\alpha_0)}\|_{L^\infty}
		\end{align*}
		(where $A\lesssim B$ means $A\leq CB$ for some constant $C>0$ depending only on $p$ and the sequence $\cpath^K$),
		since the norms in the products are uniformly bounded (recall that $\cpath^K$ is uniformly bounded in $C_t^0W_\theta^m$ and $C_t^0C_\theta^{m-1}$ and that $n_{j,L_j}\leq m$ and $n_{j,L_j-2}\leq m-1$).
		Now $f_j(c_{k-1}',c_k')^{(\alpha_0)}$ is a finite sum of terms
		$$f_j^{(\alpha)}(c_{k-1}',c_k')\left(\left(c_{k-1}^{(1+\beta_1)},c_k^{(1+\beta_1)}\right),\ldots,\left(c_{k-1}^{(1+\beta_\alpha)},c_k^{(1+\beta_\alpha)}\right)\right)$$
		with $0<\alpha\leq\alpha_0$ and $\beta_1+\ldots+\beta_\alpha=\alpha_0$, each of which is again bounded in $L^\infty$ by
		\begin{multline*}
			\|f_j^{(\alpha)}(c_{k-1}',c_k')\|_{L^\infty}\prod_{l=1}^\alpha\left\|\left(c_{k-1}^{(1+\beta_l)},c_k^{(1+\beta_l)}\right)\right\|_{L^\infty}\\
			\leq\|f_j^{(\alpha)}(c_{k-1}',c_k')\|_{L^\infty}N_K^{\alpha_0}\|(c_{k-1}',c_k')\|_{L^\infty}^\alpha
			\lesssim\|f_j^{(\alpha)}(c_{k-1}',c_k')\|_{L^\infty}N_K^{\alpha_0}.
		\end{multline*}
		In case of $\W=\WEpsFree$ (in which we set $\epsilon_K=1$) the norm of $f_j^{(\alpha)}$ is bounded independent of $K$,
		in case of $\W=\WEps[\epsilon_K]$ it is bounded up to a constant factor by $\epsilon_K^{1-\alpha}$, see \eqref{eqn:derivativeBound}.
		Thus, summarizing, we obtain
		\begin{equation*}
			\|I_{c_{k-1},c_k}^{(p)}\|_{L^1}
			\lesssim\epsilon_K^{1-p}N_K^p.
		\end{equation*}
		Consequently, we have
		\begin{equation*}
			\left|\pathenergy_{c_A,B}^{K,N_K,M_K}[\cpath^K]\!-\!\pathenergy_{c_A^{N_K},B^{N_K}}^K[\cpath^K]\right|
			\!\leq\!K\!\sum_{k=1}^K\left|\W^{M_K}[c_{k\!-\!1}^K,c_k^K]\!-\!\W[c_{k\!-\!1}^K,c_k^K]\right|
			\!\lesssim\!\frac{K^2N_K^{p}}{\epsilon_K^{p\!-\!1}M_K^p},
		\end{equation*}
		which for large enough $p$ tends to zero as desired.
		
		We still need to consider the case $\cpath\in W_t^1W_\theta^m\setminus W_t^1\immersion^m$, in which $\pathenergy_{c_A,B}[\cpath]=\infty$.
		More specifically, we need to show that then $\pathenergy_{c_A,B}^{K,N_K,M_K}[\cpath^K]$ diverges (so that the liminf inequality is satisfied).
		To this end, let $S\subset[0,1]$ be the set of times $t$ at which $\cpath'(t,\cdot)$ is not bounded away from zero and set $s=\min S$
		(note that $S$ is compact and $s>0$ since the initial curve $\cpath(0,\cdot)$ is an immersion).
		Then,
		\begin{equation*}
			\lim_{r\nearrow s}\int_0^rg_{\cpath(t)} (\dot{\cpath}(t),\dot{\cpath}(t)) \,\d t=\infty  ,
		\end{equation*}
		since a finite limit would contradict \cref{prop:SobCurveUniform}.
		Now consider our fully discrete energy, only restricted to the time interval $[0,r]$,
		that is, the sum of the time discretization only contains time steps up to time $r$ (if $r$ is not a discrete time point, the last summand only contributes partially).
		Let us denote this restricted discrete energy by $\pathenergy^{K,N_K,M_K}_{[0,r]}$.
		By the previous argument, for every $r<s$ we know that $\pathenergy^{K,N_K,M_K}_{[0,r]}$ Mosco converges to the restriction of the path energy to the same interval, thus
		\begin{equation*}
			\liminf_{K\to\infty}\pathenergy^{K,N_K,M_K}[\cpath^K]
			\geq\liminf_{K\to\infty}\pathenergy^{K,N_K,M_K}_{[0,r]}[\cpath^K]
			\geq\int_0^rg_{\cpath(t)} (\dot{\cpath}(t),\dot{\cpath}(t)) \,\d t ,
		\end{equation*}
		and letting $r\to s$ leads to the desired result.
	\end{proof}
	
	\begin{remark}[Quadrature points versus Fourier modes]
		If one picks $N_K=K$, then it obviously suffices that $M_K=K^{\beta}$ for any $\beta>1$ (and $\epsilon_K$ decreasing sufficiently slowly in case of $\W=\WEps[\epsilon_K]$),
		which implies a mild computational cost increase.
		Had we instead discretized the curves via splines of degree $r$ with $N_K$ equispaced knots,
		then in the above proof we could only form the $(r-m-1)$th derivative of the integrand $I_{c_{k-1},c_k}$ and thus by an analogous argument would have to choose $p=r-m-1$.
		If one picks, for instance, $m=2$ and quartic B-splines (thus $r=4$) and $N_K\sim K$, then one requires $M_K$ to grow faster than $K^3$.
		It is here that the spectral discretization shows its great advantage.
	\end{remark}
	
	As before, the $\Gamma$- or Mosco convergence implies convergence of minimizers.
	However, to this end, we will restrict the path energy to some norm ball
	$$B_R(0)=\{\cpath\in W_t^1W_\theta^m\,|\,\|\cpath\|_{W_t^1W_\theta^m}\leq R\}.$$
	This is no limitation since we may simply choose $R$ larger than the norm of a geodesic between $c_A$ and $B$
	in order to obtain convergence of our numerically computed discrete geodesics to a shortest continuous geodesic.
	Whether $R$ was indeed chosen large enough can be checked a posteriori
	(recall that by \cref{lem:coercivityBoundednessMetric} any computed upper bound on the Riemannian distance between $c_A$ and $B$ implies a feasible radius $R$)
	or even a priori if one can provide a finite energy path from $c_A$ to $B$.
	Such a restriction becomes necessary since the equicoercivity of the spatially discretized path energy without this restriction is not easily obtained
	and does not seem worth the effort of a proof in view of the above discussion.
	Equicoercivity may in fact hold true; for instance, one can show that the numerical quadrature of the lowest order term in $\WEps{}$ or $\WEpsFree$ scales like the exact integral
	by exploiting that the piecewise linear interpolation of a nonnegative $ M$th-order trigonometric polynomial $P$ between $N\gg M$ equispaced quadrature points is smaller than $P$ by at most a constant factor.
	However, there are inherent difficulties with the higher-order terms, which can be illustrated as follows:
	Imagine that the Riemannian metric \cref{def:SobolevMetricCurves} only contains the $m$th order term
	(in fact, the other terms can be replaced by a Poincar\'e inequality if, for instance, the curve barycenter is fixed).
	Further consider a (discrete, \ie\ piecewise linear and trigonometric) path $\cpath^K$ with $(\cpath^K)'(t,\theta)=0$,
	thus $\W[\cpath^K(\frac{k-1}K),\cpath^K(\frac kK)]=\infty$ for some $k$.
	The numerical quadrature will be finite, though, $\W^{M_K}[\cpath^K(\frac{k-1}K),\cpath^K(\frac kK)]\leq C<\infty$.
	Even if the number of quadrature points is increased to $M\gg M_K$ one will have $\W^{M}[\cpath^K(\frac{k-1}K),\cpath^K(\frac kK)]<\infty$,
	and by rescaling the path $\cpath$ with a sufficiently large factor $\lambda$ one will even obtain $\W^{M}[\lambda\cpath^K(\frac{k-1}K),\lambda\cpath^K(\frac kK)]\leq C$,
	since the $m$th order term of the path energy is homogeneous of the negative order $3-2m$.
	
	\begin{corollary}[Convergence\index{convergence} of fully discrete geodesic\index{discrete geodesic} paths]
		Under the assumptions of \cref{thm:SobolevExistence,thm:MoscoConvFullyDiscrete} and for some closed norm ball $B_R(0)$,
		given a sequence of minimizers of $\pathenergy_{c_A,B}^{K,N_K,M_K}$ on $B_R(0)$ (\ie\ of spatiotemporally discretized geodesics for $\W \in \{\WEps,\WEpsFree\}$), any subsequence contains a weakly converging subsequence, and its limit minimizes $\pathenergy_{c_A,B}$ on $B_R(0)$ (\ie\ is a shortest geodesic).
	\end{corollary}
	\begin{proof}
		Following the previous proof, it is straightforward to check that the restriction of $\pathenergy_{c_A,B}^{K,N_K,M_K}$ to $B_R(0)$ still Mosco converges to the restriction of $\pathenergy_{c_A,B}$ to $B_R(0)$.
		The restrictions are equicoercive since $B_R(0)$ is precompact with respect to weak convergence, so that convergence of minimizers follows from \cite[Thm.\,1.21]{Br02}.
	\end{proof}

	\Cref{fig:convergenceOfGeodesics} shows the convergence of discrete geodesic paths for increasing $K$ with $\W \in \{\WEps,\WEpsFree\}$, $m=2$ and weight vector $(a_0,a_1,a_2)=(10^{-4}, 1, 10^{-2})$ in \cref{def:SobolevMetricCurves}.
	For the $\epsilon$-regularized time discretization we choose $\epsilon_K = \frac{1}{K}$.
	Start and end shapes are chosen as in \cref{fig:ExampleOfParallelTransport} bottom.
	As a surrogate for the continuous geodesic, we use a discrete geodesic $\cpath$ calculated using the $\epsilon$-free approach with $K = 2^{11}$. For the space discretization, $N = 80$ Fourier modes are used.
	We observe quadratic convergence in $\tau=\frac1K$ for $\W=\WEpsFree$
	(analogous to standard piecewise affine finite element discretizations for minimizers of the Dirichlet energy with fixed boundary data)
	and linear convergence for $\W=\WEps$
	(we validated experimentally that the error scales like $\frac1K+\epsilon$).
	As already exploited in the convergence proof and observed in our experiments, the error decreases superpolynomially in $M^{-1}$.
	Note that we pick $M\geq4N$ to sufficiently resolve the highest occurring Fourier frequencies.
	The number $N$ of Fourier modes is chosen larger than required to resolve the end shapes $\cpath(0)$ and $\cpath(1)$.
	For our choice of $N$, the error is apparently dominated by the time discretization error and the $\epsilon$-regularization in the case of $\W=\WEps$.
	For qualitative results of discrete geodesics, see \cref{fig:curveSpaceGeodesicWeights,fig:curveSpaceGeodesic}.
		\begin{figure}
				\input{figures/ConvergenceOfGeodesic_Wrat} 
		\input{figures/ConvergenceOfGeodesic_Wreg}
		\caption[Convergence of discrete geodesic paths]{Error between a discrete geodesic $\cpath^K$ ($\W=\WEpsFree$ on the left, $\W=\WEps$ on the right) computed for increasing $K$ and a reference geodesic $\cpath$ (computed with $N = 80, M = 640, K = 2^{11}$). We use $\mathrm{err}=\big(\frac{1}{K+1} \sum_{k=0}^K \| \cpath^K(\tfrac kK) - \cpath(\tfrac kK)\|_{W_\theta^r}^2\big)^{1/2}$ as a time-discrete $L_t^2W_\theta^r$-norm of the error with $r=0$ (dotted), $r=1$ (dashed), and $r=2$ (solid).}
		\label{fig:convergenceOfGeodesics}
	\end{figure}
			\section{Discrete exponential and differentials}\label{sec:discreteGeodCalc}
	The concept of discrete geodesics as minimizers of the discrete path energy \eqref{eqn:discretePathEnergy} can be developed further
	to define discrete versions of the exponential map, the covariant derivative, parallel transport, or curvature.
	Again, the convergence analysis from \cite{RuWi15} needs to be adapted to the setting of Sobolev curves, which is done in the following.
	
	\subsection{Discrete exponential map} \label{sec:expMap}
	If the discrete geodesic path $(c_0, \ldots, c_K)$ is the \emph{unique} minimizer of \eqref{eqn:discretePathEnergy} for fixed end points $c_0=c_A$ and $c_K=c_B$,
	we denote the discrete ($K$-step) logarithm of $c_B$ with respect to $c_A$ by
	\begin{equation*}
		\Log^{K}_{c_A}(c_B)\coloneqq\frac{c_1-c_0}{1/K},
	\end{equation*}
	an approximation of the Riemannian logarithm $\log_{c_A}(c_B)$, which is defined as the initial velocity of the connecting geodesic.
	Its inverse, if it exists, we denote as the discrete exponential map
	\begin{equation*}
		\Exp^{K}_{c_A}(1,v) \coloneqq [\Log^{K}_{c_A}]^{-1}(v)=c_K,
	\end{equation*}
	which approximates the Riemannian exponential $\exp_{c_A}(1\cdot v)$ with $K$ steps.
	It can be computed noting that any discrete geodesic must satisfy the Euler-Lagrange equations for minimizing \eqref{eqn:discretePathEnergy}, which read
	\begin{equation}
		\label{eq:ELSobolev}
		\W_{,2}[c_{k-1},c_{k}] + \W_{,1}[c_{k},c_{k+1}] =0
		\qquad
		\text{for $k=1,\ldots, K-1$.}
	\end{equation}
	Thus, given an initial curve $c_0=c_A \in \immersion^m$, a curve variation
	$v \in W^m_\theta$ and a time step size $\tau=\tfrac1K$, we set $c_1\coloneqq c_0 + \tfrac1K v$ and then iteratively define $c_{k+1}$ as the solution of \eqref{eq:ELSobolev} for $k = 1, \ldots, K-1$ to finally arrive at $c_K=\Exp^{K}_{c_{0}}(1,v)$.
	
	We begin with a study of the first step of this iterative scheme,
	\ie of $\Exp^{2}_{c_0}(1, v)$, defined as the solution $c_2$ of
	\begin{align}
		\label{eq:ELSobolevFirst}
		\W_{,2}[c_{0},c_{1}] + \W_{,1}[c_{1},c_{2}] =0
		\qquad\text{for }c_1=c_0+v/2.
	\end{align}
	If both $\Exp^{2}_{c_0}(1,v)$ and $\Log^{2}_{c_0}(c_2)$ are well-defined and in particular unique for sufficiently small $v$ and $c_2$ close to $c_0$, then
	\begin{align*}
		c_2 = \Exp^{2}_{c_0}(1, \Log^{2}_{c_0}(c_2))\text{ and  } v = \Log^{2}_{c_0}(\Exp^{2}_{c_0}(1, v)).
	\end{align*}
	The following lemma states the well-posedness and provides an approximation result for the discrete exponential map $\Exp^2$ and the discrete logarithm $\Log^2$.
	
	\begin{lemma}[Local existence and approximation properties of $\Exp^2$ and $\Log^2$]\label{thm:existenceExp2Sobolev}
		Let $g_c(\cdot,\cdot)$ be the Sobolev metric of order $m\geq 2$, let $\mathfrak K\subset\immersion^m$ be bounded and closed, and let $\W\in\{\WEps,\WEpsFree\}$
		(set $m=2$ and $\epsilon=1$ in the latter case, while we assume $\epsilon\leq\min_{c_0\in\mathfrak K,\theta\in\Sone}|c_0'(\theta)|$ in the former).
		Then there exist constants $C,\mathcal C>0$ depending on $\mathfrak K$ such that $\Exp^2_{c_0}(1,v)$ and $\Log^2_{c_0}(c_2)$ are well-defined for all $c_0\in\mathfrak K$ whenever $\|v\|_{W^m_\theta},\|c_2-c_0\|_{W^m_\theta}\leq C\sqrt\epsilon$,
		and
		\begin{align*}
			\|\Exp^2_{c_0}(1,v)-c_0-v\|_{W^m_\theta} &\leq\mathcal C(\Vert v\Vert_{W^m_\theta}^2+\Vert v\Vert_{W^m_\theta}^4/\epsilon^2) ,\\
			\|\Log^2_{c_0}(c_2)-(c_2-c_0)\|_{W^m_\theta} &\leq\mathcal C(\Vert c_2-c_0\Vert_{W^m_\theta}^2+\Vert c_2-c_0\Vert_{W^m_\theta}^4/\epsilon^2).
		\end{align*}
		Moreover, the derivatives satisfy
		\begin{align*}
			\|\partial_v\Exp^2_{c_0}(1,v)-\Id\|_{W^m_\theta} &\leq\mathcal C(\Vert v\Vert_{W^m_\theta}+\Vert v\Vert_{W^m_\theta}^3/\epsilon^2) ,\\
			\|\partial_{c_2}\Log^2_{c_0}(c_2)-\Id\|_{W^m_\theta} &\leq\mathcal C(\Vert c_2-c_0\Vert_{W^m_\theta}+\Vert c_2-c_0\Vert_{W^m_\theta}^3/\epsilon^2).
		\end{align*}
	\end{lemma}
	The proof relies on the implicit function theorem and is given in \cref{sec:appExp2}.
	Next, we consider the convergence of the discrete exponential to the continuous exponential.  
	\begin{theorem}[Convergence\index{convergence} of $\Exp^{K}$]\label{thm:ExpKSobolev}
		Let $g_c(\cdot,\cdot)$ be the Sobolev metric of order $m\geq 2$, let $\cpath:[0,1]\to\immersion^m$ be a smooth geodesic, and let $\W\in\{\WEps,\WEpsFree\}$
		(set $m=2$ in the latter case, while in the former we assume $\epsilon$ sufficiently small depending on $\cpath$).
		There exists a constant $\kappa>0$ depending on $\cpath$ such that the discrete $K$-exponential $\Exp^{K}_{\cpath(0)}(1, \dot\cpath(0))$ is well-defined for $K>\kappa/\epsilon$, and
		$$\Vert \cpath(1)-\Exp^{K}_{\cpath(0)}(1, \dot \cpath(0)) \Vert_{W^m_\theta} \leq\kappa\cdot\begin{cases}
			\frac1K&\text{if }\W=\WEpsFree,\\
			\epsilon+\frac1{\epsilon K}&\text{if }\W=\WEps{}.
		\end{cases}$$
	\end{theorem}
	For $\W=\WEps{}$ the optimal choice of $\epsilon$ is $\tfrac{1}{\sqrt{K}}$.
	The theorem is an analog of \cite[Thm.\,5.9]{RuWi15}, and we follow the proof with some adaptations, given in \cref{sec:appConvExp}.
		
	\Cref{fig:curveSpaceGeodesicWeights,fig:curveSpaceGeodesic} show examples of numerically computed discrete geodesics and extrapolation with the discrete exponential map.
	In \cref{fig:curveSpaceGeodesicWeights}, a comparison between different choices of weights from \cref{def:SobolevMetricCurves} is given.
	Computations are performed using the $\epsilon$-regularized time discretization $\W=\WEps$.
	The grey shapes are taken from \cite{BaBrHa17}.
	In \cref{fig:curveSpaceGeodesic}, the time discretization uses $\W=\WEps$ in the first two rows and $\W=\WEpsFree$ in the others.
	The grey shapes of the middle two rows are taken from the MPEG-7 Core Experiment CE-Shape-1, and
	the grey corpus callosum shapes of the fourth row stem from \cite{KuAlSo15}.
	The results with the respective alternative time discretization appear qualitatively identical.
	However, the benefit of the $\epsilon$-regularized time discretization is that it allows to consider Sobolev metrics with $m>2$.
	In \cref{fig:curveSpaceGeodesic3rdOrder} we show qualitative results for $m=3$ again with different weights.
	All experiments use $\epsilon = 10^{-2}$, $K = 128$, and $N = 50$, $M=200$ for the space discretization.
	
	\Cref{fig:convergenceExpMap} experimentally validates the convergence rate for the discrete exponential map proven in \cref{thm:ExpKSobolev}. For this experiment the discrete exponential map is computed at the unit circle $c = (\cos, \sin)$ in direction $v=(-\frac{\cos}{2}, \sin)$ using $\W=\WEps$ with $\epsilon = 1/\sqrt{K}$ (solid) and with $\W=\WEpsFree$ (dashed).
	The numerical approximation $\Exp^{K}_{c} (1, v)$ with $K=2^{13}$ and $\W=\WEpsFree$ serves as a surrogate for the ground truth.
	Weights, Fourier modes and quadrature points are chosen as $(a_0,a_1,a_2)=(10^{-4}, 1, 10^{-2})$, $N = 30$, and $M=120$.
	
	\begin{figure}\centering
		\includegraphics[width=1\linewidth]{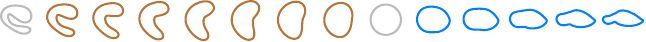}
		\includegraphics[width=1\linewidth]{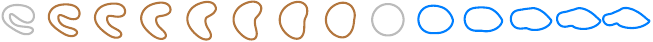}
		\includegraphics[width=1\linewidth]{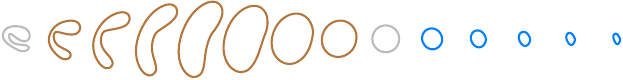}
		
		\caption[Discrete geodesics with different weights]{Discrete geodesics (orange) between the two grey curves, extrapolated beyond the second curve via the discrete exponential map (blue), computed with the weights
			$(a_0,a_1,a_2)=(10^{-4}, 1, 10^{-2})$, $(1, 1, 10^{-2})$, and $(10^{-4}, 1, 1)$ in \cref{def:SobolevMetricCurves} (top to bottom).}
		\label{fig:curveSpaceGeodesicWeights}
	\end{figure}
	
	\begin{figure}\centering
		\includegraphics[width=1\linewidth]{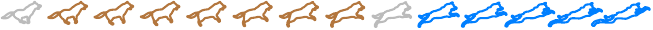}
		\includegraphics[width=1\linewidth]{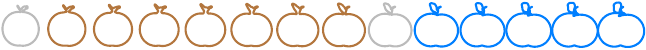}
		\includegraphics[width=1\linewidth]{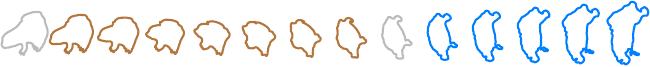}
		\includegraphics[width=1\linewidth]{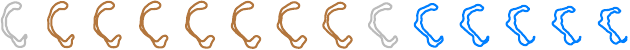}		\caption[Discrete geodesics between different shapes]{Discrete geodesics (orange) between the two grey curves, extrapolated beyond the second curve via the discrete exponential map (blue),
			using $(a_0,a_1,a_2)=(10^{-4}, 1, 10^{-2})$ in \cref{def:SobolevMetricCurves}.}
		\label{fig:curveSpaceGeodesic}
			\end{figure}
	
	\begin{figure}
		\includegraphics[width = \linewidth]{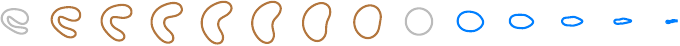}
		\includegraphics[width = \linewidth]{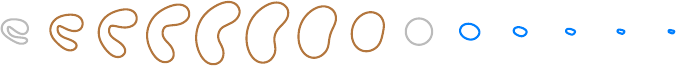}
		\includegraphics[width = \linewidth]{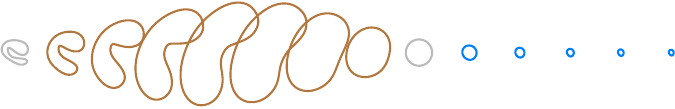}
		\caption[Discrete geodesics for the Sobolev metric of order 3]{Discrete geodesics (orange) between the two grey curves, extrapolated beyond the second curve via the discrete exponential map (blue) using the Sobolev metric of order $m=3$ with weights
			$(a_0,a_1,a_2,a_3)=(10^{-4}, 1, 10^{-2}, 10^{-4})$, $(10^{-4}, 1, 10^{-2}, 10^{-3})$, and $(10^{-4}, 1, 10^{-2}, 10^{-2})$ in \cref{def:SobolevMetricCurves} (top to bottom).}
		\label{fig:curveSpaceGeodesic3rdOrder}
	\end{figure}

	\begin{figure}
				\input{figures/ConvergenceOfExp.tex}
		\caption[Convergence of the discrete exponential map.]{ \label{fig:convergenceExpMap} Convergence of the discrete exponential map for $\W=\WEps$ with $\epsilon = 1/\sqrt{K}$ (solid) and $\W=\WEpsFree$ (dashed).}
	\end{figure}
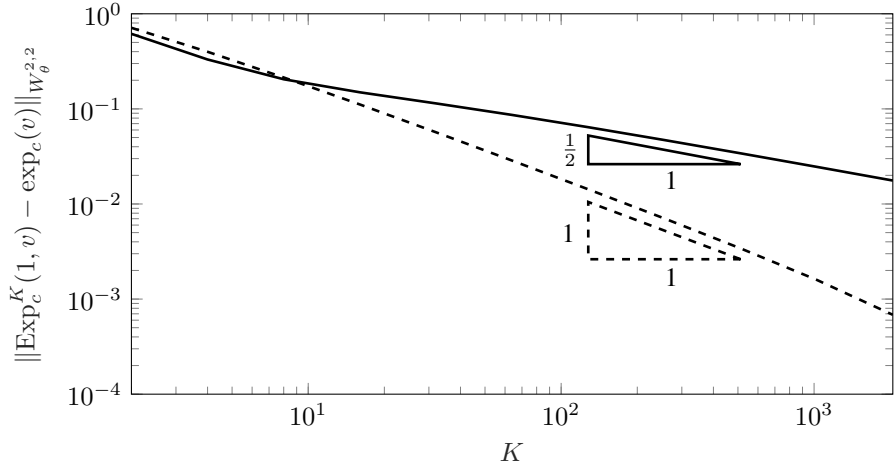
	
					\subsection{Discrete parallel transport and covariant derivative} \label{sec:PTCovDeriv}
	A  vector field $w \colon \immersion^m \to W^{m}_\theta $ is parallel along $\cpath:\R\to\immersion^m$ if it solves the differential equation
	\begin{align}\label{eq:paralleltransport}
		\tfrac\d{\d t}(w\circ\cpath)(t) = - \Gamma_{\cpath(t)}(\w\circ\cpath(t),\dot \cpath(t))
	\end{align}
	with initial data $w\circ\cpath(0) \in W^{m}_\theta$. Here, $\Christoffeloperator$ is the Christoffel operator $\Christoffeloperator: \immersion^m \times W^{m}_\theta \times W^{m}_\theta\to W^{m}_\theta, \, (c, v, w) \mapsto \Christoffel{c}{v}{w}$ defined by
	\begin{align} \label{eq:ChristoffelOperator}
		\g{c}{\Christoffel{c}{v}{w}}{z} = \frac{\Dg{c}{w}{v}{z} - \Dg{c}{z}{v}{w} + \Dg{c}{v}{w}{z}}{2}.
	\end{align}
	The associated parallel transport\index{parallel transport} map $\parTp_{\cpath(\tau \leftarrow 0)}$ from time $0$ to $\tau\in\R$, defined as
	\begin{align} \label{eq:Ptransdef}
		\parTp_{\cpath(\tau \leftarrow 0)} w\circ\cpath(0) \coloneqq w\circ\cpath(\tau),
	\end{align}
	maps initial data $w\circ\cpath(0) \in W^{m}_\theta$ to the vector $w\circ\cpath(\tau)\in W^{m}_\theta$ and is a linear isomorphism from $W^{m}_\theta$ to $W^{m}_\theta$.
	In fact, the inverse of $\parTp_{\cpath(\tau \leftarrow 0)}$ is given by the parallel transport map $\parTp_{\cpath(0 \leftarrow \tau)}$ solving the same differential equation with initial data $w\circ\cpath(\tau)$, thus $\parTp_{\cpath(\tau  \leftarrow 0)}^{-1} = \parTp_{\cpath(0 \leftarrow \tau)}$. Here, we use the notation $\cpath(r \leftarrow s)$ for the path segment $[0,1]\ni t \mapsto \cpath((1-t)s+tr)$.
	
	More generally, the covariant derivative $\cov (\w (t))$ of a vector field $\w=w\circ\cpath$ along the curve $\cpath$ is defined by
	\beqn\label{eq:cov2}
	\cov \w(t) = \dot \w(t) + \Christoffel{\cpath(t)}{\w(t)}{\dot \cpath(t)}.
	\eeqn 
	By definition $\cov \w(t) = 0$ if $\w$ is parallel along $\cpath$.
	If it is not clear from the context, we indicate the direction $v$ of the covariant differentiation, here $v=\dot c$,
	explicitly by $\covdir{v}$.
	If we consider a smooth vector field $w \colon \immersion^m \to W^{m}_\theta$ we denote the covariant derivative of $w$ at $c \in \immersion^m$ in direction $v$ by $ \left(\covdir{v} w \right)(c) \coloneqq \cov \w(0) $ with $\w(t)= w \circ \cpath(t)$ the vector field evaluated along $\cpath(t) = c + t v$. To further simplify the notation, we omit the parentheses and write $\covdir{v} w (c)$. The covariant derivative $\covdir{v} w (c)$ can also be written as the limit of difference quotients using the inverse parallel transport along $\cpath(t) = c + tv$,
	\beqn \label{eq:covdiffquot}
	\covdir{v} w(c) = \lim_{\tau \to 0} \frac{\parTp_{\cpath(\tau \leftarrow 0)} w(\cpath(\tau)) - w(\cpath(0))}{\tau}.
	\eeqn
	
	There is a well-known first-order approximation of parallel transport called \emph{Schild's ladder}\index{Schild's ladder} (\cite{EhPiSc72,KhMiNe00}) which is based on the construction of a sequence of geodesic parallelograms. This scheme can be easily transferred to discrete paths based on the discrete logarithm and the discrete exponential introduced before.
	We set $\ParTp_{c,c + \tau v} \tau w = z - c - \tau v$ where $z=\Exp^2_c(1,2(s-c))$ is the fourth corner of the discrete geodesic parallelogram $(c, c + \tau v, c + \tau w, z)$  with center $s=\Log^2_{c+\tau w}(c+\tau v)/2$, \ie $z$ and $s$ solve
	\begin{align}\label{eq:ELeqTranspCurves}
		\begin{split}
			\W_{,2}[c + \tau w, s] + \W_{,1}[s,c+\tau v] &=0,\\
			\W_{,2}[c,s] + \W_{,1}[s,z] &=0.
		\end{split}
	\end{align}
	The discrete parallel transport $\ParTp^K_{c_0,\ldots,c_K} w_0$ of $w_0$ along the discrete path $(c_0,c_1, \ldots, c_K )$ can then be iteratively defined as
	\begin{align*}
		w_{k+1} &= \tfrac{1}{\tau}(\ParTp_{c_k,c_k + \tau v_k} \tau w_{k})
	\end{align*}
	where $\tau=1/K$ and $v_k = \tfrac{c_{k+1} - c_{k}}{\tau}$. We then set $\ParTp^K_{c_0,\ldots,c_K} w_0 \coloneqq w_K$.
	
	\Cref{eq:covdiffquot} enables the definition of a discrete covariant derivative using Schild's ladder approximation of inverse parallel transport. We have $\ParTp_{c,c + \tau v}^{-1} \tau w(c + \tau v) \coloneqq \yz(\tau) - c$ 
	where $\yz(\tau)$ is the fourth corner of the discrete geodesic parallelogram $(c, c+ \tau v,c + \tau v+\tau w(c + \tau v),\yz(\tau))$ with center $\yc(\tau)$, which is illustrated in \cref{fig:CovariantSetup}.  Here $\yz(\tau)$ and $\yc(\tau)$ satisfy
	\begin{align}\label{eq:ELeqInverseTranspCurves}
		\begin{split}
			\W_{,2}[\yz(\tau),\yc(\tau)] + \W_{,1}[\yc(\tau),c+\tau v] &=0,\\
			\W_{,2}[c,\yc(\tau)] + \W_{,1}[\yc(\tau),c + \tau v+ \tau w(c + \tau v)] &=0.
		\end{split}
	\end{align}
	Given the inverse parallel transport, we approximate the covariant derivative of a vector field $w$ in direction $v$ via a one-sided discrete difference quotient
	\begin{align} \label{eq:CovPlus}
		\Covtdir{\tau}{v} w(c)\coloneqq\frac{ \ParTp_{c,c+\tau v}^{-1}  \tau w(c+\tau v) -  \tau w(c)}{\tau^2}
	\end{align}
	or via a central covariant difference quotient
	\begin{align}  \label{eq:CovPlusMinus}
		\Covtdir{\pm\tau}{v} w(c)\coloneqq\frac{ \ParTp_{c,c+\tau v}^{-1}  \tau w(c+\tau v) +  \ParTp_{c,c-\tau v}^{-1}\left( -\tau w(c-\tau v)\right)}{2 \tau^2}. 
	\end{align}
	We next show the consistency of this discrete covariant derivative by adapting and generalizing \citep[Thm.\,4.1]{EfHeRu22}:
	Apart from transferring the approach to the space of Sobolev curves and treating the additional $\epsilon$-regularization,
	we also allow small deviations in the direction of the covariant derivative,
	since this will be useful to subsequently prove the convergence of the discrete parallel transport.
	
	\begin{theorem}[Consistency of covariant difference quotients\index{consistency covariant difference quotients}]
		\label{lemma:consistencyCurveApproxDir}
		Let $g_c(\cdot,\cdot)$ be the Sobolev metric of order $m\geq 2$, $w\in C^3(\immersion^m,W^{m}_\theta)$, and let $\mathfrak K\subset\immersion^m$ be bounded and closed.
		Let $\W\in\{\WEps,\WEpsFree\}$ (with $m=2$, $\epsilon=1$ in the latter case, while we assume $\epsilon\leq\min_{c\in\mathfrak K,\theta\in\Sone}|c'(\theta)|$ in the former).
		Then there exist constants $C,\mathcal C>0$ depending on $\mathfrak K,w$ such that $\Covtdir{\tau}{v}w(c)$ is well-defined for all $c\in\mathfrak K$ whenever $\|v\|_{W^m_\theta},\tau/\sqrt\epsilon\leq C$,
		and
		\begin{equation*}
			\left\|\Covtdir{\tau}{v}  w (c) - \covdir{u} w (c) \right\|_{W^m_\theta}
			\leq \mathcal C\begin{cases}
				\tau + \Vert u-v \Vert_{W^m_\theta}&\text{if }\W=\WEpsFree,\\
				\tau + \Vert u-v \Vert_{W^m_\theta} + \epsilon&\text{if }\W=\WEps,
			\end{cases}
		\end{equation*}
		for all $u\in W^m_\theta$ as well as, if $w$ is even four times differentiable,
		\begin{equation*}
			\left\| \Covtdir{\pm \tau}{v}  w (c) - \covdir{u} w (c) \right\|_{W^m_\theta}
			\leq \mathcal C\begin{cases}
				\tau^2 + \Vert u-v \Vert_{W^m_\theta}&\text{if }\W=\WEpsFree,\\
				\tau^2 + \Vert u-v \Vert_{W^m_\theta} +\epsilon&\text{if }\W=\WEps.
			\end{cases}
		\end{equation*}
	\end{theorem}
	
	The proof is given in \cref{sec:appCovDeriv}.
	A consequence is the consistency of discrete parallel transport as a discretization of the differential equation of continuous parallel transport,
	which together with the stability of the discrete evolution operator $ \ParTp_{c_k,c_k + \tau v_k} $ implies the convergence of the discrete parallel transport, proved in \cref{sec:appPT}.
	
	\begin{theorem}[Convergence\index{convergence}\index{discrete parallel transport} of $\ParTp^K_{c_0,\ldots,c_K}$]\label{thm:PT}
		Let $\cpath:[0,1]\to \immersion^m$ be a smooth path,
		and let $\w:[0,1]\to  W^{m}_\theta$ be a parallel vector field along $\cpath$.
		Let $K\in\N$, $\tau=\frac1K$, and $c_k=\cpath(k\tau)$ for $k=1,\ldots,K$.
		For $\W\in\{\WEps,\WEpsFree\}$ (with $m=2$, $\epsilon=1$ in the latter case, while in the former we assume $\epsilon$ sufficiently small depending on $\cpath,\w$)
		there exists a constant $\kappa>0$ depending on $\cpath,\w$ such that the discrete parallel transport $\ParTp^K_{c_0,\ldots,c_K} \w(0)$ is well-defined for $K>\kappa/\epsilon$, and
		$$
		\|\ParTp^K_{c_0,\ldots,c_K} \w(0)-\w(1)\|_{W_{\theta}^{m,2}}\leq\kappa\cdot\begin{cases}\tau&\text{if }\W=\WEpsFree,\\\tau+\epsilon&\text{if }\W=\WEps.\end{cases}
		$$
	\end{theorem}
	
		\begin{figure}
		\centering
		\begin{tikzpicture}[scale=0.5, >={Triangle[length=2pt 9,width=1pt 3]}]
						\tikzstyle{every node}=[minimum size=0pt,inner sep=0pt];
			
						\node (y1) at (2,4) {};
			\node (y2) at (7,5) {};
			\node[left] at (y2.west) {$c$\ \ };
			\node[label=south:$c+\tau v$] (y3) at (12,4.5) {};
			\node (y4) at (15,5) {};
			
						\draw[solid, line width=0.5mm] (y2)-- (y3);
			\draw[fill] (y2) circle [radius=0.1];
			\draw[fill] (y3) circle [radius=0.1];
			
						\node (w2) at (6.7,9.7) {};
			\node[label=east:\ $c + \tau v + \tau w( c + \tau v)$] (w3) at (13,9) {};
			\node[black] at (15.2,8.5) {};
			\draw[->, solid, line width=0.4mm, black] (y3) -- (w3);
			\node (Pw1a) at (6,9) {};
			\node (Pw3) at (8,9.5) {};
			\node[black] at (9.37,9.8) {$\yz(\tau) = c + \ParTp_{c, c+ \tau v}^{-1} \tau w(c + \tau v )$};
						\draw[fill, black!50] (Pw3) circle [radius=0.1];
			\draw[->, solid, line width=0.4mm, black!50] (y2) -- (Pw3);
						\node (c2) at (10.5,7) {};
			\node[black] at (10.6,7.5) {$\yc(\tau)$};
			\draw[fill, black] (c2) circle [radius=0.1];
			\draw[dashed, line width=0.5mm, black!50, opacity=0.3] (y2) -- (c2) -- (w3);
			\draw[dashed, line width=0.5mm, black!50, opacity=0.3] (Pw3) -- (c2) -- (y3);
			\draw[fill] (y2) circle [radius=0.1];
			\draw[fill] (y3) circle [radius=0.1];
			\draw[fill] (w3) circle [radius=.1];
			
		\end{tikzpicture}
		\caption{A sketch of the configuration considered in \cref{lemma:consistencyCurveApproxDir} for $\ParTp_{c,c + \tau v}^{-1}  \tau w( c + \tau v)$.}
		\label{fig:CovariantSetup}
	\end{figure}
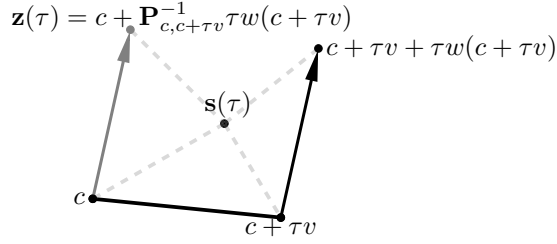
		
	We experimentally validate the convergence rate proven in \cref{lemma:consistencyCurveApproxDir} for $\W=\WEps$ and different choices of $\epsilon$. For this we compute the covariant derivative of the constant vector field $w = (\cos, -\frac{\sin}{2})$ at $c = (\cos, \sin)$ in direction $v = (-\frac{\cos}{2}, \sin)$.  In this case $\frac{D_v}{dt} w(c) = \Christoffel{c}{w}{v}$, and we can compute an analytic solution as ground truth, see \cref{example:ChristoffelWeighted} in \cref{subsec:explicit}. We use weights $(a_0,a_1,a_2)=(10^{-4}, 1, 10^{-2})$ in \cref{def:SobolevMetricCurves} and $N=20$ Fourier modes as well as $M=80$ quadrature points for the space discretization. The result is shown in \cref{fig:convergenceofcovdivcentral}. As expected, for $\epsilon= \tau$ (dotted) we obtain first order convergence while for $\epsilon = \tau^{1/2}$ (solid) the convergence rate is $O(\tau^{1/2})$. For $\epsilon= 64 \tau^{3/2}$ the $O(\epsilon) = O(\tau^{3/2})$ term dominates in the beginning while for small $\tau$ we observe the $O(\tau)$ convergence rate.
	For the central covariant derivative, we observe the same convergence rates as the $O(\epsilon)$ error dominates the convergence.
	
	We show two examples for discrete parallel transport along a discrete geodesic in \cref{fig:ExampleOfParallelTransport}. The geodesics are computed using $\W=\WEpsFree$, weights $(a_0,a_1,a_2)=(10^{-4}, 1, 10^{-2})$, $N=30$ Fourier modes, $M=120$ quadrature points, and $K=2^{10}$ time steps. For the actual discrete parallel transport
	we used $\W=\WEps$ with $\epsilon = \frac{1}{K}$.
	To visualize the parallel transport, we show every 256th curve of the path and represent the transported vector $w_k$ with
	blue arrows and $\Exp^{K}_{c_k}(1, w_k)$ as blue curve.
	For parallel transport along a geodesic path, the inner product between the transported vector and the path velocity is constant. Indeed, this is approximately the case for our computed discrete parallel transport, as is shown on the right. We compute the inner product using the discrete approximation of the metric  $\alpha_k =  \frac{1}{2} \WEps_{,11}[c_k,c_k]\left(K(c_{k+1} - c_k), w_k \right)$ for $0 \leq k \leq K$, where $K(c_{k+1} - c_k)$ is the discrete path velocity. The graphs show the absolute value of the difference quotient $|K(\alpha_{k+1}-\alpha_k )|$ along the path. Indeed, we see that it is approximately zero, and hence the inner product stays constant along the path.
	In \cref{fig:ExampleOfParallelTransport} top we transport the vector
	$w_0(\theta) = \frac{1}{5} \sin(5 \theta) v(\theta)$ with $v$ the outer normal of the unit circle; in \cref{fig:ExampleOfParallelTransport} bottom we transport $w_0(\theta) = \frac{1}{5} \tilde{v}(\theta)$ where $\tilde{v}$ is the 90 degree counter-clockwise rotation of the derivative of the parameterization of the initial shape $c_0$. 	
	In \cref{fig:ConvergenceOfParallelTransport} we show quantitative convergence results of discrete parallel transport for $\W=\WEps$ with different choices of $\epsilon$. The observed convergence rates support the theoretical result of \cref{thm:PT}. Space discretization and weights are chosen as before, start and end shape are chosen as in \cref{fig:ExampleOfParallelTransport} bottom, and $\w(0)(\theta)=\sin(5 \theta) v(\theta)$. As a reference approximation we compute $\w(1)$ using $\W=\WEpsFree$ with $K= 2^{13}$.
	
	\begin{figure}
		\input{figures/ConvergenceOfCovDiv_OneSided.tex}
				\caption[Convergence of the discrete covariant derivative]{Convergence of the discrete one-sided covariant derivative for $\W=\WEps$. The error is computed with respect to an analytically computed ground truth.
			The result for $\epsilon = \sqrt{\tau}$, $\epsilon = \tau$, and $\epsilon = 64\tau^{3/2}$ is shown as the solid, dotted, and dashed line respectively. \label{fig:convergenceofcovdivcentral}}
	\end{figure}
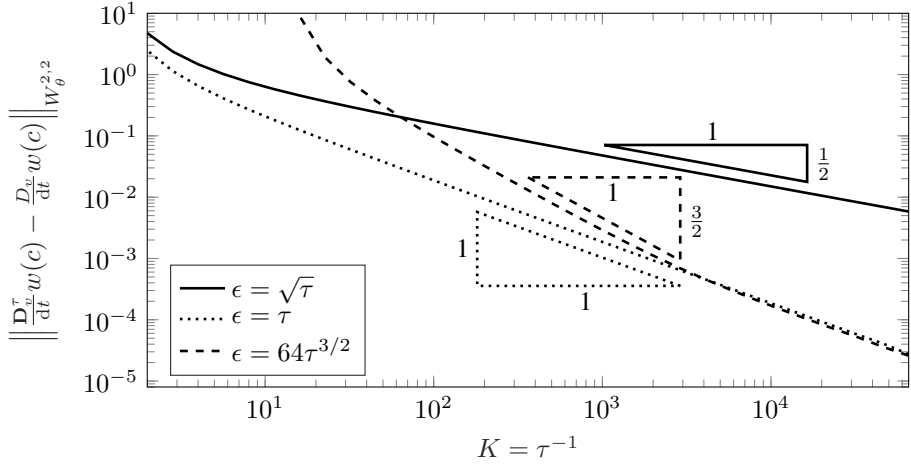
	
	\begin{figure}
		\raisebox{1\height}{ 			\includegraphics[width=0.5\linewidth]{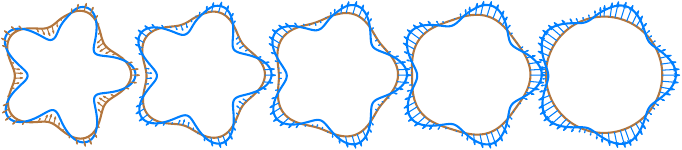} 
		}
		\hfill
				\input{figures/ParallelTransportExample=3_AngleChanges_W1only.tex}\\
		\raisebox{1\height}{ 			\includegraphics[width=0.5\linewidth]{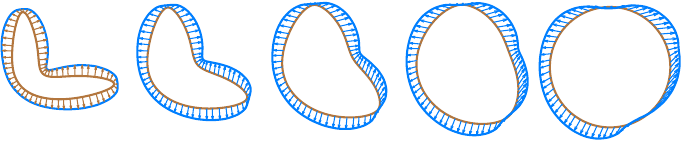} 
		}
		\hfill
		\input{figures/ParallelTransportExample=4_AngleChanges_W1only.tex}
		
		\caption[Numerically computed discrete parallel transport]{
			Two examples of discrete parallel transport along a discrete geodesic $c_0,\ldots,c_K$ (brown curves),
			showing for given initial vector $w_0$ (brown) the transported vector field $w_k$ (blue) and an exponential shooting in direction $w_k$ as the blue curves. The right plots show the absolute change of the (approximate) inner product $\alpha_k =  \frac{1}{2} \WEps_{,11}[c_k,c_k]\left(K(c_{k+1} - c_k), w_k \right)$ between the transported vector and the discrete path velocity, where a change of sign appears as a peak towards zero. As expected, the inner product is approximately constant.
		}
		\label{fig:ExampleOfParallelTransport}
	\end{figure}

	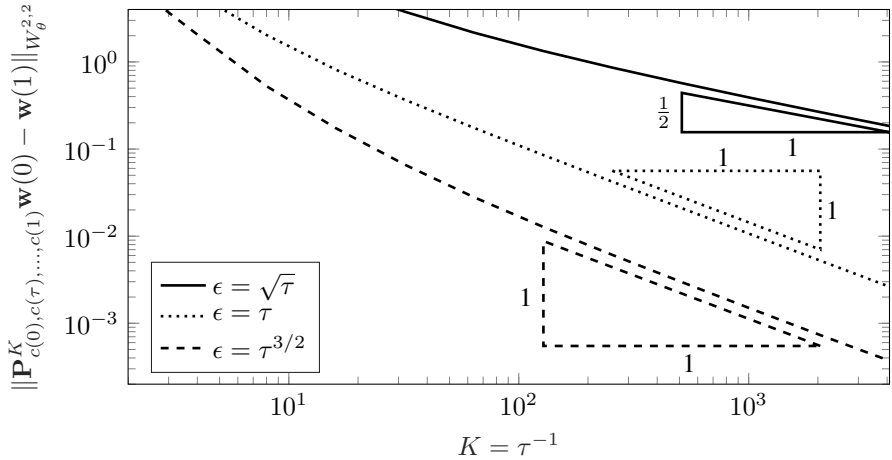
\begin{figure}
		\centering
		\input{figures/ConvergenceOfParallelTransport.tex}
		\caption[Convergence of discrete parallel transport]{Convergence of discrete parallel transport for $\W=\WEps$ and different choices of $\epsilon$.
		}
		\label{fig:ConvergenceOfParallelTransport}
	\end{figure}

					\subsection{Consistency of the discrete Riemann curvature tensor}\label{sec:consistencyCurvature}
	To discretize the Riemann curvature tensor $\riemann_c(v,w)$ for $c \in \immersion^m$
	and a pair of tangent vectors $v,w \in W_\theta^{m}$ we use the representation
	\begin{equation}\label{eqn:RiemannCurvature}
		\riemann_c(v,w) z = \left(  \covdir{v} \covdir{w} z -  \covdir{w} \covdir{v} z \right)(c)
	\end{equation}
	for all vector fields $z:\immersion^m\to W_\theta^m$.
	Since the Riemann curvature is independent of the particular choice of the vector field,
	we simply choose it to be constant.
	
	\begin{figure}
		\input{figures/ConvergenceOfOneSidedSectionalCurvature.tex}
		\input{figures/ConvergenceOfCentralSectionalCurvature.tex}
		\caption[Convergence of the discrete sectional curvature]{Convergence of the discrete sectional curvature approximation based on one-sided (left) and central covariant difference quotients (right) for the $\W=\WEps$ (solid) and $\W=\WEpsFree$ (dashed).
		}
		\label{fig:ConvergenceOfSectionalCurvature}
	\end{figure}
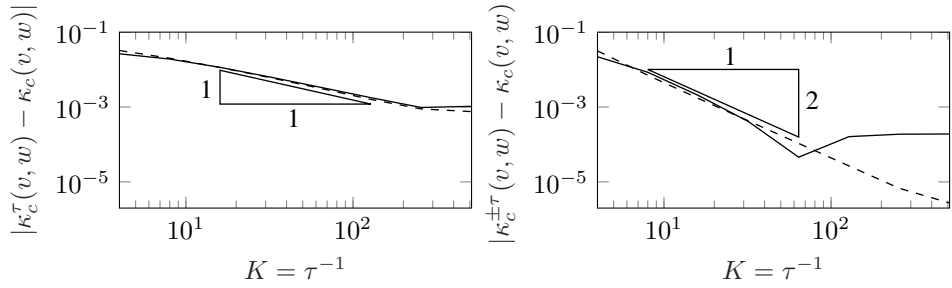
	The construction of a discrete counterpart on the Riemann curvature tensor and the consistency proof
	are analogous to \citep[Sec.\,5 \& Thm.\,5.1]{EfHeRu22},
	just adapted to the setting of Sobolev curves and potentially the $\epsilon$-regularization.
	We use nested covariant difference quotients, \ie 
	we approximate the Riemann curvature tensor based on the one-sided covariant difference quotient via
	\begin{align} \label{def:RiemannianCurvatureTensorOneSided}
		\Riemann^{\tau}_c(v,w)z \coloneqq \left(\Covtdir{\tau}{v} \Covtdir{\tau^\beta}{w} z - \Covtdir{\tau}{w} \Covtdir{\tau^\beta}{v} z \right)(c)
	\end{align}
	and based on the central covariant difference quotient via
	\begin{align} \label{def:RiemannianCurvatureTensorCentral}
		\Riemann^{\pm\tau}_c(v,w)z \coloneqq \left(\Covtdir{\pm\tau}{v} \Covtdir{\pm\tau^\beta}{w} z - \Covtdir{\pm\tau}{w} \Covtdir{\pm\tau^\beta}{v} z \right)(c)
	\end{align}
	for suitable $\beta>0$.
	In particular, we use a smaller step size for the inner difference quotient.
	If $\W=\WEps$, this will have some impact on the choice of $\epsilon$, and
	the first and second order approaches will use different values of $\epsilon$ depending on the step size $\tau$ that may differ between inner and outer discrete covariant differentiation.
	\begin{theorem}[Consistency of the discrete Riemann curvature tensor\index{consistency discrete curvature tensor}] \label{thm:consRiemCurvTensor}
				Let $c \in \immersion^m$, $g_c(\cdot,\cdot)$ be the Sobolev metric of order $m\geq 2$ and $\W\in\{\WEps,\WEpsFree\}$
		(set $m=2$ and $\epsilon=1$ in the latter case, while we assume $\epsilon$ sufficiently small in the former).
		There exist constants $C,\mathcal C>0$ such that for $\|v\|_{W_\theta^m},\|w\|_{W_\theta^m},\|z\|_{W_\theta^m},\tau/\sqrt{\epsilon_\out},\tau^\beta/\sqrt{\epsilon_\inn}\leq C$,
		\begin{align*}
			\left\|\Riemann^{\tau}_c(v,w)z-\riemann_c(v,w)z\right\|_{W_\theta^m} &\leq\mathcal C\cdot\begin{cases}\tau+\epsilon_\out+(\tau^\beta+\epsilon_\inn)/\tau&\text{if }\W=\WEps,\\\tau+\tau^{\beta-1}&\text{if }\W=\WEpsFree,\end{cases}\\
			\left\|\Riemann^{\pm\tau}_c(v,w)z-\riemann_c(v,w)z\right\|_{W_\theta^m} &\leq\mathcal C\cdot\begin{cases}\epsilon_\out+\epsilon_\inn/\tau&\text{if }\W=\WEps,\\\tau^2+\tau^{2\beta-1}&\text{if }\W=\WEpsFree,\end{cases}
		\end{align*}
		where $\epsilon_\inn$ is the value of $\epsilon$ in the inner and $\epsilon_\out$ in the outer covariant derivatives.
		For $\W=\WEps$, an optimal choice is $\beta=2$, $\epsilon_\out=\tau$, $\epsilon_\inn=\tau^2$ in the former case and $\beta=\frac32$, $\epsilon_\out=\tau^2/C$, $\epsilon_\inn=\tau^3/C$ in the latter.
	\end{theorem}
		The proof is given in \cref{sec:appCT} and relies on the consistency result for the discrete covariant derivative and the stability of discrete parallel transport.
	
	The sectional curvature can be defined as
	\begin{align}
		\kappa_{c}\left(v,w\right) \coloneqq \frac{ g_{c}\left(v, \riemann_{c}\left(v,w\right)w \right) }{
			g_{c}\left(v,v\right) g_{c}\left(w,w\right) - g_{c}\left(v,w\right)^2
		},
	\end{align}
	and we can define a discrete approximation as
	\begin{align}
		\kappa^{\tau}_{c}\left(v,w\right) \coloneqq \frac{ g_c\left(v, \Riemann^{\tau}_{c}\left(v,w\right)w \right) }{
			g_c\left(v,v\right) g_c\left(w,w\right) - g_c\left(v,w\right)^2
		}
	\end{align}
	or likewise $\kappa^{\pm\tau}_{c}\left(v,w\right)$ using $\Riemann^{\pm\tau}_{c}$.
	Analogously to \citep[Cor.\,5.3]{EfHeRu22} we obtain the following result.
	\begin{corollary}[Consistency of the discrete sectional curvature\index{consistency discrete sectional curvature}] \label{cor:consSecCurv}
		Under the assumptions of \cref{thm:consRiemCurvTensor} and for the provided optimal choice of $\beta,\epsilon_\inn,\epsilon_\out$
		we have
		\begin{align*}
			\kappa^{\tau}_{c}\left(v,w\right) &= \kappa_{c}\left(v,w\right) + O(\tau),\\
			\kappa^{\pm \tau}_{c}\left(v,w\right) &= \kappa_{c}\left(v,w\right) + O(\tau^2).
		\end{align*}
	\end{corollary}
	
	\Cref{fig:ConvergenceOfSectionalCurvature} experimentally validates the convergence rate from \cref{cor:consSecCurv} for the discrete approximation of the sectional curvature. 
	For this we numerically compute the sectional curvature $\kappa_c(v,w)$ at $c = (\cos, \sin)$ with $v = (\cos, 0)$, $w = (0, \cos)$ and compare to the exact value computed analytically in  \cref{example:SectionalCurvature}.
	Solid and dashed lines show the error for $\W=\WEps$ and $\W=\WEpsFree$, respectively, with $\beta,\epsilon$ chosen as in \cref{cor:consSecCurv}.
	We use weights $(a_0,a_1,a_2)=(1,1,1)$ as well as $N=20$ Fourier modes and $M=80$ quadrature points for the space discretization.
	Note that the error of the central difference quotient approximation for $\W=\WEps$ stagnates beyond $\tau\approx10^{-2}$.
	This is probably due to cancellation in terms such as $|\gamma|_\epsilon=\sqrt{\gamma^2+\epsilon^2}$ used in \eqref{eq:maxmin} for $\epsilon_{\inn}^2=\tau^6$.
	
		\appendix
	\section{Convergence proofs}\label{sec:appendix}
	Here we provide convergence proofs for the theorems stated in \cref{sec:discreteGeodCalc}.
	It will sometimes be useful to express the approximation $\W$ of the squared Riemannian distance as well as the metric $g_c$ as a quadratic form with nonlinear coefficients.
	To this end recall the formulation of the arclength derivative via a polynomial \eqref{eqn:arclengthDerivative} as well as the notation of \eqref{eqn:curveSpaceW2}. This implies
	\begin{equation*}
		\WEps{}[c,\tilde c] = \Geps[c,\tilde c](\tilde c-c,\tilde c-c)
		\qquad\text{and}\qquad
		\WEpsFree[c,\tilde c] = \GepsFree[c,\tilde c](\tilde c-c,\tilde c-c)
	\end{equation*}
	for the bounded, symmetric quadratic forms $\Geps[\hat c,\check c],\GepsFree[\hat c,\check c]: W^{m}_\theta \times W^{m}_\theta \to \R$ depending on $\hat c,\check c \in \immersion^m$ with
	\begin{multline}\label{eq:defGeps}
		\Geps[\hat c,\check c](u,v) \coloneqq \int_{\Sone}  \upperLength{\hat c'}{\check c'} (\theta)  u(\theta) \cdot v(\theta)
		+ \int_0^1 \sum_{j=1}^m  \left(\lowerLength{\hat c'}{\check c'} (\theta) \right)^{5-6j} \\
		\begin{split}
			& \cdot\! P_j\!\left(((1\!-\!t) \hat c'\!+\!t \check c',\ldots,(1\!-\!t) \hat c^{(j)}\!+\!t \check c^{(j)})(\theta);(u',\ldots, u^{(j)})(\theta)\right) \\
			& \cdot\! P_j\!\left(((1\!-\!t) \hat c'\!+\!t \check c',\ldots,(1\!-\!t) \hat c^{(j)}\!+\!t \check c^{(j)})(\theta);(v',\ldots, v^{(j)})(\theta)\right)
			\, \d t\, \d \theta,
		\end{split}
	\end{multline}
	\begin{multline}\label{eq:defGepsFree}
		\GepsFree[\hat c,\check c](u,v)
		\coloneqq \int_{\Sone} \tfrac{r+p}2u(\theta) \cdot v(\theta)\\
		+\tfrac{(\frac1v-1)(r+p)+(r-p)\log(\frac rp)}{r^2+p^2-2q}u'(\theta) \cdot v'(\theta)
		+ \tfrac12\left(\tfrac1{rq}+\tfrac1{pq}\right)u''(\theta) \cdot v''(\theta)\\
		-\Phi_1^T\Xi_1\Theta_1 (u'(\theta)\cdot v''(\theta)+u''(\theta)\cdot v'(\theta))
		\!+\!\Phi_2^T\Xi_2\Theta_2u'(\theta) \cdot v'(\theta)\, \d \theta.
	\end{multline}
	Similarly, the Riemannian metric can be expressed as
	\begin{equation}\label{eq:widehatGeps}
		g_c(u,v)=\GMetric[c,c](u,v)
		\qquad\text{with}\qquad
		\GMetric=\GepsFree
		\text{ or }
		\GMetric=\widehat\Geps,
	\end{equation}
	where $\widehat\Geps$ coincides with $\Geps$ except that
	$\upperLength{\hat c'}{\check c'} (\theta)$ is replaced with $\upperLength{\hat c'}{\check c'} (\theta)-\tfrac\epsilon2$
	and $\lowerLength{\hat c'}{\check c'} (\theta)$ with $\lowerLength{\hat c'}{\check c'} (\theta)+\tfrac\epsilon2$.
	Furthermore, we will denote by $L(X,Y)$ the normed space of linear maps between the normed spaces $X$ and $Y$.
	
	\subsection{Proof of well-posedness and approximation of $\Exp^2$, $\Log^2$ } \label{sec:appExp2}
	
	\begin{proof}[Proof of \cref{thm:existenceExp2Sobolev}]
		We begin with the discrete exponential map.
		Denoting $x=v/2=c_1-c_0$ and $y=c_2-c_0$ we define the nonlinear map
		$F: W^{m}_\theta \times W^{m}_\theta \to  (W^{m}_\theta)'$ with
		\begin{equation}
			\label{eq:FEL}
			F[x,y] \coloneqq \W_{,2}[c_{0},c_{0}+x] + \W_{,1}[c_{0}+x,c_{0}+y].
		\end{equation}
		Hence, solving \eqref{eq:ELSobolevFirst} for $c_2=\Exp_{c_0}^2(1,v)$ is equivalent to finding $y$
		for given $x$ such that $F[x,y]=0$, so we have phrased the computation of $\Exp_{c_0}^{2}$
		as a problem of implicit function theorem type.
		In detail, by definition of $\WEps$ and $\WEpsFree$ we know that $F[0,0]=0$, and we ask for a map $x\mapsto y[x]$ with $F[x,y[x]]=0$ in a neighbourhood of $x=0$.
		The implicit function theorem ensures the existence of such a map if $F$ is continuously differentiable
		and $\partial_y F[0,0]$ is invertible as a linear mapping in $L(W^{m}_\theta,(W^{m}_\theta)')$.
		Continuous differentiability, even smoothness of $F$ can readily be checked,
		as the integrands of both $\WEps{}[\hat c,\check c]$ and $\WEpsFree[\hat c,\check c]$ are smooth functions of all derivatives of $\hat c$ and $\check c$ up to order $m$,
		that in turn all lie in $L^\infty(\Sone,\R^d)$ except for $\hat c^{(m)},\check c^{(m)}\in L^2_\theta$ in which the integrands are quadratic.
		The invertibility of $\partial_y F[0,0]$ is also satisfied:
		For $\W=\WGalerkin$ it holds since $\WGalerkin$ satisfies \eqref{eqn:consistency} and the Riemannian metric $g_c$ is strong.
		For $\WEpsFree$ it then follows from \eqref{eqn:coercivityWEpsFree},
		and for $\WEps$ from \eqref{eqn:symmetry} and \eqref{eqn:coercivityWEps}.
		Thus we have proven that $\Exp_{c_0}^2(1,v)$ is well-posed for $v$ small enough.
		
		The implicit function theorem further implies that, if $F$ is twice continuously differentiable as in our case, then so is $y[x]$ with
		\begin{align*}
			\partial_xy[x](w)&=-\partial_yF[x,y]^{-1}\partial_xF[x,y](w),\\
			\partial_x^2y[x](w,\tilde w)&=-\partial_yF[x,y]^{-1}\big[\partial_x^2F[x,y](w,\tilde w)\\
			&\qquad+\partial_y\partial_xF[x,y](\partial_xy[x](w),\tilde w)+\partial_y\partial_xF[x,y](\partial_xy[x](\tilde w),w)\big].
		\end{align*}
		Consequently, abbreviating $S_x=\sup_{t\in[0,1]}\|\partial_x^2y[tx]\|_{L(W^m_\theta,(W^m_\theta)')}$, by Taylor's theorem we have
		\begin{align*}
			y[x]
			&=y[0]+\partial_xy[0](x)+\frac12\int_0^1\!\partial_x^2y[tx]((1-t)x,(1-t)x)\,\d t
			=\partial_xy[0](x)+e_x,\\
			\partial_xy[x]
			&=\partial_xy[0]+\int_0^1\!\partial_x^2y[tx]((1-t)x)\,\d t
			=\partial_xy[0]+\tilde e_x
		\end{align*}
		with remainders $\|e_x\|_{W^m_\theta}\leq S_x\|x\|_{W^m_\theta}^2/2$, $\|\tilde e_x\|_{W_\theta^m}\leq S_x\|x\|_{W^m_\theta}$.
		By \eqref{eqn:symmetry} it is straightforward to see $\partial_xF[0,0]=-2\partial_y F[0,0]$, thus $\partial_xy[0](w)=2w$ and
		\begin{equation*}
			\|y[x]-2x\|_{W^m_\theta}\!\leq\! S_x\|x\|_{W^m_\theta}^2
			\qquad\text{and}\qquad
			\|\partial_xy[x]-2\Id\|_{L(W^m_\theta,W^m_\theta)}\!\leq\! S_x\|x\|_{W_\theta^m}\\
		\end{equation*}
		or equivalently
		\begin{equation*}
			\|\Exp^2_{c_0}(1,v)-c_0-v\|_{W^m_\theta}\!\leq\!\tfrac{S_{v/2}}4\|v\|_{W^m_\theta}^2
			\quad\text{and}\quad
			\|\partial_v\Exp^2_{c_0}(1,v)-\Id\|_{W^m_\theta}\!\leq\!\tfrac{S_{v/2}}2\|v\|_{W^m_\theta}.
		\end{equation*}
		
		This almost finishes the proof, however, we still need to verify the existence of $C,\mathcal C>0$, independent of $c_0\in\mathfrak K$,
		such that $S_x\leq\mathcal C(1+\|x\|_{W^m_\theta}^2/\epsilon^2)$ and that the domain of $y[\cdot]$ contains the closed norm ball $B_r(0)$ of radius $r\coloneqq C\sqrt\epsilon/2$.
		To show the latter, consider the map $H_x: W^{m}_\theta \to  W^{m}_\theta$ with
		\begin{equation*}
			H_x[y] \coloneqq y- \partial_y F[0,0]^{-1} F[x,y],
		\end{equation*}
		which has $y[x]$ as a fixed-point.
		Assume we can show
		$\|H_x[0]\|_{W^m_\theta}\leq\kappa\|x\|_{W^m_\theta}$ for all $x\in B_r(0)$ and some $\kappa\geq1$
		as well as
		$\|\partial_yH_x[y]\|_{L(W^{m}_\theta,W^{m}_\theta)}\leq\frac12$ for all $x\in B_r(0)$ and $y\in B_{2\kappa r}(0)$ (the closed norm ball of radius $2\kappa r$), then
		\begin{equation*}
			\Vert H_x[y] - H_x[\tilde y] \Vert_{W^m_\theta} \leq
			\int_0^1 \Vert \partial_y H_x[\tilde y + t (y-\tilde y)] ( y-\tilde y)\Vert_{W^m_\theta} \, \d t
			\leq \Vert y-\tilde y\Vert_{W^m_\theta}/2
		\end{equation*}
		for all $y,\tilde y\in B_{2\kappa r}(0)$, thus $H_x$ is a contraction on $B_{2\kappa r}(0)$.
		Further,
		$$\|H_x[y]\|_{W^m_2}
		\leq\|H_x[0]\|_{W^m_\theta}+\|H_x[y]-H_x[0]\|_{W^m_\theta}
		\leq \kappa r+\|y\|_{W^m_\theta}/2
		\leq2\kappa r$$
		for all $y\in B_{2\kappa r}(0)$ so that $H_x$ maps the ball $B_{2\kappa r}(0)$ onto itself.
		Then Banach's fixed-point theorem implies the existence and uniqueness of $y[x]$ as desired.
		In fact, $H_x$ even maps the ball $B_{2\kappa\|x\|_{W^m_\theta}}$ onto itself so that we also know $\|y[x]\|_{W^m_\theta}\leq2\kappa\|x\|_{W^m_\theta}$.
		
		Now, we abbreviate $\ell\coloneqq\min_{c_0\in\mathfrak K,\theta\in\Sone}|c_0'(\theta)|$ and assume $R>0$ to be a radius such that $\min_\theta|c'(\theta)|>\ell/2$ for all $c_0\in\mathfrak K$ and $c\in B_R(c_0)$.
		To show the remaining estimates
		$$S_x\leq\mathcal C(1+\|x\|_{W^m_\theta}^2/\epsilon^2),\qquad \|H_x[0]\|_{W^m_\theta}\leq \kappa\|x\|_{W^m_\theta},\qquad \|\partial_yH_x[y]\|_{L(W^{m}_\theta,W^{m}_\theta)}\leq\tfrac12$$
		for all $c_0\in\mathfrak K$, $x\in B_r(0)$, $y\in B_{2\kappa r}(0)$,
		we assume, without loss of generality, that the sought constant $C$ is small enough
		such that $2\kappa r<R$ (in the case $\W=\WEps{}$ this should hold for the largest possible choice $\epsilon=\ell$ and thereby for all admissible $\epsilon$).
		This is possible since both $r$ and $\kappa$ are increasing in $C$ with $\kappa r\to0$ as $C\to0$.
		A direct consequence for the case $\W=\WEps{}$ is that ${\min}^\epsilon(|\hat c'(\theta)|,|\check c'(\theta)|)>{\min}^\epsilon(\ell/2,\ell/2)=\ell/2-\epsilon/2\geq0$ so that we may replace $\underline{\min}^\epsilon$ with ${\min}^\epsilon$ in $\WEps$.
		We further use that
		$\W{}[c,\tilde c] = G[c,\tilde c](\tilde c-c,\tilde c-c)$ with $G=\Geps$ or $G=\GepsFree$ defined in \eqref{eq:defGeps} and \eqref{eq:defGepsFree}.
		For $\tilde w\in W^{m}_\theta$ we thus have
		\begin{multline}\label{eq:ELinG}
			\W_{,2} [c_0, c_1](\tilde w) + \W_{,1} [c_1, c_2](\tilde w)\\
			= {G_{,2}}[c_0,c_1](\tilde w)(c_1-c_0,c_1-c_0)
			+ {G_{,1}}[c_1,c_2](\tilde w)(c_2-c_1,c_2-c_1)\\
			+ 2 G[c_0,c_1](c_1-c_0,\tilde w) - 2 G[c_1,c_2](c_2-c_1,\tilde w),
		\end{multline}
		which implies
		\begin{align*}
			F[x,y](\tilde w) &= G_{,2}[c_0,c_0\!+\!x](\tilde w)(x,x) \!+\! 2 G[c_0,c_0\!+\!x](x,\tilde w) \\
			&\quad +\! {G_{,1}}[c_0\!+\!x,c_0\!+\!y](\tilde w)(y\!-\!x,y\!-\!x) \!-\! 2 G[c_0\!+\!x,c_0\!+\!y](y\!-\!x,\tilde w).
		\end{align*}
		Differentiation with respect to $y$ in direction $w$ then gives
		\begin{align*}
			\partial_y F[x,y](\tilde w)(w)
			&= G_{,12}[c_0\!+\!x,c_0\!+\!y](\tilde w)(w)(y\!-\!x,y\!-\!x)  \\
			&\quad  +\! 2 {G_{,1}}[c_0\!+\!x,c_0\!+\!y](\tilde w)(y\!-\!x,w) \\
			&\quad  -\! 2 {G_{,2}}[c_0\!+\!x,c_0\!+\!y](w)(y\!-\!x,\tilde w) \!-\! 2 G[c_0\!+\!x,c_0\!+\!y](w,\tilde w),
		\end{align*}
		which evaluated at $(x,y)=(0,0)$ leads to $\partial_yF[0,0] = -2 G[c_0,c_0]$.
		Thus,
		\begin{align}
			H_x[0]&=(2G[c_0,c_0])^{-1}[G_{,2}[c_0,c_0+x](\cdot)(x,x) + 2 G[c_0,c_0+x](x,\cdot)\label{eqn:Hx}\\
			&\quad+ {G_{,1}}[c_0+x,c_0](\cdot)(x,x) + 2 G[c_0+x,c_0](x,\cdot)]\nonumber\\
			&=G[c_0,c_0]^{-1}[G_{,2}[c_0,c_0\!+\!x](\cdot)(x,x) \!+\! 2 G[c_0,c_0\!+\!x](x,\cdot)],\nonumber\\
			\partial_y H_x[y](w) &=  (2G[c_0,c_0])^{-1} \left(2G[c_0,c_0] + \partial_y F[x,y]\right)(\cdot)(w)  \label{eq:dyHeps} \\
			&= G[c_0,c_0]^{-1}
			\big(
			\tfrac12G_{,12}[c_0+x,c_0+y](\cdot)(w)(y-x,y-x)     \nonumber \\
			&\quad
			+ {G_{,1}}[c_0+x,c_0+y](\cdot)(y-x,w) \nonumber \\
			&\quad
			- {G_{,2}}[c_0+x,c_0+y](w)(y-x,\cdot)  \nonumber \\
			&\quad  + (G[c_0,c_0] -  G[c_0+x,c_0+y])(w,\cdot) \big) \nonumber .
		\end{align}
		In the case $\W=\WEps$ we further need to note that for ${\max}^\epsilon$ from \eqref{eq:maxmin} we have
		\begin{align*}
			\partial_\beta {\max}^\epsilon (\alpha, \beta) &=  \tfrac12 \left( 1 + \tfrac{\beta-\alpha}{\sqrt{(\beta-\alpha)^2+\epsilon^2}} \right)
			&&\in[0,1), \\
			\partial^2_\beta {\max}^\epsilon (\alpha, \beta) &=  \tfrac{1}{2\sqrt{(\beta-\alpha)^2+\epsilon^2}}
			\left(1-\tfrac{(\beta-\alpha)^2}{(\beta-\alpha)^2+\epsilon^2} \right)
			&&\in[0,\tfrac1{2\epsilon}),\\
			\partial^3_\beta {\max}^\epsilon (\alpha, \beta)
			&=\tfrac{-3(\beta-\alpha)}{2\sqrt{(\beta-\alpha)^2+\epsilon^2}^3}\left(1-\tfrac{3(\beta-\alpha)^2}{(\beta-\alpha)^2+\epsilon^2} \right)
			&&\in(-\tfrac2{\sqrt3\epsilon^2},\tfrac2{\sqrt3\epsilon^2}).
		\end{align*}
		Analogous estimates are obtained for $\underline{\min}^\epsilon$ instead of ${\max}^\epsilon$ and for derivatives with respect to $\alpha$ or mixed second and third derivatives.
		(Recall that we may assume $\underline{\min}^\epsilon={\min}^\epsilon$.)
		Due to $\min_\theta|c'(\theta)|\geq\ell/2$ on $B_R(c_0)$ there thus exists some constant $\gamma>0$, independent of $c_0\in\mathfrak K$, with
		\begin{equation}\label{eq:estimatesGeps}
			\left\{\!\!\!
			\begin{array}{r@{}l}
				\Vert G[c_0,c_0]^{-1} \Vert_{L(W^{m}_\theta,(W^{m}_\theta)')} &\leq\! \gamma, \\
				\Vert G[c_0+x,c_0+y] \Vert_{L(W^{m}_\theta,(W^{m}_\theta)')}
				&\leq\! \gamma, \\
				\Vert G_{,i}[c_0\!+\!x,c_0\!+\!y](\bar v) \Vert_{L(W^{m}_\theta,(W^{m}_\theta)')} &\leq\!
				\gamma \Vert \bar v\Vert_{W^m_\theta} , \\
				\Vert G_{,ij}[c_0\!+\!x,c_0\!+\!y](\bar v)(\hat v) \Vert_{L(W^{m}_\theta,(W^{m}_\theta)')} &\leq\!
				\gamma \epsilon^{-1} \Vert \bar v\Vert_{W^m_\theta}\! \Vert\hat v\Vert_{W^m_\theta},\\
				\Vert G_{\!,ijk}[c_0\!+\!x,c_0\!+\!y](\bar v)(\hat v)(\tilde v) \Vert_{L(W^{m}_\theta,(W^{m}_\theta)')} &\leq\!
				\gamma \epsilon^{-2} \Vert \bar v\Vert_{W^m_\theta}\! \Vert\hat v\Vert_{W^m_\theta}\! \Vert\tilde v\Vert_{W^m_\theta},\\
				\Vert G[c_0,c_0] \!-\!  G[c_0\!\!+\!x,c_0\!\!+\!y] \Vert_{L(W^{m}_\theta,(W^{m}_\theta)')}
				&\leq\!
				\gamma \left(\Vert x\Vert_{W^m_\theta} \!+\! \Vert y\Vert_{W^m_\theta} \right), \\
				\Vert G_{,i}[c_0,c_0 ](\bar v) \!-\!  G_{,i}[c_0\!\!+\!x,c_0\!\!+\!x](\bar v) \Vert_{L(W^{m}_\theta,(W^{m}_\theta)')} &\leq\!  \gamma (\Vert x\Vert_{W^m_\theta}\Vert \bar v \Vert_{W^m_\theta})
			\end{array}
			\right.\!
		\end{equation}
		for $i,j,k=1,2$ and $x,y\in B_R(0)$ (recall the convention $\epsilon=1$ if $\W=\WEpsFree$).
		Indeed, by \eqref{eqn:coercivityWEps} or \eqref{eqn:coercivityWEpsFree}, $G[c_0,c_0]$ coincides (up to a constant factor in the case $\W=\WEps$) with $g_{c_0}$,
		which is uniformly coercive on $\mathfrak K$ and thus implies the first estimate.
		The other estimates follow since $c_0+x,c_0+y$ lie in the bounded closed $R$-neighbourhood $B_R(\mathfrak K)$ of $\mathfrak K$
		and since the coefficients in $G[\hat c,\check c](\hat w,\check w)$ (and thus its derivatives) depend smoothly on the values and first $m-1$ derivatives of $\hat c$ and $\check c$,
		which are uniformly bounded on $B_R(\mathfrak K)$, as well as quadratically on $\hat c^{(m)},\check c^{(m)}$, which only multiply lower derivatives of $\hat w,\check w$.
		For the last estimate, we additionally use that $	\partial_\beta {\max}^\epsilon (\alpha, \beta)$ is independent of $\alpha$ and $\beta$ when evaluated at $\alpha = \beta$.
		Using these estimates in \eqref{eqn:Hx}-\eqref{eq:dyHeps},
		we achieve
		\begin{align*}
			\|H_x[0]\|_{W^m_\theta}
			&\leq\zeta(\|x\|_{W^m_\theta}^2+\|x\|_{W^m_\theta}),\\
			\Vert \partial_y H_x[y] \Vert_{L(W^{m}_\theta,W^{m}_\theta)}
			&\leq \zeta\Bigl( \epsilon^{-1} \Vert y-x\Vert_{W^m_\theta}^2
			+ \Vert y-x\Vert_{W^m_\theta}
			+ \Vert x\Vert_{W^m_\theta}
			+ \Vert y\Vert_{W^m_\theta}  \Bigr)
		\end{align*}
		for some $\zeta\geq1$.
		Hence, for the choice $\kappa=\zeta(R+1)$ and $C=1/(4\zeta(2\kappa+1))$ we obtain $r=\sqrt\epsilon/(8\zeta(2\kappa+1))$ as well as
		(assuming $2\kappa r<R$, otherwise reduce $C$ correspondingly)
		\begin{align*}
			\|H_x[0]\|_{W^m_\theta}
			&\leq\zeta(R\|x\|_{W^m_\theta}+\|x\|_{W^m_\theta})
			\leq\kappa\|x\|_{W^m_\theta},\\
			\Vert \partial_y H_x[y] \Vert_{L(W^{m}_\theta,W^{m}_\theta)}
			&\leq \zeta\left( \epsilon^{-1} (2\kappa r+r)^2 + (2\kappa r+r) + 2\kappa r+r\right)
			\leq\tfrac12
		\end{align*}
		for all $c_0\in\mathfrak K$, $x\in B_r(0)$, and $y\in B_{2\kappa r}(0)$, as desired.
		Analogously,
		\begin{align*}
			\partial_x F[x,y](\tilde w)(w)&=
			G_{,22}[c_0,c_0+x](\tilde w)(w)(x,x) + 2 {G_{,2}}[c_0,c_0+x](\tilde w)(x,w) \\
			& \quad + 2 {G_{,2}}[c_0,c_0+x](w)(x,\tilde w) + 2 G[c_0,c_0+x](\tilde w,w) \\
			& \quad + G_{,11}[c_0+x,c_0+y](\tilde w)(w)(y-x,y-x)   \\
			& \quad - 2{G_{,1}}[c_0+x,c_0+y](\tilde w)(w,y-x) \\
			& \quad -2 {G_{,1}}[c_0+x,c_0+y](w)(\tilde w,y-x)
						+ 2 G[c_0+x,c_0+y](w,\tilde w)
		\end{align*}
		implies $\|\partial_x F[x,y]\|_{L(W^m_\theta,(W^m_\theta)')}\leq\zeta(1+\epsilon^{-1}(\|x\|_{W^m_\theta}+\|y\|_{W^m_\theta})^2+\|x\|_{W^m_\theta}+\|y\|_{W^m_\theta})$
		and thus uniform boundedness of $\partial_x F[x,y]$ for $c_0\in\mathfrak K$, $x\in B_r(0)$, and $y\in B_{2\kappa r}(0)$, independent of $\epsilon$.
		Consequently, $\partial_xy[\cdot]$ is uniformly bounded on $B_r(0)$, independent of $\epsilon$ and $c_0\in\mathfrak K$.
		In the same way we can estimate $\partial_x^2F[x,y]$, $\partial_y^2F[x,y]$, and $\partial_x\partial_yF[x,y]$ to be bounded up to a constant factor
		by $1+\epsilon^{-2}(\|x\|_{W^m_\theta}+\|y\|_{W^m_\theta})^2+\epsilon^{-1}(\|x\|_{W^m_\theta}+\|y\|_{W^m_\theta})$.
		Therefore, exploiting that $\|y[x]\|_{W^m_\theta}\leq2\kappa\|x\|_{W^m_\theta}$,
		we get $\|\partial_x^2y[x]\|_{L(W^m_\theta,(W^m_\theta)')}\leq\mathcal C(1+\|x\|_{W^m_\theta}^2/\epsilon^2)$ for some $\mathcal C>0$ and thus the desired estimate for $S_x$.
		
		The arguments for the investigation of the discrete logarithm $\Log^2$ are completely analogous and
		again based on the implicit function theorem applied to
		the function $F[\cdot, \cdot]$.
		In fact, for the discrete logarithm we have
		$$
		\Log^2_{c_0}(c_2) = 2 (c_1-c_0) = 2x,
		$$
		where $x=x[y]$ solves $F[x,y] = 0$
		for $y=c_2- c_0$, so essentially the roles of $x$ and $y$ swap compared to the previous argument.
	\end{proof}
	
	\subsection{Proof of convergence of $\Exp^K$} \label{sec:appConvExp}
	\begin{proof}[Proof of \cref{thm:ExpKSobolev}.]
		Let $\cpath$ be the continuous and $(c_0,\ldots,c_K)$ the discrete geodesic with $\cpath(0)=c_0=c_A$ and $\cpath(1)=\exp_{c_A}(\dot\cpath(0))$ as well as $c_K=\Exp^K_{c_A}(1,\dot\cpath(0))$.
		Note that $\dot\cpath$ and $\ddot\cpath$ are bounded in $W^m_\theta$ by the smoothness of the geodesic.
		
		Let $N_r(\cpath)=\{c\in\immersion^m\,|\,\|c-\cpath(t)\|_{W^m_\theta}\leq r\text{ for some }t\in[0,1]\}$ denote the $r$-neighbourhood of $\cpath$
		and assume $r>0$ to be small enough such that $N_r(\cpath)$ is contained in a bounded Riemannian ball $B$ around $c_A$.
		We assume $K$ large enough such that $rK\geq\|\dot\cpath(t)\|_{W^m_\theta}$ for all $t\in[0,1]$.
		As a consequence, the linear interpolation between $\cpath(t)$ and $\cpath(\hat t)$ is contained in $N_r(\cpath)$ whenever $|t-\hat t|\leq1/K$.
		We further assume the discrete geodesic to lie in $N_{r/2}(\cpath)$ as well as $\|c_k-c_{k-1}\|_{W^m_\theta}<r/2$ for $k=1,\ldots,K$, both of which we will verify at the end of the proof for $K$ large enough.
		
		Recall from \eqref{eq:defGeps}-\eqref{eq:widehatGeps} that
		\begin{equation*}
			\W[\hat c,\check c]=G[\hat c,\check c](\check c-\hat c,\check c-\hat c)
			\quad\text{and}\quad
			g_c(v,w)=\GMetric[c,c](v,w)
		\end{equation*}
		for quadratic forms $G$ and $\GMetric$.
		Abbreviating $S=\{(\hat c,\check c)\in N_{r/2}(\cpath)\,|\,\|\hat c-\check c\|_{W^m_\theta}\leq r/2\}$,
		we will denote by $\bar C_0,\bar C_1,C_1,C_2,L_1$ bounds on $S$ for the difference between $G$ and $\GMetric$, for the first and second derivatives of $G$, and for the Lipschitz constant of the first derivative,
		\begin{equation}\label{eqn:constants}
			\left\{\begin{aligned}
				\bar C_0&=\sup_{(\hat c,\check c)\in S}\|(G-\GMetric)[\hat c,\check c]\|_{L(W^m_\theta,(W^m_\theta)')},\\
				\bar C_1&=\sup_{(\hat c,\check c)\in S,\,i\in\{1,2\}}\|(G_{,i}-\GMetric_{,i})[\hat c,\check c]\|_{L(W^m_\theta,L(W^m_\theta,(W^m_\theta)'))},\\
				C_1&=\sup_{(\hat c,\check c)\in S,\,i\in\{1,2\}}\|G_{,i}[\hat c,\check c]\|_{L(W^m_\theta,L(W^m_\theta,(W^m_\theta)'))},\\
				C_2&=\sup_{(\hat c,\check c)\in S,\,i,j\in\{1,2\}}\|G_{,ij}[\hat c,\check c]\|_{L(W^m_\theta,L(W^m_\theta,L(W^m_\theta,(W^m_\theta)')))},\\
				L_1&=\sup_{(\hat c,\check c)\in S,\,i,j\in\{1,2\}}\Vert G_{,i}[\hat c, \hat c ]- G_{,i}[ \check c , \check c ] \Vert_{L(W^m_\theta,L(W^m_\theta,(W^m_\theta)'))}/\Vert  \hat c - \check c \Vert_{W^m_\theta}.
			\end{aligned}\right.
		\end{equation}
		Exploiting \eqref{eq:estimatesGeps}, we will use that there exists a constant $\lambda>0$ depending on $\cpath$ and $r$ with
		\begin{equation}\label{eqn:constantsBounds}
			\left\{\begin{aligned}
				\bar C_0&=0,&\bar C_1&=0,&C_1&\leq\lambda,&C_2&\leq\lambda,&L_1&\leq\lambda&\text{if }\W=\WEpsFree,\\
				\bar C_0&\leq\lambda\epsilon,&\bar C_1&\leq\lambda\epsilon,&C_1&\leq\lambda,&C_2&\leq\lambda/\epsilon,&L_1&\leq\lambda&\text{if }\W=\WEps{}.
			\end{aligned}\right.
		\end{equation}

		In what follows, we denote $\tau = \tfrac1K$, $t_k=  k\tau$, $v_k = \tfrac{c_k-c_{k-1}}{\tau}$,  ${u_k = \tfrac{\cpath(t_k)-\cpath(t_{k-1})}{\tau}}$ and introduce the error quantities
		$e_k \coloneqq c_k-\cpath(t_k)$ for the position and $e^v_k\coloneqq \tfrac{e_k-e_{k-1}}{\tau} = v_{k}-u_{k}$ for the curve velocity.
		We proceed in three steps.
		
		\emph{Step 1 (reformulation of geodesic equation).} By \eqref{eq:ELSobolev} the discrete Euler--Lagrange equation reads
		\begin{align}\label{eq:ELExpDiscrete}
			0 &= \tfrac1{\tau^{2}}\left(\W_{,2}[c_{k-1},c_{k}](\psi_k) + \W_{,1}[c_{k},c_{k+1}](\psi_k)\right)\\
			&= \tfrac2{\tau} \left(G[c_{k-1},c_k] (v_k,\psi_k) - G[c_{k},c_{k+1}]  (v_{k+1},\psi_k) \right)\nonumber  \\
			& \quad + {G_{,2}}[c_{k-1},c_k](\psi_k) (v_k,v_k) + {G_{,1}}[c_{k},c_{k+1}](\psi_k)  (v_{k+1},v_{k+1})\nonumber
		\end{align}
		for all $\psi_k\in W^m_\theta$.
		Furthermore, in the remainder of this step, we will derive the following transformation of the continuous Euler--Lagrange equation \eqref{eq:geodesicIVP},
		\begin{align}\label{eq:ELExpCont}
			0 &= -2 \tfrac{\mathrm{d}}{\mathrm{d}t} \left. \left(g_{\cpath(t)} (\dot \cpath(t), \psi_k)\right)\right\vert_{t=t_k} + D_c\, g_{\cpath(t_k)} (\psi_k) (\dot \cpath(t_k),\dot \cpath(t_k)) 
			\\
			&= \!-\! \tfrac{2}{\tau} \!\left(g_{\cpath(t_{k})}(u_{k\!+\!1},\psi_k)\!-\!  g_{\cpath(t_{k\!-\!1})}(u_{k},\psi_k)\right)
									\!+\!  D_c g_{\cpath(t_{k})}(\psi_k) (u_{k},u_{k}) \!+\! O(\tau \Vert \psi_k\Vert_{W^m_\theta}) \nonumber \\
			&=- \tfrac{2}{\tau} \left(\GMetric[\cpath(t_{k}),\cpath(t_{k})](u_{k+1},\psi_k)-  \GMetric[\cpath(t_{k-1}),\cpath(t_{k-1})](u_{k},\psi_k)\right) \nonumber  \\
			&\quad +  \left({\GMetric_{,1}}[\cpath(t_{k}),\cpath(t_{k})]+{\GMetric_{,2}}[\cpath(t_{k}),\cpath(t_{k})]\right)(\psi_k) (u_{k},u_{k})
			+ O(\tau \Vert \psi_k\Vert_{W^m_\theta}) \nonumber \\
			&=- \tfrac{2}{\tau} \left(G[\cpath(t_{k}),\cpath(t_{k})](u_{k+1},\psi_k)-  G[\cpath(t_{k-1}),\cpath(t_{k-1})](u_{k},\psi_k)\right) \nonumber  \\
			&\quad +\!   \left(\!{G_{,1}}[\cpath(t_{k}),\cpath(t_{k})]\!+\!{G_{,2}}[\cpath(t_{k}),\cpath(t_{k})]\right)\!(\psi_k)\! (u_{k},u_{k})
			\!+\! O((\bar C_0\!+\!\bar C_1 \!+\! \tau) \Vert \psi_k\Vert_{W^m_\theta}\!) \nonumber \\
			&= - \tfrac{2}{\tau} \left(G[\cpath(t_{k}),\cpath(t_{k})](u_{k+1},\psi_k)-  G[\cpath(t_{k-1}),\cpath(t_{k-1})](u_{k},\psi_k)\right) \nonumber \\
			&\quad
			+  {G_{,1}}[\cpath(t_{k}),\cpath(t_{k})](\psi_k) (u_{k+1},u_{k+1}) + {G_{,2}}[\cpath(t_{k}),\cpath(t_{k})] (\psi_k)  (u_{k},u_{k}) \nonumber \\
			&\quad + O\left(\left(\bar C_0+\bar C_1+(1+C_1)\tau\right) \Vert \psi_k\Vert_{W^m_\theta}\right)
			. \nonumber
		\end{align}
		The second equality follows from the smoothness of the geodesic and the Riemannian metric, the third equality from the definition of $\GMetric$.
		The fourth equality uses the definition of $\bar C_1$ as well as the estimate
		\begin{align*}
			&\quad\tfrac1\tau\left[(\GMetric-G)[\cpath(t_{k}),\cpath(t_{k})](u_{k+1},\psi_k)-  (\GMetric-G)[\cpath(t_{k-1}),\cpath(t_{k-1})](u_{k},\psi_k)\right]\\
			&=(\GMetric\!-\!G)[\cpath(t_{k}),\cpath(t_{k})](\tfrac{u_{k+1}\!-\!u_k}\tau,\psi_k)
			+\tfrac{\left((\GMetric\!-\!G)[\cpath(t_{k}),\cpath(t_{k})]\!-\!(\GMetric\!-\!G)[\cpath(t_{k-1}),\cpath(t_{k-1})]\right)(u_k,\psi_k)}\tau\\
			&=O(\bar C_0\|\psi_k\|_{W^m_\theta})+O(\bar C_1\|\tfrac{\cpath(t_k)-\cpath(t_{k-1})}\tau\|_{W^m_\theta}\|\psi_k\|_{W^m_\theta})\\
			&=O((\bar C_0+\bar C_1)\|\psi_k\|_{W^m_\theta}),
		\end{align*}
		exploiting the boundedness of $(u_{k+1}-u_k)/\tau$ and $(\cpath(t_k)-\cpath(t_{k-1}))/\tau$ as well as that the Lipschitz constant of $\GMetric-G$ is $2\bar C_1$.
		The last equality finally uses the estimate
		\begin{align*}
			&\left|{G_{,1}}[\cpath(t_{k}),\cpath(t_{k})](\psi_k) (u_{k},u_{k})-{G_{,1}}[\cpath(t_{k}),\cpath(t_{k})](\psi_k) (u_{k+1},u_{k+1})\right|\\
			&=\tfrac\tau2\!\left|{G_{\!,1}}[\cpath(t_{k}),\cpath(t_{k})](\psi_k)(\tfrac{u_{k\!+\!1}\!-\!u_k}\tau,u_{k\!+\!1}\!\!+\!u_k\!)\!+\!{G_{\!,1}}[\cpath(t_{k}),\cpath(t_{k})](\psi_k)(u_{k\!+\!1}\!\!+\!u_k,\!\tfrac{u_{k\!+\!1}\!-\!u_k}\tau\!)\right|\\
			&=O(C_1\tau\|\psi_k\|_{W^m_\theta}).
		\end{align*}
																																		\emph{Step 2 (consistency).}
		Now we subtract \eqref{eq:ELExpDiscrete}  from \eqref{eq:ELExpCont} and obtain
		\begingroup
		\allowdisplaybreaks
		\begin{align}\label{eq:EXPdiff}
			0 &= {G_{,1}}[\cpath(t_{k}),\cpath(t_{k})](\psi_k) (u_{k+1},u_{k+1}) - {G_{,1}}[c_{k},c_{k+1}](\psi_k)  (v_{k+1},v_{k+1}) \\
			&\quad +  {G_{,2}}[\cpath(t_{k}),\cpath(t_{k})] (\psi_k)  (u_{k},u_{k}) -  {G_{,2}}[c_{k-1},c_k](\psi_k) (v_k,v_k)
			\nonumber \\
			&\quad - \tfrac{2}{\tau}  
			\big( \left( G[\cpath(t_{k}),\cpath(t_{k})](u_{k+1},\psi_k)-  G[c_{k},c_{k+1}]  (v_{k+1},\psi_k) \right) \nonumber \\
			&\qquad \quad \; -  \left(G[\cpath(t_{k-1}),\cpath(t_{k-1})](u_{k},\psi_k) - G[c_{k-1},c_k] (v_k,\psi_k) \right) \big)
			\nonumber \\
			&\quad+ O\left(\left(\bar C_0+\bar C_1+(1+C_1)\tau\right) \Vert \psi_k\Vert_{W^m_\theta}\right) \nonumber \\
						&= 
			{G_{,1}}[c_{k},c_{k+1}](\psi_k) (u_{k+1},u_{k+1}) - {G_{,1}}[c_{k},c_{k+1}] (\psi_k) (v_{k+1},v_{k+1}) ] \nonumber \\
						&\quad + \left({G_{,1}}[\cpath(t_{k}),\cpath(t_{k})]- {G_{,1}}[c_{k},c_{k}]\right) (\psi_k) (u_{k+1},u_{k+1}) \nonumber \\
			&\qquad \quad  + \left({G_{,1}}[c_{k},c_{k}]-{G_{,1}}[c_{k},c_{k+1}] \right)(\psi_k) (u_{k+1},u_{k+1}) \nonumber \\
						&\quad +  {G_{,2}}[c_{k-1},c_{k}] (\psi_k)  (u_{k},u_{k}) -  {G_{,2}}[c_{k-1},c_k] (\psi_k)(v_k,v_k) \nonumber \\
						&\quad + \left({G_{,2}}[\cpath(t_{k}),\cpath(t_{k})] - {G_{,2}}[c_{k},c_{k}] \right)(\psi_k)  (u_{k},u_{k}) \nonumber \\
			& \qquad \quad + \left({G_{,2}}[c_{k},c_{k}] - {G_{,2}}[c_{k-1},c_{k}] \right)(\psi_k)  (u_{k},u_{k}) \nonumber \\
						&\quad  + \tfrac{2}{\tau} \left(G[c_{k},c_{k}] (v_{k+1}-v_{k},\psi_k)-  G[\cpath(t_{k}),\cpath(t_{k})]  (u_{k+1}-u_{k},\psi_k)\right)  \nonumber \\
						&\quad + \tfrac2\tau(G[c_{k},c_{k+1}] -2G[c_{k},c_k]+G[c_{k},c_{k-1}])(v_k,\psi_k) \nonumber \\
						&\quad + \tfrac2\tau(G[c_{k},c_k]- G[c_{k-1},c_k]+G[c_{k-1},c_{k-1}]-G[c_{k},c_{k-1}])(v_k,\psi_k) \nonumber \\
						&\quad+\tfrac2\tau(G[c_{k},c_{k+1}]-G[c_{k},c_k])(v_{k+1}-v_k,\psi_k)\nonumber\\
						&\quad + \tfrac{2}{\tau} (G[\cpath(t_{k-1}),\cpath(t_{k-1})]-G[\cpath(t_{k}),\cpath(t_{k})]-G[c_{k-1},c_{k-1}]+G[c_{k},c_{k}])(v_k,\psi_k)]\nonumber \\
						&\quad - \tfrac{2}{\tau} [(G[\cpath(t_{k-1}),\cpath(t_{k-1})]-G[\cpath(t_{k}),\cpath(t_{k})])(e^v_k,\psi_k)\nonumber \\
			&\quad + O\left(\left(\bar C_0+\bar C_1+(1+C_1)\tau\right) \Vert \psi_k\Vert_{W^m_\theta}\right) \nonumber \\
			&= \mathrm{I}\!+\!\mathrm{II}\!+\!\mathrm{III}\!+\!\mathrm{IV}\!+\!\mathrm{V}\!+\! \mathrm{VI}\!+\!\mathrm{VII}\!+\!\mathrm{VIII}\!+\!\mathrm{IX}\!+\!\mathrm{X}\!+\!O\!\left(\left(\bar C_0\!+\!\bar C_1\!+\!(1\!+\!C_1\!)\tau\right)\! \Vert \psi_k\Vert_{W^m_\theta}\right).\nonumber
		\end{align}
		\endgroup
		Next, we estimate the different terms separately, again using the regularity of the continuous geodesic path
		and the abbreviations
		\begin{equation*}
			l^c_{k}(s) \coloneqq (1-s)\cpath(t_{k})+ s \cpath(t_{k+1}),\qquad l^d_{k}(s) \coloneqq (1-s) c_{k}+ s c_{k+1}.
		\end{equation*}
		Exploiting the relations
		\begin{align*}
			\|c_{k\!+\!1}\!-\!c_k\|_{W_\theta^m}&=\tau\|v_{k\!+\!1}\|_{W_\theta^m},\\
			\|v_{k}\|_{W_\theta^m}&\leq\|u_{k}\|_{W_\theta^m}\!+\!\|e^v_{k}\|_{W_\theta^m}=O(1\!+\!\|e^v_{k}\|_{W_\theta^m}),\\
			\|v_{k\!+\!1}\!-\!v_k\|_{W_\theta^m}&\leq\|u_{k\!+\!1}\!-\!u_k\|_{W_\theta^m}\!+\!\|e^v_{k\!+\!1}\!-\!e^v_{k}\|_{W_\theta^m}=O(\tau\!+\!\|e^v_{k\!+\!1}\|_{W_\theta^m}\!+\!\|e^v_{k}\|_{W_\theta^m}),
		\end{align*}
		we obtain
		\begingroup
		\allowdisplaybreaks
		\begin{align*}
						\vert \, \mathrm{I} \,\vert &=\vert
			{G_{,1}}[c_{k},c_{k\!+\!1}](\psi_k) (\!-\!e^v_{k\!+\!1},2u_{k\!+\!1}\!+\!e^v_{k\!+\!1})\!+\!{G_{,1}}[c_{k},c_{k\!+\!1}](\psi_k) (2u_{k\!+\!1}\!+\!e^v_{k\!+\!1},\!-\!e^v_{k\!+\!1})\vert/2  \\
			&=O\left(C_1 \left(\Vert e^v_{k+1} \Vert_{W^m_\theta} + \Vert e^v_{k+1} \Vert^2_{W^m_\theta}\right)  \Vert \psi_k  \Vert_{W^m_\theta}\right) ,\\
						\vert \, \mathrm{II} \,\vert &\leq L_1 \Vert \cpath(t_k) - c_k \Vert_{W^m_\theta} \Vert \psi_k \Vert_{W^m_\theta} \Vert u_{k+1} \Vert^2_{W^m_\theta} + C_2 \Vert c_k -c_{k+1} \Vert _{W^m_\theta}\Vert \psi_k \Vert_{W^m_\theta} \Vert u_{k+1} \Vert^2_{W^m_\theta} \\
			&=O\left(\left(L_1\Vert e_k  \Vert_{W^m_\theta} + C_2 \tau(1+\|e_{k+1}^v\|_{W^m_\theta}) \right)\Vert \psi_k  \Vert_{W^m_\theta} \right), \\
			\vert \, \mathrm{III} \,\vert &=\vert
			{G_{,2}}[c_{k\!-\!1},c_{k}] (\psi_k) (\!-\!e^v_{k},2u_{k}\!+\!e^v_{k})\!+\!{G_{,2}}[c_{k\!-\!1},c_{k}] (\psi_k) (\!-\!e^v_{k},2u_{k}\!+\!e^v_{k})
			\vert/2  \\
			&=O\left(C_1 \left(\Vert e^v_{k} \Vert_{W^m_\theta} + \Vert e^v_{k} \Vert^2_{W^m_\theta}\right)  \Vert \psi_k  \Vert_{W^m_\theta}\right) ,\\
			\vert \, \mathrm{IV} \,\vert &\leq L_1 \Vert \cpath(t_k) - c_k \Vert_{W^m_\theta} \Vert \psi_k \Vert_{W^m_\theta} \Vert u_{k} \Vert^2_{W^m_\theta} + C_2 \Vert c_k -c_{k-1} \Vert _{W^m_\theta}\Vert \psi_k \Vert_{W^m_\theta} \Vert u_{k} \Vert^2_{W^m_\theta} \\
			&=O\left(\left( L_1\Vert e_k  \Vert_{W^m_\theta} + C_2 \tau(1+\|e_{k}^v\|_{W^m_\theta}) \right)\Vert \psi_k  \Vert_{W^m_\theta} \right), \\
						\mathrm{V} &= 2 G[c_k,c_k](\tfrac{e^v_{k+1}-e^v_{k}}\tau,\psi_k)+2(G[c_k,c_k]-G[\cpath(t_k),\cpath(t_k)])(\tfrac{u_{k+1}-u_{k}}\tau,\psi_k)\\
			&= 2 G[c_k,c_k](\tfrac{e^v_{k+1}-e^v_{k}}\tau,\psi_k)+O(C_1\|e_k\|_{W_\theta^m}\|\psi_k\|_{W_\theta^m}),\\
			\vert \, \mathrm{VI} \,\vert
			&\textstyle=2\Big|\int_0^1[G_{,2}[c_k,l_{k-1}^d(s)](v_k)-G_{,2}[c_k,l_{k}^d(s)](v_{k})](v_k,\psi_k)\\
			&\qquad\quad+G_{,2}[c_k,l_{k}^d(s)](v_{k}-v_{k+1})(v_k,\psi_k)\,\d s\Big|\\
			&\textstyle=2\tau\left|\int_0^1\int_0^1G_{,22}[c_k,tl_{k\!-\!1}^d(s)\!+\!(1\!-\!t)l_k^d(s)](v_k)(v_k\!+\!s(v_{k\!+\!1}\!-\!v_k))(v_k,\psi_k)\,\d t\,\d s\right|\\
			&\quad+O(C_1\|v_{k}-v_{k+1}\|_{W_\theta^m}\|\psi_k\|_{W_\theta^m})\\
			&=O([C_2\tau(1+\|e^v_k\|_{W_\theta^m})^2(1+\|e^v_k\|_{W_\theta^m}+\|e^v_{k+1}\|_{W_\theta^m})\\
			&\quad\qquad+C_1(\tau+\|e^v_k\|_{W_\theta^m}+\|e^v_{k+1}\|_{W_\theta^m})]\|\psi_k\|_{W_\theta^m}),\\
			\vert \, \mathrm{VII} \,\vert
			&\textstyle=2\tau\left|\int_0^1\int_0^1G_{,12}[l_k^d(t),l_k^d(s)](v_k)(v_k)(v_k,\psi_k)\,\d t\,\d s\right|\\
			&=O(C_2\tau(1+\|e_k^v\|_{W_\theta^m})^3\|\psi_k\|_{W_\theta^m}),\\
			\vert \, \mathrm{VIII} \,\vert
			&=\textstyle2\left|\int_0^1G_{,2}[c_k,l_k^d(s)](v_{k+1})(v_{k+1}-v_k,\psi_k)\,\d s\right|\\
			&=O(C_1(1+\|e_{k+1}^v\|_{W_\theta^m})(\tau+\|e_{k+1}^v\|_{W_\theta^m}+\|e_{k}^v\|_{W_\theta^m})),\\
			\vert \, \mathrm{IX} \,\vert
			&\textstyle=2 \big|\int_0^1(G_{,1}+G_{,2})[l_{k-1}^d(s),l_{k-1}^d(s)](v_k)(v_k,\psi_k)\\
			&\qquad-(G_{,1}+G_{,2})[l_{k-1}^c(s),l_{k-1}^c(s)](u_k)(v_k,\psi_k)\,\d s\big|\\
			&\textstyle=2 \big|\int_0^1((G_{,1}\!\!+\!G_{,2})[l_{k\!-\!1}^d(s),l_{k\!-\!1}^d(s)]\\
			&\qquad \qquad  -(G_{,1}\!\!+\!G_{,2})[l_{k\!-\!1}^c(s),l_{k\!-\!1}^c(s)])(u_k)(v_k,\psi_k)\,\d s\big|\\
			&\quad+O(C_1\|e_k^v\|_{W_\theta^m}\|\psi_k\|_{W_\theta^m})\\
			&=O\Big((L_1(\|e_k\|_{W_\theta^m}\!+\!\|e_{k\!+\!1}\|_{W_\theta^m}\!)(1\!+\!\|e_k^v\|_{W_\theta^m}\!)\|\psi_k\|_{W_\theta^m}\\
			&\qquad \; +C_1\|e_k^v\|_{W_\theta^m}\|\psi_k\|_{W_\theta^m}\Big),\\
			\vert \, \mathrm{X} \,\vert &\textstyle=O(C_1\|e^v_k\|_{W_\theta^m}).
		\end{align*}
		\endgroup
		
						\emph{Step 3 (stability).}
		From \eqref{eq:EXPdiff} together with the above estimates of the individual terms and the uniform coercivity of $G[c_k,c_k]$, 
		we obtain the error propagation equation
		\begin{multline*}
			\frac{e_{k+1}^v-e_k^v}\tau
			=O\Big(\bar C_0\!+\!\bar C_1\!+\!\tau\!+\!C_2\tau\!+\!(L_1\!+\!C_1\!+\!\tau C_2)\left(1\!+\! \Vert e_k^v  \Vert_{W^m_\theta} \!+\! \Vert e_{k\!+\!1}^v  \Vert_{W^m_\theta}\right)^2\\
			\cdot\left( \tau\!+\!\Vert e_k  \Vert_{W^m_\theta} \!+\! \Vert e_{k\!+\!1}  \Vert_{W^m_\theta}\!+\!\Vert e_k^v  \Vert_{W^m_\theta} \!+\! \Vert e_{k\!+\!1}^v  \Vert_{W^m_\theta}\right)\Big).
		\end{multline*}
				If $\|e_j^v\|_{W^m_\theta}\leq1$ for $j=0,\ldots,k+1$, which we will verify at the end of the proof, then the above implies
		\begin{align*} 
			\frac{e_{k+1}-e_k}\tau &= e_{k+1}^v,\\
			\frac{e_{k+1}^v-e_k^v}\tau &= 
			O\!\left(\bar C_0\!+\!\bar C_1\!+\!\tau\!+\!C_2\tau\!+\! (L_1\!+\!C_1\! + \!\tau C_2)\!\left(\!\tau\!\!+\!\!\sum_{i=k}^{k\!+\!1}\!
			\left(\|e_i\|_{W^m_\theta} \! + \! \|e_i^v\|_{W^m_\theta}\right)\!\right)\!\right).
		\end{align*}
				Abbreviating $\mathrm{Err}_j\coloneqq\|e_{j}\|_{W^m_\theta}+\|e_{j}^v\|_{W^m_\theta}$ and noting that $\tau\leq1$ by definition, one obtains
		that there exist constants $\alpha,\gamma>0$ such that for $j=1,\ldots,k$
		\begin{multline*}
			\left(1-(L_1+C_1+\tau C_2)\tau\alpha\right)\mathrm{Err}_{j+1}\\
			\leq \left(1+(L_1+C_1+\tau C_2)\tau\alpha\right)\mathrm{Err}_j
			+ \tau(\bar C_0+\bar C_1+(1+L_1+C_1+C_2)\tau)\gamma
		\end{multline*}
		and using that $1-(L_1+C_1+\tau C_2)\tau\alpha>\frac12$ for 
		$\tau<1/\max(2\sqrt{C_2\alpha},4(L_1+C_1)\alpha)$
		\begin{align*}
			\mathrm{Err}_{j+1} 
			&\leq \frac{1+(L_1+C_1+\tau C_2)\tau\alpha}{1-(L_1+C_1+\tau C_2)\tau\alpha}\mathrm{Err}_j 
			+ \frac{\tau(\bar C_0+\bar C_1+(1+L_1+C_1+C_2)\tau)\gamma}{1-(L_1+C_1+\tau C_2)\tau\alpha} \\
			&=  \left(1+C_S\tau\right) \mathrm{Err}_j + C_A \tau 
		\end{align*}
		with 
		\begin{align*}
			C_S &\coloneqq 2 \frac{(L_1+C_1+\tau C_2)\alpha}{1-(L_1\!+\!C_1\!+\!\tau C_2)\tau\alpha} \leq 4 (L_1\!+\!C_1\!+\!\tau C_2)\alpha, \\
			C_A &\coloneqq \frac{(\bar C_0\!+\!\bar C_1\!+\!(1\!+\!L_1\!+\!C_1\!+\!C_2)\tau)\gamma}{1\!-\!(L_1\!+\!C_1\!+\!\tau C_2)\tau\alpha}
			\leq 2 (\bar C_0+\bar C_1+(1+L_1+C_1+C_2)\tau)\gamma.
		\end{align*}
		Thus, for $\W=\WEps{}$ we estimate 
		\begin{align*} 
			C_S &\leq 4 \lambda \alpha \left(2 +\tfrac{\tau}{\epsilon}\right)\,,\ C_A \leq \left(4 \lambda \epsilon + 2\tau + 4 \lambda \tau + 2 \lambda  \tfrac{\tau}{\epsilon} \right)\gamma,
		\end{align*}
		whereas for $\W=\WEpsFree$ we get 
		\begin{align*} 
			C_S &\leq 4 \lambda \alpha (2 +\tau),\ C_A \leq 2(1+3\lambda) \gamma \tau.
		\end{align*}
		
		Furthermore, for the initial time step, one obtains
		\begin{align*}
			\Vert e_1^v \Vert_{W^m_\theta}  			&=\left\Vert \frac{c_1-c_0}{\tau} - \frac{c(\tau)-c(0)}{\tau} \right\Vert_{W^m_\theta}
			=\left\Vert \dot c(0) - \frac{c(\tau)-c(0)}{\tau}  \right\Vert_{W^m_\theta} =O(\tau).
		\end{align*}
		This together with 
		$\mathrm{Err}_{1} = \Vert e_1 \Vert_{W^m_\theta} + \Vert e_1^v \Vert_{W^m_\theta} \leq
		\Vert e_0 + \tau e_1^v \Vert_{W^m_\theta} + \Vert e_1^v \Vert_{W^m_\theta}$
		and  $e_0 =0$ implies $ \mathrm{Err}_{1} \leq C_I \tau$ for a constant $C_I>0$.
		
		Now, the discrete Gronwall lemma applies. Explicitly,  the sequence of estimates
		$$\mathrm{Err}_{j+1} \leq (1+C_S \tau) \mathrm{Err}_{j} + C_A\tau$$ for $j=1,\ldots,k$ implies by induction 
		\begin{align}\label{eqn:errorBound}
			\mathrm{Err}_{j} \leq \frac{C_A}{C_S} \left(e^{C_S (j-1)\tau} -1\right) + e^{C_S (j-1)\tau} \mathrm{Err}_{1}.
		\end{align}
		for $j=1,\ldots,k$. For $j=K$ this would be the sought result due to $\|\Exp^K_{c_A}(1,v)-\exp_{c_A}(v)\|_{W^m_\theta}\leq\mathrm{Err}_K$.
		This set of inequalities was derived contingent on the well-definedness of the $c_j$ and the conditions
		$$c_j\in N_{r/2}(\cpath),\quad
		\|c_j-c_{j-1}\|_{W^m_\theta}\leq r/2,\quad
		\|e_j^v\|_{W^m_\theta}\leq1
		\quad\text{ for }j=1,\ldots,k.$$
		It thus remains to verify that these conditions and thus the inequalities hold for $k=1,\ldots,K$ with $K$ large enough, which we will do by induction in $k$.
		
		For $k=1$ this holds true for $K$ large enough due to $c_1=c_A+\tau\dot\cpath(0)$ and $e^v_1=\dot\cpath(0)-\tfrac{\cpath(\tau)-\cpath(0)}{\tau}$.
		Next, assume the result to hold for $k$
		and assume without loss of generality that $\epsilon$ is small and $K=\frac1\tau$ large enough
		so that the right-hand side of \eqref{eqn:errorBound} is smaller than $\min(\frac12,\frac r4)$
		and so that $(U+1)\tau\leq\min(\tfrac{C\sqrt\epsilon}2,\tfrac r8)$, $\mathcal C(4(U+1)^2\tau^2+\tfrac{16(U+1)^4\tau^4}{\epsilon^2})\leq\tfrac r8$, and $\mathcal C(4(U+1)^2\tau+\tfrac{16(U+1)^4\tau^3}{\epsilon^2})+A\tau\leq\frac12$
		with $C,\mathcal C$ from \cref{thm:existenceExp2Sobolev} for the choice $\mathfrak K=N_r(\cpath)$ and $U\coloneqq\max_{t\in[0,1]}\|\dot\cpath(t)\|_{W^m_\theta}$, $A\coloneqq\max_{t\in[0,1]}\|\ddot\cpath(t)\|_{W^m_\theta}$.
		Note that this can be achieved by simply taking $K>\kappa/\epsilon$ for $\kappa$ large enough and independent of $\epsilon$
		(recall that in the case $\W=\WEpsFree$ we set $\epsilon=1$).
		By the induction hypothesis, $\mathrm{Err}_k=\|e_{k}\|_{W^m_\theta}+\|e_{k}^v\|_{W^m_\theta}\leq\min(\frac12,\frac r4)$, thus
		\begin{equation*}
			\|c_k-c_{k-1}\|_{W^m_\theta}
			=\tau\|v_k\|_{W^m_\theta}
			\leq\tau\|u_k\|_{W^m_\theta}+\tau\|e^v_k\|_{W^m_\theta}
			\leq(U+1)\tau
			\leq\min(\tfrac{C\sqrt\epsilon}2,\tfrac r8).
		\end{equation*}
		Therefore, $c_{k+1}=\Exp^2_{c_{k-1}}(1,2(c_k-c_{k-1}))$ is well-defined by \cref{thm:existenceExp2Sobolev} with
		\begin{align*}
			&\|c_{k+1}-2c_{k}+c_{k-1}\|_{W^m_\theta} = \|\Exp^2_{c_{k-1}}(1,2(c_k-c_{k-1}))-c_{k-1} - 2(c_{k}-c_{k-1})\|_{W^m_\theta}\\
			&\leq\mathcal C\left(4\|c_k-c_{k-1}\|_{W^m_\theta}^2+\tfrac{16\|c_k-c_{k-1}\|_{W^m_\theta}^4}{\epsilon^2}\right)
			\leq\mathcal C\left(4(U+1)^2\tau^2+\tfrac{16(U+1)^4\tau^4}{\epsilon^2}\right)
			\leq\tfrac r8.
		\end{align*}
		Consequently we obtain
		$$\|c_{k+1}-c_k\|_{W^m_\theta}\leq\|c_k-c_{k-1}\|_{W^m_\theta}+\|c_{k+1}-2c_k+c_{k-1}\|_{W^m_\theta}\leq\tfrac r8+\tfrac r8=\tfrac r4,$$
		as desired.
		Furthermore,
		\begin{align*}
			\|c_{k+1}-\cpath(t_{k})\|_{W^m_\theta}
			&\leq\|c_{k+1}-c_k\|_{W^m_\theta}+\|c_k-\cpath(t_k)\|_{W^m_\theta}\\
			&\leq\tfrac r4+\|e_k\|_{W^m_\theta}
			\leq\tfrac r4+\tfrac r4
			=\tfrac r2
		\end{align*}
		so that $c_{k+1}\in N_{r/2}(\cpath)$ as desired.
		Finally, using
		\begin{align*}
			e^v_{k+1}-e^v_{k} = \tfrac{e_{k+1}-e_k}{\tau} - \tfrac{e_{k}-e_{k-1}}{\tau}
			= \tfrac{c_{k+1}- 2c_k + c_{k-1}}{\tau} - \tfrac{c(t_{k+1}-2c(t_k) + c(t_{k-1})}{\tau}
		\end{align*}
		one obtains
		\begin{align*}
			\|e^v_{k+1}-e^v_{k}\|_{W^m_\theta}
			& =\left\|\tfrac{c_{k+1}-2c_k+c_{k-1}}\tau-\tfrac{\cpath(t_{k+1})-2\cpath(t_k)+\cpath(t_{k-1})}\tau\right\|_{W^m_\theta}\\
			&\leq\mathcal C\left(4(U+1)^2\tau+\tfrac{16(U+1)^4\tau^3}{\epsilon^2}\right)+A\tau
			\leq\tfrac12,
		\end{align*}
		which together with $\|e^v_{k}\|_{W^m_\theta}\leq\frac12$ yields $\|e^v_{k+1}\|_{W^m_\theta}\leq1$, as desired, thereby closing the proof by induction.
	\end{proof}
	
			\subsection{Proof of consistency of the covariant difference quotients} \label{sec:appCovDeriv}
	
	\begin{proof}[Proof of \cref{lemma:consistencyCurveApproxDir} ]
				We consider $t\mapsto \w(t)= w\circ\cpath(t)$ to be the vector field along $\cpath(t) =c+ tv$.
		By \eqref{eq:ELeqInverseTranspCurves} and \eqref{eq:CovPlus} we have 
		$\Covtdir{\tau}{v}  w (c) = \tau^{-2} ( \yz(\tau)- c -\tau \w(0))$.
		Introducing $\Feps: \R \times W^{m}_\theta \times W^{m}_\theta \to (W^{m}_\theta)'\times(W^{m}_\theta)'$ with
		\begin{align}\label{eq:FepsCovDiff}
			\Feps[\tau, \yz, \yc](r) = 
			\begin{pmatrix} 
				\W_{,2}\myArg{\yz}{\yc}(r) + \W_{,1}\myArg{\yc}{c + \tau v}(r) \\
				\W_{,2}\myArg{c}{\yc}(r) + \W_{,1}[\yc,c + \tau v+ \tau \w(\tau)](r)
			\end{pmatrix}
		\end{align}
		for $r\in W^{m}_\theta$, condition \eqref{eq:ELeqInverseTranspCurves}
		can be rewritten as $\Feps[\tau, \yz(\tau), \yc(\tau)](r) = 0$ for all $r\in W^{m}_\theta$.
		We have
		\begin{align}
			\yc(\tau) &=  c + \tfrac{1}{2} \Log^2_c(c + \tau v +\tau \w(\tau)), \\
			\yz(\tau) &= \Exp^{2}_{c + \tau v } (1, 2(\yc(\tau) - c - \tau v  )). \label{eq:ExpLog}
		\end{align}
		Hence, \cref{thm:existenceExp2Sobolev} guarantees for $\|v\|_{W_\theta^m}$ and $\tau/\sqrt\epsilon$ small enough (depending on $\mathfrak K$ and $\sup_{t\in[0,1]}\|\w(t)\|_{W_\theta^m}$)
		the well-definedness of $\Covtdir{\tau}{v}$.
		
		In what follows, we skip the tangent vector $r$ to which $\Feps$ is applied, remembering that in all derivatives of $\W$ the first differentiation is in direction $r$.
		To compare the discrete covariant difference quotient with the continuous covariant derivative, we compute a Taylor expansion of $\yz$. For this, we differentiate $\Feps$ with respect to $\tau$ and obtain
		\begin{align}\label{eq:dTauFequalsZeroeps}
			0= \frac{\mathrm{d}}{\mathrm{d} \tau} \Feps[\tau, \yz(\tau), \yc(\tau)] =  \partial_\tau \Feps[\tau, \yz(\tau), \yc(\tau)] +  \partial_{(\yz,\yc)} \Feps[\tau, \yz(\tau), \yc(\tau)] 
			\begin{pmatrix} \dot \yz(\tau)\\ \dot \yc(\tau) \end{pmatrix},
		\end{align}
		where the dot expresses differentiation with respect to $\tau$. 
		Inserting \eqref{eq:FepsCovDiff} into \eqref{eq:dTauFequalsZeroeps} we obtain the block operator equation 
		\begin{align}  \nonumber
			&\begin{pmatrix} 
				\W_{,21}\myArg{\yz}{\yc} & \W_{,22}\myArg{\yz}{\yc} + \W_{,11}\myArg{\yc}{c+ \tau v} \\
				0 & \W_{,22}\myArg{c}{\yc} + \W_{,11}\myArg{\yc}{c + \tau v+ \tau \w}
			\end{pmatrix}
			\begin{pmatrix} \dot \yz\\ \dot \yc \end{pmatrix} \\
			&= - 
			\begin{pmatrix} 
				\W_{,12}\myArg{\yc}{c+\tau v}(v)  \\
				\W_{,12}\myArg{\yc}{c+\tau v+ \tau \w} (v+ \w + \tau \dot \w)
			\end{pmatrix}
			, \label{eq:IFTderivativeCurve}
		\end{align}
		where we abbreviated $\yc = \yc(\tau)$, $\yz = \yz(\tau)$ and $\w = \w(\tau)$.
		Below we use this equation to Taylor expand $\yc$ and $\yz$ around $\tau=0$,
		where $(\yz(0), \yc(0))=(c, c)$.
		We will use the reformulation $\W[\hat c,\check c]=G[\hat c,\check c](\check c-\hat c,\check c-\hat c)$ from \eqref{eq:defGeps} or \eqref{eq:defGepsFree}.
		Taking into accout that  $2G[c,c] = \W_{,11}\myArg{c}{c} = \W_{,22}\myArg{c}{c} = - \W_{,12}\myArg{c}{c}$ we can evaluate the operator equation \eqref{eq:IFTderivativeCurve} at $\tau=0$ and achieve
		\begin{align} \label{eq:IFTderivativeZeroCurve}
			\begin{pmatrix} 
				-2G[c,c] & 4G[c,c] \\
				0 & 4G[c,c]
			\end{pmatrix}
			\begin{pmatrix} \dot \yz(0)\\ \dot \yc(0) \end{pmatrix}
			= 
			\begin{pmatrix} 
				2G[c,c] \,v \\
				2G[c,c] \,v + 2G[c,c] \,\w(0)
			\end{pmatrix}
			,
		\end{align}
		which leads to $\dot \yc(0)=\tfrac12 (v+\w(0))$ and 
		$\dot \yz(0)
		= \w(0)$.
		Next, we differentiate \eqref{eq:IFTderivativeCurve} once more in $\tau$ and obtain
		\begin{align}  \label{eq:IFTsecondDerivativeCurve}
			&\begin{pmatrix} 
				\W_{,21}\myArg{\yz}{\yc} & \W_{,22}\myArg{\yz}{\yc} + \W_{,11}\myArg{\yc}{c +\tau v} \\
				0 & \W_{,22}\myArg{c}{\yc} + \W_{,11}\myArg{\yc}{c + \tau v+ \tau \w}
			\end{pmatrix}
			\begin{pmatrix} \ddot \yz\\ \ddot \yc \end{pmatrix}
			= 
			\begin{pmatrix} 
				R_1(\tau) \\
				R_2(\tau) 
			\end{pmatrix} ,
		\end{align}
		where  $R_1$ and $R_2$ are functionals mapping
		$\tau$
		into $(W^{m}_\theta)'$ and act on $r \in W^{m}_\theta$ with
		{
			\allowdisplaybreaks
			\begin{align*}
				R_1(\tau)(r) =& - \W_{,211}\myArg{\yz}{\yc}(r,\dot \yz,\dot \yz) -  \W_{,212}\myArg{\yz}{\yc}(r,\dot \yz,\dot \yc)- \W_{,221}\myArg{\yz}{\yc}(r,\dot \yc,\dot \yz) \\
				&- \W_{,222}\myArg{\yz}{\yc}(r,\dot \yc,\dot \yc)
				- \W_{,111}\myArg{\yc}{c + \tau v}(r,\dot \yc,\dot \yc) - \W_{,112}\myArg{\yc}{c +\tau v}(r,\dot \yc,v) \\
				&- \W_{,121}\myArg{\yc}{c +\tau v}(r,v,\dot \yc) - \W_{,122}\myArg{\yc}{c +\tau v}(r,v,v)\\
				=& - \W_{,211}\myArg{\yz}{\yc}(r,\dot \yz,\dot \yz) -  2\W_{,212}\myArg{\yz}{\yc}(r,\dot \yz,\dot \yc) \\
				&- \W_{,222}\myArg{\yz}{\yc}(r,\dot \yc,\dot \yc)
				- \W_{,111}\myArg{\yc}{c + \tau v}(r,\dot \yc,\dot \yc) \\
				&- 2\W_{,112}\myArg{\yc}{c +\tau v}(r,\dot \yc,v)
				- \W_{,122}\myArg{\yc}{c + \tau v}(r,v,v), \\
				R_2(\tau)(r) =& - \W_{,222}\myArg{c}{\yc}(r,\dot \yc,\dot \yc)  - \W_{,111}\myArg{\yc}{c + \tau (v+\w)}(r,\dot \yc,\dot \yc) \\
				&- \W_{,112}\myArg{\yc}{c + \tau (v+\w)}(r,\dot \yc,v+\w+\tau \dot \w) \\
				&- \W_{,121}\myArg{\yc}{c + \tau (v+\w)}(r,v+\w+\tau \dot \w,\dot \yc) \\
				&- \W_{,122}\myArg{\yc}{c + \tau (v+\w)}(r,v+\w+\tau \dot \w,v+\w+\tau \dot \w)\\
				& - \W_{,12}\myArg{\yc}{c + \tau (v+\w)} (r, 2 \dot \w + \tau \ddot \w) \\
				=& - \W_{,222}\myArg{c}{\yc}(r,\dot \yc,\dot \yc)  - \W_{,111}\myArg{\yc}{c + \tau (v+\w)}(r,\dot \yc,\dot \yc) \\
				&- 2\W_{,112}\myArg{\yc}{c + \tau (v+\w)}(r,\dot \yc, v+\w+\tau \dot \w)\\
				&  - \W_{,122}\myArg{\yc}{c +\tau (v+\w)}(r, v+\w+\tau \dot \w,v+\w+\tau \dot \w) \\
				&- \W_{,12}\myArg{\yc}{c +\tau (v+\w)} (r, 2 \dot \w + \tau \ddot \w)  .
			\end{align*}
		}%
		Evaluating \eqref{eq:IFTsecondDerivativeCurve} at $\tau=0$ gives
		\begin{equation*}  			\begin{pmatrix}
				-2G[c,c] &4G[c,c] \\
				0 & 4 G[c,c]
			\end{pmatrix}
			\begin{pmatrix} \ddot \yz(0)\\ \ddot \yc(0) \end{pmatrix}
			=
			\begin{pmatrix} 
				R_1(0) \\
				R_2(0) 
			\end{pmatrix},
		\end{equation*}
		whose solution (using $\dot \yc(0)=\tfrac12 (v+\w(0))$ and $\dot \yz(0) = \w(0)$ from above) reads
		{
			\allowdisplaybreaks
			\begin{align} 				G[c,c] (\ddot \yc(0),r) =  &\tfrac{R_2(0)(r)}4, \notag \\
				G[c,c](\ddot \yz(0),r)  = &\tfrac12 R_2(0)(r) - \tfrac12 R_1(0)(r) \notag\\
				=& \tfrac12\Big(
				[-\W_{,121}-\W_{,122}]\myArg{c}{c}(r,v+\w(0),v+\w(0)) \notag\\
				&\quad+\W_{,212}\myArg{c}{c}(r,\w(0),v+\w(0))
				+\W_{,112}\myArg{c}{c}(r,v+\w(0),v)  \notag\\
				&\quad+\W_{,122}\myArg{c}{c}(r,v,v)
				+\W_{,211}\myArg{c}{c}(r,\w(0),\w(0))
				\Big)\notag\\
				&- \W_{,12}\myArg{c}{c} (r, \dot \w(0)) \notag\\
				=&\tfrac12\Big(
				\quad[-\W_{,121}-\W_{,122}+\W_{,212}+\W_{,211}]\myArg{c}{c}(r,\w(0),\w(0))\notag\\
				&\quad+[-\W_{,121}-\W_{,122}+\W_{,112}+\W_{,122}]\myArg{c}{c}(r,v,v)\notag\\
				&\quad+[-\W_{,121}-\W_{,112}-2\W_{,122}+\W_{,212}+\W_{,112}]\myArg{c}{c}(r,\w(0),v)
				\Big)\notag\\
				&- \W_{,12}\myArg{c}{c} (r, \dot \w(0)) \notag\\
				=&-\tfrac12[\W_{,211}+\W_{,122}]\myArg{c}{c}(r,v,\w(0)) - \W_{,12}\myArg{c}{c} (r, \dot \w(0))  .\label{eq:metricZConsistencyC}
			\end{align}
		}%
		Note that $\W_{,12}[c,c]=\W_{,21}[c,c]$
		implies $(-\W_{,121}-\W_{,122}+\W_{,212}+\W_{,211})\myArg{c}{c} = 0$, which was used three times in the last step. To rewrite \eqref{eq:metricZConsistencyC} in terms of $G$
		we express the derivatives of $\W$ in terms of $G$,
		\begin{align*}
			\W_{,1}[c,\tilde c](\phi) &= G_{,1}[c,\tilde c](\phi)(\tilde c-c,\tilde c-c) - 2 G[c,\tilde c](\phi,\tilde c-c),\\
			\W_{,12}[c,\tilde c](\phi,\chi) &=  G_{,12}[c,\tilde c](\phi,\chi)(\tilde c -c,\tilde c-c)
			+ 	2G_{,1}[c,\tilde c](\phi)(\chi,\tilde c-c) \\
			& \quad -2G_{,2}[c,\tilde c](\chi)(\phi,\tilde c-c) - 2 G[c,\tilde c](\phi,\chi),\\
			\W_{,122}[c,\tilde c](\phi,\chi, \psi) &=  G_{,122}[c,\tilde c](\phi,\chi, \psi)(\tilde c -c,\tilde c-c)  \\
			& \quad +2G_{,12}[c,\tilde c](\phi,\chi)(\psi,\tilde c-c)
			+2G_{,12}[c,\tilde c](\phi,\psi)( \chi,\tilde c-c)  \\
			& \quad + 	2G_{,1}[c,\tilde c](\phi)(\chi,\psi)
			- 2G_{,22}[c,\tilde c](\chi,\psi)(\phi,\tilde c-c)  \\
			& \quad -2G_{,2}[c,\tilde c](\chi)(\phi,\psi) - 2 G_{,2}[c,\tilde c](\psi)(\phi,\chi),\\
			\W_{,122}[c, c](\phi,\chi, \psi) &= 	2G_{,1}[c, c](\phi)(\chi,\psi)
			-2G_{,2}[c, c](\chi)(\phi,\psi) - 2 G_{,2}[c, c](\psi)(\phi,\chi)  .
		\end{align*}
		Analogous calculations lead to
		\begin{align*}
			\W_{,211}[c, c](\phi,\chi, \psi) &= 	2G_{,2}[c, c](\phi)(\chi,\psi)
			-2G_{,1}[c, c](\chi)(\phi,\psi) - 2 G_{,1}[c, c](\psi)(\phi,\chi)  .
		\end{align*}
		Inserting both in \eqref{eq:metricZConsistencyC} leads to
		\begin{align}\label{eq:gepsForm}
			G[c,c](\ddot \yz(0),r)
			= & -({G_{,1}}[c,c] + {G_{,2}}[c,c])(r)(v, \w(0)) \\
			& + ({G_{,1}}[c,c] + {G_{,2}}[c,c])(v)(r, \w(0)) \notag \\
			& + ({G_{,1}}[c,c] + {G_{,2}}[c,c])(\w(0))(r, v)
			+ 2 G [c, c] (r, \dot \w(0))  , \notag
		\end{align}
		which can be solved for $\ddot\yz(0)$.
		Analogously it follows that $\yz$ has as many derivatives as $\w$, which depend multilinearly on $v$ and $\w$ and its derivatives and which are uniformly bounded for $c\in\mathfrak K$ depending on $w$
		(recall that $G$ is smooth).

		On the continuous side, using the definitions \eqref{eq:cov2} and \eqref{eq:ChristoffelOperator} of covariant derivative and Christoffel operator, we obtain
		\begin{align}\label{eq:CovDerivConta}
			\textstyle g_c(\covdir{u} w(c),r)\! &=
			g_c(\dot \w_u(0),r) \!+\! g_c(\Christoffel{c}{\w(0)}{u},r) \\
			&= \GMetric[c,c]( r,\dot \w_u(0) ) 			\!+\! \tfrac12 \tripleArgCommut{D_cg}{r}{\w(0)}{u} \nonumber \\
			&= \GMetric[c,c]( r,\dot \w_u(0) ) \!+\! \tfrac12 \tripleArgCommut{(\GMetric_{,1}[c,c]\!+\! \GMetric_{,2}[c,c])}{r}{\w(0)}{u} \nonumber \\
			&= \GMetric[c,c]( r,\dot \w_u(0) \!-\! \dot \w(0)) \!+\! (\GMetric[c,c] \!-\! G[c,c])(r , \dot \w(0)) \nonumber\\
			& \quad \!+\! \tfrac12\tripleArgCommut{(\GMetric_{,1}[c,c] \!+\! \GMetric_{,2}[c,c] \!-\!  G_{,1}[c,c] \!-\! G_{,2}[c,c])}{r}{\w(0)}{u}\nonumber \\
			&\quad \!+\!\tfrac{\tripleArgCommut{(G_{,1}[c,c] \!+\! G_{,2}[c,c])}{r}{\w(0)}{u-v}}2
			\!+\! \tfrac{(G[c,c] \!-\! g_c)(\ddot \yz(0), r) \!+\! g_c(\ddot \yz(0) , r)}2,\nonumber
		\end{align}
		where $\w_u(\tau) = w(c + \tau u)$, $\w(0) = \w_u(0) = w(c) $,
		$\GMetric$ is given by \eqref{eq:widehatGeps}, and we used the short-hand notation $\tripleArgCommut{B}\phi\chi\psi=-B(\phi)(\chi,\psi)+B(\chi)(\psi,\phi)+B(\psi)(\phi,\chi)$.
		Using the constants from \eqref{eqn:constants} (this time for $S=\{(c,c)\in\mathfrak K\times\mathfrak K\}$) and abbreviating the $C^0$- and $C^1$-norm of $w$ on $\mathfrak K$ by $\|w\|_{C^0}$ and $\|w\|_{C^1}$, respectively, this implies the existence of a constant $\gamma>0$ such that
		\begin{multline*}
			\left|\textstyle g_c(\covdir{u} w(c),r)-  \tfrac12 g_c(\ddot \yz(0),r)\right|
			\leq\big(\gamma\|w\|_{C^1}\|u-v\|_{W_\theta^m}+\bar C_0\|w\|_{C^1}\|v\|_{W_\theta^m}\\
			+\bar C_1\|w\|_{C^0}\|u\|_{W_\theta^m}+C_1\|w\|_{C^0}\|u-v\|_{W_\theta^m}+\bar C_0\|\ddot\yz\|_{W_\theta^m}\big)\Vert r \Vert_{W^m_\theta}.
								\end{multline*}
		Now \eqref{eq:gepsForm} implies $\|\ddot\yz(0)\|_{W_\theta^m}\leq1$ if $v$ is sufficiently small (depending on $\|w\|_{C^0}$, $\|w\|_{C^1}$, and $\mathfrak K$), thus
		\begin{equation}\label{eq:finalCovDeriv}\textstyle
			\ddot\yz(0)=2\covdir{u}w(c)+O\left((1+C_1)\|u-v\|_{W_\theta^m}+\bar C_0+\bar C_1\right).
		\end{equation}
		Next we consider the Taylor expansion of $\tau \mapsto \yz(\tau)$ at $\tau =0$, \ie
		\begin{align*}
			\yz(\tau) &= \yz(0) + \dot \yz(0) \tau + \tfrac12 \ddot \yz(0) \tau^2 + \tfrac16 \dddot \yz(0) \tau^3 + R(\tau)\\
			&\textstyle=c\!+\! w(c) \tau \!+\! \left( \covdir{u} w(c) \!+\! O\!\left((1\!+\!C_1)\|u-v\|_{W_\theta^m}\!+\!\bar C_0\!+\!\bar C_1\right) \right)\,\tau^2 \!+\! \tfrac16 \dddot \yz(0) \tau^3 \!+\! R(\tau)
		\end{align*}
		with the remainder term $R(\tau)=o(\tau^3)$ if $\w \in C^3([-1,1],W_\theta^m)$ and $R(\tau)=O(\tau^4)$ if $\w \in C^4([-1,1],W_\theta^m)$.
		This implies
		\begin{align*}\textstyle
			\covdir{u} w(c) &\textstyle= \frac{\yz(\tau)-c-\tau w(c)}{\tau^2} +O\!\left((1\!+\!C_1)\|u-v\|_{W_\theta^m}\!+\!\bar C_0\!+\!\bar C_1\!+\!\tau\right) \\
						&\textstyle=\Covtdir{\tau}{v}  w(c) + O\!\left((1\!+\!C_1)\|u-v\|_{W_\theta^m}\!+\!\bar C_0\!+\!\bar C_1\!+\!\tau\right),
		\end{align*}
		which in view of \eqref{eqn:constantsBounds} establishes first order consistency of the one-sided 		covariant difference quotient. 		Similarly, for the central covariant difference quotient, we obtain for 
		$\w \in C^{4}([-1,1],W_\theta^m)$
		\begin{align*}
			\Covtdir{\pm \tau}{v}  w(c)
			&=\frac12\left( \frac{\yz(\tau) -c -\tau w(c)}{\tau^2} + \frac{\yz(-\tau) -c +\tau w(c)}{\tau^2}\right) \\
			&= \frac{\yz(\tau) -c + \yz(-\tau) -c }{2\tau^2} \\
			&= \frac1{2\tau^2}\Big(w(c) \tau + \covdir{u} w(c)\, \tau^2 +\frac{1}{6}\dddot \yz(0) \tau^3 +  O(\tau^4)
			- w(c) \tau +\covdir{u} w(c)\, \tau^2 \\
			& \qquad -\frac{1}{6}\dddot \yz(0) \tau^3  + O(\tau^4) + O\left((1+C_1)\|u-v\|_{W_\theta^m}+\bar C_0+\bar C_1\right)\tau^{2} \Big) \\
			&= \covdir{u} w(c) + O\left((1+C_1)\|u-v\|_{W_\theta^m}+\bar C_0+\bar C_1+\tau^2\right),
		\end{align*}
		which in view of \eqref{eqn:constantsBounds}, establishes the desired consistency of the symmetric covariant difference quotient.
	\end{proof}
	
	\subsection{Proof of convergence of parallel transport} \label{sec:appPT}
	
	\begin{proof}[Proof of \cref{thm:PT}]
		We show convergence of the discrete scheme by showing consistency and stability of the discrete evolution operator $ \ParTp_{c_k,c_{k + 1}} $.
		
		\emph{Step 1 (consistency).} Let $t \in [0,1]$ and consider a single step discrete parallel transport of $\w(t)$ from $\cpath(t)$ to $\cpath(t+\tau)=\cpath(t) + \tau v$ for $v = \tfrac{\cpath(t + \tau) - \cpath(t)}{\tau}$.
		We exploit that $\ParTp_{\cpath(t),\cpath(t + \tau)}= \ParTp_{\cpath(t + \tau), \cpath(t)}^{-1}$ as $\W$ is symmetric. Hence, by \cref{lemma:consistencyCurveApproxDir} we have
		\begin{multline}\label{eq:PTconsistency}\textstyle
			\left\| \tfrac{1}{\tau}(\ParTp_{\cpath(t),\cpath(t + \tau)} \tau \w(t)   - \tau \w(t+ \tau)) \right\|_{W_\theta^m}
			=  \left\|\tau  \Covtdir{\tau}{-v} \w(t + \tau) \right\|_{W_\theta^m}\\\textstyle
			= \tau \left\| \covdir{-\dot \cpath(t+\tau)} \w(t+ \tau)\right\|_{W_\theta^m} + O\left(\tau^2 +\iota\tau \epsilon+ \tau \| v - \dot \cpath(t+\tau) \|_{W_\theta^m}\right)
			=  O(\tau^2+\iota\tau \epsilon)
		\end{multline}
		for $\tau\leq C\sqrt\epsilon$, where the last equality holds since $\cpath$ is smooth and $\w$ is parallel along $\cpath$ and where $\iota=1$ for $\W=\WEps$ and $\iota=0$ else.
				
		\emph{Step 2 (stability).}
		For $\tau\leq C\epsilon$ and $\|w\|_{W_\theta^m},\|\bar w\|_{W_\theta^m}\leq r\coloneqq2\max_{t}\|\w(t)\|_{W_\theta^m}$ we show the stability estimate
		\begin{equation}
			\Vert \tfrac{1}{\tau}\ParTp_{\cpath(t),\cpath(t + \tau)} \tau w - \tfrac{1}{\tau}\ParTp_{\cpath(t),\cpath(t + \tau) } \tau \bar{w} \Vert_{W_\theta^m} \leq(1+ \lambda\tau)  \Vert w - \bar{w}\Vert_{W_\theta^m} \label{eq:PTstability}
		\end{equation}
		with some $\lambda>0$.
		To this end, we note
		\begin{equation*}
			\tfrac{1}{\tau}\ParTp_{\cpath(t),\cpath(t + \tau)} \tau w = \tfrac{1}{\tau}(\mathrm{R} \circ \mathrm{Q} [\tau w] -\cpath(t+ \tau))
		\end{equation*}
		with $\mathrm{Q}[u] \coloneqq \tfrac{1}{2} \Log^2_{\cpath(t + \tau)}(\cpath(t) + u)+\tau v$ and $\mathrm{R}[x] \coloneqq \Exp^2_{\cpath(t)}(1,2x)$.
		(Note that by the symmetry of $\W$ one can compute the midpoint of the discrete parallelogram both using $\Log^2_{\cpath(t + \tau)}(\cpath(t) + \tau w)$ or equivalently using $\Log^2_{\cpath(t) + \tau w}(\cpath(t + \tau))$.)
		By \cref{thm:existenceExp2Sobolev}, for $\|u\|_{W_\theta^m},\|x\|_{W_\theta^m}\leq C\epsilon$ we have
		\begin{equation*}
			\partial_u \mathrm{Q}[u] = \tfrac{1}{2} \Id + O(\|u\|_{W_\theta^m})
			\quad\text{and}\quad
			\partial_x \mathrm{R}[x] = 2 \Id + O(\|x\|_{W_\theta^m}),
		\end{equation*}
		as well as $\mathrm Q[u]=O(\|u\|_{W_\theta^m})$.
		This implies $\partial_u\mathrm R\circ\mathrm Q[u]=\Id+O(\|u\|_{W_\theta^m})$ so that
		\begin{align*}
			\Vert \mathrm{R} \circ \mathrm{Q} [\tau w] - \mathrm{R} \circ \mathrm{Q} [\tau \bar{w}] \Vert \leq (1+\lambda\tau) \Vert \tau w - \tau \bar w \Vert
		\end{align*}
		(potentially decreasing $C$ depending on $\max_{t}\|\w(t)\|_{W_\theta^m}$) and thus \eqref{eq:PTstability}.
				
		\emph{Step 3 (convergence).} Now consider the discrete parallel transport $w_0 , \ldots , w_K$ with $w_0 = \w(0)$, assuming for the moment that it is well-defined for every time step and satisfies $\|w_k\|_{W_\theta^m}\leq r$. By \eqref{eq:PTconsistency} and \eqref{eq:PTstability} the error $e_{k+1} = \w((k\!+\!1)\tau) - w_{k+1}$ fulfils
		\begin{align}\label{eq:PTerrorpropagration}
			\Vert e_{k+1} \Vert_{W_\theta^m} &=\Vert \w(t_{k+1}) - \tfrac{1}{\tau}\ParTp_{c_k,c_{k+1}} \tau w_{k} \Vert_{W_\theta^m}   \\
			& \leq \Vert \w(t_{k+1})- \tfrac{1}{\tau}\ParTp_{c_k,c_{k+1}} \tau  \w(t_k) \Vert_{W_\theta^m} \nonumber \\
			& \qquad + \Vert \tfrac{1}{\tau}\ParTp_{c_k,c_{k+1}} \tau \w(t_k) - \tfrac{1}{\tau}\ParTp_{c_k,c_{k+1}} \tau w_{k} \Vert_{W_\theta^m} \nonumber \\
			& \leq \widehat{C}(\tau^2 + \iota\epsilon \tau) + (1 + \lambda\tau) \Vert e_k \Vert_{W_\theta^m}\nonumber
		\end{align}
		for some $\widehat{C}\!>\!0$.
		We show inductively that indeed it is well-defined with $\|w_k\|_{W_\theta^m}\!\leq\!r$ and
		\begin{align*}
			\Vert e_k \Vert_{W_\theta^m} \leq  \tfrac{\widehat{C}(\tau+\iota\epsilon)}{\lambda}(\exp(\lambda k\tau) -1),
		\end{align*}
		which yields the claim:
		The statement for $k=0$ is trivial.
		Now let it hold for $k\geq0$, then $w_{k+1}$ is well-defined by step 2, and
		\begin{align*}
			\Vert e_{k+1} \Vert_{W_\theta^m}& \leq \widehat{C}(\tau^2 + \iota\epsilon \tau) + (1 + \lambda\tau) \tfrac{\widehat{C}(\tau+\iota\epsilon)}{\lambda}(\exp(\lambda k\tau) -1) \\
			& =  \tfrac{\widehat{C}(\tau+\iota\epsilon)}{\lambda}[(1+\lambda\tau)\exp(\lambda k\tau) -1] \\
			& \leq \tfrac{\widehat{C}(\tau+\iota\epsilon)}{\lambda}[\exp(\lambda(k+1)\tau) -1],
		\end{align*}
		which additionally implies $\|w_{k+1}\|_{W_\theta^m}\leq r$ for $\tau,\epsilon$ small enough as can be tuned by the constant $\kappa$ from the theorem.
	\end{proof}
	
	\subsection{Proof of consistency of the discrete Riemann curvature approximation} \label{sec:appCT}
	\begin{proof}[Proof of \cref{thm:consRiemCurvTensor}]
		Setting $\iota=1$ for $\W=\WEps$ and $\iota=0$ else, we have
		\begin{align*}
			&\textstyle\left\|\left(\Covtdir{\tau}{v} \Covtdir{\tau^\beta}{w}  z - \covdir{v} \covdir{w} z \right)(c)  \right\|_{W^m_\theta}\\
			&\textstyle\leq  \left\|\Covtdir{\tau}{v} \left( \Covtdir{\tau^\beta}{w}   z -  \covdir{w}  z \right)(c)\right\|_{W^m_\theta}  +
			\left\|\left(\Covtdir{\tau}{v} \covdir{w}  z - \covdir{v} \covdir{w}  z \right)(c)\right\|_{W^m_\theta} \\
			&\textstyle \leq  \frac1{\tau^2} \left\| \ParTp_{c,c\!+\!\tau v}^{-1} \!\left(\!\tau \!\left(\! \Covtdir{\tau^\beta}{w}  z - \covdir{w}  z \!\right) (c\!+\!\tau v) \right)
			- \tau  \left(\! \Covtdir{\tau^\beta}{w}  z -  \covdir{w} z \!\right) (c) \right\|_{W^m_\theta}
			\!\!+  O(\tau\!+\! \iota\epsilon_\out) \\
			&\textstyle\leq   \frac1{\tau} (2+\lambda\tau)  \max_{\sigma \in \{0,1\}}\left\| \left( \Covtdir{\tau^\beta}{w}  z -  \covdir{w} z \right) (c\!+\!\tau \sigma v) \right\|_{W^m_\theta} \!\!+ O(\tau\!+\!\iota\epsilon_\out)\\
			&= O\left(\tau^{\beta-1}+ \tfrac{\iota\epsilon_\inn}\tau + \tau+ \iota\epsilon_\out\right),
		\end{align*}
		where we used the consistency of the covariant difference quotient \cref{lemma:consistencyCurveApproxDir} in the second and last step and the stability \eqref{eq:PTstability} of discrete parallel transport (with $\bar w=0$) in the third.
		Analogously we obtain $\big(\Covtdir{\tau}{w} \Covtdir{\tau^\beta}{v}  z - \covdir{w} \covdir{v} z \big)(c)=O(\tau^{\beta-1}+\tau+\iota(\epsilon_\out+\epsilon_\inn/\tau))$ so that the claim follows for the approximation via one-sided covariant difference quotients. 
		In the case of central covariant difference quotients, we proceed similarly and estimate
		\begin{align*}
			&\textstyle\left\|\left(\Covtdir{\pm\tau}{v} \Covtdir{\pm\tau^\beta}{w}  z - \covdir{v} \covdir{w} z \right)(c)  \right\|_{W^m_\theta}\\
			&\textstyle\leq  \left\|\Covtdir{\pm\tau}{v} \left( \Covtdir{\pm\tau^\beta}{w}   z -  \covdir{w}  z \right)(c)\right\|_{W^m_\theta}  +
			\left\|\left(\Covtdir{\pm\tau}{v} \covdir{w}  z - \covdir{v} \covdir{w}  z \right)(c)\right\|_{W^m_\theta} \\
			&\textstyle\leq   \frac2{\tau} (1+\lambda\tau)  \max_{\sigma \in \{-1,1\}}\left\| \left( \Covtdir{\pm\tau^\beta}{w}  z -  \covdir{w} z \right) (c\!+\!\tau \sigma v) \right\|_{W^m_\theta} \!\!+ O(\tau^2\!+\!\iota\epsilon_\out)\\
			&= O\left(\tau^{2\beta-1}+ \tfrac{\iota\epsilon_\inn}\tau + \tau^2+ \iota\epsilon_\out\right),
		\end{align*}
		which together with $\tau/\sqrt{\epsilon_\out},\tau/\sqrt{\epsilon_\inn}\leq C$ in the end yields the second estimate.
	\end{proof}
	
		\subsection{Maple code to verify integrals}\label{sec:mapleCode}
	We provide Maple Code to verify the explicit integrals given in \eqref{eqn:GalerkinApproxSobolevMetric}. 
			\begin{lstlisting}[language=Maple, basicstyle=\small]
			with(LinearAlgebra):
			# signs of quantities from (*@\eqref{eqn:abbreviations}@*)
			assume(r>0,p>0,q<r*p,q>-r*p);
			# upper bound on |c'| from (*@\eqref{eqn:L}@*)
			L := (1-t)*r+t*p;
			# integrands (only t-dependent terms) of the 0th to 2nd order part of W from (*@\eqref{eqn:timeIntegrals}@*)
			w0  := L;
			w1  := L/((1-t)^2*r^2+2*(1-t)*t*q+t^2*p^2);
			w2a := L/((1-t)^2*r^2+2*(1-t)*t*q+t^2*p^2)^2;
			w2b := L*(rho*(1-t)^2+sigma*t^2+2*tau*t*(1-t)) /((1-t)^2*r^2+2*(1-t)*t*q+t^2*p^2)^3;
			w2c := L*(rho*(1-t)^2+sigma*t^2+2*tau*t*(1-t))^2 /((1-t)^2*r^2+2*(1-t)*t*q+t^2*p^2)^4;
			# abbreviations from (*@\eqref{eqn:abbreviationsII}@*) and (*@\eqref{eqn:timeIntegralsII}@*)
			Phi1 := Matrix([1,(V-1)/(1-v^2)]);
			Phi2 := Matrix([1,(V-1+(1/v^2-1)/3)/(1-v^2)^2]);
			Xi1 := Matrix([[3+2*v,1],[3,1-2*v]])/8/v/(1+v);
			Xi2 := Matrix([[8*v^3+10*v^2-5,2*v^2+4*v-1,2*v-1],[15*v^2,-12*v^3+3*v^2,6*v^4-6*v^3+3*v^2]])/48/v^3/(1+v);
			Theta1 := Matrix([[sigma*r^3+rho*p^3],[r*p*((sigma+2*tau)*r+(rho+2*tau)*p)]])/(r*p)^4;
			Theta2 := Matrix([[sigma^2*r^5+rho^2*p^5], [r*p*(sigma*(sigma+4*tau)*r^3+rho*(rho+4*tau)*p^3)],[2*(r*p)^2*((rho*sigma+2*tau^2)*(r+p)+2*tau*(sigma*r+rho*p))]])/(r*p)^6;
			u := sqrt(r^2*p^2-q^2);
			v := q/r/p;
			V := (arctan((r^2-q)/u)+arctan((p^2-q)/u))/u*q;
			# formulas for integrals of w0,w1,w2a,w2b,w2c, \ie coefficients from (*@\eqref{eqn:barW}@*)
			integralW0 := (r+p)/2;
			integralW1 := ((1/v-1)*(r+p)*V+(r-p)*log(r/p))/(r^2+p^2-2*q);
			integralW2a := ((r+p)/q*V+1/r+1/p)/(r*p+q)/2;
			integralW2b := Multiply(Phi1,Multiply(Xi1,Theta1))[1,1];
			integralW2c := Multiply(Phi2,Multiply(Xi2,Theta2))[1,1];
			# check whether formulas indeed coincide with the integrals (\ie their difference is 0)
			simplify(integralW0-int(w0,t=0..1));
			simplify(integralW1-int(w1,t=0..1));
			simplify(integralW2a-int(w2a,t=0..1));
			simplify(integralW2b-int(w2b,t=0..1));
			simplify(integralW2c-int(w2c,t=0..1));
		\end{lstlisting}

	\subsection{Explicit calculation of ground truth curvature} \label{subsec:explicit}
	\input{ExactCalculation}

	\section*{Acknowledgments}%
	This work was funded by the Deutsche Forschungsgemeinschaft (DFG, German Research Foundation) under Germany's Excellence Strategy 
	EXC  2047 -- 390685813, Hausdorff Center for Mathematics at the University of Bonn,
	EXC 2044 -- 390685587, Mathematics M\"unster: Dynamics--Geometry--Structure,
	and by the DFG under CRC 1450 InSight -- 431460824.

	\bibliographystyle{amsplain}

	\bibliography{bibtex/all.bib,bibtex/own.bib,bibtex/library.bib,bibtex/references.bib,bibtex/sgp2019references.bib,bibtex/allRuWi12_siam.bib,bibtex/LiteratureDiscussion.bib}
	
\end{document}

%% file: arxivheader.tex
\ifx\undefined\headeralreadyincluded
\let\headeralreadyincluded\relax

\usepackage{times}
\usepackage{graphicx}
\usepackage{amsmath}
\usepackage{amssymb,amsthm}
\usepackage{mathtools, amsmath, amsfonts, bbm}
\usepackage{yhmath}
\usepackage{dsfont, dutchcal}
\usepackage{xspace}
\usepackage{mathrsfs}
\usepackage[linesnumbered]{algorithm2e}
\usepackage{algorithmic}
\usepackage{float}
\usepackage{listings}
\usepackage{subfigure}
\usepackage{tikz}
\usetikzlibrary{shapes.arrows, arrows.meta, fadings, calc}
\usepackage{pgfplots}
\pgfplotsset{compat=1.6}
\usepackage{enumitem}
\usepackage{color}
\usepackage{wrapfig}  
\usepackage{marginnote}
\usepackage{graphpap}
\usepackage{stmaryrd}
\usepackage{import}
\usepackage{yhmath}

\usepackage{hyperref}
\usepackage{cleveref}
\usepackage{tabularx}
\usepackage{longtable}
\usepackage{todonotes}

\usepackage[numbers]{natbib}

\numberwithin{equation}{section}
\makeatletter
\newcommand{\LeftEqNo}{\let\veqno\@@leqno}
\makeatother

\newfloat{listing}{htbp}{lop}
\floatname{listing}{Listing}

\def\XXint#1#2#3{{\setbox0=\hbox{$#1{#2#3}{\int}$}
\vcenter{\hbox{$#2#3$}}\kern-.5\wd0}}

\newcommand{\beq}{\begin{equation*}}
\newcommand{\eeq}{\end{equation*}}
\newcommand{\beqn}{\begin{equation}}
\newcommand{\eeqn}{\end{equation}}
\newcommand{\beqa}{\begin{eqnarray*}}
\newcommand{\eeqa}{\end{eqnarray*}}
\newcommand{\beqan}{\begin{eqnarray}}
\newcommand{\eeqan}{\end{eqnarray}}

\newcommand{\notinclude}[1]{}

\newcommand{\x}{{u}} 

\newcommand{\g}[3]{{g}_{#1}\left(#2, #3\right)} 
\newcommand{\Dg}[4]{\left({D}_{#1}g_{#1}\right)\left( #2 \right)\left( #3, #4 \right)} 

\newcommand{\energy}{{\mathcal{W}}}  

\newcommand{\Log}{\mathrm{Log}}
\newcommand{\Exp}{\mathrm{Exp}}

\newcommand{\parTp}{{\mathrm{P}}} 
\newcommand{\ParTp}{{\mathbf{P}}} 

\newcommand{\sphere}{\mathrm S}

\newcommand{\yz}{\mathbf{z}}
\newcommand{\yc}{\mathbf{s}}
\newcommand{\myArg}[2]{[#1, #2]}

\newcommand{\w}{\mathbf{w}}
\newcommand{\cpath}{\mathbf{c}}
\newcommand{\riemann}{\mathrm{R}}
\newcommand{\Riemann}{\mathbf{R}}
\newcommand{\cov}{{\frac{D}{\d t}}} 
\newcommand{\covdir}[2][t]{{\frac{D_{{#2}}}{\d #1}}} 
\newcommand{\covt}{{\frac{D}{\d t}}} 
\newcommand{\covs}{{\frac{D}{\d s}}} 
\newcommand{\Covtdir}[3][t]{{\frac{\mathbf{{D}}^{{#2}}_{{#3}}}{\d #1}}} 
\newcommand{\Christoffeloperator}{\Gamma} 
\newcommand{\Christoffel}[3]{\Christoffeloperator_{#1}\!\left( #2, #3 \right)} 

\newcommand{\W}{\energy}
\newcommand{\WEps}[1][\epsilon]{{\W_{\mathrm{reg}}^{#1}}}
\newcommand{\WEpsFree}{\W_{\mathrm{rat}}}
\newcommand{\WGalerkin}{\W_{\mathrm{lin}}}

\newcommand{\pathenergy}{\boldsymbol{\mathcal{E}}}

\newcommand{\Pathenergy}{\boldsymbol{\mathbf{E}}}
\newcommand{\EEps}[2]{\boldsymbol{\mathbf{E}}_{\mathrm{reg}}^{#1,#2}}
\newcommand{\EEpsFree}[1]{\boldsymbol{\mathbf{E}}_{\mathrm{rat}}^{#1}}
\newcommand{\EGalerkin}[1]{\boldsymbol{\mathbf{E}}_{\mathrm{lin}}^{#1}}

\newcommand{\out}{{\mathrm{out}}}
\newcommand{\inn}{{\mathrm{in}}}

\newcommand\restr[2]{{\left.\kern-\nulldelimiterspace #1 \vphantom{\big|}\right|_{#2}}}

\newcommand{\R}{{\mathds{R}}}
\newcommand{\N}{{\mathds{N}}}

\newcommand{\Z}{\mathds{Z}}

\renewcommand{\d}{{\mathrm{d}}}

\newcommand{\dist}{{\mathrm{dist}}}

\newcommand{\Id}{{\mathrm{Id}}}

\newcommand{\y}{y}

\newcommand{\interpolation}{\mathcal{I}} 

\newcommand{\immersion}{\operatorname{Imm}}

\newcommand{\diffs}{\partial_s}
\newcommand{\dtheta}{\partial_\theta}
\newcommand{\lowerLength}[2]{{L^{-,\epsilon}[{#1},{#2}]}}
\newcommand{\upperLength}[2]{{L^{+,\epsilon}[{#1},{#2}]}}

\newcommand{\Sone}{{\sphere^1}}

\newcommand{\Geps}{{G_{\mathrm{reg}}^\epsilon}}
\newcommand{\GepsFree}{G_{\mathrm{rat}}}
\newcommand{\GMetric}{\widehat G}

\newcommand{\Feps}{F}

\newcommand{\llangle}{\langle\!\langle}
\newcommand{\rrangle}{\rangle\!\rangle}
\newcommand{\tripleArgCommut}[4]{#1\llangle#2,#3,#4\rrangle}

\definecolor{blue}{rgb}{0.137255,0,0.862745}
\definecolor{red}{rgb}{0.862745,0.137255,0}
\definecolor{myBlue}{HTML}{377EB8}
\definecolor{myOrange}{RGB}{255, 169, 87 }
\definecolor{myGreen}{RGB}{180, 255, 162  }
\definecolor{myGrey}{RGB}{187, 187, 187  }
\definecolor{myDarkGrey}{RGB}{100, 100, 100  }

\makeatletter
\DeclareRobustCommand\onedot{\futurelet\@let@token\@onedot}
\def\@onedot{\ifx\@let@token.\else.\null\fi\xspace}

\def\eg{\emph{e.g}\onedot} 
\def\ie{\emph{i.e}\onedot} 
\def\cf{\emph{cf}\onedot}

\def\etal{\emph{et al}\onedot}
\makeatother

\graphicspath{{./images/}}

\fi

%% file: figures/ConvergenceOfGeodesic_Wrat.tex
%
%
\begin{tikzpicture}
	
\begin{axis}[%
	width=0.36\linewidth,
	height=0.185\linewidth,
	at={(0\linewidth,0\linewidth)},
	scale only axis,
	xmode=log,
	xmin=4,
	xmax=512,
	xminorticks=true,
	xlabel style={font=\color{white!15!black}},
	xlabel={$K$},
	ymode=log,
	ymin=1e-07,
	ymax=0.1,
	yminorticks=true,
	ylabel style={font=\color{white!15!black}},
	ylabel={$\mathrm{err}$},
	axis background/.style={fill=white},
	legend style={legend cell align=left, align=left, draw=white!15!black, at={(0.0,0.0)}, anchor=south west, nodes={scale=0.7, transform shape}}
]
\addplot [color=black, dotted, line width=1.0pt]
  table[row sep=crcr]{%
4	0.00729836492668258\\
8	0.0020333767032737\\
16	0.000532480679165481\\
32	0.000135722857991566\\
64	3.41852855166103e-05\\
128	8.535128659255e-06\\
256	2.09253805112432e-06\\
512	4.84375919808601e-07\\
};
\addlegendentry{$\|\cdot\|_{L_{\theta}^{2}}$}

\addplot [color=black, dashed, line width=1.0pt]
  table[row sep=crcr]{%
4	0.0147872566075112\\
8	0.00412601602310762\\
16	0.00107948352909786\\
32	0.000275053148651443\\
64	6.92687228971723e-05\\
128	1.72898451525794e-05\\
256	4.23517534824718e-06\\
512	9.79251752270393e-07\\
};
\addlegendentry{$\|\cdot\|_{W_{\theta}^{1}}$}

\addplot [color=black, line width=1.0pt]
  table[row sep=crcr]{%
4	0.0513604788405696\\
8	0.0142502533327262\\
16	0.00371407267675322\\
32	0.000945403368324906\\
64	0.000238157560061044\\
128	5.95898281467599e-05\\
256	1.47776931075243e-05\\
512	3.72650420554712e-06\\
};
\addlegendentry{$\|\cdot\|_{W_{\theta}^{2}}$}

\addplot [color=black, line width=0.5pt, forget plot]
table[row sep=crcr]{%
32	0.0025\\
256	3.90625e-05\\
256	0.0025\\
32	0.0025\\
};
\node[right, align=left, inner sep=0]
at (axis cs:90,0.008) {1};
\node[right, align=left, inner sep=0]
at (axis cs:275,0.0004) {2};
\end{axis}
\end{tikzpicture}%

%% file: figures/ConvergenceOfGeodesic_Wreg.tex
%
%
\begin{tikzpicture}

\begin{axis}[%
width=0.36\linewidth,
height=0.185\linewidth,
at={(0\linewidth,0\linewidth)},
scale only axis,
xmode=log,
xmin=4,
xmax=512,
xminorticks=true,
xlabel style={font=\color{white!15!black}},
xlabel={$K$},
ymode=log,
ymin=0.0003,
ymax=0.5,
yminorticks=true,
ylabel style={font=\color{white!15!black}},
ylabel={$\mathrm{err}$},
axis background/.style={fill=white}
]
\addplot [color=black, dotted, line width=1.0pt, forget plot]
  table[row sep=crcr]{%
4	0.0952114612472326\\
8	0.040417382278872\\
16	0.0189977836850702\\
32	0.0092572220679775\\
64	0.00457512041687294\\
128	0.00227503269590808\\
256	0.00113450689350047\\
512	0.00056653564437074\\
};
\addplot [color=black, dashed, line width=1.0pt, forget plot]
  table[row sep=crcr]{%
4	0.152489295165835\\
8	0.0645120057954012\\
16	0.0302096426453452\\
32	0.0146918333582336\\
64	0.00725420966635802\\
128	0.00360560175785133\\
256	0.00179762889220594\\
512	0.000897575903310195\\
};
\addplot [color=black, line width=1.0pt, forget plot]
  table[row sep=crcr]{%
4	0.432351955625333\\
8	0.186099890847151\\
16	0.0856222407142038\\
32	0.0410508549266945\\
64	0.0201017148190532\\
128	0.00994706992603247\\
256	0.00494774372048869\\
512	0.00246731959190064\\
};
\addplot [color=black, line width=0.5pt, forget plot]
  table[row sep=crcr]{%
32	0.06\\
256	0.0075\\
256	0.06\\
32	0.06\\
};
\node[right, align=left, inner sep=0]
at (axis cs:90,0.105) {1};
\node[right, align=left, inner sep=0]
at (axis cs:280,0.021) {1};
\end{axis}
\end{tikzpicture}%

%% file: figures/ConvergenceOfExp.tex
%
%
\begin{tikzpicture}

\begin{axis}[%
width=0.8\linewidth,
height=0.4\linewidth,
at={(0\linewidth,0\linewidth)},
scale only axis,
xmode=log,
xmin=2,
xmax=2048,
xminorticks=true,
xlabel style={font=\color{white!15!black}},
xlabel={$K$},
ymode=log,
ymin=0.0001,
ymax=1,
yminorticks=true,
ylabel style={font=\color{white!15!black}},
ylabel={$\| \Exp^K_{c}(1,v) - \exp_c(v) \|_{W_{\theta}^{2,2}}$},
axis background/.style={fill=white}
]
\addplot [color=black, line width=1.0pt, forget plot]
  table[row sep=crcr]{%
2	0.615597093532558\\
4	0.33073909489575\\
8	0.204907426919112\\
16	0.149444651850588\\
32	0.11418102478941\\
64	0.0863701733098337\\
128	0.0641826777462328\\
256	0.0470188919077982\\
512	0.0341041455815237\\
1024	0.0245756805125043\\
2048	0.0176387650717385\\
};
\addplot [color=black, dashed, line width=1.0pt, forget plot]
  table[row sep=crcr]{%
2	0.712198709356086\\
4	0.398583140712982\\
8	0.213953765022756\\
16	0.111311296900181\\
32	0.0567708075074305\\
64	0.0285951767590589\\
128	0.0142667892714466\\
256	0.00704055734124486\\
512	0.00341168714252447\\
1024	0.00159328239338484\\
2048	0.000683115884956324\\
};
\addplot [color=black, line width=1.0pt, forget plot]
  table[row sep=crcr]{%
512	0.0262339581396336\\
128	0.0524679162792672\\
128	0.0262339581396336\\
512	0.0262339581396336\\
};
\node[right, align=left, inner sep=0]
at (axis cs:256,0.018) {1};
\node[right, align=left, inner sep=0]
at (axis cs:100,0.037) {$\frac{1}{2}$};
\addplot [color=black, dashed, line width=1.0pt, forget plot]
  table[row sep=crcr]{%
512	0.00262437472501882\\
128	0.0104974989000753\\
128	0.00262437472501882\\
512	0.00262437472501882\\
};
\node[right, align=left, inner sep=0]
at (axis cs:256,0.0017) {1};
\node[right, align=left, inner sep=0]
at (axis cs:100,0.005) {1};
\end{axis}

\begin{axis}[%
width=1.227\linewidth,
height=0.615\linewidth,
at={(-0.16\linewidth,-0.068\linewidth)},
scale only axis,
xmin=0,
xmax=1,
ymin=0,
ymax=1,
axis line style={draw=none},
ticks=none,
axis x line*=bottom,
axis y line*=left
]
\end{axis}
\end{tikzpicture}%

%% file: figures/ConvergenceOfCovDiv_OneSided.tex
%
%
\begin{tikzpicture}

\begin{axis}[%
width=0.8\linewidth,
height=0.393\linewidth,
at={(0\linewidth,0\linewidth)},
scale only axis,
xmode=log,
xmin=2,
xmax=65535,
xminorticks=true,
xlabel style={font=\color{white!15!black}},
xlabel={$K = \tau^{-1}$},
ymode=log,
ymin=8e-06,
ymax=10,
yminorticks=true,
ylabel style={font=\color{white!15!black}},
ylabel={$\left\|\Covtdir{\tau}{v}  w (c) - \covdir{v} w (c) \right\|_{W_\theta^{2,2}}$},
axis background/.style={fill=white},
legend style={legend cell align=left, align=left, draw=white!15!black, anchor = south west},
legend pos= south west
]
\addplot [color=black, line width=1.0pt]
  table[row sep=crcr]{%
65536	0.0058307514777766\\
46340.9500118416	0.00693618332515116\\
32768	0.00825205126078443\\
23170.4750059208	0.00981832544057771\\
16384	0.0116835496874425\\
11585.2375029604	0.0139046110009027\\
8192	0.0165504592243096\\
5792.6187514802	0.0197032726093665\\
4096	0.023461786626561\\
2896.3093757401	0.0279445291041583\\
2048	0.0332943092816984\\
1448.15468787005	0.0396835647008946\\
1024	0.0473212575755682\\
724.077343935025	0.0564617199053384\\
512	0.0674162003312385\\
362.038671967512	0.0805682982887318\\
256	0.0963950663245298\\
181.019335983756	0.115496731410404\\
128	0.138639984079384\\
90.5096679918781	0.166823485420677\\
64	0.201381274902706\\
45.254833995939	0.244153651039224\\
32	0.297783612633789\\
22.6274169979695	0.36625811302631\\
16	0.455951535279577\\
11.3137084989848	0.577762429578254\\
8	0.751819765585639\\
5.65685424949238	1.0189270241147\\
4	1.47279192399992\\
2.82842712474619	2.37365303736209\\
2	4.74325052436554\\
};
\addlegendentry{$\epsilon = \sqrt{\tau}$}

\addplot [color=black, line width=1.0pt, forget plot]
  table[row sep=crcr]{%
1024	0.0709818863633523\\
16384	0.0177454715908381\\
16384	0.0709818863633523\\
1024	0.0709818863633523\\
};

\node[right, align=left, inner sep=0]
at (axis cs:4096,0.12) {1};
\node[right, align=left, inner sep=0]
at (axis cs:18154.017,0.035) {$\frac{1}{2}$};
\addplot [color=black, dotted, line width=1.0pt]
  table[row sep=crcr]{%
65536	2.74590180807701e-05\\
46340.9500118416	4.08114374627823e-05\\
32768	5.6965737855782e-05\\
23170.4750059208	8.03926509282382e-05\\
16384	0.00011362099305372\\
11585.2375029604	0.000160666543939992\\
8192	0.00022718822122959\\
5792.6187514802	0.00032133493325483\\
4096	0.000454470861616828\\
2896.3093757401	0.000642783023834341\\
2048	0.000909169612723494\\
1448.15468787005	0.00128603243260417\\
1024	0.00181920445202369\\
724.077343935025	0.00257379460423537\\
512	0.00364200193825718\\
362.038671967512	0.00515479030651985\\
256	0.00729844466952612\\
181.019335983756	0.0103385753508756\\
128	0.0146551997828223\\
90.5096679918781	0.020794687261009\\
64	0.029548095006655\\
45.254833995939	0.0420723160548297\\
32	0.0600839802514862\\
22.6274169979695	0.0861844791650625\\
16	0.124439250812141\\
11.3137084989848	0.181496492392231\\
8	0.268984247544874\\
5.65685424949238	0.409387319248773\\
4	0.653260200334924\\
2.82842712474619	1.14369772488226\\
2	2.47582025557971\\
};
\addlegendentry{$\epsilon = \tau$}

\addplot [color=black, dotted, line width=1.0pt, forget plot]
  table[row sep=crcr]{%
2896.3093757401	0.000357101679907967\\
181.019335983756	0.00571362687852748\\
181.019335983756	0.000357101679907967\\
2896.3093757401	0.000357101679907967\\
};

\node[right, align=left, inner sep=0]
at (axis cs:724,0.0002) {1};
\node[right, align=left, inner sep=0]
at (axis cs:135,0.0015) {1};
\addplot [color=black, dashed, line width=1.0pt]
  table[row sep=crcr]{%
65536	2.59665882749693e-05\\
46340.9500118416	3.5995794874153e-05\\
32768	5.16995881547307e-05\\
23170.4750059208	7.24839437048846e-05\\
16384	0.000102474656236146\\
11585.2375029604	0.000145760237411035\\
8192	0.000209433076018296\\
5792.6187514802	0.000305519397122606\\
4096	0.000454470861616828\\
2896.3093757401	0.000691798109282771\\
2048	0.0010790020344368\\
1448.15468787005	0.00172201673250525\\
1024	0.00280235030064563\\
724.077343935025	0.0046304497832449\\
512	0.00773732430483599\\
362.038671967512	0.0130350676445676\\
256	0.0221030776012521\\
181.019335983756	0.0377168828196258\\
128	0.0648813979181993\\
90.5096679918781	0.113032848343987\\
64	0.201381274902706\\
45.254833995939	0.374559374316668\\
32	0.76261036090172\\
22.6274169979695	1.91966688270718\\
16	9.05430733550547\\
};
\addlegendentry{$\epsilon = 64 \tau^{3/2}$}

\addplot [color=black, dashed, line width=1.0pt, forget plot]
  table[row sep=crcr]{%
362.038671967512	0.021011640041872\\
2896.3093757401	0.000928592072341154\\
2896.3093757401	0.021011640041872\\
362.038671967512	0.021011640041872\\
};

\node[right, align=left, inner sep=0]
at (axis cs:1050,0.012) {1};
\node[right, align=left, inner sep=0]
at (axis cs:3209.207,0.004) {$\frac{3}{2}$};
\end{axis}
\end{tikzpicture}%

%% file: figures/ParallelTransportExample=3_AngleChanges_W1only.tex
%
%
\begin{tikzpicture}

\begin{axis}[%
width=0.3\linewidth,
height=0.17\linewidth,
at={(0\linewidth,0\linewidth)},
scale only axis,
xmin=0,
xmax=1,
xlabel style={font=\color{white!15!black}},
xlabel={$t=k/K$},
ymode=log,
ymin=1e-07,
ymax=0.01,
yminorticks=true,
ylabel style={font=\color{white!15!black}},
ylabel={$|K(\alpha_{k+1}-\alpha_k)|$},
axis background/.style={fill=white}
]
\addplot [color=black, line width=1.0pt, forget plot]
  table[row sep=crcr]{%
0	0.000116400516844806\\
0.0009765625	0.000116360739184529\\
0.001953125	0.000116321421955945\\
0.0029296875	0.000116281531745699\\
0.00390625	0.000116241927571537\\
0.0048828125	0.00011620270856838\\
0.005859375	0.000116161896812628\\
0.0068359375	0.000116122896315574\\
0.0078125	0.000116081379019306\\
0.0087890625	0.0001160418037216\\
0.009765625	0.000116001546075495\\
0.0107421875	0.000115961387791685\\
0.01171875	0.000115919540121467\\
0.0126953125	0.000115878955512017\\
0.013671875	0.000115838850661021\\
0.0146484375	0.00011579712531784\\
0.015625	0.000115756094146491\\
0.0166015625	0.000115714980893245\\
0.017578125	0.000115673548862105\\
0.0185546875	0.000115631630251301\\
0.01953125	0.000115590484256245\\
0.0205078125	0.000115549002885018\\
0.021484375	0.000115506664542409\\
0.0224609375	0.000115464554255595\\
0.0234375	0.000115422455792213\\
0.0244140625	0.000115380716351865\\
0.025390625	0.000115337682700556\\
0.0263671875	0.000115295545583649\\
0.02734375	0.000115253115836822\\
0.0283203125	0.000115209377327119\\
0.029296875	0.000115167405056127\\
0.0302734375	0.000115124561034463\\
0.03125	0.000115080775913157\\
0.0322265625	0.000115038253170496\\
0.033203125	0.000114994956447845\\
0.0341796875	0.000114951402565566\\
0.03515625	0.000114907104716622\\
0.0361328125	0.000114863943508681\\
0.037109375	0.000114820243652503\\
0.0380859375	0.00011477595398901\\
0.0390625	0.000114731959683922\\
0.0400390625	0.000114687646146194\\
0.041015625	0.000114643289180094\\
0.0419921875	0.000114598459504123\\
0.04296875	0.000114554558876989\\
0.0439453125	0.00011451009186203\\
0.044921875	0.000114464456146379\\
0.0458984375	0.000114420124418757\\
0.046875	0.000114375193106753\\
0.0478515625	0.000114329594907758\\
0.048828125	0.000114284911660434\\
0.0498046875	0.000114239160438956\\
0.05078125	0.000114193113859073\\
0.0517578125	0.000114148348757226\\
0.052734375	0.000114102157112939\\
0.0537109375	0.000114056619850089\\
0.0546875	0.000114010580091417\\
0.0556640625	0.000113964339789163\\
0.056640625	0.000113918091528831\\
0.0576171875	0.000113871521307374\\
0.05859375	0.000113825425842151\\
0.0595703125	0.000113778134391396\\
0.060546875	0.000113731570763775\\
0.0615234375	0.000113684805910452\\
0.0625	0.000113638483981049\\
0.0634765625	0.000113590014962028\\
0.064453125	0.000113543786937953\\
0.0654296875	0.000113496421818127\\
0.06640625	0.000113448505544511\\
0.0673828125	0.000113401091539345\\
0.068359375	0.000113353360802648\\
0.0693359375	0.000113305407012376\\
0.0703125	0.000113257640350639\\
0.0712890625	0.000113210043991785\\
0.072265625	0.000113161393528571\\
0.0732421875	0.000113113360384887\\
0.07421875	0.000113064876359203\\
0.0751953125	0.000113015704982899\\
0.076171875	0.000112967973564082\\
0.0771484375	0.000112918832655851\\
0.078125	0.000112869715167108\\
0.0791015625	0.000112820843014561\\
0.080078125	0.00011277118869657\\
0.0810546875	0.000112722677158672\\
0.08203125	0.000112673023068055\\
0.0830078125	0.000112623328732298\\
0.083984375	0.000112573587330189\\
0.0849609375	0.000112524296582706\\
0.0859375	0.000112473984472672\\
0.0869140625	0.000112424455210203\\
0.087890625	0.000112373974388902\\
0.0888671875	0.000112324059273305\\
0.08984375	0.000112273449531131\\
0.0908203125	0.000112222911639037\\
0.091796875	0.000112172812805511\\
0.0927734375	0.000112121217625827\\
0.09375	0.000112071554440263\\
0.0947265625	0.000112019807829711\\
0.095703125	0.000111968946384877\\
0.0966796875	0.000111917567210185\\
0.09765625	0.000111866585029929\\
0.0986328125	0.000111814755427986\\
0.099609375	0.000111763632048678\\
0.1005859375	0.000111711354520594\\
0.1015625	0.000111659710910317\\
0.1025390625	0.000111608100269223\\
0.103515625	0.000111555712692279\\
0.1044921875	0.000111504020196662\\
0.10546875	0.000111451591692457\\
0.1064453125	0.000111399086790698\\
0.107421875	0.000111346627818421\\
0.1083984375	0.000111294458747579\\
0.109375	0.000111241232389148\\
0.1103515625	0.000111188393475459\\
0.111328125	0.000111135190536515\\
0.1123046875	0.000111082562852971\\
0.11328125	0.000111029477238844\\
0.1142578125	0.000110975388679435\\
0.115234375	0.000110922178919282\\
0.1162109375	0.000110868869796832\\
0.1171875	0.000110814585923436\\
0.1181640625	0.000110761075347909\\
0.119140625	0.000110707401290711\\
0.1201171875	0.00011065305284319\\
0.12109375	0.00011059883922826\\
0.1220703125	0.000110544632889287\\
0.123046875	0.000110489849248552\\
0.1240234375	0.000110435431452061\\
0.125	0.000110380997739412\\
0.1259765625	0.000110326123603954\\
0.126953125	0.000110271327230294\\
0.1279296875	0.00011021597947547\\
0.12890625	0.000110161366364991\\
0.1298828125	0.000110105719841158\\
0.130859375	0.000110050414377838\\
0.1318359375	0.000109995318325673\\
0.1328125	0.000109939363028388\\
0.1337890625	0.000109883820869072\\
0.134765625	0.00010982774324475\\
0.1357421875	0.000109771747702325\\
0.13671875	0.000109716042516084\\
0.1376953125	0.000109659287545583\\
0.138671875	0.000109603083956245\\
0.1396484375	0.000109546943576788\\
0.140625	0.000109490402564916\\
0.1416015625	0.000109433721718233\\
0.142578125	0.000109376966520358\\
0.1435546875	0.000109319711100397\\
0.14453125	0.000109263257854764\\
0.1455078125	0.000109205932858458\\
0.146484375	0.000109148463707243\\
0.1474609375	0.000109091130752859\\
0.1484375	0.000109033614762666\\
0.1494140625	0.000108976055344101\\
0.150390625	0.000108917968418609\\
0.1513671875	0.000108860613181605\\
0.15234375	0.000108802439171996\\
0.1533203125	0.00010874415033868\\
0.154296875	0.000108686369685529\\
0.1552734375	0.000108627301642628\\
0.15625	0.000108569661506408\\
0.1572265625	0.000108510928612304\\
0.158203125	0.000108452420590766\\
0.1591796875	0.00010839302876775\\
0.16015625	0.000108335028699003\\
0.1611328125	0.000108275300135574\\
0.162109375	0.0001082162145849\\
0.1630859375	0.000108157724753255\\
0.1640625	0.000108097482097946\\
0.1650390625	0.000108038842199676\\
0.166015625	0.000107978489040761\\
0.1669921875	0.000107918891899317\\
0.16796875	0.00010785908239086\\
0.1689453125	0.0001077990630165\\
0.169921875	0.000107739195300383\\
0.1708984375	0.000107678871017924\\
0.171875	0.000107618838910639\\
0.1728515625	0.000107557856608764\\
0.173828125	0.000107497747421803\\
0.1748046875	0.000107437195993043\\
0.17578125	0.000107376133428261\\
0.1767578125	0.000107315245259088\\
0.177734375	0.000107254791373634\\
0.1787109375	0.000107193465055389\\
0.1796875	0.00010713163624132\\
0.1806640625	0.000107071261936653\\
0.181640625	0.000107008984741697\\
0.1826171875	0.000106947792801293\\
0.18359375	0.000106885577906723\\
0.1845703125	0.000106824128806693\\
0.185546875	0.000106762181758313\\
0.1865234375	0.000106699701518664\\
0.1875	0.000106637443423097\\
0.1884765625	0.000106576060943553\\
0.189453125	0.000106513190758051\\
0.1904296875	0.000106450831708571\\
0.19140625	0.000106388385120226\\
0.1923828125	0.00010632500925567\\
0.193359375	0.000106262664303358\\
0.1943359375	0.000106199437595933\\
0.1953125	0.00010613663357617\\
0.1962890625	0.00010607333911139\\
0.197265625	0.000106009986438949\\
0.1982421875	0.000105946217217934\\
0.19921875	0.000105883203787016\\
0.2001953125	0.000105819459122358\\
0.201171875	0.000105755839513222\\
0.2021484375	0.000105691607814151\\
0.203125	0.000105627744233061\\
0.2041015625	0.000105563873603387\\
0.205078125	0.000105499521168895\\
0.2060546875	0.000105435351315464\\
0.20703125	0.00010537079833739\\
0.2080078125	0.000105306002069483\\
0.208984375	0.000105241201026729\\
0.2099609375	0.000105176986835431\\
0.2109375	0.000105111403627234\\
0.2119140625	0.000105046558019239\\
0.212890625	0.000104981732874876\\
0.2138671875	0.000104916016198331\\
0.21484375	0.000104851342939583\\
0.2158203125	0.000104784937548175\\
0.216796875	0.000104720288391036\\
0.2177734375	0.000104653673815847\\
0.21875	0.000104588540580153\\
0.2197265625	0.000104522456467748\\
0.220703125	0.000104455796872571\\
0.2216796875	0.00010439023753861\\
0.22265625	0.000104323850109722\\
0.2236328125	0.00010425747291265\\
0.224609375	0.000104191082073157\\
0.2255859375	0.000104124325844168\\
0.2265625	0.000104057534827007\\
0.2275390625	0.000103990974594126\\
0.228515625	0.000103924006907619\\
0.2294921875	0.000103856849364092\\
0.23046875	0.000103789690683698\\
0.2314453125	0.000103722157973607\\
0.232421875	0.000103654849908708\\
0.2333984375	0.000103587168496233\\
0.234375	0.000103519964568477\\
0.2353515625	0.000103451788618258\\
0.236328125	0.000103384420754082\\
0.2373046875	0.000103316196600645\\
0.23828125	0.000103248322375293\\
0.2392578125	0.000103179910865947\\
0.240234375	0.000103111607359097\\
0.2412109375	0.000103043476201492\\
0.2421875	0.000102974347782947\\
0.2431640625	0.000102905854191704\\
0.244140625	0.000102837557960811\\
0.2451171875	0.000102767962061989\\
0.24609375	0.000102700000979894\\
0.2470703125	0.000102630443507223\\
0.248046875	0.000102561189805783\\
0.2490234375	0.000102492184169023\\
0.25	0.000102422156487592\\
0.2509765625	0.000102353531701738\\
0.251953125	0.000102283165460904\\
0.2529296875	0.000102213648915495\\
0.25390625	0.000102143710364544\\
0.2548828125	0.000102074087863002\\
0.255859375	0.00010200386032011\\
0.2568359375	0.000101933571158952\\
0.2578125	0.000101863385225442\\
0.2587890625	0.000101792830719205\\
0.259765625	0.000101722338513355\\
0.2607421875	0.000101651837212557\\
0.26171875	0.00010158130066884\\
0.2626953125	0.00010151002652492\\
0.263671875	0.000101439915169976\\
0.2646484375	0.000101368176501637\\
0.265625	0.000101296770026238\\
0.2666015625	0.0001012256834656\\
0.267578125	0.000101154270623738\\
0.2685546875	0.000101082623814364\\
0.26953125	0.000101011404694873\\
0.2705078125	0.000100939307458248\\
0.271484375	0.000100867774335711\\
0.2724609375	0.000100796019069094\\
0.2734375	0.000100723551213378\\
0.2744140625	0.000100651165212184\\
0.275390625	0.00010057938970931\\
0.2763671875	0.000100506897751984\\
0.27734375	0.000100434451724141\\
0.2783203125	0.000100361640306801\\
0.279296875	0.000100289301599332\\
0.2802734375	0.00010021633352153\\
0.28125	0.000100142909332135\\
0.2822265625	0.00010007061723627\\
0.283203125	9.99974047317664e-05\\
0.2841796875	9.99240187411488e-05\\
0.28515625	9.98508660359221e-05\\
0.2861328125	9.97774741335888e-05\\
0.287109375	9.9703279829555e-05\\
0.2880859375	9.96298133486562e-05\\
0.2890625	9.95563323158422e-05\\
0.2900390625	9.94825834368385e-05\\
0.291015625	9.94082681700093e-05\\
0.2919921875	9.93347359781183e-05\\
0.29296875	9.92600275822042e-05\\
0.2939453125	9.91857098142646e-05\\
0.294921875	9.91115891793015e-05\\
0.2958984375	9.9036868505209e-05\\
0.296875	9.89621437383903e-05\\
0.2978515625	9.88871834124438e-05\\
0.298828125	9.88128485914785e-05\\
0.2998046875	9.87374005489983e-05\\
0.30078125	9.86626275789604e-05\\
0.3017578125	9.85874664820585e-05\\
0.302734375	9.85120400400774e-05\\
0.3037109375	9.84371765753167e-05\\
0.3046875	9.83609822924336e-05\\
0.3056640625	9.82854633093666e-05\\
0.306640625	9.82098224540096e-05\\
0.3076171875	9.81344742285728e-05\\
0.30859375	9.80584845819976e-05\\
0.3095703125	9.79823357738496e-05\\
0.310546875	9.79058879693184e-05\\
0.3115234375	9.78300686256262e-05\\
0.3125	9.77533538844e-05\\
0.3134765625	9.76775622802961e-05\\
0.314453125	9.7600338904158e-05\\
0.3154296875	9.75242735421489e-05\\
0.31640625	9.74469608081563e-05\\
0.3173828125	9.73704925399943e-05\\
0.318359375	9.72933075900073e-05\\
0.3193359375	9.72161626577872e-05\\
0.3203125	9.71393301369972e-05\\
0.3212890625	9.70618630162789e-05\\
0.322265625	9.69842158156098e-05\\
0.3232421875	9.69070949849993e-05\\
0.32421875	9.6829651738517e-05\\
0.3251953125	9.6752310582815e-05\\
0.326171875	9.66736758982734e-05\\
0.3271484375	9.65960373378039e-05\\
0.328125	9.65183153311955e-05\\
0.3291015625	9.64399114309344e-05\\
0.330078125	9.63620288985112e-05\\
0.3310546875	9.62831868491776e-05\\
0.33203125	9.62051990427426e-05\\
0.3330078125	9.61266614467604e-05\\
0.333984375	9.60478967044764e-05\\
0.3349609375	9.59694548328116e-05\\
0.3359375	9.58903256105259e-05\\
0.3369140625	9.58113630531443e-05\\
0.337890625	9.57323809416266e-05\\
0.3388671875	9.56532378495467e-05\\
0.33984375	9.55745224473503e-05\\
0.3408203125	9.54944835029892e-05\\
0.341796875	9.54154706960253e-05\\
0.3427734375	9.53357575781411e-05\\
0.34375	9.52562554630276e-05\\
0.3447265625	9.51764288856793e-05\\
0.345703125	9.50969538280333e-05\\
0.3466796875	9.50165429003391e-05\\
0.34765625	9.49365296492033e-05\\
0.3486328125	9.48563292695326e-05\\
0.349609375	9.47761398037983e-05\\
0.3505859375	9.46960344663239e-05\\
0.3515625	9.46152488268126e-05\\
0.3525390625	9.45345148011256e-05\\
0.353515625	9.44543303376122e-05\\
0.3544921875	9.4373305046247e-05\\
0.35546875	9.42927040341601e-05\\
0.3564453125	9.42111905715137e-05\\
0.357421875	9.41301834700425e-05\\
0.3583984375	9.40498152885993e-05\\
0.359375	9.39677177029807e-05\\
0.3603515625	9.38871028210997e-05\\
0.361328125	9.38047417093912e-05\\
0.3623046875	9.37240401981398e-05\\
0.36328125	9.36418487071933e-05\\
0.3642578125	9.35605391987337e-05\\
0.365234375	9.34782854074001e-05\\
0.3662109375	9.33966568936739e-05\\
0.3671875	9.33141388941294e-05\\
0.3681640625	9.32321331674757e-05\\
0.369140625	9.3150336851977e-05\\
0.3701171875	9.30677952055703e-05\\
0.37109375	9.29849179556186e-05\\
0.3720703125	9.29030777570006e-05\\
0.373046875	9.28199567624688e-05\\
0.3740234375	9.2737465593018e-05\\
0.375	9.26545014863223e-05\\
0.3759765625	9.25716835808998e-05\\
0.376953125	9.24884202504472e-05\\
0.3779296875	9.24052910704631e-05\\
0.37890625	9.23220620734355e-05\\
0.3798828125	9.223905158251e-05\\
0.380859375	9.21553321404645e-05\\
0.3818359375	9.20715770007519e-05\\
0.3828125	9.19884359973366e-05\\
0.3837890625	9.1903990096398e-05\\
0.384765625	9.18204798381339e-05\\
0.3857421875	9.17369013677671e-05\\
0.38671875	9.16529700134561e-05\\
0.3876953125	9.15683442599402e-05\\
0.388671875	9.14846493742516e-05\\
0.3896484375	9.14004442620353e-05\\
0.390625	9.13159076390002e-05\\
0.3916015625	9.12308089482394e-05\\
0.392578125	9.11471649942541e-05\\
0.3935546875	9.10617961835669e-05\\
0.39453125	9.09772536488163e-05\\
0.3955078125	9.08919548692211e-05\\
0.396484375	9.0807400738413e-05\\
0.3974609375	9.07219334749243e-05\\
0.3984375	9.06370448774396e-05\\
0.3994140625	9.05516312741383e-05\\
0.400390625	9.04667119812075e-05\\
0.4013671875	9.03809086594265e-05\\
0.40234375	9.02951494481385e-05\\
0.4033203125	9.02100637176773e-05\\
0.404296875	9.01239322956826e-05\\
0.4052734375	9.00380741768458e-05\\
0.40625	8.99521528481273e-05\\
0.4072265625	8.98663754469453e-05\\
0.408203125	8.97801901373896e-05\\
0.4091796875	8.9694033476917e-05\\
0.41015625	8.96077579000121e-05\\
0.4111328125	8.95213072453771e-05\\
0.412109375	8.94345848792e-05\\
0.4130859375	8.93481330876966e-05\\
0.4140625	8.92614996246266e-05\\
0.4150390625	8.91743663942179e-05\\
0.416015625	8.90876665380347e-05\\
0.4169921875	8.90009455360996e-05\\
0.41796875	8.8913623812914e-05\\
0.4189453125	8.88263807610201e-05\\
0.419921875	8.87394855908497e-05\\
0.4208984375	8.86523123426741e-05\\
0.421875	8.85645836206095e-05\\
0.4228515625	8.84765811406396e-05\\
0.423828125	8.83895163497073e-05\\
0.4248046875	8.83019579305255e-05\\
0.42578125	8.82140700468881e-05\\
0.4267578125	8.81258056324441e-05\\
0.427734375	8.8038117837641e-05\\
0.4287109375	8.79499245911575e-05\\
0.4296875	8.78624650795246e-05\\
0.4306640625	8.77735838002991e-05\\
0.431640625	8.76852182045695e-05\\
0.4326171875	8.75968453328824e-05\\
0.43359375	8.75085711413703e-05\\
0.4345703125	8.74197342000116e-05\\
0.435546875	8.73311350915174e-05\\
0.4365234375	8.72422906468273e-05\\
0.4375	8.71532085966464e-05\\
0.4384765625	8.70646490511717e-05\\
0.439453125	8.69747871092841e-05\\
0.4404296875	8.68863799041719e-05\\
0.44140625	8.67966984969826e-05\\
0.4423828125	8.67079891122557e-05\\
0.443359375	8.6617969600411e-05\\
0.4443359375	8.65283584516874e-05\\
0.4453125	8.64394733071094e-05\\
0.4462890625	8.63489449329791e-05\\
0.447265625	8.62595341004635e-05\\
0.4482421875	8.61699220422452e-05\\
0.44921875	8.60794366417394e-05\\
0.4501953125	8.59901163039467e-05\\
0.451171875	8.58996429542458e-05\\
0.4521484375	8.58088576478622e-05\\
0.453125	8.57188138070342e-05\\
0.4541015625	8.56285121244582e-05\\
0.455078125	8.55381310884695e-05\\
0.4560546875	8.5447754145207e-05\\
0.45703125	8.53570304570894e-05\\
0.4580078125	8.526602641723e-05\\
0.458984375	8.51749070989172e-05\\
0.4599609375	8.50840360726579e-05\\
0.4609375	8.49937655402755e-05\\
0.4619140625	8.49021039357467e-05\\
0.462890625	8.48112217681773e-05\\
0.4638671875	8.47198525661952e-05\\
0.46484375	8.46285431634897e-05\\
0.4658203125	8.45367096644623e-05\\
0.466796875	8.44455553306034e-05\\
0.4677734375	8.43539505694935e-05\\
0.46875	8.42621554966172e-05\\
0.4697265625	8.41704786580522e-05\\
0.470703125	8.40784889533097e-05\\
0.4716796875	8.3986510617251e-05\\
0.47265625	8.38944235965755e-05\\
0.4736328125	8.38024666336423e-05\\
0.474609375	8.37100942590041e-05\\
0.4755859375	8.3618009057318e-05\\
0.4765625	8.35250891668693e-05\\
0.4775390625	8.34330819543538e-05\\
0.478515625	8.3340133414822e-05\\
0.4794921875	8.32480636745458e-05\\
0.48046875	8.31547624784434e-05\\
0.4814453125	8.30618905638403e-05\\
0.482421875	8.29691391572851e-05\\
0.4833984375	8.28760973945464e-05\\
0.484375	8.2783410562115e-05\\
0.4853515625	8.26901459731744e-05\\
0.486328125	8.25967356377078e-05\\
0.4873046875	8.2503709336379e-05\\
0.48828125	8.24098810880969e-05\\
0.4892578125	8.23168306851585e-05\\
0.490234375	8.22228876131703e-05\\
0.4912109375	8.21296041522146e-05\\
0.4921875	8.20354134702939e-05\\
0.4931640625	8.19419462914084e-05\\
0.494140625	8.18482501472317e-05\\
0.4951171875	8.17538175397203e-05\\
0.49609375	8.16601516362425e-05\\
0.4970703125	8.15652347228024e-05\\
0.498046875	8.14718905530754e-05\\
0.4990234375	8.13768526768399e-05\\
0.5	8.12831342500431e-05\\
0.5009765625	8.11884651739092e-05\\
0.501953125	8.10935664503631e-05\\
0.5029296875	8.09993246093654e-05\\
0.50390625	8.09047378425021e-05\\
0.5048828125	8.08095253432839e-05\\
0.505859375	8.07147630439431e-05\\
0.5068359375	8.06196026132966e-05\\
0.5078125	8.05249867426028e-05\\
0.5087890625	8.04295411853673e-05\\
0.509765625	8.0334653830505e-05\\
0.5107421875	8.02388713054825e-05\\
0.51171875	8.01443516138534e-05\\
0.5126953125	8.00483064722357e-05\\
0.513671875	7.99528679635841e-05\\
0.5146484375	7.98575665612589e-05\\
0.515625	7.97616110048693e-05\\
0.5166015625	7.96660485775647e-05\\
0.517578125	7.95702412688115e-05\\
0.5185546875	7.94742518337443e-05\\
0.51953125	7.9378111195183e-05\\
0.5205078125	7.9281867328973e-05\\
0.521484375	7.91861330071697e-05\\
0.5224609375	7.90894864621805e-05\\
0.5234375	7.89933046689839e-05\\
0.5244140625	7.8897055118432e-05\\
0.525390625	7.88006959737686e-05\\
0.5263671875	7.87040610248368e-05\\
0.52734375	7.86074465395359e-05\\
0.5283203125	7.85104982696794e-05\\
0.529296875	7.84138794642786e-05\\
0.5302734375	7.83177458743012e-05\\
0.53125	7.82196814270719e-05\\
0.5322265625	7.81234105033946e-05\\
0.533203125	7.80260181727499e-05\\
0.5341796875	7.79286435772519e-05\\
0.53515625	7.78319742948952e-05\\
0.5361328125	7.77342290803062e-05\\
0.537109375	7.76373301505373e-05\\
0.5380859375	7.7539557878481e-05\\
0.5390625	7.74420709603874e-05\\
0.5400390625	7.73446690800483e-05\\
0.541015625	7.72469595631264e-05\\
0.5419921875	7.7149330309112e-05\\
0.54296875	7.70508711411821e-05\\
0.5439453125	7.69532882713975e-05\\
0.544921875	7.68554386922915e-05\\
0.5458984375	7.67573965276824e-05\\
0.546875	7.66592693253187e-05\\
0.5478515625	7.65608567689924e-05\\
0.548828125	7.64627386615757e-05\\
0.5498046875	7.63646742143465e-05\\
0.55078125	7.62654781283345e-05\\
0.5517578125	7.61677722493914e-05\\
0.552734375	7.60689993057895e-05\\
0.5537109375	7.59701163133286e-05\\
0.5546875	7.58717292228539e-05\\
0.5556640625	7.57728039388894e-05\\
0.556640625	7.56738706968463e-05\\
0.5576171875	7.55747316816269e-05\\
0.55859375	7.54759792016557e-05\\
0.5595703125	7.53767590140342e-05\\
0.560546875	7.52778432797641e-05\\
0.5615234375	7.51786080854799e-05\\
0.5625	7.50790570691606e-05\\
0.5634765625	7.49799389723194e-05\\
0.564453125	7.48804893646593e-05\\
0.5654296875	7.47808203414024e-05\\
0.56640625	7.46812650049833e-05\\
0.5673828125	7.45815414120443e-05\\
0.568359375	7.44821400076034e-05\\
0.5693359375	7.43821913147258e-05\\
0.5703125	7.4282023433625e-05\\
0.5712890625	7.41824601391272e-05\\
0.572265625	7.40823620617448e-05\\
0.5732421875	7.39822980904137e-05\\
0.57421875	7.38819685466297e-05\\
0.5751953125	7.37818063498707e-05\\
0.576171875	7.3681664389369e-05\\
0.5771484375	7.3581275955803e-05\\
0.578125	7.3480816809024e-05\\
0.5791015625	7.33806048174301e-05\\
0.580078125	7.32799576326215e-05\\
0.5810546875	7.31793800241576e-05\\
0.58203125	7.30783729068207e-05\\
0.5830078125	7.29778223558242e-05\\
0.583984375	7.2876980084402e-05\\
0.5849609375	7.27757044387545e-05\\
0.5859375	7.26754119568795e-05\\
0.5869140625	7.25740185316681e-05\\
0.587890625	7.24729377452604e-05\\
0.5888671875	7.23715945696313e-05\\
0.58984375	7.22705517546274e-05\\
0.5908203125	7.21689855254226e-05\\
0.591796875	7.20678590369062e-05\\
0.5927734375	7.19662878054805e-05\\
0.59375	7.18650246653851e-05\\
0.5947265625	7.17631858151435e-05\\
0.595703125	7.16617967100319e-05\\
0.5966796875	7.15598191618483e-05\\
0.59765625	7.14583650278655e-05\\
0.5986328125	7.13563615590829e-05\\
0.599609375	7.12545415808563e-05\\
0.6005859375	7.11523384779866e-05\\
0.6015625	7.10503040863841e-05\\
0.6025390625	7.09483192622429e-05\\
0.603515625	7.08460204350558e-05\\
0.6044921875	7.07440315181884e-05\\
0.60546875	7.06418852587376e-05\\
0.6064453125	7.05391205428896e-05\\
0.607421875	7.04372057498404e-05\\
0.6083984375	7.03341606822505e-05\\
0.609375	7.02317609011516e-05\\
0.6103515625	7.01290898632578e-05\\
0.611328125	7.00266650710546e-05\\
0.6123046875	6.99239628829673e-05\\
0.61328125	6.98207336427004e-05\\
0.6142578125	6.97185462286143e-05\\
0.615234375	6.96150511885207e-05\\
0.6162109375	6.95122769229783e-05\\
0.6171875	6.94092129833734e-05\\
0.6181640625	6.93060515004618e-05\\
0.619140625	6.92030564550805e-05\\
0.6201171875	6.90997137553495e-05\\
0.62109375	6.89966959726007e-05\\
0.6220703125	6.88928971612768e-05\\
0.623046875	6.87895997089072e-05\\
0.6240234375	6.86862786096754e-05\\
0.625	6.85823908952443e-05\\
0.6259765625	6.84788674334413e-05\\
0.626953125	6.83754483361554e-05\\
0.6279296875	6.82717475228856e-05\\
0.62890625	6.81678097862459e-05\\
0.6298828125	6.80641874168941e-05\\
0.630859375	6.79598590522801e-05\\
0.6318359375	6.78559540574497e-05\\
0.6328125	6.77521352372423e-05\\
0.6337890625	6.76481590744515e-05\\
0.634765625	6.75435521770851e-05\\
0.6357421875	6.74394839279557e-05\\
0.63671875	6.73354045375163e-05\\
0.6376953125	6.72309952278738e-05\\
0.638671875	6.7126398334949e-05\\
0.6396484375	6.70224562782096e-05\\
0.640625	6.69176129122206e-05\\
0.6416015625	6.68130878693773e-05\\
0.642578125	6.67086624162039e-05\\
0.6435546875	6.66038031340577e-05\\
0.64453125	6.64991380290303e-05\\
0.6455078125	6.63944540519879e-05\\
0.646484375	6.62894344714005e-05\\
0.6474609375	6.61843855596089e-05\\
0.6484375	6.60794573832391e-05\\
0.6494140625	6.59747831832647e-05\\
0.650390625	6.58692883916956e-05\\
0.6513671875	6.57644138755131e-05\\
0.65234375	6.56590439120919e-05\\
0.6533203125	6.55536814520019e-05\\
0.654296875	6.54489533644664e-05\\
0.6552734375	6.53433494335331e-05\\
0.65625	6.52378009817767e-05\\
0.6572265625	6.51320949600631e-05\\
0.658203125	6.50269685138483e-05\\
0.6591796875	6.49212647658715e-05\\
0.66015625	6.48158652438724e-05\\
0.6611328125	6.47096812826931e-05\\
0.662109375	6.46040775791334e-05\\
0.6630859375	6.44982455924037e-05\\
0.6640625	6.43929454327008e-05\\
0.6650390625	6.42866325506475e-05\\
0.666015625	6.4180799654423e-05\\
0.6669921875	6.4074831243488e-05\\
0.66796875	6.39686606973555e-05\\
0.6689453125	6.3862635215628e-05\\
0.669921875	6.37565340184665e-05\\
0.6708984375	6.36499485153763e-05\\
0.671875	6.35439696452522e-05\\
0.6728515625	6.34376524430991e-05\\
0.673828125	6.33312747595483e-05\\
0.6748046875	6.32246510576806e-05\\
0.67578125	6.31178536423249e-05\\
0.6767578125	6.30118327080709e-05\\
0.677734375	6.29052137810504e-05\\
0.6787109375	6.27983258709719e-05\\
0.6796875	6.26919620572153e-05\\
0.6806640625	6.25849190782901e-05\\
0.681640625	6.24780989255669e-05\\
0.6826171875	6.23716507561767e-05\\
0.68359375	6.22640759502247e-05\\
0.6845703125	6.21579413291329e-05\\
0.685546875	6.2050585711404e-05\\
0.6865234375	6.19433585598017e-05\\
0.6875	6.18363312696601e-05\\
0.6884765625	6.17292234892375e-05\\
0.689453125	6.16222152984847e-05\\
0.6904296875	6.15146391282906e-05\\
0.69140625	6.14073153428762e-05\\
0.6923828125	6.1300269635467e-05\\
0.693359375	6.11928203397838e-05\\
0.6943359375	6.10854310707509e-05\\
0.6953125	6.09780038303143e-05\\
0.6962890625	6.08699538133806e-05\\
0.697265625	6.07631150160159e-05\\
0.6982421875	6.06548876476154e-05\\
0.69921875	6.05476716373232e-05\\
0.7001953125	6.04400031534169e-05\\
0.701171875	6.03316480010108e-05\\
0.7021484375	6.02244347192027e-05\\
0.703125	6.01161032136588e-05\\
0.7041015625	6.0008241689502e-05\\
0.705078125	5.99002312355879e-05\\
0.7060546875	5.97923742589046e-05\\
0.70703125	5.96843449329754e-05\\
0.7080078125	5.95761316617427e-05\\
0.708984375	5.94682098835619e-05\\
0.7099609375	5.93599802414246e-05\\
0.7109375	5.92515871176147e-05\\
0.7119140625	5.91433749832504e-05\\
0.712890625	5.90350423408381e-05\\
0.7138671875	5.8926583506036e-05\\
0.71484375	5.88182101637358e-05\\
0.7158203125	5.87097019888461e-05\\
0.716796875	5.860145324732e-05\\
0.7177734375	5.84926117426221e-05\\
0.71875	5.83840471790609e-05\\
0.7197265625	5.82753973503713e-05\\
0.720703125	5.81668732593243e-05\\
0.7216796875	5.80580615405779e-05\\
0.72265625	5.79490956624795e-05\\
0.7236328125	5.78407516513835e-05\\
0.724609375	5.77316732233157e-05\\
0.7255859375	5.76223833377298e-05\\
0.7265625	5.75140500131965e-05\\
0.7275390625	5.74044786390004e-05\\
0.728515625	5.72959452256327e-05\\
0.7294921875	5.71866657992359e-05\\
0.73046875	5.7077340898104e-05\\
0.7314453125	5.69683850244473e-05\\
0.732421875	5.68594216474594e-05\\
0.7333984375	5.67498873351724e-05\\
0.734375	5.66407409223757e-05\\
0.7353515625	5.65316850043018e-05\\
0.736328125	5.64218796625937e-05\\
0.7373046875	5.6312681408599e-05\\
0.73828125	5.62031466415647e-05\\
0.7392578125	5.60937498903513e-05\\
0.740234375	5.59840611913387e-05\\
0.7412109375	5.58747833565576e-05\\
0.7421875	5.57647365440062e-05\\
0.7431640625	5.56556344690762e-05\\
0.744140625	5.55460349005443e-05\\
0.7451171875	5.54361854483432e-05\\
0.74609375	5.53263662368408e-05\\
0.7470703125	5.52165181488817e-05\\
0.748046875	5.5106868785515e-05\\
0.7490234375	5.49969340681855e-05\\
0.75	5.4886944781174e-05\\
0.7509765625	5.47769027434697e-05\\
0.751953125	5.46672013115312e-05\\
0.7529296875	5.45573366252938e-05\\
0.75390625	5.44467202416854e-05\\
0.7548828125	5.433702085611e-05\\
0.755859375	5.42266789125279e-05\\
0.7568359375	5.41168142262904e-05\\
0.7578125	5.40060871117021e-05\\
0.7587890625	5.38962242444541e-05\\
0.759765625	5.37858763891563e-05\\
0.7607421875	5.367561584535e-05\\
0.76171875	5.35651511199831e-05\\
0.7626953125	5.34546802555269e-05\\
0.763671875	5.33444072061684e-05\\
0.7646484375	5.32341261987312e-05\\
0.765625	5.31234757090715e-05\\
0.7666015625	5.3013117167211e-05\\
0.767578125	5.29024800925981e-05\\
0.7685546875	5.27920453805564e-05\\
0.76953125	5.2681422630485e-05\\
0.7705078125	5.25703676430567e-05\\
0.771484375	5.24601596225693e-05\\
0.7724609375	5.23493661148677e-05\\
0.7734375	5.2238266562199e-05\\
0.7744140625	5.21277327152347e-05\\
0.775390625	5.20168093771645e-05\\
0.7763671875	5.19060552051087e-05\\
0.77734375	5.1795017952827e-05\\
0.7783203125	5.16841676017066e-05\\
0.779296875	5.15732053827378e-05\\
0.7802734375	5.146219064045e-05\\
0.78125	5.1351162710489e-05\\
0.7822265625	5.12398537466652e-05\\
0.783203125	5.11295525029709e-05\\
0.7841796875	5.10178049353271e-05\\
0.78515625	5.09066915128642e-05\\
0.7861328125	5.07958675370901e-05\\
0.787109375	5.06842545746622e-05\\
0.7880859375	5.05730797613069e-05\\
0.7890625	5.04620875290129e-05\\
0.7900390625	5.0350370202068e-05\\
0.791015625	5.02391324062046e-05\\
0.7919921875	5.01281469951209e-05\\
0.79296875	5.00165622270288e-05\\
0.7939453125	4.99051268434414e-05\\
0.794921875	4.9793879952631e-05\\
0.7958984375	4.968229154656e-05\\
0.796875	4.95708700327668e-05\\
0.7978515625	4.94593728035397e-05\\
0.798828125	4.9347781441611e-05\\
0.7998046875	4.92361928081664e-05\\
0.80078125	4.91244445584016e-05\\
0.8017578125	4.90131546939665e-05\\
0.802734375	4.89009553348296e-05\\
0.8037109375	4.87896436425217e-05\\
0.8046875	4.86779074435617e-05\\
0.8056640625	4.85661510083446e-05\\
0.806640625	4.84538061300555e-05\\
0.8076171875	4.83426924802188e-05\\
0.80859375	4.8230493121082e-05\\
0.8095703125	4.81188551475498e-05\\
0.810546875	4.80069038530928e-05\\
0.8115234375	4.78948454656347e-05\\
0.8125	4.77832368233067e-05\\
0.8134765625	4.76710158636706e-05\\
0.814453125	4.75594069939689e-05\\
0.8154296875	4.74472199130105e-05\\
0.81640625	4.7334962118839e-05\\
0.8173828125	4.7223174078681e-05\\
0.818359375	4.7110917421378e-05\\
0.8193359375	4.69992521630047e-05\\
0.8203125	4.68865973743959e-05\\
0.8212890625	4.6774481461398e-05\\
0.822265625	4.66627705009159e-05\\
0.8232421875	4.65500834252452e-05\\
0.82421875	4.64381894289545e-05\\
0.8251953125	4.63257022147445e-05\\
0.826171875	4.62133882592752e-05\\
0.8271484375	4.61012289179052e-05\\
0.828125	4.59889736248442e-05\\
0.8291015625	4.58765432540531e-05\\
0.830078125	4.57644541711488e-05\\
0.8310546875	4.56516927442863e-05\\
0.83203125	4.55396611869219e-05\\
0.8330078125	4.54268133580626e-05\\
0.833984375	4.5314792487261e-05\\
0.8349609375	4.52018921350827e-05\\
0.8359375	4.5090006551618e-05\\
0.8369140625	4.49769031547476e-05\\
0.837890625	4.48648613655678e-05\\
0.8388671875	4.47521672413131e-05\\
0.83984375	4.46391479727026e-05\\
0.8408203125	4.45270291038469e-05\\
0.841796875	4.4413982777769e-05\\
0.8427734375	4.43016979261301e-05\\
0.84375	4.41889976627863e-05\\
0.8447265625	4.40762676134909e-05\\
0.845703125	4.39637310591934e-05\\
0.8466796875	4.38509600826364e-05\\
0.84765625	4.37384162523813e-05\\
0.8486328125	4.36254920259671e-05\\
0.849609375	4.35129268225865e-05\\
0.8505859375	4.34001349276514e-05\\
0.8515625	4.32870631357218e-05\\
0.8525390625	4.31746354934148e-05\\
0.853515625	4.3062004579042e-05\\
0.8544921875	4.29486383382027e-05\\
0.85546875	4.2836268676183e-05\\
0.8564453125	4.2723073875095e-05\\
0.857421875	4.26100505137583e-05\\
0.8583984375	4.24973945882812e-05\\
0.859375	4.23843509906874e-05\\
0.8603515625	4.22712971612782e-05\\
0.861328125	4.21587124037615e-05\\
0.8623046875	4.20454571212758e-05\\
0.86328125	4.19326700011879e-05\\
0.8642578125	4.18192378219828e-05\\
0.865234375	4.17065348301549e-05\\
0.8662109375	4.15936715398857e-05\\
0.8671875	4.14803023431887e-05\\
0.8681640625	4.13674554238241e-05\\
0.869140625	4.12543734000792e-05\\
0.8701171875	4.11413604979316e-05\\
0.87109375	4.10284997087729e-05\\
0.8720703125	4.09150547966419e-05\\
0.873046875	4.0802007106322e-05\\
0.8740234375	4.06886786095129e-05\\
0.875	4.05757159569475e-05\\
0.8759765625	4.04623920076119e-05\\
0.876953125	4.03494066176791e-05\\
0.8779296875	4.02362786644517e-05\\
0.87890625	4.01232350668579e-05\\
0.8798828125	4.00097846977587e-05\\
0.880859375	3.98963563839061e-05\\
0.8818359375	3.97836151933006e-05\\
0.8828125	3.96700802411942e-05\\
0.8837890625	3.95568426938553e-05\\
0.884765625	3.94437172417383e-05\\
0.8857421875	3.93304824228835e-05\\
0.88671875	3.92170079521748e-05\\
0.8876953125	3.91035282518715e-05\\
0.888671875	3.89904971598298e-05\\
0.8896484375	3.88773300983303e-05\\
0.890625	3.87639845484955e-05\\
0.8916015625	3.86501599223266e-05\\
0.892578125	3.85373314202297e-05\\
0.8935546875	3.84238840069884e-05\\
0.89453125	3.83102410523861e-05\\
0.8955078125	3.81970003218157e-05\\
0.896484375	3.80836502245074e-05\\
0.8974609375	3.79702680675109e-05\\
0.8984375	3.78571155579266e-05\\
0.8994140625	3.77434064375848e-05\\
0.900390625	3.76301206870266e-05\\
0.9013671875	3.7517227610806e-05\\
0.90234375	3.74031039882539e-05\\
0.9033203125	3.72902293293009e-05\\
0.904296875	3.71760843336233e-05\\
0.9052734375	3.70629429653491e-05\\
0.90625	3.69496267467184e-05\\
0.9072265625	3.68361718301458e-05\\
0.908203125	3.67224422461732e-05\\
0.9091796875	3.66092958756781e-05\\
0.91015625	3.64959691978584e-05\\
0.9111328125	3.63820036000106e-05\\
0.912109375	3.62689495432278e-05\\
0.9130859375	3.6155222915113e-05\\
0.9140625	3.60415854174789e-05\\
0.9150390625	3.592844313971e-05\\
0.916015625	3.5814928878608e-05\\
0.9169921875	3.57016369889607e-05\\
0.91796875	3.55880658844399e-05\\
0.9189453125	3.54743031039106e-05\\
0.919921875	3.5360908441362e-05\\
0.9208984375	3.52473889506655e-05\\
0.921875	3.51337553183839e-05\\
0.9228515625	3.50202642493969e-05\\
0.923828125	3.49069400726876e-05\\
0.9248046875	3.47933732882666e-05\\
0.92578125	3.46796246049053e-05\\
0.9267578125	3.45663434018206e-05\\
0.927734375	3.44527334164013e-05\\
0.9287109375	3.43389942827343e-05\\
0.9296875	3.42256576004729e-05\\
0.9306640625	3.41122492955037e-05\\
0.931640625	3.39987639108585e-05\\
0.9326171875	3.38851809829066e-05\\
0.93359375	3.37716166995961e-05\\
0.9345703125	3.36580114890239e-05\\
0.935546875	3.35443701260374e-05\\
0.9365234375	3.3431207157264e-05\\
0.9375	3.3317372754027e-05\\
0.9384765625	3.32040006014722e-05\\
0.939453125	3.30903260419291e-05\\
0.9404296875	3.2976877719193e-05\\
0.94140625	3.28634166635311e-05\\
0.9423828125	3.2749863976278e-05\\
0.943359375	3.26363419844711e-05\\
0.9443359375	3.25227936173178e-05\\
0.9453125	3.24091213315114e-05\\
0.9462890625	3.22959631375852e-05\\
0.947265625	3.21823349622719e-05\\
0.9482421875	3.20688236570277e-05\\
0.94921875	3.1955031090547e-05\\
0.9501953125	3.18421223255427e-05\\
0.951171875	3.1728373187434e-05\\
0.9521484375	3.16146606564871e-05\\
0.953125	3.15013521685614e-05\\
0.9541015625	3.138794681945e-05\\
0.955078125	3.12742354253714e-05\\
0.9560546875	3.11605622300704e-05\\
0.95703125	3.10473240006104e-05\\
0.9580078125	3.09339066006942e-05\\
0.958984375	3.08204935208778e-05\\
0.9599609375	3.07068758047535e-05\\
0.9609375	3.05934177049494e-05\\
0.9619140625	3.04805871564895e-05\\
0.962890625	3.03661040561565e-05\\
0.9638671875	3.02533035210217e-05\\
0.96484375	3.01397021758021e-05\\
0.9658203125	3.00263934605027e-05\\
0.966796875	2.99128184906294e-05\\
0.9677734375	2.97995402434026e-05\\
0.96875	2.96862544928445e-05\\
0.9697265625	2.95727365937637e-05\\
0.970703125	2.94594490242162e-05\\
0.9716796875	2.93461616820423e-05\\
0.97265625	2.92325053123932e-05\\
0.9736328125	2.911941783168e-05\\
0.974609375	2.90059504095552e-05\\
0.9755859375	2.88925014046981e-05\\
0.9765625	2.87791804112203e-05\\
0.9775390625	2.86660686015239e-05\\
0.978515625	2.85525848084944e-05\\
0.9794921875	2.84392251614918e-05\\
0.98046875	2.83258837043832e-05\\
0.9814453125	2.82127907667018e-05\\
0.982421875	2.80995334378531e-05\\
0.9833984375	2.79863024843507e-05\\
0.984375	2.78728150533425e-05\\
0.9853515625	2.77600261142652e-05\\
0.986328125	2.76462756119145e-05\\
0.9873046875	2.75332895398606e-05\\
0.98828125	2.7420018795965e-05\\
0.9892578125	2.73067435045959e-05\\
0.990234375	2.71933795374935e-05\\
0.9912109375	2.70806372100196e-05\\
0.9921875	2.69672395916132e-05\\
0.9931640625	2.68539092758147e-05\\
0.994140625	2.67411778622773e-05\\
0.9951171875	2.66275299054541e-05\\
0.99609375	2.65147555182921e-05\\
0.9970703125	2.64019743099198e-05\\
0.998046875	2.62883845607575e-05\\
};
\end{axis}
\end{tikzpicture}%

%% file: figures/ParallelTransportExample=4_AngleChanges_W1only.tex
%
%
\begin{tikzpicture}

\begin{axis}[%
width=0.3\linewidth,
height=0.17\linewidth,
at={(0\linewidth,0\linewidth)},
scale only axis,
xmin=0,
xmax=1,
xlabel style={font=\color{white!15!black}},
xlabel={$t=k/K$},
ymode=log,
ymin=1.0e-07,
ymax=0.01,
yminorticks=true,
ylabel style={font=\color{white!15!black}},
ylabel={$|K(\alpha_{k+1}-\alpha_k)|$},
axis background/.style={fill=white}
]
\addplot [color=black, line width=1.0pt, forget plot]
  table[row sep=crcr]{%
0	0.00157893435573442\\
0.0009765625	0.00157747760795246\\
0.001953125	0.00157602723879791\\
0.0029296875	0.00157458302373925\\
0.00390625	0.00157314647401563\\
0.0048828125	0.00157171609760098\\
0.005859375	0.00157029367505856\\
0.0068359375	0.00156887579214526\\
0.0078125	0.00156746538414154\\
0.0087890625	0.00156606099642431\\
0.009765625	0.0015646651482939\\
0.0107421875	0.00156327433285242\\
0.01171875	0.00156189058088785\\
0.0126953125	0.00156051449982897\\
0.013671875	0.00155914344486519\\
0.0146484375	0.00155778060729972\\
0.015625	0.00155642331674244\\
0.0166015625	0.00155507308124925\\
0.017578125	0.00155373006384707\\
0.0185546875	0.00155239204923419\\
0.01953125	0.00155106273405181\\
0.0205078125	0.00154973896746924\\
0.021484375	0.00154842231154362\\
0.0224609375	0.00154711077709635\\
0.0234375	0.00154580727678422\\
0.0244140625	0.00154451110188347\\
0.025390625	0.00154322133369078\\
0.0263671875	0.00154193672744896\\
0.02734375	0.0015406602742587\\
0.0283203125	0.00153938999756065\\
0.029296875	0.00153812507301154\\
0.0302734375	0.00153686952614862\\
0.03125	0.00153561837794314\\
0.0322265625	0.0015343742445566\\
0.033203125	0.00153313678890754\\
0.0341796875	0.00153190716912377\\
0.03515625	0.00153068265331058\\
0.0361328125	0.0015294649755333\\
0.037109375	0.00152825458519601\\
0.0380859375	0.00152705073321613\\
0.0390625	0.00152585409318817\\
0.0400390625	0.00152466172232835\\
0.041015625	0.00152347781897788\\
0.0419921875	0.00152229925595293\\
0.04296875	0.00152112837338336\\
0.0439453125	0.0015199639618686\\
0.044921875	0.00151880527050707\\
0.0458984375	0.00151765492728373\\
0.046875	0.00151650933662495\\
0.0478515625	0.00151537065755747\\
0.048828125	0.00151423854595123\\
0.0498046875	0.00151311392744446\\
0.05078125	0.00151199556137271\\
0.0517578125	0.00151088157599588\\
0.052734375	0.00150977593955304\\
0.0537109375	0.00150867629383811\\
0.0546875	0.00150758443635368\\
0.0556640625	0.00150649756096755\\
0.056640625	0.00150541661571424\\
0.0576171875	0.00150434394743115\\
0.05859375	0.00150327642393222\\
0.0595703125	0.00150221606384093\\
0.060546875	0.0015011609236808\\
0.0615234375	0.00150011272864958\\
0.0625	0.00149907109482683\\
0.0634765625	0.00149803586816688\\
0.064453125	0.00149700727740765\\
0.0654296875	0.00149598534642337\\
0.06640625	0.0014949685295278\\
0.0673828125	0.00149395809512498\\
0.068359375	0.00149295557866935\\
0.0693359375	0.00149195859023621\\
0.0703125	0.00149096707298213\\
0.0712890625	0.00148998204656436\\
0.072265625	0.00148900389433493\\
0.0732421875	0.00148803206297998\\
0.07421875	0.00148706546178801\\
0.0751953125	0.00148610587882558\\
0.076171875	0.00148515223327195\\
0.0771484375	0.00148420541347605\\
0.078125	0.0014832645780416\\
0.0791015625	0.00148232847436702\\
0.080078125	0.00148140036424138\\
0.0810546875	0.00148047748984936\\
0.08203125	0.00147956121554671\\
0.0830078125	0.00147865072653985\\
0.083984375	0.00147774662900702\\
0.0849609375	0.00147684670162107\\
0.0859375	0.00147595556734359\\
0.0869140625	0.00147506887151394\\
0.087890625	0.00147418859796744\\
0.0888671875	0.00147331432481224\\
0.08984375	0.0014724461265132\\
0.0908203125	0.0014715831371177\\
0.091796875	0.00147072698268857\\
0.0927734375	0.00146987627942963\\
0.09375	0.00146903246866259\\
0.0947265625	0.00146819342808158\\
0.095703125	0.00146736134672665\\
0.0966796875	0.00146653384911133\\
0.09765625	0.00146571198422407\\
0.0986328125	0.0014648972176019\\
0.099609375	0.0014640876232761\\
0.1005859375	0.00146328350649583\\
0.1015625	0.00146248504370305\\
0.1025390625	0.00146169230686155\\
0.103515625	0.00146090576583902\\
0.1044921875	0.0014601250928763\\
0.10546875	0.0014593496616726\\
0.1064453125	0.00145857960751528\\
0.107421875	0.00145781563298897\\
0.1083984375	0.00145705692841602\\
0.109375	0.00145630272834296\\
0.1103515625	0.00145555639664963\\
0.111328125	0.00145481426795868\\
0.1123046875	0.00145407752449955\\
0.11328125	0.00145334698129318\\
0.1142578125	0.00145262125590762\\
0.115234375	0.00145190140676732\\
0.1162109375	0.0014511858488504\\
0.1171875	0.00145047633327522\\
0.1181640625	0.0014497733063763\\
0.119140625	0.00144907425988094\\
0.1201171875	0.00144838081041598\\
0.12109375	0.00144769201051531\\
0.1220703125	0.00144700952387211\\
0.123046875	0.00144633167974462\\
0.1240234375	0.00144565936147956\\
0.125	0.00144499160808209\\
0.1259765625	0.00144432949718976\\
0.126953125	0.00144367243308352\\
0.1279296875	0.00144301976297356\\
0.12890625	0.00144237228892052\\
0.1298828125	0.00144173052672159\\
0.130859375	0.00144109384132207\\
0.1318359375	0.00144046151581279\\
0.1328125	0.00143983472537457\\
0.1337890625	0.0014392126703342\\
0.134765625	0.0014385946233233\\
0.1357421875	0.00143798244835125\\
0.13671875	0.00143737399787369\\
0.1376953125	0.00143677111418583\\
0.138671875	0.00143617244475536\\
0.1396484375	0.00143557942726602\\
0.140625	0.00143499051694107\\
0.1416015625	0.00143440634030867\\
0.142578125	0.0014338259154556\\
0.1435546875	0.00143325232340885\\
0.14453125	0.00143268066892688\\
0.1455078125	0.00143211488204997\\
0.146484375	0.00143155340265366\\
0.1474609375	0.00143099686567894\\
0.1484375	0.00143044424476102\\
0.1494140625	0.00142989701942042\\
0.150390625	0.00142935339545147\\
0.1513671875	0.00142881275860418\\
0.15234375	0.001428277574405\\
0.1533203125	0.00142774590551653\\
0.154296875	0.00142721946417623\\
0.1552734375	0.00142669678734819\\
0.15625	0.00142617888366203\\
0.1572265625	0.00142566359693319\\
0.158203125	0.00142515437357815\\
0.1591796875	0.00142464706050305\\
0.16015625	0.00142414539197944\\
0.1611328125	0.00142364719351917\\
0.162109375	0.00142315242203495\\
0.1630859375	0.0014226623943614\\
0.1640625	0.00142217508323483\\
0.1650390625	0.00142169218247545\\
0.166015625	0.0014212135880598\\
0.1669921875	0.00142073730842185\\
0.16796875	0.00142026590856403\\
0.1689453125	0.00141979846580398\\
0.169921875	0.00141933336522015\\
0.1708984375	0.00141887393499474\\
0.171875	0.00141841557501721\\
0.1728515625	0.00141796157129193\\
0.173828125	0.00141751159856085\\
0.1748046875	0.00141706392025753\\
0.17578125	0.00141662033058765\\
0.1767578125	0.001416179425064\\
0.177734375	0.00141574367967223\\
0.1787109375	0.00141530884366148\\
0.1796875	0.00141487910843807\\
0.1806640625	0.00141445037104404\\
0.181640625	0.00141402739029672\\
0.1826171875	0.00141360539544166\\
0.18359375	0.00141318643136401\\
0.1845703125	0.0014127713156995\\
0.185546875	0.00141235911371496\\
0.1865234375	0.00141194984269077\\
0.1875	0.00141154282232492\\
0.1884765625	0.0014111399190142\\
0.189453125	0.00141073748432063\\
0.1904296875	0.00141034001046592\\
0.19140625	0.0014099442778388\\
0.1923828125	0.0014095512831318\\
0.193359375	0.00140915952465548\\
0.1943359375	0.00140877208195889\\
0.1953125	0.0014083866474266\\
0.1962890625	0.0014080030380228\\
0.197265625	0.00140762271371386\\
0.1982421875	0.00140724344475984\\
0.19921875	0.00140686829013248\\
0.2001953125	0.00140649312459118\\
0.201171875	0.00140612168138432\\
0.2021484375	0.00140575246155095\\
0.203125	0.00140538430798642\\
0.2041015625	0.00140501937255522\\
0.205078125	0.00140465590538952\\
0.2060546875	0.00140429442046752\\
0.20703125	0.00140393368167224\\
0.2080078125	0.00140357631778443\\
0.208984375	0.00140322012134675\\
0.2099609375	0.00140286528596789\\
0.2109375	0.00140251322682161\\
0.2119140625	0.00140216191948639\\
0.212890625	0.00140181225890501\\
0.2138671875	0.00140146446381095\\
0.21484375	0.00140111825771783\\
0.2158203125	0.00140077337152889\\
0.216796875	0.00140042887767322\\
0.2177734375	0.00140008633627531\\
0.21875	0.00139974606690885\\
0.2197265625	0.00139940552048756\\
0.220703125	0.00139906767799403\\
0.2216796875	0.00139872918214223\\
0.22265625	0.00139839297753497\\
0.2236328125	0.0013980563126097\\
0.224609375	0.00139772115551295\\
0.2255859375	0.0013973879816831\\
0.2265625	0.00139705467051954\\
0.2275390625	0.00139672309694561\\
0.228515625	0.00139639165229255\\
0.2294921875	0.00139606007871862\\
0.23046875	0.0013957298006062\\
0.2314453125	0.00139539922656695\\
0.232421875	0.00139507020276142\\
0.2333984375	0.00139474180468824\\
0.234375	0.00139441339103996\\
0.2353515625	0.00139408498273497\\
0.236328125	0.00139375593641944\\
0.2373046875	0.0013934287740085\\
0.23828125	0.00139310205156562\\
0.2392578125	0.00139277396669968\\
0.240234375	0.00139244638000946\\
0.2412109375	0.00139211842952136\\
0.2421875	0.00139179034624703\\
0.2431640625	0.00139146277274449\\
0.244140625	0.00139113474938313\\
0.2451171875	0.00139080575308981\\
0.24609375	0.00139047752918486\\
0.2470703125	0.00139014845603924\\
0.248046875	0.00138981735142352\\
0.2490234375	0.00138948823735063\\
0.25	0.00138915756792812\\
0.2509765625	0.00138882585997635\\
0.251953125	0.00138849374900474\\
0.2529296875	0.0013881621737255\\
0.25390625	0.00138782718966013\\
0.2548828125	0.00138749256700521\\
0.255859375	0.00138715594755467\\
0.2568359375	0.00138682083411368\\
0.2578125	0.00138648223628479\\
0.2587890625	0.00138614412185234\\
0.259765625	0.00138580279087819\\
0.2607421875	0.00138546128152939\\
0.26171875	0.00138511817431208\\
0.2626953125	0.00138477504208367\\
0.263671875	0.00138442823106288\\
0.2646484375	0.00138408265650014\\
0.265625	0.00138373296704231\\
0.2666015625	0.00138338225781354\\
0.267578125	0.00138303004632689\\
0.2685546875	0.00138267740396714\\
0.26953125	0.0013823206362531\\
0.2705078125	0.00138196284513015\\
0.271484375	0.00138160366736884\\
0.2724609375	0.00138124056559263\\
0.2734375	0.00138087801644815\\
0.2744140625	0.00138051072212875\\
0.275390625	0.00138014332435432\\
0.2763671875	0.00137977248346033\\
0.27734375	0.00137939961132361\\
0.2783203125	0.00137902327696793\\
0.279296875	0.00137864558553247\\
0.2802734375	0.00137826577406486\\
0.28125	0.0013778816745571\\
0.2822265625	0.00137749500345308\\
0.283203125	0.001377106806558\\
0.2841796875	0.00137671548486651\\
0.28515625	0.00137632111386665\\
0.2861328125	0.00137592284886523\\
0.287109375	0.00137552218382098\\
0.2880859375	0.00137511829041159\\
0.2890625	0.00137471158086555\\
0.2900390625	0.0013743010657663\\
0.291015625	0.00137388726568588\\
0.2919921875	0.0013734703454702\\
0.29296875	0.00137304965107887\\
0.2939453125	0.00137262559348983\\
0.294921875	0.00137219872044625\\
0.2958984375	0.00137176677321804\\
0.296875	0.001371332061467\\
0.2978515625	0.00137089260397261\\
0.298828125	0.00137045018277604\\
0.2998046875	0.0013700030704058\\
0.30078125	0.00136955210234646\\
0.3017578125	0.00136909812238173\\
0.302734375	0.00136863919738062\\
0.3037109375	0.00136817507438991\\
0.3046875	0.00136770861490731\\
0.3056640625	0.00136723606522082\\
0.306640625	0.0013667604314378\\
0.3076171875	0.00136627975336978\\
0.30859375	0.00136579344268739\\
0.3095703125	0.00136530296128967\\
0.310546875	0.00136480895321256\\
0.3115234375	0.00136431062878728\\
0.3125	0.001363805223491\\
0.3134765625	0.00136329587076034\\
0.314453125	0.00136278240222509\\
0.3154296875	0.00136226211645862\\
0.31640625	0.0013617383511928\\
0.3173828125	0.00136120933143502\\
0.318359375	0.00136067332766743\\
0.3193359375	0.00136013366841325\\
0.3203125	0.00135958752866827\\
0.3212890625	0.00135903654052072\\
0.322265625	0.00135848052070742\\
0.3232421875	0.00135791775574035\\
0.32421875	0.00135734931643583\\
0.3251953125	0.00135677649836907\\
0.326171875	0.00135619689285704\\
0.3271484375	0.00135561109004811\\
0.328125	0.00135501946317618\\
0.3291015625	0.00135442161797528\\
0.330078125	0.00135381937082002\\
0.3310546875	0.00135320982190024\\
0.33203125	0.00135259334774673\\
0.3330078125	0.00135197094891737\\
0.333984375	0.00135134279207705\\
0.3349609375	0.00135070867327158\\
0.3359375	0.00135006631671786\\
0.3369140625	0.00134941931537469\\
0.337890625	0.00134876345509838\\
0.3388671875	0.00134810332451707\\
0.33984375	0.00134743510056978\\
0.3408203125	0.00134676020002189\\
0.341796875	0.00134607946620235\\
0.3427734375	0.00134538913425786\\
0.34375	0.00134469316992636\\
0.3447265625	0.00134399054127243\\
0.345703125	0.00134328114393156\\
0.3466796875	0.0013425639056095\\
0.34765625	0.00134183853970171\\
0.3486328125	0.00134110688850342\\
0.349609375	0.00134036665053827\\
0.3505859375	0.00133961996334619\\
0.3515625	0.00133886545359019\\
0.3525390625	0.00133810308079774\\
0.353515625	0.00133733442089579\\
0.3544921875	0.00133655595993787\\
0.35546875	0.00133577228439208\\
0.3564453125	0.00133497838089625\\
0.357421875	0.00133417678239311\\
0.3583984375	0.0013333678463141\\
0.359375	0.00133255095181539\\
0.3603515625	0.00133172527625902\\
0.361328125	0.00133089151381682\\
0.3623046875	0.00133004997519492\\
0.36328125	0.00132920016369553\\
0.3642578125	0.00132834138332782\\
0.365234375	0.00132747452266813\\
0.3662109375	0.00132659892324227\\
0.3671875	0.00132571501239909\\
0.3681640625	0.00132482172875825\\
0.369140625	0.00132392144848836\\
0.3701171875	0.00132301181668026\\
0.37109375	0.00132209203991351\\
0.3720703125	0.0013211647898288\\
0.373046875	0.0013202289627543\\
0.3740234375	0.00131928269433956\\
0.375	0.00131832918702912\\
0.3759765625	0.00131736547177752\\
0.376953125	0.00131639184337473\\
0.3779296875	0.00131541003224811\\
0.37890625	0.0013144195689847\\
0.3798828125	0.00131341816347685\\
0.380859375	0.00131240912219255\\
0.3818359375	0.00131139020891169\\
0.3828125	0.00131036086190761\\
0.3837890625	0.00130932295371622\\
0.384765625	0.00130827547229728\\
0.3857421875	0.00130721660525523\\
0.38671875	0.00130615113266686\\
0.3876953125	0.00130507379037681\\
0.388671875	0.00130398659155162\\
0.3896484375	0.00130288958825986\\
0.390625	0.00130178414201509\\
0.3916015625	0.00130066760857517\\
0.392578125	0.00129954126805387\\
0.3935546875	0.00129840529530156\\
0.39453125	0.00129725844578843\\
0.3955078125	0.00129610218164089\\
0.396484375	0.00129493525389535\\
0.3974609375	0.00129375953565614\\
0.3984375	0.00129257231969859\\
0.3994140625	0.00129137348847053\\
0.400390625	0.00129016676203264\\
0.4013671875	0.00128894850342931\\
0.40234375	0.00128771858089749\\
0.4033203125	0.00128648063889614\\
0.404296875	0.00128523028138261\\
0.4052734375	0.00128397025991944\\
0.40625	0.00128269933702541\\
0.4072265625	0.00128141780760416\\
0.408203125	0.00128012556137946\\
0.4091796875	0.00127882353262976\\
0.41015625	0.00127751030424861\\
0.4111328125	0.00127618521582917\\
0.412109375	0.00127484968629688\\
0.4130859375	0.0012735035983269\\
0.4140625	0.00127214660858499\\
0.4150390625	0.00127077785316487\\
0.416015625	0.00126939979043073\\
0.4169921875	0.00126800934299354\\
0.41796875	0.00126660840567183\\
0.4189453125	0.0012651957692924\\
0.419921875	0.00126377260619392\\
0.4208984375	0.0012623379783463\\
0.421875	0.00126089163052256\\
0.4228515625	0.00125943497801018\\
0.423828125	0.00125796816178081\\
0.4248046875	0.00125648822734092\\
0.42578125	0.00125499859018419\\
0.4267578125	0.00125349527593244\\
0.427734375	0.00125198272269245\\
0.4287109375	0.00125045781317112\\
0.4296875	0.0012489207474573\\
0.4306640625	0.00124737396595265\\
0.431640625	0.00124581601266982\\
0.4326171875	0.00124424454031669\\
0.43359375	0.00124266266959694\\
0.4345703125	0.00124106983503225\\
0.435546875	0.0012394640700677\\
0.4365234375	0.00123784893867196\\
0.4375	0.00123621983493649\\
0.4384765625	0.00123458115763242\\
0.439453125	0.0012329286031445\\
0.4404296875	0.0012312666553953\\
0.44140625	0.00122959187876859\\
0.4423828125	0.00122790562841146\\
0.443359375	0.00122620711431409\\
0.4443359375	0.00122449905120448\\
0.4453125	0.00122277698642392\\
0.4462890625	0.00122104518277411\\
0.447265625	0.00121930022191918\\
0.4482421875	0.00121754373390104\\
0.44921875	0.00121577653976601\\
0.4501953125	0.00121399671616018\\
0.451171875	0.00121220463859117\\
0.4521484375	0.00121040153646845\\
0.453125	0.00120858675154523\\
0.4541015625	0.00120676085521154\\
0.455078125	0.00120492039025066\\
0.4560546875	0.00120307241616047\\
0.45703125	0.00120120994131412\\
0.4580078125	0.0011993360036513\\
0.458984375	0.00119745055212661\\
0.4599609375	0.00119555352148382\\
0.4609375	0.0011936449394625\\
0.4619140625	0.00119172411984891\\
0.462890625	0.00118979018270693\\
0.4638671875	0.00118784680955741\\
0.46484375	0.00118589104215516\\
0.4658203125	0.00118392371905429\\
0.466796875	0.00118194421270346\\
0.4677734375	0.00117995243090263\\
0.46875	0.00117795014000421\\
0.4697265625	0.00117593605455113\\
0.470703125	0.00117390833304398\\
0.4716796875	0.00117187155251486\\
0.47265625	0.00116982193685544\\
0.4736328125	0.00116775987260098\\
0.474609375	0.00116568678652129\\
0.4755859375	0.00116360364677348\\
0.4765625	0.00116150680071314\\
0.4775390625	0.00115939920749497\\
0.478515625	0.00115728021171435\\
0.4794921875	0.00115514897299818\\
0.48046875	0.00115300684819886\\
0.4814453125	0.0011508534865925\\
0.482421875	0.0011486879052427\\
0.4833984375	0.00114651095952922\\
0.484375	0.00114432326267888\\
0.4853515625	0.00114212319135731\\
0.486328125	0.00113991126158908\\
0.4873046875	0.00113768879305098\\
0.48828125	0.00113545541944404\\
0.4892578125	0.00113321032995373\\
0.490234375	0.00113095400297425\\
0.4912109375	0.00112868603400784\\
0.4921875	0.0011264075295685\\
0.4931640625	0.00112411702423287\\
0.494140625	0.00112181644772136\\
0.4951171875	0.0011195040151506\\
0.49609375	0.00111718026028029\\
0.4970703125	0.00111484546687279\\
0.498046875	0.00111250076349734\\
0.4990234375	0.00111014279195842\\
0.5	0.00110777638712989\\
0.5009765625	0.00110539702723145\\
0.501953125	0.00110300820028897\\
0.5029296875	0.00110060790643729\\
0.50390625	0.00109819788281129\\
0.5048828125	0.00109577462171728\\
0.505859375	0.00109334456578836\\
0.5068359375	0.00109089941042839\\
0.5078125	0.00108844688952558\\
0.5087890625	0.00108598216593236\\
0.509765625	0.00108350761047404\\
0.5107421875	0.00108102328363202\\
0.51171875	0.00107852746714343\\
0.5126953125	0.00107602069169843\\
0.513671875	0.00107350453367872\\
0.5146484375	0.00107097623026675\\
0.515625	0.00106844072661261\\
0.5166015625	0.00106589417077885\\
0.517578125	0.00106333624012223\\
0.5185546875	0.00106076903352914\\
0.51953125	0.00105819092345882\\
0.5205078125	0.00105560411100214\\
0.521484375	0.00105300785480722\\
0.5224609375	0.0010503994994906\\
0.5234375	0.00104778331058242\\
0.5244140625	0.00104515669477223\\
0.525390625	0.00104251991865567\\
0.5263671875	0.00103987456463983\\
0.52734375	0.0010372187084613\\
0.5283203125	0.00103455464000035\\
0.529296875	0.00103187898901069\\
0.5302734375	0.00102919598259632\\
0.53125	0.00102650217877454\\
0.5322265625	0.00102379911072603\\
0.533203125	0.00102108839939774\\
0.5341796875	0.00101836805140465\\
0.53515625	0.00101563811267624\\
0.5361328125	0.00101289889983036\\
0.537109375	0.00101015183156505\\
0.5380859375	0.00100739528875238\\
0.5390625	0.00100463002127071\\
0.5400390625	0.00100185581879941\\
0.541015625	0.000999073147681884\\
0.5419921875	0.000996281483821804\\
0.54296875	0.000993482325839068\\
0.5439453125	0.000990674014701654\\
0.544921875	0.000987856509254925\\
0.5458984375	0.000985032458743262\\
0.546875	0.000982199756094815\\
0.5478515625	0.0009793584305271\\
0.548828125	0.000976508217377159\\
0.5498046875	0.000973651361618977\\
0.55078125	0.000970786028915427\\
0.5517578125	0.000967913835779655\\
0.552734375	0.000965033077818589\\
0.5537109375	0.000962144363938933\\
0.5546875	0.000959247981427325\\
0.5556640625	0.000956344091378014\\
0.556640625	0.000953434248572194\\
0.5576171875	0.000950514978399042\\
0.55859375	0.000947589140082528\\
0.5595703125	0.000944657003856264\\
0.560546875	0.00094171683872446\\
0.5615234375	0.000938769546792173\\
0.5625	0.000935815130105766\\
0.5634765625	0.000932854635266267\\
0.564453125	0.000929887297729692\\
0.5654296875	0.000926912105455813\\
0.56640625	0.000923931665624877\\
0.5673828125	0.000920942835705318\\
0.568359375	0.000917950368489073\\
0.5693359375	0.000914948887839273\\
0.5703125	0.000911942558673218\\
0.5712890625	0.000908929286651983\\
0.572265625	0.000905910016399503\\
0.5732421875	0.00090288537876404\\
0.57421875	0.000899854422641511\\
0.5751953125	0.000896817775128511\\
0.576171875	0.000893774460337227\\
0.5771484375	0.00089072633249998\\
0.578125	0.000887671594000494\\
0.5791015625	0.000884612436948373\\
0.580078125	0.000881547778135428\\
0.5810546875	0.000878477362334706\\
0.58203125	0.000875401214784688\\
0.5830078125	0.000872321701763212\\
0.583984375	0.00086923560024843\\
0.5849609375	0.000866144685687686\\
0.5859375	0.000863048897599583\\
0.5869140625	0.000859949134564886\\
0.587890625	0.000856843054634737\\
0.5888671875	0.000853733486678721\\
0.58984375	0.000850619312359413\\
0.5908203125	0.000847501451630706\\
0.591796875	0.000844377457042356\\
0.5927734375	0.000841250179405506\\
0.59375	0.000838119038576224\\
0.5947265625	0.000834982877790935\\
0.595703125	0.000831844059348441\\
0.5966796875	0.000828700339638999\\
0.59765625	0.000825553866434348\\
0.5986328125	0.000822402361336572\\
0.599609375	0.000819247476101737\\
0.6005859375	0.000816089968111555\\
0.6015625	0.000812927896049587\\
0.6025390625	0.000809763723964352\\
0.603515625	0.000806595049425596\\
0.6044921875	0.000803424125820129\\
0.60546875	0.000800249983285539\\
0.6064453125	0.00079707302791121\\
0.607421875	0.000793892915112338\\
0.6083984375	0.000790709508351028\\
0.609375	0.000787524695624597\\
0.6103515625	0.000784336313699896\\
0.611328125	0.000781145363930591\\
0.6123046875	0.000777951806071542\\
0.61328125	0.000774757939552728\\
0.6142578125	0.000771559181657722\\
0.615234375	0.000768360611800745\\
0.6162109375	0.000765157776868364\\
0.6171875	0.000761953967185036\\
0.6181640625	0.00075874863205172\\
0.619140625	0.00075554200066108\\
0.6201171875	0.000752331309513465\\
0.62109375	0.000749121300714251\\
0.6220703125	0.000745909491001839\\
0.623046875	0.000742695784651914\\
0.6240234375	0.000739481365485517\\
0.625	0.000736264608349302\\
0.6259765625	0.000733047220592198\\
0.626953125	0.000729829160718509\\
0.6279296875	0.000726610110746151\\
0.62890625	0.000723389862514523\\
0.6298828125	0.000720169576538865\\
0.630859375	0.00071694754467444\\
0.6318359375	0.000713726043386487\\
0.6328125	0.000710503504137705\\
0.6337890625	0.000707280787651143\\
0.634765625	0.00070405702683729\\
0.6357421875	0.000700833261021216\\
0.63671875	0.000697610701763551\\
0.6376953125	0.000694386187205964\\
0.638671875	0.0006911625033581\\
0.6396484375	0.00068793912066667\\
0.640625	0.000684716358364312\\
0.6416015625	0.000681492548892493\\
0.642578125	0.000678270312846507\\
0.6435546875	0.000675047524964612\\
0.64453125	0.000671826455118207\\
0.6455078125	0.000668605115720311\\
0.646484375	0.000665385019942732\\
0.6474609375	0.000662165804214965\\
0.6484375	0.000658946792555071\\
0.6494140625	0.000655729760637769\\
0.650390625	0.000652513062050275\\
0.6513671875	0.000649297625500367\\
0.65234375	0.000646083611513859\\
0.6533203125	0.00064287104783034\\
0.654296875	0.000639659892158306\\
0.6552734375	0.00063644982265032\\
0.65625	0.000633241149103014\\
0.6572265625	0.000630034650612288\\
0.658203125	0.000626830225201047\\
0.6591796875	0.000623626090941798\\
0.66015625	0.000620425929128032\\
0.6611328125	0.000617226614053834\\
0.662109375	0.000614028648328713\\
0.6630859375	0.000610834305916796\\
0.6640625	0.000607641663805225\\
0.6650390625	0.000604451175490794\\
0.666015625	0.000601262599275287\\
0.6669921875	0.00059807651950905\\
0.66796875	0.000594893858988144\\
0.6689453125	0.000591714320421488\\
0.669921875	0.000588535495239739\\
0.6708984375	0.000585361089406433\\
0.671875	0.000582187171630721\\
0.6728515625	0.000579019457063623\\
0.673828125	0.000575852483734707\\
0.6748046875	0.000572688729221227\\
0.67578125	0.000569528766732219\\
0.6767578125	0.000566372018283801\\
0.677734375	0.000563218712841262\\
0.6787109375	0.000560067458422964\\
0.6796875	0.000556920641088254\\
0.6806640625	0.000553777406366862\\
0.681640625	0.000550638034724216\\
0.6826171875	0.000547501234336778\\
0.68359375	0.000544368149576258\\
0.6845703125	0.000541240192433179\\
0.685546875	0.000538114607024909\\
0.6865234375	0.000534993925953131\\
0.6875	0.000531877128764791\\
0.6884765625	0.000528763460465598\\
0.689453125	0.00052565494195278\\
0.6904296875	0.000522550489677087\\
0.69140625	0.000519450194360616\\
0.6923828125	0.000516353384909962\\
0.693359375	0.000513262333356579\\
0.6943359375	0.000510175428985349\\
0.6953125	0.000507092057205227\\
0.6962890625	0.000504013939007564\\
0.697265625	0.00050094021150926\\
0.6982421875	0.000497870863569005\\
0.69921875	0.00049480757593301\\
0.7001953125	0.000491747656724328\\
0.701171875	0.000488694000750911\\
0.7021484375	0.000485643743559194\\
0.703125	0.000482599253246008\\
0.7041015625	0.00047956005460037\\
0.705078125	0.00047652473836024\\
0.7060546875	0.000473496678068841\\
0.70703125	0.00047047233806552\\
0.7080078125	0.000467454028239445\\
0.708984375	0.000464440789528453\\
0.7099609375	0.000461432745510137\\
0.7109375	0.000458430235312335\\
0.7119140625	0.000455433308275133\\
0.712890625	0.000452442340701964\\
0.7138671875	0.000449456023375205\\
0.71484375	0.00044647596462255\\
0.7158203125	0.0004435020590563\\
0.716796875	0.000440533037249224\\
0.7177734375	0.000437571180214036\\
0.71875	0.000434613482525492\\
0.7197265625	0.000431662512482944\\
0.720703125	0.000428717100476206\\
0.7216796875	0.000425778684757461\\
0.72265625	0.000422845797515947\\
0.7236328125	0.000419919412593117\\
0.724609375	0.000416998443370176\\
0.7255859375	0.000414083492273676\\
0.7265625	0.000411176213447106\\
0.7275390625	0.000408273461857789\\
0.728515625	0.000405378514074073\\
0.7294921875	0.000402489750740642\\
0.73046875	0.000399607267922875\\
0.7314453125	0.000396730850752647\\
0.732421875	0.000393860928170398\\
0.7333984375	0.000390998793818653\\
0.734375	0.000388141902476491\\
0.7353515625	0.000385292471833054\\
0.736328125	0.000382448992809259\\
0.7373046875	0.000379612513370375\\
0.73828125	0.000376783677438652\\
0.7392578125	0.000373961004243029\\
0.740234375	0.000371144782093324\\
0.7412109375	0.00036833565127381\\
0.7421875	0.000365533879289615\\
0.7431640625	0.000362738280614394\\
0.744140625	0.000359950556003241\\
0.7451171875	0.000357168831442323\\
0.74609375	0.000354395046542777\\
0.7470703125	0.000351628001794779\\
0.748046875	0.000348868626929288\\
0.7490234375	0.000346116356695347\\
0.75	0.000343371789199409\\
0.7509765625	0.000340632901156823\\
0.751953125	0.000337901854663869\\
0.7529296875	0.000335178493628518\\
0.75390625	0.000332462144911005\\
0.7548828125	0.000329754010408578\\
0.755859375	0.00032705229534713\\
0.7568359375	0.000324358912507705\\
0.7578125	0.000321672488098557\\
0.7587890625	0.000318992633992821\\
0.759765625	0.000316321359491667\\
0.7607421875	0.000313656789899142\\
0.76171875	0.000311001250565823\\
0.7626953125	0.000308352725141958\\
0.763671875	0.000305710954080496\\
0.7646484375	0.000303078392448697\\
0.765625	0.000300452177611987\\
0.7666015625	0.000297833015679316\\
0.767578125	0.000295223585339954\\
0.7685546875	0.000292621706535101\\
0.76953125	0.00029002579265125\\
0.7705078125	0.000287438654822836\\
0.771484375	0.00028486021972185\\
0.7724609375	0.000282288836729094\\
0.7734375	0.000279725003338172\\
0.7744140625	0.00027716915315068\\
0.775390625	0.000274621732160085\\
0.7763671875	0.000272081291200266\\
0.77734375	0.000269549450308659\\
0.7783203125	0.000267025765651852\\
0.779296875	0.000264509429371174\\
0.7802734375	0.000262001824921754\\
0.78125	0.00025950185909096\\
0.7822265625	0.000257009376582573\\
0.783203125	0.000254525765285507\\
0.7841796875	0.000252050288509054\\
0.78515625	0.000249582404535431\\
0.7861328125	0.000247123047188325\\
0.787109375	0.000244670735128238\\
0.7880859375	0.000242227390572225\\
0.7890625	0.000239792685192697\\
0.7900390625	0.000237365683460666\\
0.791015625	0.000234946476780351\\
0.7919921875	0.000232535463055683\\
0.79296875	0.000230132322144527\\
0.7939453125	0.000227738512762699\\
0.794921875	0.000225352864845263\\
0.7958984375	0.000222974280859489\\
0.796875	0.000220604645164713\\
0.7978515625	0.000218243495851311\\
0.798828125	0.000215890151253006\\
0.7998046875	0.000213545566111861\\
0.80078125	0.000211208442465249\\
0.8017578125	0.000208879529168371\\
0.802734375	0.000206559680236751\\
0.8037109375	0.000204248070531321\\
0.8046875	0.00020194379737859\\
0.8056640625	0.000199648813349995\\
0.806640625	0.000197361962477771\\
0.8076171875	0.000195083968037579\\
0.80859375	0.000192812986369972\\
0.8095703125	0.000190551439573028\\
0.810546875	0.000188297498198153\\
0.8115234375	0.000186052572530571\\
0.8125	0.000183815376772145\\
0.8134765625	0.000181587363385916\\
0.814453125	0.000179367079681469\\
0.8154296875	0.000177154574998895\\
0.81640625	0.000174951340454754\\
0.8173828125	0.000172756856954948\\
0.818359375	0.000170569398960652\\
0.8193359375	0.000168391415627411\\
0.8203125	0.000166221251220122\\
0.8212890625	0.000164059416647433\\
0.822265625	0.000161906674748025\\
0.8232421875	0.000159761669124237\\
0.82421875	0.000157624900452902\\
0.8251953125	0.000155496883280648\\
0.826171875	0.000153377846572766\\
0.8271484375	0.000151265252611665\\
0.828125	0.000149162591242202\\
0.8291015625	0.00014706820900301\\
0.830078125	0.000144981695370916\\
0.8310546875	0.000142903610935718\\
0.83203125	0.000140834056765016\\
0.8330078125	0.000138772878244708\\
0.833984375	0.000136720806040103\\
0.8349609375	0.000134675285721642\\
0.8359375	0.000132639481762453\\
0.8369140625	0.000130611397139546\\
0.837890625	0.000128592230566937\\
0.8388671875	0.000126580699429724\\
0.83984375	0.000124577883070742\\
0.8408203125	0.000122583483630478\\
0.841796875	0.000120597858767724\\
0.8427734375	0.000118619804652553\\
0.84375	0.000116649879828401\\
0.8447265625	0.000114688986741385\\
0.845703125	0.000112735827201504\\
0.8466796875	0.00011079079740739\\
0.84765625	0.000108855210783076\\
0.8486328125	0.00010692712078253\\
0.849609375	0.000105007546494562\\
0.8505859375	0.000103095512258733\\
0.8515625	0.00010119273838427\\
0.8525390625	9.92969065691796e-05\\
0.853515625	9.74105895465982e-05\\
0.8544921875	9.55322326490204e-05\\
0.85546875	9.36614825377546e-05\\
0.8564453125	9.17988055562091e-05\\
0.857421875	8.99449157714116e-05\\
0.8583984375	8.80996174146276e-05\\
0.859375	8.6262275772242e-05\\
0.8603515625	8.44323380988499e-05\\
0.861328125	8.26109740046377e-05\\
0.8623046875	8.07985288702184e-05\\
0.86328125	7.89929702023073e-05\\
0.8642578125	7.71965297872157e-05\\
0.865234375	7.54078747604581e-05\\
0.8662109375	7.36269062144856e-05\\
0.8671875	7.18549405291924e-05\\
0.8681640625	7.00898389141003e-05\\
0.869140625	6.83339821989648e-05\\
0.8701171875	6.65857064632291e-05\\
0.87109375	6.48447422690879e-05\\
0.8720703125	6.31128024224381e-05\\
0.873046875	6.13880524724664e-05\\
0.8740234375	5.9672280031009e-05\\
0.875	5.79637444388936e-05\\
0.8759765625	5.62644960382386e-05\\
0.876953125	5.45708825256952e-05\\
0.8779296875	5.28874012388769e-05\\
0.87890625	5.12105848429201e-05\\
0.8798828125	4.95425485951273e-05\\
0.880859375	4.78823807270601e-05\\
0.8818359375	4.62291793610348e-05\\
0.8828125	4.45841758391907e-05\\
0.8837890625	4.29481756327732e-05\\
0.884765625	4.13180330269824e-05\\
0.8857421875	3.96973665601763e-05\\
0.88671875	3.80831169195517e-05\\
0.8876953125	3.64780182735558e-05\\
0.888671875	3.48810318655524e-05\\
0.8896484375	3.3290516512352e-05\\
0.890625	3.17086839913827e-05\\
0.8916015625	3.0133901532281e-05\\
0.892578125	2.85670528228366e-05\\
0.8935546875	2.70084927933567e-05\\
0.89453125	2.54562672807879e-05\\
0.8955078125	2.3912172537166e-05\\
0.896484375	2.23764023985495e-05\\
0.8974609375	2.08479660841476e-05\\
0.8984375	1.93268610928499e-05\\
0.8994140625	1.78138175215281e-05\\
0.900390625	1.63077572779002e-05\\
0.9013671875	1.480951141275e-05\\
0.90234375	1.33202589722714e-05\\
0.9033203125	1.18368787980216e-05\\
0.904296875	1.03616666820017e-05\\
0.9052734375	8.89289663064119e-06\\
0.90625	7.4333482871225e-06\\
0.9072265625	5.97989264861098e-06\\
0.908203125	4.53435529834678e-06\\
0.9091796875	3.09640040541126e-06\\
0.91015625	1.66575944149372e-06\\
0.9111328125	2.42730948230019e-07\\
0.912109375	1.17342767680384e-06\\
0.9130859375	2.58191710145184e-06\\
0.9140625	3.98303916426812e-06\\
0.9150390625	5.37748076112621e-06\\
0.916015625	6.76361946716497e-06\\
0.9169921875	8.14290626749425e-06\\
0.91796875	9.51477704802528e-06\\
0.9189453125	1.08798710698466e-05\\
0.919921875	1.22370003055039e-05\\
0.9208984375	1.35884112069107e-05\\
0.921875	1.49305047898451e-05\\
0.9228515625	1.62672074566217e-05\\
0.923828125	1.7595353597244e-05\\
0.9248046875	1.89186016541498e-05\\
0.92578125	2.02329039211691e-05\\
0.9267578125	2.15402790217922e-05\\
0.927734375	2.28402660695792e-05\\
0.9287109375	2.41346623397476e-05\\
0.9296875	2.54206312320093e-05\\
0.9306640625	2.66999911673338e-05\\
0.931640625	2.79726709777606e-05\\
0.9326171875	2.92378540507343e-05\\
0.93359375	3.04969851185888e-05\\
0.9345703125	3.17486492349417e-05\\
0.935546875	3.29930022644476e-05\\
0.9365234375	3.42309252800987e-05\\
0.9375	3.54618828168896e-05\\
0.9384765625	3.66857847211577e-05\\
0.939453125	3.7903073121015e-05\\
0.9404296875	3.91144535569765e-05\\
0.94140625	4.0317842604054e-05\\
0.9423828125	4.1514706936141e-05\\
0.943359375	4.27052159466257e-05\\
0.9443359375	4.38887584550685e-05\\
0.9453125	4.50652810286556e-05\\
0.9462890625	4.62362461348675e-05\\
0.947265625	4.73991153739917e-05\\
0.9482421875	4.85565285543998e-05\\
0.94921875	4.97067836704446e-05\\
0.9501953125	5.08500887690388e-05\\
0.951171875	5.19875859481544e-05\\
0.9521484375	5.31181821088467e-05\\
0.953125	5.42426779475136e-05\\
0.9541015625	5.53598495116603e-05\\
0.955078125	5.6470963272659e-05\\
0.9560546875	5.75758051581943e-05\\
0.95703125	5.86740910648587e-05\\
0.9580078125	5.97656646732503e-05\\
0.958984375	6.08514973237106e-05\\
0.9599609375	6.19303476696587e-05\\
0.9609375	6.30025923555877e-05\\
0.9619140625	6.40698726783739e-05\\
0.962890625	6.51291704798496e-05\\
0.9638671875	6.61833565800407e-05\\
0.96484375	6.72312097549366e-05\\
0.9658203125	6.82719554561118e-05\\
0.966796875	6.93073974389335e-05\\
0.9677734375	7.03363010643443e-05\\
0.96875	7.13584068989803e-05\\
0.9697265625	7.23758259937313e-05\\
0.970703125	7.33853516976524e-05\\
0.9716796875	7.4389942255948e-05\\
0.97265625	7.53884685309458e-05\\
0.9736328125	7.63803134304908e-05\\
0.974609375	7.73670415128436e-05\\
0.9755859375	7.83462897970821e-05\\
0.9765625	7.93210837173319e-05\\
0.9775390625	8.02888635007548e-05\\
0.978515625	8.12511810863725e-05\\
0.9794921875	8.2206881984348e-05\\
0.98046875	8.31577335702605e-05\\
0.9814453125	8.41018278379124e-05\\
0.982421875	8.50407144525889e-05\\
0.9833984375	8.59726044382114e-05\\
0.984375	8.68998170062696e-05\\
0.9853515625	8.7821285433165e-05\\
0.986328125	8.87359840362478e-05\\
0.9873046875	8.96460437616042e-05\\
0.98828125	9.05495819552016e-05\\
0.9892578125	9.14482753842094e-05\\
0.990234375	9.23396461303128e-05\\
0.9912109375	9.32268671931524e-05\\
0.9921875	9.410793245479e-05\\
0.9931640625	9.49833722643234e-05\\
0.994140625	9.58532790491518e-05\\
0.9951171875	9.67173604067284e-05\\
0.99609375	9.75758065351329e-05\\
0.9970703125	9.84288822110102e-05\\
0.998046875	9.92763553995246e-05\\
};
\end{axis}
\end{tikzpicture}%

%% file: figures/ConvergenceOfParallelTransport.tex
%
%
\begin{tikzpicture}

\begin{axis}[%
width=0.8\linewidth,
height=0.394\linewidth,
at={(0\linewidth,0\linewidth)},
scale only axis,
xmode=log,
xmin=2,
xmax=4096,
xminorticks=true,
xlabel style={font=\color{white!15!black}},
xlabel={$K  = \tau^{-1}$},
ymode=log,
ymin=0.0002,
ymax=4,
yminorticks=true,
ylabel style={font=\color{white!15!black}},
ylabel={$ \|\ParTp^K_{c(0), c(\tau), \ldots, c(1)} \w(0)-\w(1)\|_{W_{\theta}^{2,2}} $},
axis background/.style={fill=white},
legend style={legend cell align=left, align=left, draw=white!15!black, anchor = south west},
legend pos= south west
]

\addplot [color=black, line width=1.0pt]
table[row sep=crcr]{%
	16	6.87296597314287\\
	32	3.76714706134867\\
	64	2.16538162167805\\
	128	1.33294446621998\\
	256	0.859182255452376\\
	512	0.570585470528765\\
	1024	0.38646597213701\\
	2048	0.265252419656208\\
	4096	0.183713654571719\\
};
\addlegendentry{$\epsilon = \sqrt{\tau}$}

\addplot [color=black, line width=1.0pt, forget plot]
table[row sep=crcr]{%
	4096	0.156156606385961\\
	512	0.441677581210366\\
	512	0.156156606385961\\
	4096	0.156156606385961\\
};

\node[right, align=left, inner sep=0]
at (axis cs:1448.155,0.105) {1};
\node[right, align=left, inner sep=0]
at (axis cs:390,0.263) {$\frac{1}{2}$};

\addplot [color=black, dotted, line width=1.0pt]
  table[row sep=crcr]{%
4	5.94797075992928\\
8	2.0371149907697\\
16	0.815543007974292\\
32	0.367437958319331\\
64	0.174739414769685\\
128	0.0852582199156647\\
256	0.0421259676955327\\
512	0.0209491569714527\\
1024	0.0104566639788469\\
2048	0.00523452261300437\\
4096	0.00263004650181772\\
};
\addlegendentry{$\epsilon = \tau$}

\addplot [color=black, dotted, line width=1.0pt, forget plot]
  table[row sep=crcr]{%
256	0.0561679569273769\\
2048	0.00702099461592211\\
2048	0.0561679569273769\\
256	0.0561679569273769\\
};

\node[right, align=left, inner sep=0]
at (axis cs:724.077,0.08) {1};
\node[right, align=left, inner sep=0]
at (axis cs:2155.789,0.02) {1};

\addplot [color=black, dashed, line width=1.0pt]
  table[row sep=crcr]{%
2	8.41190135829959\\
4	2.06397492938934\\
8	0.526723976293073\\
16	0.17387033308212\\
32	0.0660568896798822\\
64	0.0279374854639981\\
128	0.0127403596849719\\
256	0.00607702571712992\\
512	0.00296881894299119\\
1024	0.00146805635984289\\
2048	0.000731932831627133\\
4096	0.000371377233966241\\
};
\addlegendentry{$\epsilon = \tau^{3/2}$}

\addplot [color=black, dashed, line width=1.0pt, forget plot]
  table[row sep=crcr]{%
2048	0.00054894962372035\\
128	0.00878319397952559\\
128	0.00054894962372035\\
2048	0.00054894962372035\\
};

\node[right, align=left, inner sep=0]
at (axis cs:512,0.00035) {1};
\node[right, align=left, inner sep=0]
at (axis cs:100,0.002) {1};
\end{axis}
\end{tikzpicture}%

%% file: figures/ConvergenceOfOneSidedSectionalCurvature.tex
%
%
\begin{tikzpicture}

\begin{axis}[%
	width=0.37\linewidth,
	height=0.185\linewidth,
	at={(0\linewidth,0\linewidth)},
	scale only axis,
	xmode=log,
	xmin=4,
	xmax=512,
	xminorticks=true,
	xlabel style={font=\color{white!15!black}},
	xlabel={$K = \tau^{-1}$},
	ymode=log,
	ymin=2e-06,
	ymax=0.1,
	yminorticks=true,
	ylabel style={font=\color{white!15!black}},
	ylabel={$|\kappa^\tau_c(v,w) - \kappa_c(v,w)|$},
	axis background/.style={fill=white}
	]
\addplot [color=black, line width=0.5pt, forget plot]
  table[row sep=crcr]{%
4	0.0263979226998822\\
8	0.0190845586780693\\
16	0.0114566621639978\\
32	0.00630746162876716\\
64	0.00335672627056145\\
128	0.0017763483498841\\
256	0.00097198317751665\\
512	0.00104370840879467\\
};
\addplot [color=black, dashed, line width=0.5pt, forget plot]
  table[row sep=crcr]{%
4	0.0323556987389642\\
8	0.0201303543337879\\
16	0.0113742925540564\\
32	0.00607044056459811\\
64	0.00313958818098911\\
128	0.00159741557024927\\
256	0.000889789399964855\\
512	0.000754372475433182\\
};
\addplot [color=black, line width=0.5pt, forget plot]
  table[row sep=crcr]{%
128	0.00119806167768695\\
16	0.00958449342149562\\
16	0.00119806167768695\\
128	0.00119806167768695\\
};
\node[right, align=left, inner sep=0]
at (axis cs:40,0.0005) {1};
\node[right, align=left, inner sep=0]
at (axis cs:12,0.003) {1};
\end{axis}
\end{tikzpicture}%

%% file: figures/ConvergenceOfCentralSectionalCurvature.tex
%
%
\begin{tikzpicture}

\begin{axis}[%
	width=0.37\linewidth,
	height=0.185\linewidth,
	at={(0\linewidth,0\linewidth)},
	scale only axis,
	xmode=log,
	xmin=4,
	xmax=512,
	xminorticks=true,
	xlabel style={font=\color{white!15!black}},
	xlabel={$K = \tau^{-1}$},
	ymode=log,
	ymin=2e-06,
	ymax=0.1,
	yminorticks=true,
	ylabel style={font=\color{white!15!black}},
	ylabel={$|\kappa^{\pm\tau}_c(v,w) - \kappa_c(v,w)|$},
	axis background/.style={fill=white}
	]
\addplot [color=black, line width=0.5pt, forget plot]
  table[row sep=crcr]{%
4	0.0220007436188657\\
8	0.00847633471384891\\
16	0.00205527034832376\\
32	0.000409022287152491\\
64	4.57328370380844e-05\\
128	0.000161915049810946\\
256	0.000187956710003351\\
512	0.000190427082803257\\
};
\addplot [color=black, dashed, line width=0.5pt, forget plot]
  table[row sep=crcr]{%
4	0.031563213556288\\
8	0.00706008084321899\\
16	0.00171142009973804\\
32	0.000426224098526329\\
64	0.000106690346440425\\
128	2.67105792723793e-05\\
256	6.74525911190105e-06\\
512	2.73830759357774e-06\\
};
\addplot [color=black, line width=0.5pt, forget plot]
  table[row sep=crcr]{%
8	0.0100858297760271\\
64	0.000157591090250424\\
64	0.0100858297760271\\
8	0.0100858297760271\\
};

\node[right, align=left, inner sep=0]
at (axis cs:22.627,0.023) {1};
\node[right, align=left, inner sep=0]
at (axis cs:70,0.0015) {2};
\end{axis}
\end{tikzpicture}%

%% file: ExactCalculation.tex
\input{arxivheader.tex}
\newcommand{\colvec}[2]{\begin{pmatrix} #1 \\[2pt] #2 \end{pmatrix}}
\newcommand{\tildeg}[3]{\tilde{\text{g}}_{#1}\left(#2, #3\right)}
\newcommand{\parDif}[2]{\frac{\partial #1}{\partial #2}}
\newcommand{\parparDif}[3]{\frac{\partial^{2} #1}{\partial #2 \partial #3}}
\newcommand{\dd}[2]{\frac{\textnormal{d}#1}{\textnormal{d}#2}}
\newcommand{\Rie}[3]{\riemann_{#1}\left(#2, #3\right)}

\newcommand{\m}{\mathfrak{m}}
\newcommand{\M}{\mathcal{M}}
\newcommand{\TpM}[1]{\mathrm{T}_{#1}\M}
\newcommand{\secCurvature}[3]{\kappa_{#1}\left(#2 , #3\right)}
\newcommand{\eucNorm}[1]{\left| #1 \right|}
\newcommand{\dTheta}{\partial_{\theta}}

\newcommand{\FT}{\mathcal{F}}

For numerical validation of our discrete curvature approximation we here show how the exact (sectional) curvature can be computed explicitly at special Sobolev curves $c$, exemplarily for $m=2$.
Without loss of generality we may restrict to constant speed parameterized $c$, since applying the same reparameterization to all Sobolev curves leaves metric and curvature properties invariant. Moreover, we only consider the constant speed one since other cases would only lead to a further constant in the respective terms.

Curvature involves the first and second derivative of the metric.
From \cref{def:SobolevMetricCurves} and \eqref{eq:SobolevMetricTwo} it is straightforward to derive
\begin{align*}
\Dg{c}{\nu}{\xi}{\zeta}=
\int_{\Sone}
&a_0\frac{c'\cdot\nu'}{|c'|}\xi\cdot\zeta
-a_1\frac{c'\cdot\nu'}{|c'|^3}\xi'\cdot\zeta'
-3a_2\frac{c'\cdot\nu'}{|c'|^5}\xi''\cdot\zeta''\\
&+a_2\frac{c' \cdot c''}{|c'|^7}\left(2(\nu'\cdot c')'-7\tfrac{(c'\cdot c'' )(c'\cdot\nu')}{|c'|^2}\right)\xi'\cdot\zeta'\\
&-\frac{a_2}{|c'|^5}\left((c' \cdot \nu')'- 5 \tfrac{\left(c' \cdot c'' \right) \left(c' \cdot \nu' \right)}{|c'|^2}\right)( \xi' \cdot \zeta')'
\,\d\theta
\end{align*}
and hence for unit speed parameterization $|c'|=1$ (thus $c'\cdot c''=0$)
\begin{align*}
\g{c}{\xi}{\zeta}
&=\int_{\Sone}a_0\xi \cdot \zeta+a_1\xi' \cdot \zeta'+a_2\xi'' \cdot \zeta''\,\d\theta,\\
\Dg{c}{\nu}{\xi}{\zeta}
&=\int_{\Sone}c'\cdot\nu'[a_0\xi\cdot\zeta-a_1\xi'\cdot\zeta'-3a_2\xi''\cdot\zeta'']-a_2(c' \cdot \nu')'( \xi' \cdot \zeta')'\,\d\theta,\\
D_c^2g_c(\nu,\eta)(\xi,\zeta)
&=\int_{\Sone}a_0(P-Q)\xi\cdot\zeta-a_1(P-3Q)\xi'\cdot\zeta'-a_2(3P-15Q)\xi''\cdot\zeta''\\
&\qquad+2a_2(\nu'\cdot c')'(\eta'\cdot c')'\xi'\cdot\zeta'
-a_2[P'-5Q']( \xi' \cdot \zeta')'\,\d\theta
\end{align*}
with the abbreviations
\begin{equation*}
P=\nu'\cdot\eta',
\quad
Q=(c'\cdot\nu')(c'\cdot\eta').
\end{equation*}

Via Fourier series this leads to an explicit expression for the Christoffel operator.

\begin{lemma}[Christoffel operator]\label{thm:Christoffel}
Let $c\in\immersion^2(\Sone,\R^d)$ be unit speed parameterized and $\xi,\zeta\in W^{2, 2}(\Sone,\R^d)$. Then
\begin{gather*}
\Gamma_{c}(\xi,\zeta)
=\FT^{-1}\left(k\mapsto\tfrac{\FT(F(c,\xi,\zeta)/2)(k)}{a_0+a_1k^2+a_2k^4}\right)
\qquad\text{with}\\
\begin{aligned}
F(c,\xi,\zeta)=
&\quad a_0(c'\cdot\zeta')\xi'+a_1((c'\cdot\zeta')\xi')'-3a_2((c'\cdot\zeta')\xi'')''-a_2((c'\cdot\zeta')''\xi')'\\
&+a_0(c'\cdot\xi')\zeta'+a_1((c'\cdot\xi')\zeta')'-3a_2((c'\cdot\xi')\zeta'')''-a_2((c'\cdot\xi')''\zeta')'\\
&+a_0((\xi\cdot\zeta)c')'-a_1((\xi'\cdot\zeta')c')'-3a_2((\xi''\cdot\zeta'')c')'+a_2((\xi'\cdot\zeta')''c')'
\end{aligned}
\end{gather*}
and $\FT(f)(k) =\frac{1}{\sqrt{2\pi}} \int_{\Sone} f(\theta) e^{-ik\arg\theta} \,\d\theta$, $k\in\Z$, the Fourier series transform.
\end{lemma}
\begin{proof}
From the previous expression for $D_cg_c$ one readily finds for $z\in W^{2,2}(\Sone,\R^d)$
\begin{equation*}
\Dg{c}{\zeta}{\xi}{z}+\Dg{c}{\xi}{\zeta}{z}-\Dg{c}{z}{\xi}{\zeta}=\int_{\Sone}F(c,\xi,\zeta)\cdot z\,\d\theta,
\end{equation*}
using (up to two) integration by parts, from which one can see $F(c,\xi,\zeta)\in W^{-2,2}(\Sone,\R^d)$.
By Parseval's identity and \eqref{eq:ChristoffelOperator} we thus have
\begin{multline*}
\sum_{k\in\Z}\FT[F(c,\xi,\zeta)](k)\cdot\FT[z](k)
=\int_{\Sone}F(c,\xi,\zeta)\cdot z\,\d\theta
=2\g{c}{\Christoffel{c}{\xi}{\zeta}}{z}\\
=2\sum_{k\in\Z}\FT[\Christoffel{c}{\xi}{\zeta}](k)\FT[z](k)(a_0+a_1k^2+a_2k^4).
\end{multline*}
The arbitrariness of $z$ now leads to the desired result.
\end{proof}

Furthermore, one can express the (numerator of the sectional) curvature solely in terms of the Christoffel operator and metric derivatives.

\begin{lemma}[Numerator of sectional curvature]\label{lemma:NumeratorSectionalCurvature}
Let $c\in\immersion^m(\Sone,\R^d)$ and $v, w, z, \nu \in W^{m, 2}(\Sone, \R^d)$, then
\begin{align*} \label{eq:NumeratorSectionalCurvature}
\g{c}{v}{\Rie{c}{v}{w}z}
&=\g{c}{v}{D_c\Christoffel{c}{z}{w}(v)}-\g{c}{v}{D_c\Christoffel{c}{z}{v}(w)} \\
&\qquad+\tfrac{\Dg{c}{\Christoffel{c}{z}{w}}{v}{v}}{2}
+\tfrac{\Dg{c}{v}{\Christoffel{c}{z}{v}}{w}}{2} \\
&\qquad-\tfrac{\Dg{c}{w}{\Christoffel{c}{z}{v}}{v}}{2}
-\tfrac{\Dg{c}{\Christoffel{c}{z}{v}}{w}{v}}{2}
\end{align*}
\begin{align*}
\text{with}\qquad
\g{c}{z}{D_c\Christoffel{c}{v}{w}(\nu)}
&=\tfrac{D_c^2g_c(\nu,w)(v,z)+D_c^2g_c(\nu,v)(w,z)-D_c^2g_c(\nu,z)(v,w)}2\\
&\qquad-\Dg{c}{\nu}{\Christoffel{c}{v}{w}}{z}.
\end{align*}
\end{lemma}
\begin{proof}
The expression for $\g{c}{z}{D_c\Christoffel{c}{v}{w}(\nu)}$ follows from differentiating \eqref{eq:ChristoffelOperator} with respect to $c$.
As for the Riemannian curvature, extending $z$ to a constant vector field on $\immersion(\Sone, \mathbb{R}^d)$, \eqref{eqn:RiemannCurvature} yields
	\begin{align*}
		\Rie{c}{v}{w}z
		&=\covt\left( \Christoffel{\mathbf{c}(t,0)}{z}{w} \right)(0)
		-  \covs\left(\Christoffel{\mathbf{c}(0,s)}{z}{v}\right)(0)\\
    &=  D_c\Christoffel{c}{z}{w}(v)	+ \Christoffel{c}{\Christoffel{c}{z}{w}}{v}
		-\! D_c\Christoffel{c}{z}{v}(w)-\! \Christoffel{c}{\Christoffel{c}{z}{v}}{w}.
	\end{align*}
	Inserting
	\begin{alignat*}{2}
		2 \, \g{c}{v}{\Christoffel{c}{\Christoffel{c}{z}{w}}{v}}
		&= 	\Dg{c}{v}{\Christoffel{c}{z}{w}}{v} &&- \Dg{c}{v}{\Christoffel{c}{z}{w}}{v} \\
		& &&+ \Dg{c}{\Christoffel{c}{z}{w}}{v}{v}
		\\&=  \Dg{c}{\Christoffel{c}{z}{w}}{v}{v} &&
	\end{alignat*}
	and
	\begin{align*}
		2 \, \g{c}{v}{ \Christoffel{c}{\Christoffel{c}{z}{v}}{w} }
		= 	\Dg{c}{w}{\Christoffel{c}{z}{v}}{v} &- \Dg{c}{v}{\Christoffel{c}{z}{v}}{w} \\
		&+ \Dg{c}{\Christoffel{c}{z}{v}}{w}{v}
	\end{align*}
	leads to the desired formula.
\end{proof}

With the above expressions it is straightforward to evaluate by hand (or computer algebra) the sectional curvature at $c(\theta)=(\cos\theta,\sin\theta)^T$ (which is unit speed parameterized and has a finite Fourier series transform)
in any directions $v,w\in W^{2,2}(\Sone,\R^2)$ with finitely many Fourier modes.

\begin{example}[Explicit sectional curvature]  \label{example:SectionalCurvature} 
Let $a_0 = a_1 = a_2 = 1$ and consider the curve $c(\theta)=(\cos\theta,\sin\theta)^T$, and $v(\theta)=(\cos\theta,0)^T$, $w(\theta)=(0,\cos\theta)^T$.
Then \cref{thm:Christoffel} yields
\begin{align*}
F(c,v,w)(\theta) &=
	\colvec{
		30 \cos(\theta)^2 \sin(\theta)  - 11 \sin(\theta)^3
	}{
			- 33 \cos(\theta) \sin(\theta)^2 + 8 \cos(\theta)^3
	}
	=
	\colvec{
		-\frac{3}{4} \sin(\theta) - \frac{41}{4} \sin(3\theta)
	}{
		- \frac{9}{4} \cos(\theta) +\frac{41}{4} \cos(3\theta)
	},\\
\Christoffel{c}{v}{w}(\theta) &= \colvec{
			-\frac{1}{8} \sin(\theta) + \frac{41}{728} \sin(3 \theta)
		}{
			-\frac{3}{8} \cos(\theta)+\frac{41}{728} \cos(3 \theta)
		}.
\end{align*}
Analogously,
\begin{align*}
F(c,w,w)(\theta) &=
		\colvec{
			9 \sin \left(\theta\right)^{2} \cos \left(\theta\right)
		}{
		54 \sin \left(\theta\right) \cos \left(\theta\right)^{2}-19 \sin \left(\theta\right)^{3}
		}
		,\\
		\Christoffel{c}{w}{w} &=
		\colvec{
			\frac{3 \cos \left(\theta\right)}{8}-\frac{9 \cos \left(3 \theta\right)}{728}
		}{
			-\frac{\sin \left(\theta\right)}{8}+\frac{73 \sin \left(3 \theta\right)}{728}
		}
		.
\end{align*}
	We can now compute every term in \cref{lemma:NumeratorSectionalCurvature} (each is an integral of a trigonometric polynomial) and obtain
	$\g{c}{v}{\Rie{c}{v}{w}w} = \frac{-31 \pi}{13}$.
	Furthermore, we compute $\g{c}{v}{v} = \g{c}{w}{w} = 3 \pi$ and $\g{c}{v}{w} = 0$.
	Consequently,
	\begin{align*}
		\secCurvature{c}{v}{w} = \frac{\g{c}{v}{\Rie{c}{v}{w}w}}{\g{c}{v}{v} \g{c}{w}{w} - \g{c}{v}{w}^2}
		= \frac{-31}{117\pi}.
	\end{align*}
\end{example}

Likewise, other weights can be used, as shown in the following example.
\begin{example}[Explicit Christoffel operator for weighted Sobolev metric] \label{example:ChristoffelWeighted}
	Let $a_0 = 10^{-4}$, $a_1 = 1$, $a_2 = 10^{-2}$, $c(\theta) := \colvec{\cos(\theta)}{\sin(\theta)}$, $v(\theta) := \colvec{-\frac{\cos(\theta)}{2}}{\sin(\theta)}$, and $w(\theta) := \colvec{\cos(\theta)}{\frac{-\sin(\theta)}{2}}$.
	We obtain
	\begin{align*}
			{F}(c,v,w)(\theta) &=
		\colvec{
			\frac{50901}{160000} \cos(\theta) - \frac{116081}{160000} \cos(3\theta)
		}{ 
			-\frac{113883}{160000} \sin(\theta) + \frac{87313}{160000} \sin(3 \theta)
		}, \text{ thus} \\
		\Christoffel{c}{v}{w}(\theta) 
		&=
		\frac{
			\colvec{ 
					1664479667 \cos(\theta) - 390844727 \cos(3 \theta) 
					}{
					- 3724012061 \sin(\theta) + 293982871 \sin(3\theta)
					}			
		}
		{10569794144}
		.
	\end{align*}	
	Furthermore, if we consider $w$ as constant vector field and $\w = w \circ \mathbf{c}$ as constant vector field along the path $\mathbf{c}(t) = c + t v$ then we have
	$$\frac{D}{dt}\w(0) = \frac{D_v}{dt} w(c) = \Gamma_{c}(w, v) = \Gamma_{c}(v, w).$$  
\end{example}